\documentclass[10pt,reqno]{amsart}
\numberwithin{equation}{section}
\usepackage{amsmath,amsthm}
\usepackage{mathptmx}  
\usepackage{amssymb}
\usepackage{multicol}
\usepackage{mathtools}
\DeclareFontFamily{U}{BOONDOX-calo}{\skewchar\font=45 }
\DeclareFontShape{U}{BOONDOX-calo}{m}{n}{
  <-> s*[1.05] BOONDOX-r-calo}{}
\DeclareFontShape{U}{BOONDOX-calo}{b}{n}{
  <-> s*[1.05] BOONDOX-b-calo}{}
\DeclareMathAlphabet{\mathcalboondox}{U}{BOONDOX-calo}{m}{n}
\SetMathAlphabet{\mathcalboondox}{bold}{U}{BOONDOX-calo}{b}{n}
\DeclareMathAlphabet{\mathbcalboondox}{U}{BOONDOX-calo}{b}{n}
\DeclareMathOperator{\sign}{sign}

\newcommand{\mcb}[1]{{\mathcalboondox #1}}
\usepackage{amssymb}
\usepackage{amsmath}
\usepackage{amsthm}
\usepackage{diagbox}
\usepackage{booktabs}
\usepackage{enumitem}
\usepackage[sans]{dsfont} 
\usepackage{bbm}
\usepackage{changebar}
\usepackage[colorlinks]{hyperref}
\usepackage{a4wide}
\usepackage[usenames,dvipsnames,svgnames,table]{xcolor}
\usepackage{xcolor}

\usepackage{mhequ}

\definecolor{colA}{RGB}{214, 60,  46}   
\definecolor{colB}{RGB}{52,  114, 185}  
\definecolor{colC}{RGB}{60,  160, 80}   

\usepackage{tikz}
\usetikzlibrary{arrows,shapes,automata,petri,positioning,calc}
\tikzset{
    place/.style={
        circle,
        thick,
        draw=black,
        fill=gray!50,
        minimum size=20mm,
    },
        state/.style={
        circle,
        thick,
        draw=blue!75,
        fill=blue!20,
        minimum size=20mm,
    },
}

\tikzset{
    cross/.pic = {
    \draw[rotate = 45] (-0.2,0) -- (0.2,0);
    \draw[rotate = 45] (0,-0.2) -- (0, 0.2);
    }
}

\usepackage{ulem}
\usepackage{setspace}

\newtheorem{thm}{Theorem}[section]
\newtheorem{lem}[thm]{Lemma}
\newtheorem{cor}[thm]{Corollary}

\newtheorem{prop}[thm]{Proposition}
\newcommand\supp{\mathrm{supp}}

\newtheorem{definition}[thm]{Definition}

\newtheorem{cond}{Condition}
\newtheorem{hyp}{Hypothesis}

\newtheorem{rem}[thm]{Remark}

\newcommand\cC{{\mathcal C}}

\newcommand\cS{{\mathcal S}}

\newcommand\dd{{\mathrm d}}

\newcommand\EE{{\mathbb E}}

\newcommand\NN{{\mathbb N}}

\newcommand\R{{\mathbb R}}

\newcommand\Z{{\mathbb Z}}

\newcommand\E{{\mathbb E}}

\usepackage{comment}

\newcommand{\bb}[1]{{\mathbb #1}}

\title[{ ABC model with long jumps}
]{Equilibrium fluctuations for a multi-species particle system with long jumps}
\author{Giuseppe Cannizzaro, Pedro Cardoso,  Lukas Gr{\"a}fner, Alessandra Occelli}
\newcommand{\Addresses}{{
	\footnotesize

    Giuseppe Cannizzaro, \textsc{\noindent Department of Statistics, 
    University of Warwick, Coventry CV4 7AL, UK
    }\par\nopagebreak
	\textit{E-mail address}:
    \texttt{giuseppe.cannizzaro@warwick.ac.uk}

    \medskip

	Pedro Cardoso, \textsc{\noindent  Institute of Applied Mathematics,
	University of Bonn, Endenicher Allee 60, D-53115 Bonn, Germany}\par\nopagebreak
	\textit{E-mail address}: 
    \texttt{pgondimc@uni-bonn.de}\medskip

    Lukas Gr{\"a}fner, \textsc{\noindent Department of Statistics, 
    University of Warwick, Coventry CV4 7AL, UK
    }\par\nopagebreak
	\textit{E-mail address}:
    \texttt{Lukas.Grafner@warwick.ac.uk}

    \medskip

    Alessandra Occelli, \textsc{\noindent Laboratoire Angevin de Recherche en Mathématiques,
    Université d'Angers,
    2 bouevard Lavoisier 49045 Angers, France
    }\par\nopagebreak
	\textit{E-mail address}:
    \texttt{alessandra.occelli@univ-angers.fr}

}}

\begin{document}
\begin{abstract}
In the present paper, we study the equilibrium fluctuations of a particle system in infinite volume with two conserved quantities and long-range dependence. More specifically, the model of interest is the so-called ABC model, in which three types of particles (A, B and C) exchange their locations 
between $x\in\mathbb{Z}$ and $x+z\in\mathbb{Z}$ at a rate that depends on the type of particles involved 
and is proportional to $|z|^{-\gamma-1}$ for $\gamma>0$. After rigorously identifying the normal modes associated to the 
conserved quantities (the density of particles of types $A$ and $B$, say),  
we prove that their fluctuations converge to independent fractional 
stochastic partial differential equations (SPDEs), which are either 
Gaussian or the Stochastic Burgers equation, and  
whose nature is determined by the microscopic range of dependence and 
the strength of the asymmetry.

\end{abstract}
\maketitle

\tableofcontents

\section{Introduction}

The KPZ universality class occupies a central position in the study of stochastic dynamics,
gathering a wide collection of models whose large-scale fluctuations are conjectured to be, in one
space dimension, the same as those of a Stochastic Partial Differential Equation (SPDE), the Kardar--Parisi--Zhang (KPZ) \cite{KPZ} equation
\begin{equation}\label{e:KPZ}
\partial_t h = \tfrac{1}{2}(\nabla h)^2 + \Delta h + \xi,
\end{equation}
where $\xi$ denotes a space-time white noise. A rigorous derivation of KPZ-type behaviour
from microscopic particle systems has been a major challenge in probability theory for the
past three decades. The programme initiated in \cite{GJburgers,asymjara} has established
that weakly asymmetric exclusion-type dynamics, under an appropriate fluctuation scaling,
converge to stationary energy solutions of the Stochastic Burgers Equation (SBE), which
is the spatial derivative of the KPZ equation. This line of work, however, has largely
focused on models possessing a single conserved quantity (the particle density), the simple
exclusion process being the canonical example. The present paper extends this programme
in two directions simultaneously: to \textit{long-range}
dynamics, and to a particle system with \textit{two} conserved quantities.

More precisely, we study the so-called ABC model, which is a three-species exclusion process on $\mathbb{Z}$ in which each site is
occupied by exactly one particle of type $A$, $B$, or $C$; species exchange positions at
rates that depend only on the types involved. Both the density of $A$-particles and the
density of $B$-particles are conserved (the $C$-density is then determined by
complementarity), so the system possesses two independent conserved quantities. The
nearest-neighbour periodic version of this model was analysed in \cite{ABC}. The key insight of \cite{ABC} is that, by diagonalising
the pair of fluctuation fields $(\mathcal{Y}^{n,A}, \mathcal{Y}^{n,B})$ via Spohn's theory of Nonlinear Fluctuations Hydrodynamics~\cite{Spohn}, one obtains two scalar fields
$\mathcal{Z}^{n,+}$ and $\mathcal{Z}^{n,-}$, the {\it normal fields}, that (almost) decouple asymptotically and, in the limit, satisfy SPDEs driven by independent noises each of which only depends on the ``total mass'' of the other. This latter spurious feature is inherently due to periodicity and should disappear in the infinite volume setting. The nature of these limiting equations - whether
they are Ornstein--Uhlenbeck (Gaussian) or Burgers/KPZ type - depends on the strength of the asymmetry and on the specific linear combination defining the normal fields. 

The present paper replaces nearest-neighbour interaction 
with a heavy-tailed jump kernel and considers the system in 
infinite volume. In mathematical terms, this amounts to study 
the ABC model on $\mathbb{Z}$ and to allow particles at 
positions $x\in\mathbb{Z}$ and $x+z$, $z\in\mathbb{Z}$, 
to exchange at a rate proportional to $|z|^{-1-\gamma}$ for $\gamma>0$. 
Such long-range dependence introduces a 
new competition between diffusion and asymmetry. Indeed, 
the symmetric part of the dynamics
gives rise to the fractional Laplacian $(-\Delta)^{\gamma/2}$ for $\gamma\in(0,2)$, and to
the ordinary Laplacian for $\gamma\ge 2$, both 
at the macroscopic (hydrodynamic) scale and at the mesoscopic 
(fluctuation) scale (where further one obtains a fractional noise). 
The latter was studied, in the symmetric case, in \cite{jarafluc}, 
which identified the fractional Ornstein--Uhlenbeck (OU) process as the
universal Gaussian limit, and, in the asymmetric case, in \cite{GJburgers}, 
where instead, for $\gamma\in(3/2,2)$, the fluctuations field 
was shown to converge
to a stationary energy solution of a fractional SBE. Let us also mention the papers \cite{CGJ2,symflucped} that address the symmetric and slow-barrier settings. A common thread in all these works is
that the nature of the limiting SPDE is sensitive to the exponent $\gamma$ 
in that both the relevant time scale and the limiting operators 
depend on whether 
$\gamma$ is below, equal to, or above $2$.
\medskip

\noindent\textbf{Main results and novelties.} In the  
ABC model we consider the exchange rates between particles of types $\alpha$ and $\beta$, $\alpha,\beta\in\{A,B,C\}$, are of the form $p(\cdot)r^n_{\alpha,\beta}$, where $p(z)$ is proportional to $|z|^{-\gamma-1}$ (see \eqref{prob}) and $r^n_{\alpha,\beta}=1+K_n(E_\alpha-E_\beta)$ (see \eqref{ralbet}) for  $E_\alpha\in\mathbb{R}$ 
and $(K_n)_{ n \in \mathbb{N}_+ }$ a bounded sequence converging to some $K\in[0,\infty)$. {Here we adopt the notation $\mathbb{N}_+:=\{1, 2, \ldots \}$.} 
The sequence $(K_n)_{ n \in \mathbb{N}_+ }$ is responsible for tuning 
the strength of the asymmetry and we will choose it according to 
the time scale $\Theta(n)=n^{\gamma\wedge2}$ for $\gamma\neq 2$ and $n^2/\log n$ for $\gamma=2$ (see~\eqref{timescale}) 
at which the system is run\footnote{This is the natural choice for the symmetric part of the dynamics to give rise to a meaningful limit. }. 
In this respect, we consider two settings for $K_n$: Hypothesis~\ref{hyplin}, that covers all the regimes in which the asymmetric part of the dynamics is either absent
or too weak to produce an effect at the fluctuation scale, and Hypothesis~\ref{hypnonlin}, where the asymmetric part plays a non-trivial role. 

As done in~\cite{ABC}, at first we identify the two diagonalised fluctuation fields $\mathcal{Z}^{n,\pm}_t$ (see \eqref{defZtnHpm}), 
given by a precise linear combination of the standard fluctuations fields 
$\mathcal{Y}^{n,A}_t$ and $\mathcal{Y}^{n,B}_t$ associated 
the particle densities of $A$ and $B$.  
Our two main results are Theorems~\ref{clt} and~\ref{clt2}. 
The former shows that, under Hypothesis~\ref{hyplin}, for any value of $\gamma>0$, the sequence
$\{(\mathcal{Z}^{n,+}_t, \mathcal{Z}^{n,-}_t)_{0\le t\le T}: n\in\mathbb{N}_{  +  } \}$ converges
in distribution to a pair of \textit{independent} stationary 
solutions of the fractional
Ornstein--Uhlenbeck (OU) equations of the form
\[
\dd \mathcal{Z}^\pm_t = (\mathcal{L}^\gamma + \hat{\mathcal{L}}^\gamma_{\lambda^\pm})\mathcal{Z}^\pm_t \; \dd t
+ \sqrt{2D^\pm_3}\,\mathcal{P}_\gamma \; \dd W^\pm_t,
\]
where the operators $\mathcal{L}^\gamma$, $\mathcal{P}_\gamma$ and $\hat{\mathcal{L}}^\gamma_{\lambda^\pm}$ should respectively be thought of 
as the fractional Laplacian of index $\gamma$, its square-root 
and the fractional derivative (see~\eqref{defLlamb},~\eqref{defPgamH} and~\eqref{defLhatlamb} for precise definitions) and 
$D^\pm_3$ are the corresponding diffusion coefficient given by \eqref{defD3}. In particular, the large-scale behaviour of
$(\mathcal{Z}^{n,+}, \mathcal{Z}^{n,-})_{n \in \mathbb{N}_+}$ is Gaussian. 
Under Hypothesis~\ref{hypnonlin} instead, the effect of the 
asymmetry becomes tangible and KPZ-type behaviour emerges. 
Theorem~\ref{clt2} states that the same two fields 
converge to a pair of {\it independent} stationary energy solutions of fractional SBEs formally given by 
\[
\dd \mathcal{Z}^\pm_t = (\mathcal{L}^\gamma + \hat{\mathcal{L}}^\gamma_{\lambda^\pm})\mathcal{Z}^\pm_t\ \; \dd t
+ \kappa^\pm_\gamma\nabla(\mathcal{Z}^\pm_t)^2 \; \dd t+ \sqrt{2D^\pm_3}\,\mathcal{P}_\gamma \; \dd W^\pm_t,
\]
whose coupling constants $\kappa^\pm_\gamma$ are explicit (see \eqref{defkgampm}). This
regime is limited to $\gamma\in [3/2,2)\cup(2,\infty)$.
\medskip

These results constitute a complete long-range generalisation of Theorem~2.8 of \cite{ABC} and there are 
several features which are worth emphasising. 
\begin{itemize}

\item Compared to~\cite{ABC}, the limiting fields are {\it genuinely} 
independent and there is no spurious dependence of one on the mass of the other. This is due to the fact that we work in infinite volume and the two fields evolve at different velocities. 
Not only this is meaningful from the physics perspective as it 
confirms the original prediction~\cite{prediction}, but also from the mathematics viewpoint, as it forced us to develop a more robust approach compared to 
that in~\cite{ABC} to control the crossed fields (see Section~\ref{sec:Cross}). 
We believe this is of independent interest as it 
could be applied to other systems with 
multiple conserved quantities. 

\item The coupling constants $\kappa^+_\gamma$ and/or $\kappa^-_\gamma$
vanish, so that the corresponding field satisfies an OU equation 
even under Hypothesis~\ref{hypnonlin}, if
and only if $(E_B-E_A)(E_B-E_C)=0$ (see Remark~\ref{remgam32} and the discussion
following it). This corresponds precisely to Cases~(I) and~(II) of \cite{ABC}, and
generalises their analysis to the long-range setting. In the generic Case~(III), both fields
exhibit KPZ behaviour simultaneously, a feature that is structurally robust under
long-range perturbations.

\item The operator $\hat{\mathcal{L}}^\gamma_\lambda$ in
\eqref{defLhatlamb} is a fractional antisymmetric operator that vanishes identically for
$\gamma\ge 2$; it encodes the macroscopic drift generated by the asymmetric part of the
jump kernel in the sub-diffusive regime $\gamma\in(0,2)$. Its presence 
is specific of the long-range setting and thus missing in~\cite{ABC}. 
Further, since the presence of operators of this type had not been considered in the SPDE literature, we prove uniqueness of the corresponding SPDE 
by adapting the framework of \cite{GPP}. Even though the adaptation is 
relatively straightforward, we detail in Appendix \ref{app:Unique} 
how the proof works as it might be useful in other contexts. 

\item Under Hypothesis~\ref{hypnonlin} and $\gamma=3/2$, 
the limiting equation is a fractional SBE which is {\it scale-invariant}, 
in the sense that it is invariant under a suitable superdiffusive scaling. 
Equations of this type are also called {\it scaling-critical} 
in the sense that recent pathwise techniques, as Hairer's theory of regularity structures~\cite{Hai} 
or Gubinelli--Imkeller--Perkowski's paracontrolled calculus approach \cite{GubinelliImkellerPerkowski2015}, are not applicable 
and cannot provide a solution theory for it. In particular, the methods 
explored in the preprint~\cite{huang2025scalinglimitweaklyasymmetric} 
would not apply in this context.    
\end{itemize}

We conclude the introduction by summarising the strategy of proof 
and highlighting some of the technical difficulties we had to overcome. 
\medskip 

\noindent\textbf{Strategy of the proof.}
The overall scheme follows the usual tightness-plus-limit characterisation approach of~\cite{ABC,asymjara,jarafluc}. The starting point is Dynkin's formula applied to $\mathcal{Z}^n(H)$, for $\mathcal{Z}^n$ either $\mathcal{Z}^{\pm,n}$ and $H$ a smooth compactly supported test function, which yields the
martingale decomposition \eqref{decompZ}:
\[
\mathcal{Z}^n_t(H) = \mathcal{Z}^n_0(H) +  \mathcal{M}^n_t(H) + \mcb B^n_t(H) +  \mcb I^n_t(H).
\]
where $\mathcal M^n_t(H)$ is the Dynkin martingale, 
$\mcb I^n_t(H)$ is the linear term, i.e. the time 
integral of $\mathcal{Z}^n$ tested against some discrete
operator acting on $H$, and $\mcb B^n_t(H)$ is 
the non-linear term (see \eqref{defBtn}--\eqref{defBtnalbe}) 
which is due to the current and captures the interaction
between species (and of a specie with itself).  
The most delicate part of the argument concerns this latter term. 
Its behaviour distinguishes the two regimes of Hypotheses~\ref{hyplin} and~\ref{hypnonlin}.
Under Hypothesis~\ref{hyplin}, Lemma~\ref{varnlinvan} shows that $\mcb B^n_t(H)$ converges
to zero in $ \mathcal{C} ([0,T],\mathbb{R})$, so the limit is necessarily Gaussian (OU regime).
Under Hypothesis~\ref{hypnonlin}, Lemma~\ref{varnlintight} instead establishes its 
tightness. In the regime $\gamma<2$, we move away from 
the technical multiscale analysis described in Section 4.4.3 of \cite{jarafluc}, and show that it is possible to make use 
of only two steps namely the \textit{one-block} and \textit{two-block} estimates (see Lemmas \ref{lemobe} and \ref{lemtbe}). 

The characterisation of limit points (Section~\ref{seccharac}) 
amounts to verify that every limit point (that exists by tightness) 
is a solution 
to the martingale problems associated to either the fractional OU or 
the fractional SBE. Since these martingale problems are well-posed (see Proposition~\ref{uniqlaw}, which is novel for $\gamma=3/2$), 
the conclusion follows. 
Once again the trickiest point 
is to identify the limit of the nonlinear term $\mcb B^n_t(H)$. 
This relies on a sequence of approximations that allows to rewrite 
it in terms of products of the form $\mathcal{Z}^{n,\pm}_r(f_\varepsilon^{w/n})\cdot\mathcal{Z}^{n,\pm}_r(g_\varepsilon^{w/n})$, see \eqref{KPZbeh} for more details. Cross-terms involving $\mathcal{Z}^{n,+}\cdot\mathcal{Z}^{n,-}$ are shown to vanish in
the limit using Theorem~\ref{thm:Cross}. As mentioned above, 
this is a crucial new ingredient
with respect to \cite{ABC}: working in infinite volume (rather than on the discrete torus)
forces us to replace the Fourier basis used in~\cite{ABC} by a wavelet orthonormal basis of $L^2(\mathbb{R})$ \cite{Dau1,Dau} with
suitable regularity and decay properties, and makes the overall analysis 
more challenging. 
\medskip

\noindent\textbf{Organisation of the paper.}
In Section \ref{secABC}, we provide a detailed description of 
the ABC model of interest, present the solution theory of the limiting 
SPDEs and state the main results.  
Section \ref{sectight} is dedicated to the tightness. The 
two crucial lemmas, Lemma~\ref{varnlinvan} (vanishing under
Hypothesis~\ref{hyplin}) and Lemma~\ref{varnlintight} (tightness under
Hypothesis~\ref{hypnonlin}) can be found herein, but 
their proof is presented in Section~\ref{secvarnlinvan}. In Section~\ref{seccharac}, we complete the proof of the main results,  Theorems~\ref{clt}
and~\ref{clt2}. Section~\ref{sec:Cross} states and proves the crossed fields
Theorem~\ref{thm:Cross}, using the wavelet theory collected in
Appendix~\ref{a:Wavelets}. In Appendix \ref{appa}, we present some  algebraic identities regarding the generator of our Markov process. Finally, in Appendix \ref{useest} and \ref{appc}, we collect useful estimates 
and technical results that are used in the rest of the paper. At last, Appendix~\ref{app:Unique} sketches the proof of uniqueness
for the fractional SBE with $\gamma=3/2$ (Proposition~\ref{uniqlaw}). 

\section{The ABC model} \label{secABC}

In this work we study the evolution of a three-species particle system, {namely an exclusion process allowing at most one particle per site, that can be either of type $A$, or type $B$, or type $C$. As in \cite{ABC}, we call it the ABC model.} Let $S = \{A, B, C\}$ be the set of possible species. The space of configurations is $\Omega:=S^{\bb Z}$, whose elements are denoted by $\eta$. Thus $\eta(x) \in S$, for any $x \in \bb Z$ {and thus represents the type of particle allocated to the position $x$}. Moreover, given $\eta \in \Omega$ and $x, y \in \bb Z$, the new configuration $\eta^{x,y}$ obtained after exchanging the particles at $x$ and $y$ is given by
\begin{equation}
\eta^{x,y}(z)=
\begin{cases}
\eta(y), \quad z=x, \\
\eta(x), \quad z=y, \\
\eta(z), \quad z \notin \{x, y \}.
\end{cases}
\end{equation}
Let $n \geq 1$ and $r^n_{A, B}, r^n_{B, C}, r^n_{C, A}, r^n_{B, A}, r^n_{C, B}, r^n_{A, C} \in [0, \; \infty)$ be the rates of exchanges, which may depend on $n$. More exactly, if $x \neq y \in \mathbb{Z}$ and $\eta(x) \neq \eta(y)$, the rate of an exchange of particles between $x$ and $y$ is $r^n_{\alpha, \beta}$, where $\alpha=\eta(x)$ and $\beta=\eta(y)$. Next, for any $\alpha \in S$ and $x \in \mathbb{Z}$, define $\xi^{\alpha}_x: \Omega \rightarrow \{0, \; 1\}$ by
\begin{align} \label{defxialp}
\forall \eta \in \Omega, \quad \xi_x^{\alpha}(\eta):= \mathbbm{1}_{ \{ \alpha \} } \big( \eta(x) \big) =
\begin{cases}
1, \quad \eta(x)=\alpha, \\
0, \quad \eta(x) \neq \alpha.
\end{cases}
\end{align}
Thus, {$\xi^{\alpha}_x$ is the variable that dictates if there is a particle at site $x$ of type $\alpha$. Also} if $x\neq y$, $\eta\in\Omega$ and $\eta(x)\neq \eta(y)$, an exchange of particles between $x$ and $y$ occurs with rate
$r^n_
x,y(\eta)$. {In order to include all the possible exchanges of $\eta(x) \neq \eta(y)$, we define then  $r^n_{x,y}(\eta)$ by}
\begin{align} \label{rateABC}
r^n_{x,y}(\eta):= \sum_{\alpha \neq \beta \in S} r^n_{\alpha, \; \beta} \xi_x^{\alpha} (\eta) \xi_y^{\beta} (\eta).
\end{align}
Last display is analogous to equation (2.2) in \cite{ABC} and can be rewritten as
\begin{align*}
r^n_{x,y}(\eta) =& r^n_{A, B} \xi_x^{A} (\eta) \xi_y^{B} (\eta) +  r^n_{A, C} \xi_x^{A} (\eta) \xi_y^{C} (\eta) + r^n_{B, C} \xi_x^{B} (\eta) \xi_y^{C} (\eta)   \\
 + &  r^n_{B, A} \xi_x^{B} (\eta) \xi_y^{A} (\eta)  +  r^n_{C, A} \xi_x^{C} (\eta) \xi_y^{A} (\eta) + r^n_{C, B} \xi_x^{C} (\eta) \xi_y^{B} (\eta). 
\end{align*}
Therefore, for a local function $f: \Omega \rightarrow \bb R$, the infinitesimal generator of the ABC model acts on $f$ as 
\begin{equation} \label{genABC}
(\mcb L^n f) (\eta):= \Theta(n) \sum_{x,y} p(y-x) r^n_{x,y}(\eta) [f(\eta^{x,y}) - f(\eta)  ],
\end{equation}
where $\Theta(n)$ is the time scale according to which we speed up our Markov process (in order to observe a {non-trivial macroscopic behavior}), given by \eqref{timescale}; and $p(\cdot)$ is a probability measure on $\mathbb{Z}$. In \eqref{genABC} and in what follows, unless stated otherwise, the discrete variables in a summation range over $\mathbb{Z}$.
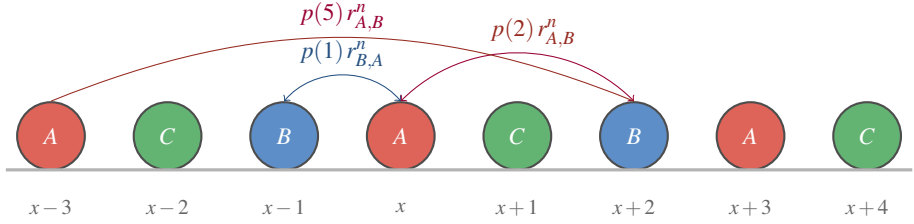
\begin{figure}[h]
\centering
\begin{tikzpicture}[
    =>,
    site/.style={
        circle, thick,
        draw=black!70,
        minimum size=9mm,
        inner sep=0pt,
        font=\bfseries\small
    },
    partA/.style={site, fill=colA!80, text=white},
    partB/.style={site, fill=colB!80, text=white},
    partC/.style={site, fill=colC!80, text=white},
    ratelbl/.style={font=\small, inner sep=2pt},
]

\def\dx{1.55} 

\node[partA] (s1) at (0*\dx, 0)      {$A$};
\node[partC] (s2) at (1*\dx, 0)      {$C$};
\node[partB] (s3) at (2*\dx, 0)      {$B$};
\node[partA] (s4) at (3*\dx, 0)      {$A$};
\node[partC] (s5) at (4*\dx, 0)      {$C$};
\node[partB] (s6) at (5*\dx, 0)      {$B$};
\node[partA] (s7) at (6*\dx, 0)      {$A$};
\node[partC] (s8) at (7*\dx, 0)      {$C$};

\draw[black!30, very thick] (-0.6, -0.45) -- (7*\dx+0.6, -0.45);

\foreach \i/\lbl in {1/{x-3}, 2/{x-2}, 3/{x-1}, 4/{x}, 5/{x+1}, 6/{x+2}, 7/{x+3}, 8/{x+4}}{
    \node[font=\footnotesize, text=black!60] at (\i*\dx - \dx, -0.95) {$\lbl$};
}

\draw[<->, colB!70!black, bend left=50]
    (s3.north) to
    node[above, ratelbl, colB!70!black]
        {$p(1)\,r^n_{B,A}$}
    (s4.north);

\draw[<->, purple!80!black, bend left=45]
    (s4.north) to
    node[above, ratelbl, colA!70!black, xshift=6pt]
        {$p(2)\,r^n_{A,B}$}
    (s6.north);

 \draw[colA!70!black, bend left=22]
    (s1.north) to
    node[above, ratelbl, purple!80!black]
        {$p(5)\,r^n_{A,B}$}
    (s6.north);


\end{tikzpicture}
\caption{The ABC model with long jumps on $\mathbb{Z}$.
Each site $x\in\mathbb{Z}$ carries exactly one particle of type $A$ (red), $B$ (blue) or
$C$ (green). Two particles at sites $x$ and $y$ with distinct types $\alpha=\eta(x)$,
$\beta=\eta(y)$ exchange at rate $\Theta(n)\,p(y-x)\,r^n_{\alpha,\beta}$, where
$p(z)\sim c^\pm|z|^{-1-\gamma}$ is the long-range jump kernel \eqref{genABC} and
$r^n_{\alpha,\beta}=1+K_n(E_\alpha-E_\beta)$ are the species-dependent rates \eqref{rateABC}.}
\label{fig:dynamics}
\end{figure}

\begin{rem}
In \cite{ABC}, $\Theta(n)=n^{2}$ and $p: \bb Z \rightarrow [0, 1]$ is given by $p(\cdot)=\mathbbm{1}_{\{1\}}(\cdot)$, which assigns weight $1$ to jumps of size one (i.e. the transition probability of the TASEP). 
\end{rem}
Next, as in \cite{ABC} we impose that the rates will satisfy the pairwise balance condition, so that we are able to identify a class of invariant measures for the generator $\mcb L^n$.
\begin{cond}
For every $n \in \mathbb{N}$, the rates $r^n_{\alpha, \beta}, \alpha \neq \beta \in S$, satisfy
\begin{equation} \label{pairbal}
r^n_{A,B} + r^n_{B,C} + r^n_{C,A} = r^n_{B,A} + r^n_{C,B} + r^n_{A,C}.
\end{equation}
\end{cond}
{As a consequence of the previous condition we get:}
\begin{prop} \label{propABCinv}
Let $\nu$ be a probability measure on $\Omega$ satisfying 
\begin{equation} \label{invtrans}
\forall x,y \in \bb Z, \; \forall \eta \in \Omega, \quad \nu(\eta^{x,y})=\nu(\eta).
\end{equation}
Then, under condition \eqref{pairbal}, the measure $\nu$ is invariant with respect to the dynamics given by $\mcb L^n$. 
\end{prop}

\begin{proof}
Since $\nu$ satisfies \eqref{invtrans}, we get from \eqref{genABC} then $\nu$ is invariant with respect to $\mcb L^n$ if, and only if,
\begin{align} \label{condinv}
\forall \eta_1 \in \Omega, \quad \nu(\eta_1) \sum_{\eta_2 \neq \eta_1} \mcb L^n( \eta_2, \eta_1) - \nu(\eta_1) \sum_{\eta_2 \neq \eta_1}   \mcb L^n( \eta_1, \eta_2) =0.
\end{align}
Now for every $\eta \in \Omega$ and $\alpha \neq \beta \in S$, define $K_{\alpha,\beta}(\eta)$ and $K_{\alpha}(\eta)$ by
\begin{align*}
K_{\alpha,\beta}(\eta):= &   \sum_{x,y} p(y-x) \xi_x^{\alpha} (\eta) \xi_y^{\beta} (\eta)  \nu (\eta), \quad K_{\alpha}(\eta):=    \sum_{x,y} p(y-x) \xi_x^{\alpha} (\eta)   \nu (\eta).
\end{align*}
In particular, for every $\eta \in \Omega$ and $\alpha \in S$, it holds
\begin{align*}
K_{\alpha}(\eta) = \sum_{\beta \in S} K_{\alpha,\beta}(\eta) = \sum_{\beta \in S} K_{\beta,\alpha}(\eta). 
\end{align*}
Thus, we conclude that for every $\eta \in \Omega$, it holds
\begin{equation} \label{nab-nba}
K_{A, B}(\eta) - K_{B, A}(\eta) = K_{B, C}(\eta) - K_{C, B}(\eta) =  K_{C, A}(\eta) - K_{A, C}(\eta).
\end{equation}
Now, keeping \eqref{condinv} in mind, we observe that for every $\eta_1 \in \Omega$, it holds
\begin{align*}
\nu(\eta_1) \sum_{\eta_2 \neq \eta_1} \mcb L^n( \eta_1, \eta_2) =& \sum_{\alpha \neq \beta \in S} K_{\alpha,\beta}(\eta_1) r^n_{\alpha,\beta}, \quad \nu(\eta_1) \sum_{\eta_2 \neq \eta_1} \mcb L^n( \eta_2, \eta_1) = \sum_{\alpha \neq \beta \in S} K_{\alpha,\beta}(\eta_1) r^n_{\beta,\alpha}.
\end{align*}
Therefore, due to \eqref{nab-nba} the display in \eqref{condinv} can be rewritten as
\begin{align*}
\big[ K_{B, A}(\eta_1) - K_{A, B}(\eta_1) \big] \big\{  \big( r^n_{A,B} - r^n_{B,A} \big) + \big( r^n_{B,C} - r^n_{C,B} \big) + \big( r^n_{C,A} - r^n_{A,C} \big)  \big\} =0.
\end{align*}
The last equality holds for any $\eta_1 \in \Omega$ if and only if $\big( r^n_{A,B} - r^n_{B,A} \big) + \big( r^n_{B,C} - r^n_{C,B} \big) + \big( r^n_{C,A} - r^n_{A,C} \big) =0$, which is equivalent to \eqref{pairbal} and  this ends the proof.
\end{proof}

\subsection{Choice for the rates and the transition probability}

{Now we fix a choice for the jump rates. To that end,} we fix $E_A, E_B, E_C \in \mathbb{R}$ such that
\begin{equation} \label{ralbet}
\forall n \in \mathbb{N}_+, \; \forall \alpha \neq \beta \in S, \quad r^n_{\alpha,\beta}:= 1 + K_n (E_{\alpha}-E_{\beta}),
\end{equation}
where $(K_n)_{n \in \mathbb{N}_+} \subset (0, \infty)$ is a sequence such that
\begin{equation} \label{limKn}
\lim_{n \rightarrow \infty} K_n = K \in [0, \infty).
\end{equation}
In order to ensure that the display in \eqref{ralbet} is positive for $n$ large enough, we will assume for the remainder of this work that
\begin{align*}
2 K \max \{| E_{\alpha}-E_{\beta}|: \alpha \neq \beta \in S \} < 1. 
\end{align*}
Therefore,  from \eqref{rateABC}, there exists a positive constant $C_0$ such that for every $\eta \in \Omega$ and every $x,y  \in \mathbb{Z}$ satisfying $\eta \neq \eta^{x,y}$, it holds
\begin{equation} \label{lbrate}
r^n_{x,y}(\eta) \geq C_0, 
\end{equation} 
whenever $n$ is large enough. {The role of $K_n$ corresponds to the one of $(2 N^{\gamma})^{-1}$ in the end of page 8 in \cite{ABC}. 

In order to observe {different macroscopic behaviors} of our particle systems, in this work we will make use of (asymmetric) long jumps. More exactly, we fix $\gamma >0$ and 
define $p: \bb Z \rightarrow [0, 1]$ by
\begin{equation} \label{prob}
\forall z \in \mathbb{Z}, \quad	p(z):= \frac{\mathbbm{1}_{ \{z > 0 \} } c^{+}   }{|z|^{\gamma+1}} + \frac{\mathbbm{1}_{ \{z < 0 \} } c^{-}   }{|z|^{\gamma+1}}.
\end{equation}
In \eqref{prob}, $c^{+},c^{-}$ are non-negative constants such that $p(\cdot)$ is a probability measure on $\bb Z$. 

We denote by $(\eta_t^n)_{t \geq 0}$ the Markov process whose generator is $\Theta(n) \mcb L^n$ with the choice of the time scale depending on the value of $\gamma$ as: 
{\begin{equation} \label{timescale}
\forall n \in \mathbb{N}_+, \quad \Theta(n):=
\begin{dcases}
n^{\gamma\wedge2}, \quad & \gamma\neq 2; \\
n^{2} /\log(n), \quad & \gamma =2; 
\end{dcases}
\end{equation}}
\subsection{{Invariant measures}}

Fix $\rho:=(\rho_A, \; \rho_B, \; \rho_C  ) \in [0, \; 1]^3$ such that $\rho_C=1 - \rho_A - \rho_B$. Define the product measure $\nu_{\rho}$ on $\Omega$, whose marginals satisfy
\begin{align} \label{invprod}
\forall x \in \mathbb{Z}, \; \forall \alpha \in S, \quad  \nu_{\rho} \big(\eta: \eta(x) = \alpha \big) = \rho_{\alpha}.
\end{align} 
We claim that $\nu_{\rho}$ satisfies \eqref{invtrans}. Indeed, let $\eta \in \Omega$ and $x,y \in \mathbb{Z}$. Moreover, denote $\alpha:=\eta(x)$ and $\beta:=\eta(y)$. Thus,
\begin{align*}
\frac{\nu_{\rho}(\eta^{x,y})}{\nu_{\rho}(\eta)} = \frac{ \nu_{\rho} \big(\eta: \eta(x) = \beta \big) \nu_{\rho} \big(\eta: \eta(y) = \alpha \big) }{\nu_{\rho} \big(\eta: \eta(x) = \alpha \big) \nu_{\rho} \big(\eta: \eta(y) = \beta \big)} = \frac{\rho_{\beta} \rho_{\alpha}}{ \rho_{\alpha} \rho_{\beta} } = 1,
\end{align*}
which leads to \eqref{invtrans}, and the claim is proved. 

In what follows, we fix a finite time horizon $T$. Given a metric space $(N, \| \cdot \|_N)$, we denote by $\mathcal{D}([0,T], N)$ the space of the càdlàg (right-continuous and with left limits) trajectories. Furthermore, we denote by $\mathbb{P}_{\nu_{\rho}}$ the probability on $\mathcal{D}([0,T], \Omega)$ induced by the Markov process with initial measure $\nu_{\rho}$ and generator $\Theta (n) \mcb L^n f$. Moreover, we denote by $\mathbb{E}_{\nu_{\rho}}$ the expectation with respect to $\mathbb{P}_{\nu_{\rho}}$. 

From Proposition \ref{propABCinv}, $\nu_{\rho}$ is invariant with respect to the dynamics given by $\mcb L^n$, thus
\begin{equation} \label{expinv}
\forall x \in \bb Z, \; \forall t \in [0, T], \; \forall \alpha \in S, \quad \E_{\nu_{\rho}} [ \xi_x^{\alpha}( \eta_t^n ) ] = \rho_{\alpha}.
\end{equation}
\begin{rem} \label{remind}
For every $t \in [0,T]$, the random variables $\{\xi_x^{A}( \eta_t^n ), \xi_x^{B}( \eta_t^n ), \; x \in \bb Z \}$ are independent under $\nu_{\rho}$. In particular, by denoting $\chi(u):=u(1-u)$ for any $u \in [0, 1]$, it holds
\begin{equation} \label{corrzero}
	\forall t \in [0,T], \; \forall x,y \in \mathbb{Z}, \; \forall \alpha \in S, \quad \E_{\nu_{\rho}} \big[ \overline{\xi}_x^{\alpha}( \eta_t^n ) \; \overline{\xi}_y^{\alpha}( \eta_t^n ) \big] = \chi(\rho_{\alpha}) \mathbbm{1}_{ \{ x = y \} },
\end{equation}
where $\overline{\xi}_x^{\alpha}( \eta ):=\xi_x^{\alpha}( \eta ) - \rho_{\alpha}$, for every $\eta ´\in \Omega$, every $x \in \bb Z$ and every $\alpha \in S$. 
\end{rem}

\subsection{{Density fluctuation field}}
Given $\alpha \in S$, we define the density fluctuation field $\mathcal{Y}_t^{n, \alpha}$ as
\begin{align} \label{defYalp}
\forall n \in \mathbb{N}_+, \; \forall H \in L^{2}(\mathbb{R}),\; \forall t \in [0,T], \quad	\mathcal{Y}_t^{n, \alpha}(H):= \frac{ 1 }{\sqrt{n}} \sum_{x } H ( \tfrac{x}{n}  ) \overline{\xi}_x^{\alpha}( \eta_t ).  
\end{align} 
Before we proceed, we observe that for every $n \in \mathbb{N}_+$ and $t \in [0,T]$, 	$\E_{\nu_{\rho}} \big[   \mathcal{Y}_t^{n, \alpha}(  H )  \big]=0$ and
\begin{align}  \label{bndL2alpf}
	\E_{\nu_{\rho}} \big[  \big\{  \mathcal{Y}_t^{n, \alpha}(  H ) \big\} ^{2} \big] 
	=
	   \chi(\rho_{\alpha})   \| H \|^{2}_{2,n},  
\end{align}
where the equality comes from \eqref{corrzero}. In the last line, $\| \cdot \|_{2,n}: L^{2}(\mathbb{R}) \mapsto \mathbb{R}$ is given by
\begin{equation} \label{L2discnorm}
\forall n \in \mathbb{N}_+, \; \forall H \in L^{2}(\mathbb{R}), \quad \| H \|^{2}_{2,n} := \frac{1}{n} \sum_{ x }  [  G ( \tfrac{x  }{n}  )   ]^2 \leq C(H) \| H \|_{L^{2}(\mathbb{R})}^{2},
\end{equation}
where $C(H)$ is some constant depending only on $H$. Now, since $\xi_x^{C}( \eta )=1 - \xi_x^{A}( \eta ) - \xi_x^{B}( \eta )$ for any $x \in \mathbb{Z}$ and $\eta \in \Omega$, the field $\mathcal{Y}_t^{n, C}$ is completely determined by $\mathcal{Y}_t^{n, A}$ and $\mathcal{Y}_t^{n, B}$, thus in this work we focus on the species $A$ and $B$. More exactly, motivated by the results in \cite{ABC}, we will analyze $\mathcal{Z}_t^{n}$, a special transformation of $\mathcal{Y}_t^{n, A}$ and $\mathcal{Y}_t^{n, B}$.  

First, for any $u \in \bb R$ and any $n \in \mathbb{N}_+$ \textit{fixed}, we denote the translation operator $T^n_{u}:L^{2}(\mathbb{R}) \mapsto L^{2}(\mathbb{R})$ by
\begin{align} \label{translop}
\forall v \in \mathbb{R}, \; \forall H \in L^{2}(\mathbb{R}), \quad	(T^n_{u} H ) ( v  ) := H ( v- \tfrac{u}{n}  ).
\end{align}
{Our goal is to obtain a system of SPDEs governing the evolution of the conserved quantities of the system. As any linear combination of the conserved quantities is again conserved, we consider then} a generic field $\mathcal{Z}_t^n$ given in terms of $\mathcal{Y}_t^{n, A}$ and $\mathcal{Y}_t^{n, B}$ as: 
\begin{equation} \label{defZtnH}
\forall t \in [0,T], \; \forall H \in 	L^{2}(\mathbb{R}), \quad \mathcal{Z}_t^n(H):= D_1 \mathcal{Y}_t^{n, A} ( T^n_{ t v_n  } H ) + D_2 \mathcal{Y}_t^{n, B} ( T^n_{t v_n } H ),
\end{equation}
where $D_1, D_2$ and $v: \mathbb{N}_+ \mapsto \mathbb{R}$ will be fixed later on.   In \eqref{defZtnH} and in the remainder of this article, $v_n:=v(n)$.  In order to avoid a heavy notation, we will omit the dependence {of the field $\mathcal{Z}_t^n$} on $D_1$, $D_2$ and $v$.

Now from Remark \ref{remind} and \eqref{defZtnH}, we get that
\begin{equation} \label{1bndL2Z}
\forall H \in L^{2}(\mathbb{R}), \quad \lim_{n \rightarrow \infty} \Big\{ \sup_{t  \in [0,T]} \big\{  \E_{\nu_{\rho}}  \big[   | \mathcal{Z}_t^n(H) |^{2} \big] \} \Big\} = D_3 \| H \|_{L^{2}(\mathbb{R})}^{2},
\end{equation}
where $D_3=D_3(D_1, D_2, \rho_A, \rho_B)$ is given by
	\begin{equation} \label{defD3}
			D_3 :=  (D_1)^2 \chi(\rho_{A}) + (D_2)^2 \chi(\rho_{B}) - 2 D_1 D_2 \rho_A \rho_B.
	\end{equation}	

We denote the Schwarz space by $\mathcal{S}(\mathbb  R)$ and its dual by $\mathcal S'(\mathbb  R)$. We will see that for some particular choice of the values of $D_1$ and $D_2$ in \eqref{defZtnH}, the sequence $\big\{ (\mathcal{Z}_t)_{0 \leq t \leq T}^n, \; n \in \mathbb{N}_+ \big\}$ converges in law in $\mathcal{C} \big( [0, T], \;  \mathcal{S}'(\mathbb{R})\; \big)$ to $(\mathcal{Z}_t)_{0 \leq t \leq T}$, a stationary solution of an SPDE. For many regimes, the asymmetry of the model will not have a macroscopic effect, thus the limit object is going to be Gaussian, more exactly the Ornstein-Uhlenbeck process. On the other hand, for some critical regimes the limit of $\big\{ \mathcal{Z}_{0 \leq t \leq T}^n, \; n \in \mathbb{N}_+ \big\}$ will be a stationary energy solution of the stochastic Burgers equation. Keeping this in mind, we present the exact definitions of the corresponding SPDEs in the next subsection.

\subsection{Martingale problems}

We begin by introducing the operators that will appear in the SPDEs describing the scaling limit of our model. In order to clarify what features of the microscopic dynamics produces them, 
we will directly present them with the constants that will be relevant in our context. 

{We can associate to the probability  measure $p(\cdot)$ given by \eqref{prob} its symmetric and   asymmetric components that we denote by  $s(\cdot)$ and  $a(\cdot)$, and are defined via} 
\begin{equation} \label{defpsa}
\forall z \in \mathbb{Z}, \quad s(z):= \frac{p(z) + p(-z)}{2}, \quad a(z):= \frac{p(z) - p(-z)}{2}. 
\end{equation}
From \eqref{prob}, we get that $s(z):=c_{\gamma} |z|^{-\gamma-1} \mathbbm{1}_{ \{ z \neq 0 \} }$, where $c_{\gamma}$ is given by
\begin{equation} \label{defcgam}
c_{\gamma}:=\frac{ c^{+} + c^{-} }{2} = \Bigg\{ 2 \sum_{z=1}^{\infty} z^{-1-\gamma} \Bigg\}^{-1} < \infty.
\end{equation}
The \textit{symmetric} component of the dynamics leads to the operator $\mathbb{L}^{ \gamma}\colon \mathcal{S}(\mathbb{R}) \to L^2(\mathbb{R})$ 
defined, for $H\in \mathcal{S}(\mathbb{R})$ and $u\in\R$,  by
\begin{equation} \label{defLlamb}
\big( \mathbb{L}^{\gamma} H \big)(u):=
\begin{dcases}
 c_{\gamma} \int_{\mathbb{R}} \frac{H(u+v) + H(u-v) -2 H(u)}{|v|^{1+\gamma}} \dd v, \quad & \gamma \in (0,2), \\
\widehat{c}_{\gamma} \; \Delta H (u), \quad & \gamma \geq 2, \end{dcases}
\end{equation}
where for $\gamma \geq 2$, $\widehat{c}_{\gamma}$ is
\begin{equation} \label{defkgam}
\widehat{c}_{\gamma}:=
\begin{cases}
2 c_{\gamma} = 2 c_2, \quad & \gamma=2, \\
4 c_{\gamma} \sum_{z=1}^{\infty} z^{1-\gamma}, \quad & \gamma >2\,.
\end{cases}
\end{equation} 
We observe that $\mathbb{L}^{\gamma}$ is a multiple of the fractional (resp. classical) Laplacian if $\gamma \in (0,2)$ (resp. $\gamma \geq 2$). 

Part of the contribution of the \textit{asymmetric} component of the particle system results in the operator  
$\widehat{\mathbb{L}}_{\lambda}^{ \gamma}\colon \mathcal{S}(\mathbb{R}) \to L^2(\mathbb{R})$, which is defined\footnote{For a representation of 
$\widehat{\mathbb{L}}_{\lambda}^{ \gamma}$ as a Fourier multiplier see Appendix~\ref{app:Unique}.}, for $G\in \mathcal{S}(\mathbb{R})$ and $u\in\R$, as
\begin{equation} \label{defLhatlamb}
 \big( \widehat{\mathbb{L}}_{\lambda}^{ \gamma} G)(u):=
\begin{dcases}
K\frac{c^+ - c^{-}}{3} \lambda  \int_{\mathbb{R}} \frac{ \sign(v) }{|v|^{1+\gamma} } \big[ G(u+v) - G(u) - \theta^{\gamma} (v) \nabla G(u)  \big] \dd v, \; &\gamma \in (0,2), \\
0, \;& \gamma \geq 2, 
\end{dcases}
\end{equation}
where $K\in[0,\infty)$ is that in~\eqref{limKn}, $c^+, c^{-}>0$ are the same constants as in~\eqref{prob}, 
{$\sign(\cdot)$} is the classical sign function, 
and $\theta^{\gamma}\colon \mathbb{R} \to \mathbb{R}$ is given, for $v\in\R$, by 
\begin{align*}
\theta^{\gamma}(v):=
\begin{cases}
0, \quad & \gamma \in (0, 1), \\
v \; \mathbbm{1}_{ \{ [-1, 1 ] \} }(v), \quad & \gamma =1, \\
v, \quad & \gamma \in (1, 2)\,.
\end{cases}
\end{align*}
Furthermore, above $\lambda$ is a real number which will be fixed later. 

Associated to the previous operators, we consider the seminorm 
$\mathcal{P}^{ \gamma}$ defined on $H\in\cS(\R)$ as  
\begin{equation} \label{defPgamH}
\mathcal{P}^{ \gamma} H:=
\begin{dcases}
\widehat{c}_{\gamma} \| \nabla H \|^{2}_{L^{2}(\mathbb{R}) }, \quad & \gamma \geq 2, \\
c_{\gamma} \iint_{\mathbb{R}^{2}} \frac{[H(v) - H(u)]^{2}}{|v-u|^{1+\gamma}} \dd u \; \dd v, 	\quad & \gamma \in (0, \; 2).
\end{dcases}
\end{equation}

Thanks to the previous definitions, we are ready to introduce the martingale formulation of 
the solution theory for the fractional Ornstein-Uhlenbeck process that will be used 
throughout. Below and in the remainder of this article, $\dd  \mathcal{W}_t$ is a space-time white noise.

\begin{definition} \label{defspde}
Let $\gamma > 0$, $\lambda \in \mathbb{R}$. We say that a $\mathcal{S}'(\mathbb{R})$-valued 
stochastic process $(\mathcal{Z}_t)_{ 0 \leq t \leq T}$ is a stationary solution of the 
fractional Ornstein-Uhlenbeck equation with diffusivity $D>0$ (in short, 
$OU(\gamma,\lambda, D)$), 
\begin{equation} \label{spde}
\dd \mathcal{Z}_t = ( \mathbb{L}^{ \gamma  } + \widehat{\mathbb{L}}_{\lambda}^{ \gamma}) \; \mathcal{Z}_t \; \dd t + \sqrt{2D \; \mathcal{P}^{\gamma} } \; \dd  \mathcal{W}_t
\end{equation}
if $(\mathcal{Z}_t)_{ 0 \leq t \leq T}$ is such that  
\begin{enumerate}
\item for all $t\in[0,T]$, 
$\mathcal{Z}_t$ is a mean zero Gaussian field whose covariance 
satisfies 
\begin{equation} \label{covinifie}
\mathbb{E}_{P} [ \mathcal{Z}_t(H_1) \mathcal{Z}_t(H_2) ] = D \int_{\mathbb{R}}  H_1(u) \; H_2(u) \, \dd u
\end{equation}
for every $H_1, \; H_2 \in \mathcal{S}(\mathbb{R})$; 
\item
 for all $H \in \mathcal{S}(\mathbb{R})$, the processes $\mathcal{M}_t(H)$ and $\mathcal{N}_t(H)$ given by
\begin{equation}  \label{MartNou}
\mathcal{M}_t(H) := \mathcal{Z}_t(H) - \mathcal{Z}_0(H) - \mathcal{I}_t(H), \quad \mathcal{N}_t(H) := \mathcal{M}_t(H)^2  -2  D t \;    \mathcal{P}^{\gamma} H,
\end{equation}
are martingales with respect to the canonical filtration 
generated by $(\mathcal{Z}_t)_{ 0 \leq t \leq T}$, where, for $H\in\cS(\R)$, $\mathcal{I}_t(H)$ is given by
\begin{equation} \label{intprocess}
\forall t \in [0, \; T], \quad \mathcal{I}_t(H):= \int_0^t\mathcal{Z}_s (( \mathbb{L}^{ \gamma  
} + \widehat{\mathbb{L}}_{\lambda}^{ \gamma}) H) \; \dd s.
\end{equation}

\end{enumerate}
\end{definition}

\begin{rem}\label{rem:intprocess}
Condition (1) of the previous definition states that the process 
$(\mathcal{Z}_t)_{ 0 \leq t \leq T}$ is stationary and it has invariant measure given by a Gaussian white noise 
on $\R$. 
\end{rem}

We now move to the nonlinear equation of interest, namely the fractional Stochastic Burgers Equation (SBE). 
First, we need to define the Energy Estimate. 

\begin{definition}\label{defEE}
Let $(\mathcal{Z}_t)_{0 \leq t \leq T}$ be a given $\mathcal{S}'(\mathbb{R})$-valued stochastic process defined on some probability space $(X, \mathcal{F}, P)$. 
We say that $(\mathcal{Z}_t)_{0 \leq t \leq T}$ satisfies an \textit{Energy Estimate} (EE) if there exist $\kappa_0 >0$, $\omega >0$ such that for every $H \in \mathcal{S}(\mathbb{R})$, any $0 < \delta < \varepsilon < 1/2$ 
and any $0 \leq s < t \leq T$, we have 
\begin{equation} \label{eqdefEE}
\mathbb{E}_P \big[ \big( \mathcal{B}_{s,t}^{\varepsilon}(H) - \mathcal{B}_{s,t}^{\delta}(H)  \big)^2 \big] \leq \kappa_0 \varepsilon^{\omega} (t-s) \| \nabla H \|^2_{L^2(\mathbb{R})}, 
\end{equation}
where, for any $\varepsilon \in (0, 1/2)$, $0 \leq s \leq t \leq T$ and $H \in \mathcal{S}(\mathbb{R})$, 
the process $\mathcal{B}_{s,t}^{\varepsilon}(H)$ is defined as 
\begin{equation} \label{defprocB}
\mathcal{B}_{s,t}^{\varepsilon}(H):= \int_s^t \int_{\mathbb{R}} \nabla H (u)  \mathcal{Z}_r \big(\overleftarrow{\iota}^u_{\varepsilon}\big) \mathcal{Z}_r \big( \overrightarrow{\iota}_{\varepsilon}^u \big) \; \dd u \; \dd r,
\end{equation}
and $\overleftarrow{\iota}_{\varepsilon},\,\overrightarrow{\iota}_{\varepsilon}$ 
are the normalised indicator functions 
\begin{align} \label{defiotaeps}
\overleftarrow{\iota}_{\varepsilon}(v):=\frac{1}{ \varepsilon } \mathbbm{1}_{ \{  [- \varepsilon, \; 0)  \}  } (v), \quad \overrightarrow{\iota}_{\varepsilon}(v):=\frac{1}{ \varepsilon } \mathbbm{1}_{ \{  (0, \; \varepsilon ]  \}  } (v)\,,\qquad  v \in \mathbb{R}
\end{align}
while, for $u \in \mathbb{R}$ fixed, and a function $h$, the shifted function $h^{u}$ 
is 
\begin{align} \label{smoothshift}
 h^{u}(v):=h(v-u)\,,\qquad  \,v \in \mathbb{R}\,.
\end{align}
\end{definition}

Notice that if a process $(\mathcal{Z}_t)_{0 \leq t \leq T}$ satisfies (EE) then the sequence 
$(\mathcal{B}_{s,t}^{\varepsilon}(H) )_{\varepsilon > 0}$
is Cauchy in $L^2(P)$ and thus it has a unique limit 
$\mathcal{B}_{s,t}(H):= \lim_{\varepsilon \rightarrow 0^+} \mathcal{B}_{s,t}^{\varepsilon}(H)$ 
for any $0 \leq s < t \leq T$ and any $H \in \mathcal{S}(\mathbb{R})$. $\mathcal{B}_{\cdot,\cdot}(H)$ 
should be thought of as the, apriori ill-defined, 
product  $ \int_s^t \mathcal{Z}_t^2(\nabla H) \dd s$. 
If further $(\mathcal{Z}_t)_{0 \leq t \leq T}$ is stationary then $(\mathcal{B}_t)_{0 \leq t \leq T}$ given by
\begin{equ}[e:NonLin]
\mathcal{B}_t(H):= \lim_{\varepsilon \rightarrow 0^+} \mathcal{B}_{0,t}^{\varepsilon}(H)
\end{equ}
is a well defined stochastic process (see~\cite[Prop. 2.10]{jarafluc} 
and~\cite[Thm 2.2]{asymjara} for more details). 

We are now ready to give the definition of energy solution for the fractional SBE. 

\begin{definition} \label{defspdefbe}
Let $\gamma \geq 3/2$, $\lambda \in \mathbb{R}$ and $\kappa \in\R$. 
We say that a $\mathcal{S}'(\mathbb{R})$-valued stochastic process 
$(\mathcal{Z}_t)_{ 0 \leq t \leq T}$ is a stationary energy solution of the 
fractional SBE with coupling constant $\kappa$ and diffusivity $D>0$ 
(in short, $SBE(\gamma,\lambda, \kappa, D)$) 
\begin{equation} \label{spdefbe}
\dd \mathcal{Z}_t = ( \mathbb{L}^{ \gamma  } + \widehat{\mathbb{L}}_{\lambda}^{ \gamma}) \; \mathcal{Z}_t \; \dd t + \kappa \nabla \mathcal{Z}_t^2  \; \dd t+  \sqrt{2 D \; \mathcal{P}^{\gamma} } \; \dd \mathcal{W}_t
\end{equation}
 if $(\mathcal{Z}_t)_{ 0 \leq t \leq T}$ satisfies (EE) in the sense of Definition \ref{defEE}, and
\begin{enumerate}
\item for all $t\in[0,T]$, $\mathcal{Z}_t$ is a mean zero Gaussian field with covariance as in~\eqref{covinifie};
\item
 for any $H \in \mathcal{S}(\mathbb{R})$, the processes $\mathcal{M}_t(H)$ and $\mathcal{N}_t(H)$ given by
\begin{align*}
& \mathcal{M}_t(H) := \mathcal{Z}_t(H) - \mathcal{Z}_0(H) - \mathcal{I}_t(H) + \kappa_{\gamma} \mathcal{B}_t(H),  \quad \mathcal{N}_t(H) :=[ \mathcal{M}_t(H)]^2  -2  D_3 t \;    \mathcal{P}^{\gamma} H,
\end{align*} 
 are continuous martingales with respect to the canonical filtration 
generated by $(\mathcal{Z}_t)_{t\in[0,T]}$, where the processes $(\mathcal{I}_t)_{t\in[0,T]}$ and 
$(\mathcal{B}_t)_{t\in[0,T]}$ are respectively defined in~\eqref{intprocess} and~\eqref{e:NonLin}.  
\item the reversed process $\{\widehat{\mathcal{Z}}_{t}:= \mathcal{Z}_{T-t}: t\in[0,T]\}$ 
satisfies items (1) and (2) with $\kappa$, 
and $\mathcal{B}_t$ respectively replaced by $-\kappa$ and 
$\widehat{\mathcal{B}}_t:= \mathcal{B}_{T-t} -\mathcal{B}_T$. 
\end{enumerate}
\end{definition}

\begin{rem}\label{rem:OUvsSBE}
It is clear that, for any $\gamma\geq 3/2$ and $D>0$, the notion of solutions for 
$OU(\gamma,\lambda, D)$ and $SBE(\gamma,\lambda, 0,D)$ coincide. 
\end{rem}


Before proceeding, we state a uniqueness result for the martingale problems 
associated to $OU(\gamma, \mathcolor{orange}{\lambda,}  D)$ and $SBE(\gamma, \lambda, \kappa,D)$ in Definitions~\ref{defspde} and~\ref{defspdefbe}. 

\begin{prop} \label{uniqlaw}
Let $\gamma >0$, $\lambda \in \mathbb{R}$, $D>0$ and $\kappa\in\R$. Assume that, 
for $\gamma<3/2$, $\kappa=0$ and that the constant $K$ in the definition 
of $\widehat{\mathbb{L}}_{\lambda}^{ \gamma}$ in~\eqref{defLhatlamb} is different from $0$ only if $\gamma=3/2$. 
Then, two stationary solutions of 
either $OU(\gamma,\lambda,D)$ as in Definition~\ref{defspde} or 
$SBE(\gamma,\lambda,\kappa, D)$ as in Definition~\ref{defspdefbe} have the same law.
\end{prop}
\begin{proof}
Now, in the case of $OU(\gamma,\lambda,D)$, the result is a direct 
consequence of~\cite[Prop. 2.7]{jarafluc} if $\gamma\in(0,2)$ and of~\cite[Prop. 2.3]{symflucped} 
if $\gamma\geq 2$. 

The treatment of $SBE(\gamma,\lambda,\kappa, D)$ with $\kappa\neq0$ is more subtle. Note that if $\gamma>3/2$ then 
the operator $\widehat{\mathbb{L}}_{\lambda}^{ \gamma}$ in~\eqref{defLhatlamb} is $0$ (as $K=0$). Hence, 
$SBE(\gamma,\lambda,\kappa, D)$ coincides with the fractional SBE whose martingale 
problem has been shown to admit a unique solution  
in~\cite[Sec. 4.4 and 5.2]{GubinelliPerkowski2020}. For $\gamma=3/2$ instead a finer argument that suitably adapts 
that in~\cite{GPP} is needed. For completeness, we outline it in Appendix~\ref{app:Unique}. 
\end{proof}

\subsection{Main results}

Comparing \eqref{spde} and \eqref{spdefbe}, we observe that they differ on the presence (or not) of the nonlinearity $\nabla \mathcal{Z}_t^2$. In order to present our main theorems, we state now two hypothesis respectively corresponding to 
the case in which the nonlinearity does not or does appear. A crucial difference between them is the value of $K^{\star}$, defined by
\begin{align} \label{defKstar}
K^{\star} : =\lim_{n \rightarrow \infty } \frac{\Theta(n) K_n}{n \sqrt{n} }. 
\end{align}
\begin{hyp} \label{hyplin}
Either $c^{+}=c^{-}$ holds; or exactly one of the following holds:
	\begin{itemize}
		\item
		$0 < \gamma \neq 2$ and $K^{\star}=0$; 
		\item
		$\gamma =2$ and $\lim_{n \rightarrow \infty}  K_n \sqrt{n} \; [ \log(n) ]^{-3/4} =0$.
	\end{itemize}
\end{hyp}

\begin{hyp} \label{hypnonlin}
It holds $c^{+} \neq c^{-}$, $0 < \gamma \neq 2$ and $K^{\star} \in (0, \infty)$. 
\end{hyp}
\begin{rem} \label{remgam32}
Recall from \eqref{limKn} that $(K_n)_{n \in \mathbb{N}_+}$ is a bounded sequence. Thus we get from \eqref{timescale} and \eqref{defKstar} that $K^{\star} >0$ only when $\gamma \geq 3/2$.
\end{rem}
Our goal in this work is to state a long-range version of the main result in \cite{ABC}, namely Theorem 2.8. In order to simplify the presentation, this theorem was stated under the condition
\begin{equation} \label{simp}
\rho_A = \rho_B = \rho_C = 1/3 \quad \text{and} \quad E_A \neq E_C.
\end{equation}
In order to simplify the presentation of this work, we will also impose \eqref{simp} in what follows.

Keeping Theorem 2.8 of \cite{ABC} in mind, we will produce a sequence of pairs of processes $\big\{ \big( \mathcal{Z}_t^{n,+}, \; \mathcal{Z}_t^{n,-}  \big)_{0 \leq t \leq T}  \big\}$, where $\mathcal{Z}_t^{n,+}, \; \mathcal{Z}_t^{n,-}$ correspond to \textit{distinct} choices for the pair $(D_1,D_2)$ and the velocity $v_n$ in \eqref{defZtnH}. In order to avoid the degenerate case $(D_1,D_2)=(0,0)$, in the remainder of this work we impose the restriction
\begin{equation} \label{defDelta}
D_1 \neq 0, \quad \Delta:=(E_A - E_B)^2 -  (E_A-E_C) (E_C - E_B)>0.
\end{equation}
The value of $D_2^{\pm}$ in the definition of $\mathcal{Z}_t^{n,+}$ will be given by 
\begin{align} \label{defD2pm}
	D_2^{\pm}:= D_1 \frac{E_A - E_B + \lambda^{\pm}}{E_A-E_C}, 
\end{align}
where $\lambda^{\pm}$ is defined by
\begin{equation} \label{deflambda}
\lambda^{\pm}:= \pm \sqrt{\Delta} = \pm \sqrt{(E_A - E_B)^2 -  (E_A-E_C) (E_C - E_B)}. 
\end{equation}
\begin{rem} \label{remdelDel}
	We observe that the values of $\Delta$ and $\lambda$ are independent of $D_1$ and $D_2$. Moreover, keeping $\delta$ given by (2.17) in \cite{ABC} in mind, we get
\begin{equation*}
\Delta = (E_A - E_B)^2 -  (E_A-E_C) (E_C - E_B) = (E_A - E_C)^2 + (E_B - E_C)^2 - (E_A-E_C) (E_B - E_C) = (3 \delta/2)^{2}.
\end{equation*}	
Furthermore, the values of $D_2$ in \eqref{defD2pm} are exactly the values of $D_2$ expressed by equation (2.11) of \cite{ABC}, after making the choice $D_1 = 1$.
\end{rem}
It still remains to define the velocity $v_n$ in \eqref{defZtnH}. It will be given by
\begin{align} \label{vsdif}
v_n^{\pm} = - \frac{2 K_n}{3 } m_n^{\gamma} \lambda^{\pm},
\end{align}
where $m_n^{\gamma}$ is defined in an analogous way as in (2.3) of \cite{jarafluc}, namely
\begin{align} \label{mngamma}
m_n^{\gamma}:=
\begin{dcases}
0, \quad &  \gamma \in (0, 1); \\
n \sum_{r=-n}^n r a(r) = 2 \Theta(n) \sum_{r=1}^n r a(r), \quad &  \gamma =1; \\
\Theta(n) m_a, \quad &  \gamma \in (1, \infty).
\end{dcases}
\end{align}
In the last display, $m_a$ is given by
\begin{equation} \label{defma}
m_a:= \sum_{r } r a(r) = 2 \sum_{r=1}^{\infty} r a(r)
\end{equation}
and $a: \mathbb{Z} \mapsto [-1, 1]$ is given by \eqref{defpsa}. In particular, for $\gamma \in (0, 1)$, we have that $v_n=0$.  
Now we present our first main theorem.
\begin{thm}  \label{clt}
Let $\gamma >0$ and $D_1 \neq 0$. Assume that Hypothesis \ref{hyplin} and \eqref{simp} both hold. Consider the Markov process $\{ \eta_{t}^n : t \in [0,T] \}:=\{ \eta_{t \Theta(n)} : t \in [0,T] \}$  with generator given by \eqref{genABC} with $\Theta(n)$  given by \eqref{timescale} and  suppose that it starts from the invariant state $\nu_{\rho}$. For every $n \in \mathbb{N}_+$, define $ \big( \mathcal{Z}_t^{n,+}, \; \mathcal{Z}_t^{n,-}  \big)_{0 \leq t \leq T}$ by 
\begin{equation} \label{defZtnHpm}
\forall t \in [0, T], \; \forall H \in \mathcal{S}(\mathbb{R}), \quad  	\mathcal{Z}_t^{n, \pm}(H):= D_1 \Bigg\{ \mathcal{Y}_t^{n, A} ( T^n_{ t v_n^{\pm} } H ) + \frac{E_A - E_B + \lambda^{\pm}}{E_A-E_C} \mathcal{Y}_t^{n, B} ( T^n_{t v_n^{\pm}} H ) \Bigg\}, 
\end{equation}
where $\lambda^{\pm}$, resp. $v_n^{\pm}$ is given by \eqref{deflambda}, resp. by \eqref{vsdif}. With respect to the Skorohod topology of $ \mcb{D} \big( [0,T], \;  \mathcal{S}'(\mathbb{R}) \times \mathcal{S}'(\mathbb{R}) \big)$, $\big\{ \big( \mathcal{Z}_t^{n,+}, \; \mathcal{Z}_t^{n,-}  \big)_{0 \leq t \leq T}: \; n \in \mathbb{N}_+ \big\}$ converges in distribution to $\big( \mathcal{Z}_t^{+}, \; \mathcal{Z}_t^{-}  \big)_{0 \leq t \leq T}$, where $(\mathcal{Z}_t^{+})_{0 \leq t \leq T}$ and $(\mathcal{Z}_t^{-})_{0 \leq t \leq T}$ are uncorrelated stationary solutions of the Ornstein-Uhlenbeck equations of the form
\begin{equation} \label{spdeteo}
\dd \mathcal{Z}^{\pm}_t = ( \mathbb{L}^{ \gamma  } + \widehat{\mathbb{L}}_{\lambda^{\pm}}^{ \gamma}) \; \mathcal{Z}_t \; \dd t  + \sqrt{2 D_3^{\pm} \; \mathcal{P}^{\gamma} } \; \dd \mathcal{W}^{\pm}_t.
\end{equation}
In the last display, $D_3^{\pm}$ is given by \eqref{deflambda}, \eqref{defD2pm}, \eqref{simp} and \eqref{defD3}. Moreover, $(\mathcal{W}_t^{+})_{0 \leq t \leq T}$ and $(\mathcal{W}_t^{-})_{0 \leq t \leq T}$ are independent $\mathcal{S}'(\mathbb{R})$-valued Brownian motions with covariances given by
\begin{equation} \label{covWpm}
\forall G, H \in \mathcal{S}(\mathbb{R}), \; \forall 0 \leq s \leq t \leq T, \quad \mathbb{E} \big[ \mathcal{W}_t^{\pm}(G) \; \mathcal{W}_t^{\pm}(H)    \big] = s \int_{\mathbb{R}} G(u) \; H(u) \; du.
\end{equation}
\end{thm}
Next we present our second main result.
\begin{thm}  \label{clt2}
Let $\gamma >0$ and $D_1 \neq 0$. Assume that Hypothesis \ref{hypnonlin} and \eqref{simp} both hold. Consider the Markov process $\{ \eta_{t}^n : t \in [0,T] \}:=\{ \eta_{t \Theta(n)} : t \in [0,T] \}$  with generator given by \eqref{genABC} with $\Theta(n)$  given by \eqref{timescale} and  suppose that it starts from the invariant state $\nu_{\rho}$. For every $n \in \mathbb{N}_+$, define $ \big( \mathcal{Z}_t^{n,+}, \; \mathcal{Z}_t^{n,-}  \big)_{0 \leq t \leq T}$ by \eqref{defZtnHpm}. With respect to the Skorohod topology of $ \mcb{D} \big( [0,T], \;  \mathcal{S}'(\mathbb{R}) \times \mathcal{S}'(\mathbb{R}) \big)$, $\big\{ \big( \mathcal{Z}_t^{n,+}, \; \mathcal{Z}_t^{n,-}  \big)_{0 \leq t \leq T}: \; n \in \mathbb{N}_+ \big\}$ converges in distribution to $\big( \mathcal{Z}_t^{+}, \; \mathcal{Z}_t^{-}  \big)_{0 \leq t \leq T}$, where $(\mathcal{Z}_t^{+})_{0 \leq t \leq T}$ and $(\mathcal{Z}_t^{-})_{0 \leq t \leq T}$ are uncorrelated stationary energy solutions of the stochastic Burgers equations of the form
\begin{equation} \label{spdeteofbe}
\dd \mathcal{Z}^{\pm}_t = ( \mathbb{L}^{ \gamma  } + \widehat{\mathbb{L}}_{\lambda^{\pm}}^{ \gamma}) \; \mathcal{Z}_t \; \dd t + \kappa_{\gamma}^{\pm} (\nabla \mathcal{Z}_t)^2 \; \dd t  + \sqrt{2 D_3^{\pm} \; \mathcal{P}^{\gamma} } \; \dd \mathcal{W}^{\pm}_t.
\end{equation}
In the last display, $D_3^{\pm}$ is given by \eqref{defD3}; $(\mathcal{W}_t^{+})_{0 \leq t \leq T}$ and $(\mathcal{W}_t^{-})_{0 \leq t \leq T}$ are independent $\mathcal{S}'(\mathbb{R})$-valued Brownian motions with covariances given by \eqref{covWpm}; and $\kappa_{\gamma}^{\pm}$ is defined by
\begin{equation} \label{defkgampm}
\kappa_{\gamma}^{\pm}:= - 2 K^{\star} m_a  \frac{(E_A-2 E_B + E_C \mp  \sqrt{\Delta})(E_A-E_C)}{\pm 2  D_1 \sqrt{\Delta} }.
\end{equation}
In the last line, $K^{\star}$, resp. $m_a$, resp. $\Delta$ is given by \eqref{defKstar}, resp. \eqref{defma}, resp. \eqref{defDelta}.
\end{thm}

\begin{rem}
We get from \eqref{defkgampm} that $\kappa_{\gamma}^{\pm}=0$ only if $|E_A-2 E_B + E_C| = \sqrt{\Delta}$. This condition holds if and only if $(E_B-E_A)(E_B-E_C)=0$, due to \eqref{defDelta}.

Therefore, in the general case $(E_B-E_A)(E_B-E_C) \neq 0$ (this corresponds to case (III) in Section 3.4 of \cite{ABC}), we conclude that $\mathcal{Z}^{\pm}$ displays KPZ behavior.  

Now we analyze the particular case $(E_B-E_A)(E_B-E_C) = 0$.
\begin{itemize}
\item
If $E_A - E_C = E_B - E_C = E \neq 0$ (this corresponds to case (I) in Section 3.2 of \cite{ABC}), we get from \eqref{defDelta} that $\Delta=E^2$, thus $\kappa_{\gamma}^{-}=0$ and $\kappa_{\gamma}^{+} \neq 0$. This means that $\mathcal{Z}^{+}$ displays KPZ behavior, but $\mathcal{Z}^{-}$ is the solution of an Ornstein-Uhlenbeck equation. 
\item
If $E_B - E_A = E_C - E_A = E \neq 0$ (this corresponds to case (II) in Section 3.3 of \cite{ABC}), we get from \eqref{defDelta} that $\Delta=E^2$, thus $\kappa_{\gamma}^{-}=0$ and $\kappa_{\gamma}^{+} \neq 0$. This means that $\mathcal{Z}^{+}$ displays KPZ behavior, but $\mathcal{Z}^{-}$ is the solution of an Ornstein-Uhlenbeck equation.
\end{itemize}
\end{rem}
{
In order to prove Theorems \ref{clt} and \ref{clt2}, the first step is to prove the following result.}
\begin{prop} \label{tightseqZpm}
{Let $\gamma >0$ and assume that $D_1$, $D_2$, $\lambda$ and $v$ are chosen in such a way that \eqref{defD2pm}, \eqref{deflambda} and \eqref{vsdif} are satisfied. Moreover, assume that Hypothesis \ref{hyplin} or Hypothesis \ref{hypnonlin} holds.}

{For every $n \in \mathbb{N}_+$, define $ \big( \mathcal{Z}_t^{n,+}, \; \mathcal{Z}_t^{n,-}  \big)_{0 \leq t \leq T}$ by \eqref{defZtnHpm}. Then the sequence $\big\{ \big( \mathcal{Z}_t^{n,+}, \; \mathcal{Z}_t^{n,-}  \big)_{0 \leq t \leq T}: \; n \in \mathbb{N}_+ \big\}$ is tight with respect to the Skorohod topology of $ \mcb{D} \big( [0,T], \;  \mathcal{S}'(\mathbb{R}) \times \mathcal{S}'(\mathbb{R}) \big)$.}
\end{prop}
{Section \ref{secvarnlinvan} is dedicated to the proof of Lemmas \ref{varnlinvan} and \ref{varnlintight}, which are required in order to obtain Theorems \ref{clt} and \ref{clt2}.}

{Afterwards, in Section \ref{seccharac}, we prove that under \eqref{simp} and Hypothesis \ref{hyplin}, resp. Hypothesis \ref{hypnonlin}, the limit points of $\big\{ \big( \mathcal{Z}_t^{n,+}, \; \mathcal{Z}_t^{n,-}  \big)_{0 \leq t \leq T}: n \in \mathbb{N}_+\big\}$ satisfy the conditions stated in Definition \ref{defspde}, resp. \ref{defspdefbe}. Combining this with Proposition \ref{uniqlaw}, the proof for  Theorems \ref{clt} and \ref{clt2} ends.
}

In the next section, we obtain Proposition \ref{tightseqZpm}.

\section{ Proof of Proposition \ref{tightseqZpm}} \label{sectight}

We begin this section by observing that Proposition \ref{tightseqZpm} is a direct consequence of Mitoma's criterion (see \cite{mitoma}), the (classical) fact that $\mathcal{S}(\mathbb{R})$ is a nuclear Frechét space and the following lemma.
\begin{lem} \label{lemtightZnH}
Let $\gamma >0$ and assume that $D_1$, $D_2$, $\lambda$ and $v$ are chosen in such a way that \eqref{defD2pm}, \eqref{deflambda} and \eqref{vsdif} are satisfied.
\begin{itemize}
    \item 
If $H \in \mathcal{S}(\mathbb{R})$ and either Hypothesis \ref{hyplin} or Hypothesis \ref{hypnonlin} holds, the sequences $\big\{ \big( \mathcal{Z}_t^{n,+} (H) \; \big)_{0 \leq t \leq T}: \; n \in \mathbb{N}_+\big\}$ and $\big\{ \big( \mathcal{Z}_t^{n,-} (H) \; \big)_{0 \leq t \leq T}: \; n \in \mathbb{N}_+\big\}$ are tight with respect to the Skorohod topology of $\mathcal{D}([0, T], \mathbb{R})$.
\item 
If $H \in C_c^2 (\mathbb{R})$ and Hypothesis \ref{hypnonlin} holds, then
\begin{equation} \label{boundtimeZhol}	
\forall n \in \mathbb{N}_+, \; \forall 0 \leq s \leq t \leq T, \quad	  \mathbb{E}_{\nu_{\rho}} \big[ \big\{ \mathcal{Z}_t^{n} (H) - \mathcal{Z}_s^{n} (H) \big\}^{2} \big] 
	\leq   C(H )  (t-s).	  
\end{equation}
\end{itemize}
\end{lem}
\begin{rem}
In \eqref{boundtimeZhol}, $C(H)$ denotes some positive constant independent of $n$, $s$, $t$ satisfying the crucial property (necessary to obtain \eqref{e:HolP})
\begin{equation} \label{Cinvshift}
\forall u \in \mathbb{R}, \; \forall H \in C_c^2(\mathbb{R}), \quad C(H) = C(H^u), 
\end{equation}
where $H^u$ is given by \eqref{smoothshift}. Thus, $C(H)$ denotes a constant invariant under "shifts" of $H$. 
\end{rem}
In order to obtain Lemma \ref{lemtightZnH}, in the remainder of this section we fix $H \in \mathcal{S}(\mathbb{R}) { \;\cup \; C_c^{2}( \mathbb{R} ) }$. In order to simplify the notation, in the remainder of this section we denote $\big( \mathcal{Z}_t^{n,\pm} (H) \; \big)_{0 \leq t \leq T}$ by $\big( \mathcal{Z}_t^{n} (H) \; \big)_{0 \leq t \leq T}$. From Dynkin's formula, see Appendix 1.5 of \cite{kipnis1998scaling}, we have that
\begin{align} 
\mathcal{M}_t^n(H):= & \mathcal{Z}_{t}^n(H) - \mathcal{Z}_{0}^n(H) - \int_0^t \partial_s \mathcal{Z}_{s}^n(H) \dd s - \int_0^t \Theta(n) \mcb L^n \mathcal{Z}_{s}^n(H) \dd s, \label{dynk} \\
\mathcal{N}_t^n(H):= & \big[ \mathcal{M}_t^n(H) \big]^{2} -  \int_0^{t} \Theta(n) \big\{  \mcb L^n[\mathcal{Z}_{s}^n(H)]^{2} -2 \mathcal{Z}_{s}^n(H) \mcb L^n \mathcal{Z}_{s}^n(H)  \big\} \dd s \label{dynkN}
\end{align}
are both martingales with respect to the natural filtration $\mcb F_t:=\sigma(\eta_s^n, \; 0 \leq s \leq t)$. In particular, denoting the quadratic variation of $\mathcal{M}_t^n(H)$ by $\langle \mathcal{M}^n(H) \rangle_t$, it holds
\begin{equation} \label{quadvarMt}
\langle \mathcal{M}^n(H) \rangle_t =  \int_0^{t} \Theta(n) \big\{  \mcb L^n[\mathcal{Z}_{s}^n(H)]^{2} -2 \mathcal{Z}_{s}^n(H) \mcb L^n \mathcal{Z}_{s}^n(H)  \big\}\dd s.
\end{equation}
We now want to investigate the tightness of the sequence $\big\{ \big( \mathcal{Z}_t^{n}(H)  \big)_{0 \leq t \leq T}: \; n \in \mathbb{N}_+\big\}$ in the Skorohod topology of $\mathcal{D}([0, T], \mathbb{R})$. We begin by obtaining the convergence of the initial field $\mathcal{Z}_0^n(H)$, as $n \rightarrow \infty$.

\subsection{Convergence of $\big( \mathcal{Z}_0^n(H) \; \big)_{n \in \mathbb{N}_+}$}\label{tightZ}

Recall from \eqref{defD3} the definition of $D_3$.
\begin{prop} \label{convinifie}
 For any $t \in [0, T]$, it holds
\begin{align*}
\forall u \in \mathbb{R}, \quad \lim_{n \rightarrow \infty} \E_{\nu_{\rho}}  \big[   i u \mathcal{Z}_t^n(H)  \big] \} = \exp \Bigg\{ - \frac{u^2}{2} D_3 \| H \|_{L^{2}(\mathbb{R})}^{2} \Bigg\}. 
\end{align*}
Therefore, the sequence $\big( \mathcal{Z}_t^n(H) \; \big)_{n \in \mathbb{N}_+}$ converges in distribution to a mean zero Gaussian random variable with variance $D_3 \| H \|_{L^{2}(\mathbb{R})}^{2}$. In particular, $( \mathcal{Z}_0^n)_{n \in \mathbb{N}_+}$ converges in distribution to $\mathcal{Z}_0$, a mean zero Gaussian field with covariance given by \eqref{covinifie}, with $D$ being replaced by $D_3$.
\end{prop}
\begin{proof}
Similarly to Proposition 3.2 of \cite{tertufluc}, we obtain that
\begin{align*}
\lim_{n \rightarrow \infty} \log \Big( \E_{\nu_{\rho}}  \big[   i u \mathcal{Z}_t^n(H)  \big] \}  \Big) =& - \frac{ u^2 }{2} \lim_{n \rightarrow \infty} \Bigg\{ \frac{1}{n} \sum_x H^2 ( \tfrac{x-t v_n}{n}  )  \E_{\nu_{\rho}}  \big[ \big\{ D_1 \overline{\xi}_x^{A}( \eta^n_t ) + D_2 \overline{\xi}_x^{B}( \eta^n_t )  \big\}^2  \Bigg\}.
\end{align*}
The proof ends by applying \eqref{1bndL2Z}.
\end{proof}
In the next subsection we prove the convergence of the  martingale $\mathcal{M}^n_t(H)$, as $n \rightarrow \infty$.

\subsection{Convergence of $\big\{ \big( \mathcal{M}_t^{n}  \big)_{0 \leq t \leq T}: \; n \in \mathbb{N}_+  \big\}$}\label{tightM}

The goal of this result is to obtain the following proposition.
\begin{prop} \label{convmart} [\textbf{Convergence of the Dynkin martingale}]
Let $H \in \mathcal{S}(\mathbb{R})$. The sequence of martingales $\big\{ \big(  \mathcal{M}_t^n(H)  \big)_{0 \leq t \leq T}, \; n \in \mathbb{N}_+\big\}$ converges with respect to the Skorohod topology of $\mathcal{D}([0, T], \mathbb{R})$ to a mean zero Gaussian process $\mathcal{M}(H)$, as $n \rightarrow \infty$. Furthermore, $\mathcal{M}(H)$ is a continuous martingale on $[0, T]$ whose quadratic variation is $\langle \mathcal{M}(H) \rangle_t = t D_3 \mathcal{P}^{ \gamma} H$.
\end{prop}
In order to get the last result, we state a proposition on the convergence of the quadratic variation.
\begin{prop} \label{propexpquadvar}
Recall from \eqref{defPgamH} the definition of $\mathcal{P}^{ \gamma} H$. It holds
\begin{align} \label{limexpvarmart}
  \forall 0 \leq s \leq t \leq T, \quad \lim_{n \rightarrow \infty} E_{\nu_{\rho}} \big[ \langle \mathcal{M}^n(H) \rangle_t - \langle \mathcal{M}^n(H) \rangle_s \big] = (t-s) D_3  \mathcal{P}^{ \gamma} H.
\end{align}	
Furthermore, if $H \in C_c^1(\mathbb{R})$, it holds
	\begin{align} \label{ubcompvarmart}
\forall n \in \mathbb{N}_+, \; \forall 0 \leq s \leq t \leq T, \quad		\ E_{\nu_{\rho}} \big[ \langle \mathcal{M}^n(H) \rangle_t - \langle \mathcal{M}^n(H) \rangle_s \big] \leq   C(H)(t-s).
	\end{align}
\end{prop}
Next we state a result which controls the variance of $\langle \mathcal{M}^n(H) \rangle_{t}$. 
\begin{prop} \label{propvarvarquad}
It holds
\begin{align}  \label{limpropvarvarquad}
 \lim_{n \rightarrow \infty} E_{\nu_{\rho}} \Big[ \sup_{t \in [0, T]} \Big( \langle \mathcal{M}^n(H) \rangle_t - E_{\nu_{\rho}} \big[   \langle \mathcal{M}^n(H) \rangle_t \big] \Big)^{2} \Big]=0.
\end{align} 	
\end{prop} 
Now we are ready to show Proposition \ref{convmart}.
\begin{proof}[Proof of Proposition \ref{convmart}]
From \eqref{limexpvarmart}  and \eqref{limpropvarvarquad}, we have that for every $t \in [0, T]$ \textit{fixed}, the sequence of random variables $\big( \langle \mathcal{M}^n(H) \rangle_t \big)_{n \in \mathbb{N}_+}$ converges in probability to $t D_3 \mathcal{P}^{ \gamma} H$. Next, observe from \eqref{quadvarMt} that the trajectories of the process $\langle \mathcal{M}^n(H) \rangle_t$ are continuous almost surely. Finally, by exactly the same arguments described in \cite{ABC}, we have that the limit in (4.15) of \cite{ABC} also holds in our setting. The proof ends by applying Theorem 4.2 in \cite{ABC}.
\end{proof}
The reader may have noticed that we did not make use of \eqref{ubcompvarmart} in the last proof. Nevertheless, we stated it since this upper bound is crucial for obtaining \eqref{boundtimeZhol}.

In the remainder of this subsection we prove Propositions \ref{propexpquadvar} and \ref{propvarvarquad}. Keeping this in mind, let $f_x: \Omega \mapsto \mathbb{R}$ denote the application $f_x(\eta):= D_1 \overline{\xi}_x^{A}( \eta ) + D_2 \overline{\xi}_x^{B}( \eta )$ for any $x \in \mathbb{Z}$. We proceed by rewriting the quadratic variation $\langle \mathcal{M}^n(H) \rangle_{t}$ in a convenient way.
\begin{prop} \label{propquadvarM}
Let $H \in \mathcal{S}(\mathbb{R})$, $s \in [0,T]$ and $n \in \mathbb{N}_+$. Then
\begin{align}
	& \Theta(n) \big\{  \mcb L^n[\mathcal{Z}_{s}^n(H)]^{2} -2 \mathcal{Z}_{s}^n(H) \mcb L^n \mathcal{Z}_{s}^n(H)  \big\} \nonumber \\
	= &\frac{\Theta(n)}{2n} \sum_{ z, w } \big[ H ( \tfrac{z - s v_n  }{n}  ) - H ( \tfrac{w - s v_n  }{n}  ) \big]^{2} [f_w( \eta_s^n ) - f_z( \eta_s^n )]^{2} [ p(w-z) r^n_{z,w}(\eta_s^n) + p(z-w) r^n_{w,z}(\eta_s^n) ]. \label{claimquadvar}
\end{align}	
\end{prop}
\begin{proof}
The proof is a consequence of \eqref{genxiza}, \eqref{genxizb}, \eqref{ezalneza},  \eqref{ezblneza}, \eqref{ezalnezb} and \eqref{ezblnezb}.
\end{proof}
Before showing Proposition \ref{propexpquadvar}, we observe that
\begin{equation} \label{exprfzw}
\forall \eta \in \Omega, \;	\forall w \neq z \in \mathbb{Z}, \quad	 \E_{\nu_{\rho}} \big[ \{ f_w( \eta ) - f_z( \eta ) \}^{2}r^n_{z,w}(\eta) \big] = D_3.
\end{equation}
This comes from the fact that, under $\nu_\rho$, $\{ f_w( \eta ) - f_z( \eta ) \}^{2}$ and $r^n_{z,w}$ depend only on the couple $\eta(w), \eta(z)$, so the expectation reduces to a finite sum over $\alpha\neq\beta$. 
\begin{proof}[Proof of Proposition \ref{propexpquadvar}]
From \eqref{quadvarMt}, Proposition \ref{propquadvarM}, Fubini's Theorem and \eqref{exprfzw}, the expectation $E_{\nu_{\rho}} \big[ \langle \mathcal{M}^n(H) \rangle_t - \langle \mathcal{M}^n(H) \rangle_s \big]$ can be rewritten as
\begin{align*}
	& D_3 \int_{ s }^{t} \frac{\Theta(n)}{2n} \sum_{ z, w } \big[ H ( \tfrac{w + q_{n}^{ r }  }{n}  ) - H ( \tfrac{z + q_{n}^r  }{n}  ) \big]^{2}  [p(w-z) + p(z-w)] \dd r,
\end{align*}
where for every $n \in \mathbb{N}_+$ and $r \in [0,T]$, $q_{n}^{r}$ is given by
\begin{equation} \label{defqnr}
q_{n}^{ r } := - r v_n - \lfloor - r v_n \rfloor \in [0, 1).
\end{equation}
In the last line, $\lfloor \cdot \rfloor: \mathbb{R} \mapsto \mathbb{Z}$ denotes the "floor" function.

Thus, we get from \eqref{exprfzw} that
\begin{equation} \label{expquadvar}
	\E_{\nu_{\rho}} \big[ \langle \mathcal{M}^n(H) \rangle_t - \langle \mathcal{M}^n(H) \rangle_s  \big] = D_3   \int_{s}^{t} \frac{\Theta(n)}{n} \sum_{ z, w } \big[ H ( \tfrac{w + q_{n}^r  }{n}  ) - H ( \tfrac{z + q_{n}^r  }{n}  ) \big]^{2} s(w-z) \; \dd r,
\end{equation}
where $s: \mathbb{Z} \mapsto [0, 1]$ is the \textit{ symmetric} part of $p$, defined by \eqref{defpsa}. If $\gamma \in (0,2)$, resp. $\gamma \geq 2$, as a direct consequence of the computations in Section 4.2 of \cite{jarafluc}, resp. Proposition A.5 in \cite{symflucped}, we have that
\begin{align}
&\lim_{n \rightarrow \infty} \frac{\Theta(n)}{n}   \sum_{ x, y } \big\{ H ( \tfrac{y   }{n}  ) - H ( \tfrac{ x  }{n}  )  \big\}^{2} s(y-x) =  \mathcal{P}^{ \gamma} H,  \label{limexpquadvar0a} \\
&\lim_{n \rightarrow \infty} \frac{\Theta(n)}{n}  \int_{s}^{t} \sum_{ z, w } \big\{ H ( \tfrac{w   }{n}  ) - H ( \tfrac{ z  }{n}  )  \big\}^{2} s(w-z) \; \dd r = (t-s) \; \mathcal{P}^{ \gamma} H.  \label{limexpquadvar}
\end{align}
In Appendix \ref{useest1} we show that for every $n \in \mathbb{N}_+$, it holds
\begin{align}
	  \sup_{r \in [0,T]} \Bigg\{ \frac{\Theta(n)}{n} \sum_{ z, w } \big\{ \big[ H ( \tfrac{ w + q_{n}^r  }{n}  ) - H ( \tfrac{w   }{n}  ) \big] - \big[ H ( \tfrac{ z + q_{n}^r }{n}  ) - H ( \tfrac{ z  }{n}  ) \big] \big\}^{2} s(w-z) \Bigg\} \leq  \widetilde{C}_H \Bigg( \frac{1}{n^2} + \frac{\Theta(n)}{n^4}\Bigg). \label{errorexpquadvar}  
\end{align}
 Combining the last display with \eqref{timescale} and \eqref{limexpquadvar}, we get \eqref{limexpvarmart}. Above and in the remainder of this paper, $\widetilde{C}_H$ is some constant depending only on $H$, that may change from line to line. 

In Appendix \ref{useest2} we prove that if $a_H < b_H \in \mathbb{R}$ are such that $H(u)=0$ whenever $u \notin (a_H, b_H)$, then for every $n \in \mathbb{N}_+$, it holds
\begin{equation}  \label{boundexpquadvar}
  \sup_{r \in [0,T]} \Bigg\{  \frac{\Theta(n)}{n} \sum_{ z, w } \big\{ \big[ H ( \tfrac{ w + q_{n}^r  }{n}  ) -  H ( \tfrac{ z + q_{n}^r }{n}  )  \big]^{2} s(w-z) \Bigg\} \leq  C(\gamma) (a_H-b_H+1) [\| H \|_{\infty}^2 + \| \nabla H \|_{\infty}^2].
\end{equation}
Above and in the remainder of this text, $C(\gamma)$ denotes some positive constant depending only on $\gamma$, which may change from line to line. Combining the last display with \eqref{expquadvar}, we get \eqref{ubcompvarmart} and the proof ends.

\end{proof}
We end this subsection by presenting the proof of Proposition \ref{propvarvarquad}.
\begin{proof}[Proof of Proposition \ref{propvarvarquad}]

From \eqref{quadvarMt}, Proposition \ref{propquadvarM} and Fubini's Theorem, we have that
\begin{align} \label{centvarquad}
\langle \mathcal{M}^n(H) \rangle_t - E_{\nu_{\rho}} \big[ \langle \mathcal{M}^n(H) \rangle_t \big] = \int_0^{t} \frac{\Theta(n)}{2n} \sum_{ z, w } \big[ H ( \tfrac{w - s v_n  }{n}  ) - H ( \tfrac{z - s v_n  }{n}  ) \big]^{2} Z_{z,w}(\eta_s^n) \; \dd s,
\end{align}
where for every $z, w \in \mathbb{Z}$, $Z_{z,w}: \Omega \mapsto \mathbb{R}$ is given by
\begin{align*}
Z_{z,w}(\eta) :=& \{f_w( \eta ) - f_z( \eta )\}^{2} \{ p(w-z) r^n_{z,w}(\eta) + p(z-w) r^n_{w,z}(\eta) \} \\
- &  \E_{\nu_{\rho}} \big[ \{f_w( \eta ) - f_z( \eta )\}^{2} \{ p(w-z) r^n_{z,w}(\eta) + p(z-w) r^n_{w,z}(\eta) \} \big].
\end{align*}
Then from \eqref{exprfzw}, for every $\eta \in \Omega$ and every $x,y \in \mathbb{Z}$,
\begin{equation} \label{Zzweta}
  Z_{x,y}(\eta) = \{f_y( \eta ) - f_x( \eta )\}^{2} [p(y-x) r^n_{x,y}(\eta) + p(x-y) r^n_{y,x}(\eta)] - D_3 s(y-x)  = Z_{y,x}(\eta),
\end{equation}
where $s(\cdot)$ is given by \eqref{defpsa}. In particular, 
\begin{equation} \label{uncorrZzw}
\forall x,y, z,w \in \mathbb{Z}: \{ x,y \} \cap \{ z,w \} =\varnothing, \quad \E_{\nu_{\rho}} \big[ Z_{x,y}(\eta_s^n) Z_{z,w}(\eta)  \big] = 0.
\end{equation}
Observe that there exists $C_1$ (depending on $E_A$, $E_B$, $E_C$, $D_1$ and $D_2$) such that
\begin{align*}
\forall x,y \in \mathbb{Z}, \; \forall n \in \mathbb{N}_+, \quad \{f_y( \eta ) - f_x( \eta )\}^{2} |r^n_{x,y}(\eta)| \leq C_1. 
\end{align*}
Moreover, $p(y-x) + s(y-x) \leq C(\gamma) \mathbbm{1}_{y \neq x} |y-x|^{-1-\gamma}$, for any $x,y \in \mathbb{Z}$. Thus, from \eqref{Zzweta} we get
\begin{equation} \label{boundZzweta}
\forall x \neq y \in \mathbb{Z}, \quad  |Z_{x,y}(\eta)| \leq C_2 \mathbbm{1}_{ \{ y \neq x \} } |y-x|^{-1-\gamma},
\end{equation}
for $C_2:=(D_3 + 2 C_1) C(\gamma)$. Now from \eqref{centvarquad}, the expectation in \eqref{limpropvarvarquad} can be rewritten as
\begin{align*}
&\E_{\nu_{\rho}} \Bigg[  \sup_{t \in [0, T]}  \Bigg( \int_0^{t}  \frac{\Theta(n)}{2n} \sum_{ z, w } \big[ H ( \tfrac{w - s v_n  }{n}  ) - H ( \tfrac{z - s v_n  }{n}  ) \big]^{2} Z_{z,w}(\eta_s^n) \; \dd s  \Bigg)^2 \Bigg] \\
\leq& \frac{ T  [\Theta(n)]^{2}}{4 n^2}     \int_0^{ T }  \; \dd r \sum_{x,y, z, w } \big[ H ( \tfrac{y - r v_n  }{n}  ) - H ( \tfrac{x - r v_n  }{n}  ) \big]^{2} \big[ H ( \tfrac{w - r v_n  }{n}  ) - H ( \tfrac{z - r v_n  }{n}  ) \big]^{2} \E_{\nu_{\rho}} \big[ Z_{x,y}(\eta_r^n) Z_{z,w}(\eta_r^n)  \big].   
\end{align*} 
The first inequality follows from Cauchy--Schwarz, while the last identity from Fubini's theorem. Now from \eqref{uncorrZzw}, the sum over $x,y,z,w$ in the second line of the last display can be rewritten as
\begin{align*}
& 2\sum_{x,y } \big[ H ( \tfrac{y - r v_n  }{n}  ) - H ( \tfrac{x - r v_n  }{n}  ) \big]^{4} \E_{\nu_{\rho}} \big[ Z_{x,y}(\eta_r^n) Z_{x,y}(\eta_r^n)  \big] \\
+ & 4\sum_{x,y,  z: z \neq y } \big[ H ( \tfrac{y - r v_n  }{n}  ) - H ( \tfrac{x - r v_n  }{n}  ) \big]^{2} \big[ H ( \tfrac{z - r v_n  }{n}  ) - H ( \tfrac{x - r v_n  }{n}  ) \big]^{2} \E_{\nu_{\rho}} \big[ Z_{x,y}(\eta_r^n) Z_{x,z}(\eta_r^n)  \big]. 
\end{align*}
Now from \eqref{boundZzweta}, the last display is bounded from above by
\begin{align*}
& 2(C_2)^2 \sum_{x,y } \big[ H ( \tfrac{y - r v_n  }{n}  ) - H ( \tfrac{x - r v_n  }{n}  ) \big]^{4}   \mathbbm{1}_{ \{ y \neq x \} } |y-x|^{-2-2\gamma} \\
+ & 4 (C_2)^2 \sum_{x } \Bigg\{ \sum_{y } \big[ H ( \tfrac{y - r v_n  }{n}  ) - H ( \tfrac{x - r v_n  }{n}  ) \big]^{2}  \mathbbm{1}_{ \{ y \neq x \} } |y-x|^{-1-\gamma}  \Bigg\}^2.
\end{align*}
When $\gamma \in (0,2)$, the argument follows  Appendix D.2 of \cite{jarafluc}, while for $\gamma \geq 2$, we conclude following the same arguments as the ones below Proposition 4.2 in \cite{GJburgers}.
\end{proof}

To prove Lemma \ref{lemtightZnH}, we also need to study the last two terms of \eqref{dynk}. This is done in the next subsection. 

\subsection{Expansion of the Dynkin martingale}\label{tight3a}

We begin by computing the time derivative term of \eqref{dynk}. From \eqref{defZtnH}, we get
\begin{equation} \label{timedynk}
- \int_0^t \partial_s \mathcal{Z}_{s}^n( H)  \dd r 
 = \frac{1}{\sqrt{n}} \frac{ v_n }{n} \int_0^t    \sum_{x } \nabla H ( \tfrac{ x - v_n s}{n}  ) [D_1  \bar{\xi}_x^{A}( \eta_s^n )  + D_2 \bar{\xi}_x^{B}( \eta_s^n ) ]  \dd s,
\end{equation}
In order to expand the final term in \eqref{dynk}, for every $f:\mathbb{Z} \mapsto [0,1] $, define the operator $\mathbb{L}_n^{\gamma,f}$ by 
\begin{align} \label{defKnsa}
\forall n \in \mathbb{N}_+, \; \forall G \in 	L^{\infty}(\mathbb{R}), \; \forall u \in \mathbb{R}, \quad (\mathbb{L}_n^{\gamma,f} G )( \tfrac{u }{n}  ):= 2  \sum_{r } [ G ( \tfrac{u + r }{n}  )  - G ( \tfrac{u }{n}  )  ]f(r),
\end{align}
Combining the last display with \eqref{simp}, \eqref{ralbet}, \eqref{timedynk} and \eqref{dynk}, after applying some algebraic manipulations, we get
\begin{align}
	&\mathcal{M}_t^n(H)=\mathcal{Z}_{t}^n(H) - \mathcal{Z}_{0}^n(H) - \mcb{B}_t^{n}(H) - \frac{\Theta(n)}{\sqrt{n}} \int_0^t  \sum_{z }  \; \mathbb{L}_n^{\gamma,s} H ( \tfrac{z - s v_n }{n}  ) [D_1  \bar{\xi}_z^{A}( \eta_s^n )  + D_2 \bar{\xi}_z^{B}( \eta_s^n ) ] \dd s   \label{prindyn}  \\
+ & \frac{D_1 \Theta(n) K_n}{\sqrt{n}}  \frac{1}{3} \int_0^t \dd s \sum_{z } ( \mathbb{L}_n^{\gamma,a}   H )( \tfrac{z -s v_n}{n}  ) [ (E_B - E_A) \bar{\xi}_z^{A}( \eta_s^n ) + (E_B - E_C)   \bar{\xi}_z^{B}( \eta_s^n ) ] ] \label{remdyn1} \\
 + & \frac{D_2 \Theta(n) K_n}{\sqrt{n}} \frac{1 }{3} \int_0^t \dd s \sum_{z } ( \mathbb{L}_n^{\gamma,a}   H )( \tfrac{z -s v_n}{n}  ) [  (E_A - E_C)\bar{\xi}_z^{A}( \eta_s^n ) -(E_B - E_A) \bar{\xi}_z^{B}( \eta_s^n )] \label{remdyn2} \\
	+ & \frac{1}{\sqrt{n}} \frac{ v_n }{n} \int_0^t    \sum_{z } \nabla H ( \tfrac{z - v_n s}{n}  ) [D_1  \bar{\xi}_z^{A}( \eta_s^n )  + D_2 \bar{\xi}_z^{B}( \eta_s^n ) ] \dd s. \label{remdyn3}
\end{align}
where $s(\cdot)$ and $a(\cdot)$ are given by \eqref{defpsa}. Let $D_4:=D_1 (E_C - E_A)$, $D_5:=D_2 (E_C - E_B)$ and $D_6:= \big[ D_1(E_B-E_C)+D_2(E_A-E_C) \big]/2$. In \eqref{prindyn}, $\mcb{B}_t^{n}(H)$ is defined by
\begin{equation} \label{defBtn}
\mcb{B}_t^{n}(H):=2 \big[ D_4 \mcb{B}_t^{n,A,A}(H) + D_5 \mcb{B}_t^{n,B,B}(H) - D_6 \mcb{B}_t^{n,A,B}(H) - D_6 \mcb{B}_t^{n,B,A}(H) \big],  
\end{equation}
and for any $\alpha, \beta \in \{A, B\}$, $\mcb{B}_t^{n, \alpha , \beta}(H)$ is given by
\begin{equation} \label{defBtnalbe}
\mcb{B}_t^{n, \alpha , \beta}(H):= \frac{ \Theta(n) K_n}{\sqrt{n}}  \int_0^t \dd s \sum_x \sum_{y=x+1}^{\infty} [  H ( \tfrac{y - s v_n }{n}  ) - H ( \tfrac{x - s v_n }{n}  )  ] a(y-x)   \bar{\xi}_x^{\alpha}( \eta_s^n ) \bar{\xi}_y^{\beta}( \eta_s^n ).
\end{equation}
Now we state a result which provides necessary conditions under which the contribution of $\mcb{B}_t^{n}(H)$ can be neglected. Recall from \eqref{prob} the constants $c^{+}$, $c^{-}$. 

Lemma \ref{varnlinvan} below shows that under Hypothesis 1 the nonlinear contribution $\mcb{B}_t^{n}(H)$ is negligible at the fluctuation scale. Hence, the limiting field is necessarily Gaussian and governed by a Ornstein–Uhlenbeck type equation.
This connects the lemma directly to Theorem \ref{clt} and explains why the nonlinear term disappears in some regimes.

\begin{lem} \label{varnlinvan}
	Assume Hypothesis \ref{hyplin}. Then, there exist some $\delta_1 \in (0, \;1]$ and some function $g: \mathbb{N} \mapsto \mathbb{R}$ satisfying $\lim_{n \rightarrow \infty} g(n)=0$, such that
	\begin{align} 
\forall \alpha,\beta \in S, \; \forall n \in \mathbb{N}_+, \; \forall	0 \leq s \leq t \leq T, \quad	\mathbb{E}_{\nu_{\rho}} \big[ \big\{ \mcb{B}_t^{n, \alpha , \beta}(H) - \mcb{B}_s^{n, \alpha , \beta}(H) \big\}^{2} \big] \leq  C (t-s)^{1 + \delta_1} g(n). \label{varnlinzero}
	\end{align}
	In the last line, $C$ is some constant independent of $n$, $s$ and $t$.
	
	In particular, the sequence of stochastic processes $\big\{ \big( \mcb{B}_t^{n}(H) \; \big)_{0 \leq t \leq T}, \; n \in \mathbb{N}_+ \big\}$ is tight with respect to the uniform topology of $\mathcal{C} \big( [0, T], \mathbb{R} \big)$, due to the Kolmogorov-Centsov Theorem. Moreover, any of its limit points is identically equal to zero.
\end{lem}
\begin{proof}
If $c^{+}=c^{-}$, we get from \eqref{prob} and \eqref{defpsa} that $a \equiv 0$, thus $\mcb{B}_r^{n, \alpha , \beta}(H)=0$ for any $r \in [0, T]$ and the result follows. 

When $c_+\neq c_-$, the proof consists in estimating the second moment of the increment $\mcb{B}_t^{n, \alpha , \beta}(H) - \mcb{B}_s^{n, \alpha , \beta}(H)$ by decomposing the spatial sum according to the jump
size $|y-x|$ (far, intermediate, and near contributions). Hypothesis \ref{hyplin} ensures that the prefactor $\Theta(n)K_n\sqrt{n}$ 
vanishes (or is sufficiently small) in all regimes where asymmetry is not critical. This is the reason why the quadratic moment decays as $n\to\infty$.
Each piece yields a bound of the form
$C (t-s)^{1+\delta_1} \, \mathrm{Err}_n$, where $\mathrm{Err}_n\to 0$ under Hypothesis \ref{hyplin}. This produces \eqref{varnlinzero}
(with some $\delta_1\in(0,1]$) and allows one to apply the Kolmogorov--Centsov Theorem.
In particular, the proof is a direct consequence of	
	\begin{itemize}	
		\item 
		Corollary \ref{cornlgamless1}, if $0 < \gamma < 1$;
		\item 
		Corollaries \ref{corngamgr1} and \ref{cornlgamless32}, if $1 \leq \gamma < 3/2$;
		\item 
		Corollaries \ref{corngamgr1} and \ref{cornlgamgrea32}, if $\gamma \in [3/2, 2)$ and $\lim_{n \rightarrow \infty}  (K_n)^{2} n^{2\gamma - 3 }=0$;
		\item 
		Corollaries \ref{corngamgr1} and \ref{cornlgamgrea32}, if $\gamma =2$ and $\lim_{n \rightarrow \infty}  K_n \sqrt{n} \; [ \log(n) ]^{-3/4} =0$;
		\item Lemma \ref{lemgamgrea2} and Corollary \ref{corngamgr2}, if $\gamma >2$ and $\lim_{n \rightarrow \infty}  K_n \sqrt{n}  =0$.
	\end{itemize}	
Observe that the precise value of $\delta_1$ is irrelevant for our purposes: any gain over linear time growth,
namely any exponent strictly larger than $1$, is sufficient to invoke the Kolmogorov--Centsov Theorem and obtain tightness in $ \mathcal{C}([0,T],\R)$. We therefore do not attempt to optimize $\delta_1$.
\end{proof}

Next we state a result which provides weaker conditions under which the sequence of stochastic processes $\big\{ \big( \mcb{B}^{n}(H) \; \big)_{0 \leq t \leq T}, \; n \in \mathbb{N}_+ \big\}$ is tight with respect to the uniform topology of $\mathcal{C} \big( [0, T], \mathbb{R} \big)$ (due to the Kolmogorov-Centsov Theorem), but its limit points are not necessary equal to zero.
Lemma~\ref{varnlintight} shows that, under Hypothesis \ref{hypnonlin}, the nonlinear term
$\big\{ \big( \mcb{B}_t^{n}(H) \; \big)_{0 \leq t \leq T}, \; n \in \mathbb{N}_+ \big\}$ remains tight in $ \mathcal{C}([0,T],\R)$, but in general does \textit{not} vanish in the limit.
This result identifies the regime in which the asymmetry of the microscopic dynamics produces a macroscopic nonlinear effect, leading to a Burgers/KPZ-type correction in the limiting SPDE.
\begin{lem} \label{varnlintight}
	Assume Hypothesis \ref{hypnonlin}. Then, there exists some $\delta_1 \in (0, \;1]$ and some bounded function $g: \mathbb{N} \mapsto \mathbb{R}$, such that
	\begin{align} \label{ubvarnlintight1}
	\forall \alpha,\beta \in S, \; \forall n \in \mathbb{N}_+, \; \forall	0 \leq s \leq t \leq T, \quad	\mathbb{E}_{\nu_{\rho}} \big[ \big\{ \mcb{B}_t^{n, \alpha , \beta}(H) - \mcb{B}_s^{n, \alpha , \beta}(H) \big\}^{2} \big] \leq  C(t-s)^{1 + \delta_1} g(n). 		
	\end{align}
	In the last line, $C$ is some constant independent of $n$, $s$ and $t$. Furthermore, if $H \in C_c^{2}(\mathbb{R})$, then the role of $C$ in \eqref{ubvarnlintight1} is fulfilled by some constant $C(H)$ satisfying \eqref{Cinvshift}.
\end{lem}	
\begin{proof}
The proof follows the same spatial decomposition as in Lemma~\ref{varnlinvan} (far, intermediate and
near interactions), but the bounds are now uniform in $n$ rather than vanishing. Hypothesis \ref{hypnonlin} ensures that the relevant prefactors remain bounded, yielding time regularity sufficient for tightness. In details, the result is a direct consequence of Corollaries \ref{corngamgr1} and \ref{cornlgamgrea32} if $\gamma \in [3/2, 2)$, and Lemma \ref{lemgamgrea2} and Corollary \ref{corngamgr2} if $\gamma >2$. As in Lemma~\ref{varnlinvan}, the exponent $1+\delta_1>1$ is only used to guarantee tightness through the Kolmogorov--Centsov Theorem. No attempt is made to optimize $\delta_1$.
\end{proof}
From the following result, we obtain an upper bound for the variance of the  final term of \eqref{prindyn}. In Lemma \ref{lemvarprinc} below and in the remainder of this text, $C(\gamma,H)$ denotes some positive constant $C$ depending only on $\gamma$ and $H$. The next result is proved in Appendix \ref{useest31}.
\begin{lem} \label{lemvarprinc}
For every $n \in \mathbb{N}_+$, it holds 
\begin{align} \label{suplapfracdisc1}
 \sup_{r \in [0,T]} \Bigg\{ \frac{1}{n} \sum_{x} \big[  \Theta(n) ( \mathbb{L}_n^{\gamma,s} H )( \tfrac{x + q_n^r }{n}  ) \big]^{2}  \Bigg\} \leq C(\gamma,H).
\end{align}
Moreover, if $H \in C_c^{2}(\mathbb{R})$, then the role of $C(\gamma,H)$ above is fulfilled by some constant $C(H)$ satisfying \eqref{Cinvshift}.
\end{lem}
It remains to treat the terms in \eqref{remdyn1}, \eqref{remdyn2} and \eqref{remdyn3}. From this point on, we will assume that $D_1$, $D_2$, $\lambda$ and $v$ are chosen in such a way that \eqref{defD2pm}, \eqref{deflambda} and \eqref{vsdif} are satisfied. Under this assumption, it is not hard to prove that the sum of \eqref{remdyn1}, \eqref{remdyn2} and \eqref{remdyn3} can be rewritten as
\begin{align} \label{sumremdyn}
	\frac{1}{\sqrt{n}}  \int_0^t    \sum_{z } \widehat{\mathbb{L}}_{n, \lambda}^{ \gamma} H ( \tfrac{z - v_n s}{n}  ) [D_1  \bar{\xi}_z^{A}( \eta_s^n )  + D_2 \bar{\xi}_z^{B}( \eta_s^n ) ] \dd s.
\end{align}
The discrete operator $\widehat{\mathbb{L}}_{n, \lambda}^{ \gamma}$ in the last line is given by
\begin{equation} \label{Lnbgam}
\widehat{\mathbb{L}}_{n, \lambda}^{ \gamma} H ( \tfrac{u}{n}  )=   \frac{2 \lambda K_n}{3 } \Bigg\{ \Theta(n) \sum_{r } [ H ( \tfrac{u + r }{n}  )  - H ( \tfrac{u }{n}  )  ]a(r) -   \frac{ m_n^{\gamma} }{n} \nabla H ( \tfrac{u}{n}  ) \Bigg\},
\end{equation}
where $m_n^{\gamma}$ is given by \eqref{mngamma}. 

The variance of the term in \eqref{sumremdyn} can be estimated by making use of the following lemma, which is proved in Appendix \ref{useest32}.
\begin{lem} \label{lemvarext}
For every $n \in \mathbb{N}_+$, it holds
\begin{align} \label{suplapfraclatdisc1}
 \sup_{r \in [0,T]} \Bigg\{ \frac{1}{n} \sum_{x} \big[  \widehat{\mathbb{L}}_{n, \lambda}^{ \gamma} H ( \tfrac{x + q_n^r }{n}  ) \big]^{2}  \Bigg\} \leq | \lambda| K_n C(\gamma,H).
\end{align}
Moreover, if $H \in C_c^{2}(\mathbb{R})$, then the role of $C(\gamma,H)$ above is fulfilled by some constant $C(H)$ satisfying \eqref{Cinvshift}.
\end{lem}
In order to prove Lemma \ref{lemtightZnH}, we get from \eqref{prindyn} and \eqref{sumremdyn} that
\begin{equation} \label{decompZ}
	\mathcal{Z}_{t}^n(H) = \mathcal{Z}_{0}^n(H) + \mathcal{M}_t^n(H)  + \mcb{B}_t^{n}(H) + \mcb{I}_t^{n}(H),
\end{equation}
where $\mcb{I}_t^{n}(H)$ is the integral term given by
\begin{equation} \label{defItn}
	\mcb{I}_t^{n}(H):= \frac{1}{\sqrt{n}} \int_0^t  \sum_{z } \big\{ \Theta(n) \; \mathbb{L}_n^{\gamma} H ( \tfrac{z + q_n^s }{n}  ) - \widehat{\mathbb{L}}_{n, \lambda}^{ \gamma} H ( \tfrac{z + q_n^s}{n}  )\big\} [D_1  \bar{\xi}_z^{A}( \eta_s^n )  + D_2 \bar{\xi}_z^{B}( \eta_s^n ) ] \dd s, 
\end{equation}
where $q_n^s$ is given by \eqref{defqnr}. In particular, for any $0 \leq s \leq t \leq T$. it holds
\begin{equation} \label{compZn}
		\mathcal{Z}_{t}^n(H) - \mathcal{Z}_{s}^n(H) = \big[	\mathcal{M}_t^n(H) - \mathcal{M}_s^n(H) \big]  + \big[ \mcb{B}_t^{n}(H) - \mcb{B}_s^{n}(H) \big] + \big[ \mcb{I}_t^{n}(H) - \mcb{I}_s^{n}(H) \big].
\end{equation}

\begin{rem}\label{rem:KCtight}
    Combining the Kolmogorov-Centsov Theorem with \eqref{suplapfracdisc1} and \eqref{suplapfraclatdisc1}, resp. \eqref{varnlinzero} and \eqref{ubvarnlintight1}, we get that $\big\{ \big( \mcb{I}_t^{n}(H) \;  \big)_{0 \leq t \leq T}: \; n \in \mathbb{N}_+ \big\}$, resp. $\big\{ \big( \mcb{B}_t^{n}(H) \;  \big)_{0 \leq t \leq T}: \; n \in \mathbb{N}_+ \big\}$ is a tight sequence with respect to the Skorohod topology of $\mathcal{D}([0, T], \mathbb{R})$. 
\end{rem}
We end this subsection by presenting the proof of Lemma \ref{lemtightZnH}.
\begin{proof}[Proof of Lemma \ref{lemtightZnH}]
Recall that $D_1$, $D_2$, $\lambda$ and $v$ are chosen in such a way that \eqref{defD2pm}, \eqref{deflambda} and \eqref{vsdif} are satisfied.
\begin{itemize}
    \item 
If $H \in \mathcal{S}(\mathbb{R})$ and either Hypothesis \ref{hyplin} or Hypothesis \ref{hypnonlin} holds, by combining \eqref{decompZ} and Remark \ref{rem:KCtight} with Propositions \ref{convinifie} and \ref{convmart}, we conclude that the sequence $\big\{ \big( \mathcal{Z}_t^{n} (H) \; \big)_{0 \leq t \leq T}: \; n \in \mathbb{N}_+  \big\}$ is tight with respect to the Skorohod topology of $\mathcal{D}([0, T], \mathbb{R})$.
\item 
If $H \in C_c^2 (\mathbb{R})$ and Hypothesis \ref{hypnonlin} holds, by combining \eqref{compZn} and \eqref{ubcompvarmart} with Lemmas \ref{varnlintight} and \ref{lemvarext}, we get \eqref{Cinvshift}. This ends the proof. 
\end{itemize}
\end{proof}
It still remains to detail the proof for Lemmas \ref{varnlinvan} and \ref{varnlintight}. This is done in the next section.

\section{Proof of Lemmas \ref{varnlinvan} and \ref{varnlintight}} \label{secvarnlinvan}

In this section we provide the proof for Lemmas \ref{varnlinvan} and \ref{varnlintight}, thus for the remainder of this section we fix $H \in \mathcal{S}(\mathbb{R}) { \;\cup \; C_c^{2}( \mathbb{R} ) }$ and $\alpha,\beta \in S$. We observe that for the latter lemma, we have the additional requirement that if $H \in C_c^{2}(\mathbb{R})$, then the role of $C$ in \eqref{ubvarnlintight1} is fulfilled by some constant $C(H)$ satisfying \eqref{Cinvshift}. Thus, for many (but not all) of the results listed below, we need to state (and prove) an analogous condition. This will be avoided only if a particular lemma below is required for obtaining Lemma \ref{varnlinvan}, but not Lemma \ref{varnlintight}. We recall that $v_n:=v(n)$, for $v:\mathbb{N}_+ \mapsto \mathbb{R}$. We begin with the following estimate.
\begin{lem} \label{CSfubest}
Let $\gamma >0$, $0 \leq s \leq t \leq T$, $n \in \mathbb{N}_+$ and $A_n$ be an arbitrary subset of $\mathbb{Z}^{2}$. Let also $B_{r,v}^{H,n}: \mathbb{Z}^{2} \mapsto \mathbb{R}$ be a deterministic map depending only on $v,H,n$ and $r \in [0,T]$. Then
\begin{align*}
&	\mathbb{E}_{\nu_{\rho}} \Bigg[ \Bigg\{ \int_s^t \dd r    \sum_{(x,y) \in A_n} B_{r,v}^{H,n}(x,y)  a(y-x)   \bar{\xi} _x^{\alpha}( \eta_r^n ) \bar{\xi} _y^{\beta}( \eta_r^n ) \Bigg\}^{2} \Bigg]  \\
\leq & \chi(\rho_{\alpha}) \chi(\rho_{\beta})	  (t-s)  \int_s^t \sum_{(x,y) \in A_n} \big[  B_{r,v}^{H,n} (x,y) \big]^{2}  a^{2}(y-x) \; \dd r.
\end{align*}
\end{lem}
\begin{proof}
The proof follows from a simple application of the  Cauchy-Schwarz inequality in the time integral, then Fubini's Theorem and using the fact that the variables are centered. We leave the details to the reader. 
\end{proof}
Next we state a result which allows treating the velocity terms in  \eqref{varnlinzero}. In what follows, $(K_n)_{n \in \mathbb{N}_+} \subset \mathbb{R}$ denotes the sequence defined in \eqref{ralbet} and \eqref{limKn}. Below and in the remainder of this section, $\widehat{C}$ denotes some constant independent of $n \in \mathbb{N}_+$ and $0 \leq s \leq t$, which may change from line to line.
\begin{lem} \label{lemaproxintpart}
Let $\gamma >0$, $0 \leq s \leq t \leq T$, $n \in \mathbb{N}_+$ and $A_n$ be an arbitrary subset of $\mathbb{Z}^{2}$. Then
\begin{align*} 
	&	\mathbb{E}_{\nu_{\rho}} \Bigg[ \Bigg\{ \int_s^t \dd r \frac{ \Theta(n) K_n }{\sqrt{n}}    \sum_{(x,y) \in A_n} B_{r,v}^{H,n}(x,y) a(y-x)   \bar{\xi} _x^{\alpha}( \eta_r^n ) \bar{\xi} _y^{\beta}( \eta_r^n ) \Bigg\}^{2} \Bigg] \leq \widehat{C} (t-s)^{2} \frac{  \big[ \Theta(n) K_n \big]^{2}  }{ n^{4}  } \mathbbm{1}_{ \{ \gamma \geq 1 \} },
\end{align*}
where for every $z,w \in \mathbb{Z}^{2}$, $B_{r,v}^{H,n}(z,w):=[  H ( \tfrac{w - r v_n }{n}  ) - H ( \tfrac{z - r v_n }{n}  )  ] - [  H ( \tfrac{ \lfloor w - r v_n \rfloor }{n}  ) - H ( \tfrac{ \lfloor z - r v_n \rfloor }{n}  )  ]$. Furthermore, if $H \in C_c^2(\mathbb{R})$, the role of $\widehat{C}$ is fulfilled by some constant $C(H)$ satisfying \eqref{Cinvshift}. 
\end{lem}
\begin{proof}
We first observe that for  $\gamma \in (0, 1)$, we get from \eqref{vsdif} and \eqref{mngamma} that $v_n=0$, thus $B_{r,v}^{H,n} \equiv 0$ and the expectation in the last display is trivially equal to zero. On the other hand, in Appendix \ref{useest4} we prove that, under $\gamma \geq 1$,
\begin{equation} \label{claimuseest4}
\sup_{r \in [0,T]}  \Bigg\{ \sum_{x,y} \big\{ [  H ( \tfrac{y+q_{n}^{ r } }{n}  ) - H ( \tfrac{x+q_{n}^{ r } }{n}  )  ] - [  H ( \tfrac{ y }{n}  ) - H ( \tfrac{x }{n}  )  ] \big\} ^{2}  a^{2}(y-x) \Bigg\} \leq \frac{\widehat{C}}{n^3},   
\end{equation}
where the role of $\widehat{C}$ in \eqref{claimuseest4} is fulfilled by some constant $C(H)$ satisfying \eqref{Cinvshift} for $H \in C_c^2(\mathbb{R})$. Combining \eqref{claimuseest4} with Lemma \ref{CSfubest}, the proof ends.
\end{proof}
The next result is a consequence of  Lemma \ref{CSfubest} and allows treating  terms $(x,y)$ in \eqref{varnlinzero} that are sufficiently far from each other.
\begin{lem} \label{varfarterms}
Let $\gamma >0$, $0 \leq s \leq t \leq T$ and $n \in \mathbb{N}_+$. Then
\begin{align*}  
&\mathbb{E}_{\nu_{\rho}} \Bigg[ \Bigg\{ \frac{\Theta(n) K_n}{\sqrt{n}}  \int_s^t \dd r \sum_x \sum_{y =x + n}^{\infty} [  H ( \tfrac{y - r v_n }{n}  ) - H ( \tfrac{x - r v_n }{n}  )  ] a(y-x)   \bar{\xi} _x^{\alpha}( \eta_r^n ) \bar{\xi} _y^{\beta}( \eta_r^n ) \Bigg\}^{2} \Bigg] \leq  \widehat{C}  (t-s)^{2}\frac{ [ \Theta(n) K_n ]^{2} }{n^{2 \gamma+1}}. 
\end{align*}
Furthermore, if $H \in C_c^0(\mathbb{R})$, the role of $\widehat{C}$ is fulfilled by some constant $C(H)$ satisfying \eqref{Cinvshift}.  
\end{lem}
\begin{proof}
In Appendix \ref{useest5}, we prove that
\begin{equation} \label{claimuseest5}
\sup_{r \in [0,T]} \Bigg\{  \sum_{x,y:|y - x| \geq n} [  H ( \tfrac{y - r v_n }{n}  ) - H ( \tfrac{x - r v_n }{n}  )  ]^{2}  a^{2}(y-x) \Bigg\} \leq \frac{\widehat{C}}{n^{2 \gamma}},   
\end{equation}
where the role of $\widehat{C}$ in \eqref{claimuseest5} is fulfilled by some constant $C(H)$ satisfying \eqref{Cinvshift} for $H \in C_c^0(\mathbb{R})$. Combining \eqref{claimuseest5} with Lemma \ref{CSfubest}, the proof ends.
\end{proof}
Next we state a result which treats the contribution of the terms $(x,y)$ in \eqref{varnlinzero} such that $x$ and $y$ are slightly far from each other. In the remainder of this section, let $L_n \in \{1, \ldots, n-1\}$, for any $n \in \mathbb{N}_+$. 
\begin{lem} \label{varmidterms}
Let $\gamma >0$, $0 \leq s \leq t \leq T$ and $n \in \mathbb{N}_+$. Then
	\begin{align} 
	&	\mathbb{E}_{\nu_{\rho}} \Bigg[ \Bigg\{ \frac{\Theta(n) K_n}{\sqrt{n}} \int_s^t \dd r \sum_x \sum_{y=x+L_n}^{
	x+n-1}  [  H ( \tfrac{y - r v_n }{n}  ) - H ( \tfrac{x - r v_n }{n}  )  ] a(y-x)   \bar{\xi} _x^{\alpha}( \eta_r^n ) \bar{\xi} _y^{\beta}( \eta_r^n ) \Bigg\}^{2} \Bigg] \nonumber \\
		\leq  & \widehat{C} (t-s)^{2} \Bigg\{ \frac{  \big[ \Theta(n) K_n \big]^{2}  }{ n^{4}  } \mathbbm{1}_{ \{ \gamma \geq 1 \} } + \frac{ [ \Theta(n) K_n  ]^{2} }{n^{2}} \sum_{x=L_n}^{n} x^{- 2 \gamma} \Bigg\}. \label{upvarmidterms}
	\end{align}
	Furthermore, if $H \in C_c^2(\mathbb{R})$, the role of $\widehat{C}$ is fulfilled by some constant $C(H)$ satisfying \eqref{Cinvshift}. 
\end{lem}
\begin{proof}
The leftmost term in \eqref{upvarmidterms} comes from the application of Lemma \ref{lemaproxintpart} with the choice $A_n:=\{(z,w) \in \mathbb{Z}^{2}: \; L_n \leq |w - z| < n \}$. We are then reduced to estimate
\begin{align*}
	\mathbb{E}_{\nu_{\rho}} \Bigg[ \Bigg\{ \frac{\Theta(n) K_n}{\sqrt{n}} \int_s^t \dd r \sum_x \sum_{y=x+L_n}^{
	x+n-1} [  H ( \tfrac{ \lfloor y - r v_n \rfloor }{n}  ) - H ( \tfrac{ \lfloor  x - r v_n \rfloor }{n}  )  ] a(y-x)   \bar{\xi} _x^{\alpha}( \eta_r^n ) \bar{\xi} _y^{\beta}( \eta_r^n ) \Bigg\}^{2} \Bigg].
\end{align*}
In Appendix \ref{useest6}, we prove that
\begin{equation} \label{claimuseest6}
\sup_{r \in [0,T]} \Bigg\{  \sum_{x,y: \; L_n \leq |y - x| \leq n} \big[  H ( \tfrac{ \lfloor y - r v_n \rfloor }{n}  ) - H ( \tfrac{ \lfloor  x - r v_n \rfloor }{n}  ) \big]^{2}  a^{2}(y-x) \Bigg\} \leq \frac{\widehat{C}}{n} \sum_{x=L_n}^{n} x^{- 2 \gamma},   
\end{equation}
where the role of $\widehat{C}$ in \eqref{claimuseest6} is fulfilled by some constant $C(H)$ satisfying \eqref{Cinvshift} for $H \in C_c^2(\mathbb{R})$. Combining \eqref{claimuseest6} with Lemma \ref{CSfubest}, the proof ends.
\end{proof}
Recall \eqref{defBtnalbe}. From the next result, and for $\gamma \in (0, \;1)$,  we conclude that the contribution of $\mcb{B}_t^{n}(H)$ can be neglected. 
\begin{cor} \label{cornlgamless1}
Let $\gamma >0$, $0 \leq s \leq t \leq T$ and $n \in \mathbb{N}_+$. Then
\begin{align} \label{varcornlgamless1}
 \mathbb{E}_{\nu_{\rho}} \Bigg[ \Bigg\{ \frac{\Theta(n) K_n}{\sqrt{n}} \int_s^t \dd r \sum_x \sum_{y=x+1}^{\infty} [  H ( \tfrac{y - r v_n }{n}  ) - H ( \tfrac{x - r v_n }{n}  )  ] a(y-x)   \bar{\xi} _x^{\alpha}( \eta_s^n ) \bar{\xi} _y^{\beta}( \eta_r^n ) \Bigg\}^{2} \Bigg] \leq \widehat{C} (t-s)^{2} f_1(n),
 \end{align}
where $f_1:\mathbb{N}_+ \mapsto \mathbb{R}$ is given by
\begin{align*}
\forall n \in \mathbb{N}_+, \quad f_1(n):= \begin{dcases}
	 1/n, \quad & \gamma \in (0 , \; 1/2), \\
	\log(n)/n, \quad & \gamma=1/2, \\
	n^{2 \gamma-2}, \quad & \gamma > 1/2.
\end{dcases}
\end{align*} 
\end{cor}
\begin{proof}	
From  Lemmas \ref{lemaproxintpart},  \ref{varfarterms} and \ref{varmidterms} (for the choice $L_n=1$ for any $n$),  the expectation in \eqref{varcornlgamless1} is bounded from above by
\begin{align*}
	\widehat{C} (t-s)^{2} \Bigg\{ \frac{ [ \Theta(n) K_n ]^{2} }{n^{2}} \sum_{x=1}^{n} x^{- 2 \gamma} + \frac{ [ \Theta(n) K_n ]^{2} }{n^{2 \gamma+1}} \Bigg\} 
	 \leq \widehat{C}  (t-s)^{2} \Bigg\{  \sum_{x=1}^{n} \frac{ n^{2 \gamma} }{n^{2}} x^{- 2 \gamma}  + \frac{ 1 }{n} \Bigg\}.
\end{align*}
The  upper bound in the last display holds due to the fact that $\Theta(n) = n^{\gamma}$ for any $\gamma \in (0 , \; 1)$ and  that the sequence $(K_n)_{n \in \mathbb{N}_+}$ is uniformly bounded. Now, the rightmost sum in the last display can be bounded from above by:
\begin{align*}
\sum_{x=1}^{n} \frac{ n^{2 \gamma} }{n^{2}} x^{- 2 \gamma} =
\begin{dcases}
	\frac{1}{n} \Big\{ \frac{1}{n} \sum_{x=1}^{n} \Big( \frac{x}{n} \Big)^{- 2 \gamma} \Big\} \leq \frac{C(\gamma)}{n} \int_0^{1} u^{- 2 \gamma} du \leq \frac{C(\gamma)}{n}, \quad & \gamma \in (0 , \; 1/2), \\
	\frac{ n }{n^{2}} \sum_{x=1}^{n} x^{-1} = 	\frac{ 1 }{n} \sum_{x=1}^{n} x^{-1} \leq \frac{ 1 + \log(n) }{n}, \quad & \gamma=1/2, \\
	\frac{ n^{2 \gamma} }{n^{2}}  \sum_{x=1}^{n} x^{- 2 \gamma}  \leq \frac{ n^{2 \gamma} }{n^{2}} \sum_{x=1}^{\infty} x^{- 2 \gamma}  \leq C(\gamma) \frac{ n^{2 \gamma} }{n^{2}}, \quad & \gamma >1/2.
\end{dcases}
\end{align*}
\end{proof}
From the last result, we obtain Lemma \ref{varnlinvan} for the case $0 < \gamma < 1$.  

In the regime $\gamma \geq 1$, from  Lemmas \ref{varfarterms} and \ref{varmidterms}, it is enough to estimate the next term:
\begin{align*}
	\mathbb{E}_{\nu_{\rho}} \Bigg[ \Bigg\{ \frac{\Theta(n) K_n}{\sqrt{n}}  \int_s^t \dd r \sum_x \sum_{y=x+1}^{x+L_n-1} [  H ( \tfrac{y - r v_n }{n}  ) - H ( \tfrac{x - r v_n }{n}  )  ] a(y-x)   \bar{\xi} _x^{\alpha}( \eta_r^n ) \bar{\xi} _y^{\beta}( \eta_r^n ) \Bigg\}^{2} \Bigg],
\end{align*}
which can be rewritten as
\begin{align*}
	\mathbb{E}_{\nu_{\rho}} \Bigg[ \Bigg\{ \frac{\Theta(n) K_n}{\sqrt{n}}  \int_s^t \dd r \sum_x \sum_{z=1}^{L_n-1} [  H ( \tfrac{x +z - r v_n }{n}  ) - H ( \tfrac{x - r v_n }{n}  )  ] a(z)   \bar{\xi} _x^{\alpha}( \eta_r^n ) \bar{\xi} _{x+z}^{\beta}( \eta_r^n ) \Bigg\}^{2} \Bigg].
\end{align*}
As a consequence of Lemma \ref{lemaproxintpart}, the last expectation is bounded from above by $\widehat{C}$, times
\begin{align*}
	(s-t)^2 \frac{  \big[ \Theta(n) K_n \big]^{2}  }{ n^{4}  } + \mathbb{E}_{\nu_{\rho}} \Bigg[ \Bigg\{ \frac{\Theta(n) K_n}{\sqrt{n}} \int_s^t \dd r \sum_{x} \sum_{z=1}^{L_n-1} [  H ( \tfrac{ \lfloor x+z - r v_n \rfloor }{n}  ) - H ( \tfrac{ \lfloor x - r v_n \rfloor }{n}  )  ] a(z)   \bar{\xi} _x^{\alpha}( \eta_r^n ) \bar{\xi} _{x+z}^{\beta}( \eta_r^n ) \Bigg\}^{2} \Bigg].
\end{align*}
In order to estimate the last expectation, in the next lemma we show that it is possible to replace in the last display $\bar{\xi} _{x+z}^{\beta}( \eta_r^n )$ by $\bar{\xi} _{x+1}^{\beta}( \eta_r^n )$.

\begin{lem} \label{lem31jarabur}
Let $\gamma > 0$, $0  \leq t \leq T$ and $n \in \mathbb{N}_+$. Then
\begin{align}
	&	\mathbb{E}_{\nu_{\rho}} \Bigg[ \Bigg\{ \frac{\Theta(n) K_n}{  \sqrt{n} } \int_0^t \dd r \sum_{x} \sum_{z=1}^{L_n-1}   [  H ( \tfrac{ \lfloor x+z - r v_n \rfloor }{n}  ) - H ( \tfrac{ \lfloor x - r v_n \rfloor }{n}  )  ] a(z)   \bar{\xi}_x^{\alpha}( \eta_r^n ) \big[ \bar{\xi} _{x+z}^{\beta}( \eta_r^n ) - \bar{\xi} _{x+1}^{\beta}( \eta_r^n ) \big] \Bigg\}^{2} \Bigg] \label{varlem36} \\
	&\quad \quad \quad\quad \leq  C t (  K_n )^2 f_2(n), \nonumber
\end{align} 
where $C$ is a constant independent of $n$ and $t$; and $f_2:\mathbb{N}_+ \mapsto \mathbb{R}$ is given by
\begin{align} \label{deff2n}
\forall n \in \mathbb{N}_+, \quad f_2(n):= \begin{dcases}
    (L_n/n)^{2 - \gamma}, \quad & 0 <\gamma  <2, \\
  \log(L_n)  / \log(n), \quad & \gamma  =2, \\
 1, \quad & \gamma > 2. 
\end{dcases}
\end{align}
	Furthermore, if $H \in C_c^1(\mathbb{R})$, the role of $C$ is fulfilled by some constant $C(H)$ satisfying \eqref{Cinvshift}.
\end{lem}
\begin{proof}
From Lemma 4.3 in \cite{CLO}, the expectation in \eqref{varlem36} is bounded from above by a constant independent of $n$, $t$ and $H$, times 
\begin{equation} \label{suplem36}
\begin{split} \int_0^{t} \dd r \sup_{g \in L^2(\nu_{\rho} ) } \Bigg\{ &2 \frac{\Theta(n) K_n}{  \sqrt{n} } \sum_{x} \sum_{z=1}^{L_n-1}   [  H ( \tfrac{ \lfloor x+z - r v_n \rfloor }{n}  ) - H ( \tfrac{ \lfloor x - r v_n \rfloor }{n}  )  ] \times \\ &\times a(z) \int_{\Omega} \bar{\xi}_x^{\alpha}( \eta ) \big[ \bar{\xi}_{x+z}^{\beta}( \eta ) - \bar{\xi} _{x+1}^{\beta}( \eta ) \big] g(\eta) d \nu_{\rho} (\eta)-  \langle g, - \mcb S^n g \rangle_{\nu_{\rho}} \Bigg\},
\end{split}
\end{equation}
where $\mcb S^n$ is the symmetric part of the generator $\mcb L^n$ given by \eqref{genABC}. In the last display we applied the fact that $\langle g,  \mcb L^n g \rangle_{\nu_{\rho}} = \langle g, \mcb S^n g \rangle_{\nu_{\rho}}$, for any $g \in L^2(\nu_{\rho} )$. Performing the change of variables $\eta_1:=\eta^{x+1,x+z}$, for which the measure is invariant,  we get
\begin{align*}
& \int_{\Omega} \bar{\xi}_x^{\alpha}( \eta ) \big[ \bar{\xi}_{x+z}^{\beta}( \eta ) - \bar{\xi} _{x+1}^{\beta}( \eta ) \big] g(\eta) \dd \nu_{\rho} (\eta)= \int_{\Omega} \bar{\xi}_x^{\alpha}( \eta )  \bar{\xi}_{x+1}^{\beta}( \eta )  \big[g(\eta^{x+1,x+z}) - g(\eta) \big] \dd \nu_{\rho} (\eta).
\end{align*}
From Young's inequality, for every $x \in \mathbb{Z}$ and $z \in \{1, \ldots, L_n-1 \}$, it holds
\begin{align*}
&\Bigg| \int_{\Omega} \bar{\xi}_x^{\alpha}( \eta )  \bar{\xi}_{x+1}^{\beta}( \eta )  \big[g(\eta^{x+1,x+z}) - g(\eta) \big] \dd \nu_{\rho} (\eta) \Bigg| \leq \frac{\beta(n,r,x,z)}{2} + \frac{I_{x+1,x+z}(g,  \nu_{\rho} )}{2 \beta(n,r,x,z)},
\end{align*}
for every $\beta(n,r,x,z) >0$, where 
\begin{equation} \label{defIywg}
I_{y,w}(g,  \nu_{\rho} ):= \int_{\Omega}  \big[g(\eta^{y,w}) - g(\eta) \big]^2 \dd \nu_{\rho} (\eta).
\end{equation} 
Now, using  Lemma 3.2 in \cite{GJburgers}, namely equation (3.12) in \cite{GJburgers} we conclude  that
\begin{align*}
\Bigg|  \int_{\Omega} \bar{\xi}_x^{\alpha}( \eta ) \big[ \bar{\xi}_{x+z}^{\beta}( \eta ) - \bar{\xi} _{x+1}^{\beta}( \eta ) \big] g(\eta) \; \dd \nu_{\rho} \Bigg| \leq \frac{\beta(n,r,x,z)}{2} + \frac{4z -7}{2 \beta(n,r,x,z)} \sum_{y=x+1}^{x+z-1} I_{y,y+1}(g,  \nu_{\rho} ).
\end{align*} 
By choosing $\beta(n,r,x,z)= C_n |  H ( \tfrac{ \lfloor x+z - \hat{r} v_n \rfloor }{n}  ) - H ( \tfrac{ \lfloor x - \hat{r} v_n \rfloor }{n}  )  |$ (for some constant $C_n$, depending only on $n$, to be chosen later), and making a change of variables $w= \lfloor x - \hat{r} v_n \rfloor$,  the leftmost term in the supremum in \eqref{suplem36} is bounded from above by 
\begin{align}
 \frac{\Theta(n) K_n}{  \sqrt{n} } \Bigg\{ C_n \sum_{w} \sum_{z=1}^{L_n-1}   |  H ( \tfrac{ w+z  }{n}  ) - H ( \tfrac{ w }{n}  )  |^2 \; | a(z) | + \frac{2}{ C_n } \sum_{z=1}^{L_n-1}   | a(z) | z^2  \sum_{w} I_{w,w+1}(g,  \nu_{\rho} ) \Bigg\}. \label{ub36a1}
\end{align}
In Appendix \ref{useest7}, we prove that
\begin{equation} \label{claimuseest7}
   \sum_{w} \sum_{z=1}^{L_n-1}   |  H ( \tfrac{ w+z  }{n}  ) - H ( \tfrac{ w }{n}  )  |^2 \; | a(z) | \leq  \frac{\widetilde{C}_H}{n} C(\gamma) \sum_{z=1}^{L_n-1} z^{1-\gamma},
\end{equation}
where $\widetilde{C}_H$ is some constant depending only on $H$. Moreover, if $H \in C_c^1(\mathbb{R})$, the role of $C$ is fulfilled by some constant $C(H)$ satisfying \eqref{Cinvshift}. We stress that in \eqref{claimuseest7} and in the remainder of this proof, $\widetilde{C}_H$ is a \textit{fixed} constant. Now from \eqref{claimuseest7}, the display in \eqref{ub36a1} is bounded from above by
\begin{align} \label{ub36a2}
   \frac{\Theta(n) K_n}{  \sqrt{n} }  \sum_{z=1}^{L_n-1}  z^{1-\gamma}  \Bigg\{  \frac{C_n  C(\gamma) C_H  }{n}  + \frac{  C(\gamma)  }{ C_n }   \sum_{w} I_{w,w+1}(g,  \nu_{\rho} ) \Bigg\}, 
\end{align}
Moreover, from \eqref{genABC} and \eqref{lbrate}, we get that $\langle g, - \mcb S^n g \rangle_{\nu_{\rho}}$ can be rewritten as 
\begin{align} \label{lbnddirg}
\frac{\Theta(n)}{2} \sum_{y,w} s(w-y) \int_{\Omega} r^n_{y,w}(\eta) \big[g(\eta^{y,w}) - g(\eta) \big]^2 \dd \nu_{\rho} (\eta) \geq \Theta(n) C_1(\gamma) \sum_{w}  I_{w, w+1}(g,  \nu_{\rho} ),   
\end{align}
From \eqref{lbnddirg} and \eqref{ub36a2}, we get that the display inside the supremum in \eqref{suplem36} is bounded from above by
\begin{align} \label{ub36a3}
    C(\gamma) C_H   \frac{\Theta(n) K_n}{  \sqrt{n} } \sum_{z=1}^{L_n-1} z^{1-\gamma} \frac{C_n}{n} + \Theta(n)  \sum_{w}  I_{w, w+1}(g,  \nu_{\rho} ) \Bigg\{  C(\gamma)  \frac{ K_n}{  \sqrt{n} C_n } \sum_{z=1}^{L_n-1} z^{1-\gamma}   -  C_1(\gamma)  \Bigg\}.  
\end{align}
By choosing $C_n= K_n C(\gamma)  [ \sqrt{n} C_1(\gamma) ]^{-1}  \sum_{z=1}^{L_n-1} z^{1-\gamma},
$
the display in \eqref{ub36a3} can be rewritten as
\begin{align*}
\frac{ \Theta(n) }{n   C_1(\gamma)   }  C(\gamma)   \Bigg\{  C(\gamma) C_H   \frac{ K_n}{  \sqrt{n}  } \sum_{z=1}^{L_n-1} z^{1-\gamma} \Bigg\}^2 \leq  C(\gamma) C_H   (K_n)^2  f_2(n),
\end{align*}
where $f_2$ is given by \eqref{deff2n}. This ends the proof of \eqref{varlem36}.
\end{proof}
After replacing $\bar{\xi} _x^{\alpha}( \eta_r^n ) \bar{\xi} _{x+z}^{\beta}( \eta_r^n )$ by $\bar{\xi} _x^{\alpha}( \eta_r^n ) \bar{\xi} _{x+1}^{\beta}( \eta_r^n )$, we are reduced to estimating
\begin{align*}
	\mathbb{E}_{\nu_{\rho}} \Bigg[ \Bigg\{ \frac{\Theta(n) K_n}{\sqrt{n}} \sum_{z=1}^{L_n-1} a(z) \int_0^t \dd r \sum_{x}  [  H ( \tfrac{ \lfloor x+z - r v_n \rfloor }{n}  ) - H ( \tfrac{ \lfloor x - r v_n \rfloor }{n}  )  ]    \bar{\xi} _x^{\alpha}( \eta_r^n ) \bar{\xi} _{x+1}^{\beta}( \eta_r^n ) \Bigg\}^{2} \Bigg].
\end{align*}
From a  second-order Taylor expansion of $H$ we are reduce to estimate the next two terms:
\begin{align}
	& \mathbb{E}_{\nu_{\rho}} \Bigg[ \Bigg\{ \frac{\Theta(n) K_n}{2 n^2 \sqrt{n} } \sum_{z=1}^{L_n-1}  z^{2} a(z)  \int_0^t \dd r \sum_{x}  \Delta H (\omega_{x,z}^{r,n})     \bar{\xi} _x^{\alpha}( \eta_r^n ) \bar{\xi} _{x+1}^{\beta}( \eta_r^n ) \Bigg\}^{2} \Bigg], \label{varxznear2order} \\
	& \mathbb{E}_{\nu_{\rho}} \Bigg[ \Bigg\{ \frac{\Theta(n) K_n}{n \sqrt{n} } \sum_{z=1}^{L_n-1} za(z) \int_0^t \dd r \sum_{x}   \nabla H ( \tfrac{ \lfloor x - r v_n \rfloor }{n}  )      \bar{\xi} _x^{\alpha}( \eta_r^n ) \bar{\xi} _{x+1}^{\beta}( \eta_r^n ) \Bigg\}^{2} \Bigg], \label{varxznear}
\end{align}
where 
 $\omega_{x,z}^{r,n} \in (\lfloor x - r v_n \rfloor / n, \; \lfloor x +z- r v_n \rfloor / n) $. 
 Now we state a lemma that allows us to estimate the expectation in \eqref{varxznear2order}. 
\begin{lem} \label{lem46}
	Let $\gamma \geq 1$, $0 \leq s \leq t$ and $n \in \mathbb{N}_+$. Then the expectation in \eqref{varxznear2order} is bounded from above by
	\begin{align}
		 \widehat{C} (t-s)^{2} \frac{\big[ \Theta(n) K_n \big]^{2} }{n^{4} } \Bigg\{1 + \int_1^{L_n} u^{2 - 2 \gamma} \; \dd u \Bigg\}. \label{upbndlem46}
	\end{align} 
	Furthermore, if $H \in C_c^2(\mathbb{R})$, the role of $\widehat{C}$ is fulfilled by some constant $C(H)$ satisfying \eqref{Cinvshift}. 
\end{lem}
\begin{proof}
In Appendix \ref{useest8}, we prove that under $\gamma \geq 1$,
\begin{equation} \label{claimuseest8}
\sup_{r \in [0,T]} \Bigg\{ \sum_{x} \sum_{z=1}^{L_n-1} \big( z^{2} a(z) \Delta H (\omega_{x,z}^{r,n}) ^2 \big) \Bigg\} \leq n  \widehat{C} \Bigg\{1 + \int_1^{L_n} u^{2 - 2 \gamma} \; \dd u \Bigg\},  
\end{equation}
where the role of $\widehat{C}$ is fulfilled by some constant $C(H)$ satisfying \eqref{Cinvshift} for $H \in C_c^2(\mathbb{R})$. Combining \eqref{claimuseest8} with Lemma \ref{CSfubest}, the proof ends. 
\end{proof}
In order to treat the expectation in \eqref{varxznear}, we will replace $\bar{\xi}_x^{\alpha}( \eta_r^n ) \bar{\xi} _{x+1}^{\beta}( \eta_r^n )$ by a convenient microscopic quantity, in an analogous way as it is done in (6.1) of \cite{ABC}. Keeping this in mind, for every $L \in \mathbb{N}$ and $x \in \mathbb{Z}$, define the applications $\overleftarrow{\xi}_x^{\alpha, L}, \overrightarrow{\xi}_x^{\beta, L}, \Psi_{x,L}^{\alpha,\beta}: \Omega  \mapsto [-1, \;1]$ by  
\begin{equation} \label{defpsixL}
	\begin{split}
		\overleftarrow{\xi}_x^{\alpha, L}(\eta):=& \frac{1}{L} \sum_{j=1}^{L}\bar{\xi} _{x-j}^{\alpha}( \eta ), \quad \overrightarrow{\xi}_x^{\beta, L}:=\frac{1}{L} \sum_{j=1}^{L}\bar{\xi} _{x+j}^{\beta}( \eta ), \quad \Psi_{x,L}^{\alpha,\beta}(\eta):=\overleftarrow{\xi}_x^{\alpha, L}(\eta)   \overrightarrow{\xi}_x^{\beta, L}(\eta).
	\end{split}
\end{equation}
Intuitively, $\overleftarrow{\xi}_x^{\alpha, L}$, resp. $\overrightarrow{\xi}_x^{\beta, L}$, is an average over a box immediately to the left of $x$, resp. to the right of $x$. 

For an application $F:\Omega \mapsto \mathbb{R}$ we denote its support by  $\supp(F)$ which is defined as the minimal set $\widetilde{\Lambda} \subset \mathbb{Z}$ such that 
for all $ y \in \widetilde{\Lambda}$ and for all $ \eta_1, \eta_2 \in \Omega$ such that $ \eta_1(y)= \eta_2(y)$ then $ F(\eta_1)=F(\eta_2)$. 


We stress that the $\supp \big( \overleftarrow{\xi}_x^{\alpha, L} \big) \cap \supp \big( \overrightarrow{\xi}_x^{\beta, L} \big) = \varnothing$, thus $\overleftarrow{\xi}_x^{\alpha, L}(\eta_r^n)$ and the functions in \eqref{defpsixL} have all mean zero. 
Moreover, for every $x \in \mathbb{Z}$ every $r \in [0,T]$, every $n \in \mathbb{N}_+$ and every $L \in \mathbb{N}_+$, it holds
\begin{align}
	\mathbb{E}_{\nu_{\rho}} \big[ \big\{ \Psi_{x,L}^{\alpha,\beta}(\eta_r^n) \big\}^{2} \big] = & \frac{1}{L^{4}} \sum_{j_1=1}^{L} \sum_{j_2=1}^{L} \sum_{k_1=1}^{L} \sum_{k_2=1}^{L} \mathbb{E}_{\nu_{\rho}} \big[  \bar{\xi} _{x-j_1}^{\alpha}(\eta_r^n) \bar{\xi} _{x-j_2}^{\alpha}(\eta_r^n) \bar{\xi} _{x+k_1}^{\beta}(\eta_r^n) \bar{\xi} _{x+k_2}^{\beta}(\eta_r^n) \big] \nonumber \\
= & \frac{1}{L^{4}} \sum_{j=1}^{L} \sum_{k=1}^{L} \mathbb{E}_{\nu_{\rho}} \big[ \big\{ \bar{\xi} _{x-j}^{\alpha}(\eta_r^n) \big\}^{2} \big\{ \bar{\xi} _{x+k}^{\beta}(\eta_r^n) \big\}^{2} \big] \leq \frac{\chi(\rho_{\alpha}) \chi(\rho_{\beta})}{L^{2}}. 	 \label{varpsi}
\end{align}
The second equality in the last display holds due to \eqref{corrzero}. From $(u+v)^{2} \leq 2(u^{2}+v^{2})$,} in order to estimate the term in \eqref{varxznear}, it is enough to obtain an upper bound for
\begin{align}
	&	\mathbb{E}_{\nu_{\rho}} \Bigg[ \Bigg\{ \frac{\Theta(n) K_n}{n \sqrt{n} } \sum_{z=1}^{L_n-1}  za(z) \int_0^t \dd r \sum_{x}   \nabla H ( \tfrac{ \lfloor x - r v_n \rfloor }{n}  )   \Psi_{x,L_n}^{\alpha,\beta}(\eta_r^n)  \Bigg\}^{2} \Bigg], \label{varxzmed} \\
	&	\mathbb{E}_{\nu_{\rho}} \Bigg[\Bigg\{ \frac{\Theta(n) K_n}{n \sqrt{n} } \sum_{z=1}^{L_n-1} za(z)  \int_0^t \dd r \sum_{x}   \nabla H ( \tfrac{ \lfloor x - r v_n \rfloor }{n}  )   \big[  \bar{\xi} _x^{\alpha}( \eta_r^n ) \bar{\xi} _{x+1}^{\beta}( \eta_r^n ) - \Psi_{x,L_n}^{\alpha,\beta}(\eta_r^n) \big] \Bigg\}^{2} \Bigg]. \label{varxznearrep}
\end{align}
Before treating the term in the first line of the last display, we state a lemma which will be useful later.
\begin{lem} \label{lemsupCS}
	Let $\gamma >0$, $0 \leq s \leq t \leq T$, $n \in \mathbb{N}_+$ and $ \widehat{b}_n \in \mathbb{R}$. Moreover, let $(F_x)_{x \in \mathbb{Z}}: \Omega \mapsto \mathbb{R}$ be a sequence of mappings such that 
\begin{align}
 \forall x \in \mathbb{Z}, \; \forall s \in [0,T], \; \forall n \in \mathbb{N}_+, \quad & \mathbb{E}_{\nu_{\rho}} \big[ \big\{ F_x(\eta_s^n) \big\}^{2} \big] \leq \frac{C_1}{L^{2}}, \label{hiplemsupCS1} \\
 \forall x,y \in \mathbb{Z}: |y-x| \geq 2 K, \; \forall s \in [0,T], \; \forall n \in \mathbb{N}_+, \quad & \mathbb{E}_{\nu_{\rho}} \big[ F_x(\eta_r^n) \; F_y(\eta_r^n) \big]=0. \label{hiplemsupCS2}
\end{align}	
	for some constants $C_1, L>0$ and $K \in \mathbb{N}_+$. Recall the definition of $\| \cdot \|_{2,n}$ in \eqref{L2discnorm}. It holds
	\begin{align*}
\forall n \in \mathbb{N}_+, \quad		\mathbb{E}_{\nu_{\rho}} \Bigg[ \Bigg\{ \widehat{b}_n \int_s^t \dd r \sum_{x}   \nabla H ( \tfrac{ \lfloor x - r v_n \rfloor }{n}  )    F_x(\eta_r^n)  \Bigg\}^{2} \Bigg] \leq 2 C_1 (t-s)^{2} n \frac{(\widehat{b}_n)^2 K}{L^{2}} \| \nabla H \|^{2}_{2,n}.
	\end{align*}
\end{lem} 
\begin{proof}
From the Cauchy-Schwarz inequality and Fubini's Theorem, the expectation in the last line is bounded from above by
\begin{align}
& ( \widehat{b}_n)^2 (t-s) \int_s^t \dd r \; \mathbb{E}_{\nu_{\rho}} \Bigg[ \Bigg\{ \sum_{x}   \nabla H ( \tfrac{ \lfloor x - r v_n \rfloor }{n}  )    F_x(\eta_r^n)  \Bigg\}^{2} \Bigg]. \label{ublemsupCS1}
\end{align}
Next, for every $j \in \{1, 2, \ldots, 2 K\}$, we define $\mathbb{A}_j$, the subset of $\mathbb{Z}$ given by
\begin{align*}
\mathbb{A}_j:=\big\{ y \in \mathbb{Z}: \quad \exists q \in \mathbb{Z}: y = 2K q + j-1 \big\}.
\end{align*}
Therefore, the expectation in \eqref{ublemsupCS1} can be rewritten as
\begin{align*}
&   \mathbb{E}_{\nu_{\rho}} \Bigg[ \Bigg\{ \sum_{j=1}^{2K} \sum_{x \in \mathbb{A}_j}   \nabla H ( \tfrac{ \lfloor x - r v_n \rfloor }{n}  )    F_x(\eta_r^n)  \Bigg\}^{2} \Bigg] \leq 2K    \Bigg\{  \sum_{j=1}^{2K} \sum_{x,y \in \mathbb{A}_j} \nabla H ( \tfrac{ \lfloor x - r v_n \rfloor }{n}  ) \nabla H ( \tfrac{ \lfloor y - r v_n \rfloor }{n}  ) \mathbb{E}_{\nu_{\rho}} \big[ F_x(\eta_r^n) \; F_y(\eta_r^n) \big] \Bigg\}, \\
\end{align*}
where the inequality comes from the Cauchy-Schwarz inequality. From \eqref{hiplemsupCS2}, the rightmost term in the last display can be rewritten as
\begin{align*}
& 2K   \Bigg\{  \sum_{j=1}^{2K} \sum_{x \in \mathbb{A}_j} \big[ \nabla H ( \tfrac{ \lfloor x - r v_n \rfloor }{n}  ) \big]^2  \mathbb{E}_{\nu_{\rho}} \big[ \{ F_x(\eta_r^n) \}^2 \big] \Bigg\} \leq  2K   \Bigg\{  \sum_{x} \big[ \nabla H ( \tfrac{ \lfloor x - r v_n \rfloor }{n}  ) \big]^2 \frac{C_1}{L^2} \Bigg\}, \\
\end{align*}
where the inequality comes from \eqref{hiplemsupCS1}. Making the change of variables $y = \lfloor x - r v_n \rfloor$, we conclude that the term in \eqref{ublemsupCS1} is bounded from above by
\begin{align*}
 \frac{2KC_1}{L^2} ( \widehat{b}_n)^2 (t-s) \int_s^t \dd r  \Bigg\{  \sum_{y} \big[ \nabla H ( \tfrac{ y }{n}  ) \big]^2  \Bigg\} = &\frac{2KC_1}{L^2} ( \widehat{b}_n)^2 (t-s)^2 n \Bigg\{ \frac{1}{n} \sum_{y} \big[ \nabla H ( \tfrac{ y }{n}  ) \big]^2  \Bigg\}.
\end{align*}
\end{proof}
Combining the last result with  \eqref{varpsi} and the fact that $\supp\big( \Psi_{x,L_n}^{\alpha,\beta} \big) \cap \supp\big( \Psi_{y,L_n}^{\alpha,\beta} \big) = \varnothing$ whenever $|y-x| \geq 2L_n$, we obtain the following corollary, which allows us to estimate the expectation in \eqref{varxzmed}.
\begin{cor} \label{cor410}
	Let $\gamma >0$, $0 \leq s \leq t \leq T$, $n \in \mathbb{N}_+$ and $ \widehat{b}_n \in \mathbb{R}$. Then
	\begin{align} \label{cor28a}
		&	\mathbb{E}_{\nu_{\rho}} \Bigg[ \Bigg\{ \widehat{b}_n \int_s^t \dd r \sum_{x}   \nabla H ( \tfrac{ \lfloor x - r v_n \rfloor }{n}  )    \Psi_{x,L_n}^{\alpha,\beta}(\eta_r^n)  \Bigg\}^{2} \Bigg] \leq  2 \chi(\rho_{\alpha}) \chi(\rho_{\beta}) (t-s)^{2} n \frac{(\widehat{b}_n)^2 }{L_n} \| \nabla H \|^{2}_{2,n}.
	\end{align}
\end{cor}
The reader likely noticed that we need to choose $\widehat{b}_n=\Theta(n) K_n n^{-3/2} \sum_{z=1}^{L_n-1} za(z)$ in \eqref{cor28a}, in order to obtain an upper bound for the term in \eqref{varxzmed}. 

In order to treat the expectation in \eqref{varxznearrep}, we make use of {well-known} \textit{blocks} estimates adapted to our setting.  Namely, we first replace $\bar{\xi} _x^{\alpha}( \eta_r^n ) \bar{\xi} _{x+1}^{\beta}( \eta_r^n )$ by $\Psi_{x, \ell_n}^{\alpha,\beta}(\eta_r^n)$, where $1 \leq \ell_n < L_n$. We then replace $\Psi_{x, \ell_n}^{\alpha,\beta}(\eta_r^n)$ by $\Psi_{x,L_n}^{\alpha,\beta}(\eta_r^n)$. The first step is going to be achieved by making use of the following result, which is proved in Appendix \ref{prooflemobe}.

\begin{lem} [\textbf{One-block estimate}] \label{lemobe}
	Let $\gamma >0$, $0 \leq t \leq T$, $n \in \mathbb{N}_+$ and $ \widehat{b}_n \in \mathbb{R}$. Then
	\begin{align*}
		\mathbb{E}_{\nu_{\rho}} \Bigg[ \Bigg\{ \widehat{b}_n \int_0^t \dd r \sum_{x}   \nabla H ( \tfrac{ \lfloor x - r v_n \rfloor }{n}  )    \big[  \bar{\xi} _x^{\alpha}( \eta_r^n ) \bar{\xi} _{x+1}^{\beta}( \eta_r^n ) - \Psi_{x, \ell_n}^{\alpha,\beta}(\eta_r^n) \big]  \Bigg\}^{2} \Bigg] \leq C t  \frac{ n (\widehat{b}_n \ell_n)^2  }{\Theta(n)}  \| \nabla H \|^{2}_{2,n},
	\end{align*}
	where $C$ is a constant independent of $H$, $n$ and $t$. 	
\end{lem}
Now we state the two-blocks estimate, which is proved in Appendix \ref{prooflemtbe}.
\begin{lem} [\textbf{Two-blocks estimate}] \label{lemtbe}
	Let $\gamma >0$, $n \in \mathbb{N}_+$ and $ \widehat{b}_n \in \mathbb{R}$. For $1 \leq \ell_n \leq  L_n$, it holds
	\begin{align*}
	\mathbb{E}_{\nu_{\rho}} \Bigg[ \Bigg\{ \widehat{b}_n \int_0^t \dd r \sum_{x}   \nabla H ( \tfrac{ \lfloor x - r v_n \rfloor }{n}  )    \big[   \Psi_{x, \ell_n}^{\alpha,\beta}(\eta_r^n) - \Psi_{x, L_n}^{\alpha,\beta}(\eta_r^n)  \big]  \Bigg\}^{2} \Bigg] \leq C t  \frac{ n (\widehat{b}_n )^2 L_n^{\gamma-1}  }{\Theta(n)}  \| \nabla H \|^{2}_{2,n},
	\end{align*}
for any $t \in [0,T]$, where $C$ is a constant independent of $H$, $n$ and $t$. 	
\end{lem}
At this point we recall that in order to prove Lemmas \ref{varnlinvan} and \ref{varnlintight}, we need to obtain upper bounds (in time) of the form $(t-s)^{1 + \delta_1}$, for some $\delta_1 >0$. However, the upper bounds provided by Lemmas \ref{lem31jarabur}, \ref{lemobe} and \ref{lemtbe} are (on time) linear. In order to solve this issue, we now obtain bounds for the respective expectations, but which are quadratic on time. We begin by stating the estimate corresponding to Lemma \ref{lem31jarabur}.
\begin{lem} \label{lemrepxzx1}
	Let $\gamma \geq 1$, $0 \leq s \leq t \leq T$ and $n \in \mathbb{N}_+$. Then
	\begin{align}
		&	\mathbb{E}_{\nu_{\rho}} \Bigg[ \Bigg\{ \frac{\Theta(n) K_n}{  \sqrt{n} } \int_s^t \dd r \sum_{x} \sum_{z=1}^{L_n-1}   [  H ( \tfrac{ \lfloor x+z - r v_n \rfloor }{n}  ) - H ( \tfrac{ \lfloor x - r v_n \rfloor }{n}  )  ] a(z)   \bar{\xi}_x^{\alpha}( \eta_r^n ) \big[ \bar{\xi} _{x+z}^{\beta}( \eta_r^n ) - \bar{\xi} _{x+1}^{\beta}( \eta_r^n ) \big] \Bigg\}^{2} \Bigg] \nonumber \\
		\leq & \widehat{C}(t-s)^2 \Bigg( \frac{\Theta(n) K_n}{n} \Bigg)^{2} \big\{ \mathbbm{1}_{ \{ \gamma > 1 \} } + \big[ \log(L_n) \big]^{2} \mathbbm{1}_{ \{ \gamma = 1 \} } \big\}. \label{uboundlemrepxzx1}
	\end{align} 
	Furthermore, if $H \in C_c^2(\mathbb{R})$, the role of $\widehat{C}$ is fulfilled by some constant $C(H)$ satisfying \eqref{Cinvshift}.
\end{lem}
\begin{proof}
In Appendix \ref{useest9}, we prove that under $\gamma \geq 1$,
\begin{equation} \label{claimuseest9}
\begin{split}
& \sum_{y} \sum_{z=1}^{L_n-1}   \big[  H ( \tfrac{ y+z }{n}  ) - H ( \tfrac{ y }{n}  )  \big]^{2} a^2(z) +  \sum_{y} \Bigg\{ \sum_{z=1}^{L_n-1}   [  H ( \tfrac{ y + z }{n}  ) - H ( \tfrac{y }{n}  )  ] a(z)  \Bigg\}^2 \\
\leq &\widehat{C}  \big\{ \mathbbm{1}_{ \{ \gamma > 1 \} } + \big[ \log(L_n) \big]^{2} \mathbbm{1}_{ \{ \gamma = 1 \} } \big\}, 
\end{split}
\end{equation}
where the  role of $\widehat{C}$ is fulfilled by some constant $C(H)$ satisfying \eqref{Cinvshift} for $H \in C_c^2(\mathbb{R})$. 

From a convex inequality,  the expectation in the statement  is bounded from above by
\begin{align}
	& 2 \mathbb{E}_{\nu_{\rho}} \Bigg[ \Bigg\{ \frac{\Theta(n) K_n}{  \sqrt{n} } \int_s^t \dd r \sum_{x} \sum_{z=1}^{L_n-1}   [  H ( \tfrac{ \lfloor x+z - r v_n \rfloor }{n}  ) - H ( \tfrac{ \lfloor x - r v_n \rfloor }{n}  )  ] a(z)   \bar{\xi}_x^{\alpha}( \eta_r^n )  \bar{\xi} _{x+z}^{\beta}( \eta_r^n )  \Bigg\}^{2} \Bigg] \label{lemrepxzx1a} \\
	+ & 2 \mathbb{E}_{\nu_{\rho}} \Bigg[ \Bigg\{ \frac{\Theta(n) K_n}{  \sqrt{n} } \int_s^t \dd r \sum_{x} \sum_{z=1}^{L_n-1}   [  H ( \tfrac{ \lfloor x+z - r v_n \rfloor }{n}  ) - H ( \tfrac{ \lfloor x - r v_n \rfloor }{n}  )  ] a(z)   \bar{\xi}_x^{\alpha}( \eta_r^n )   \bar{\xi} _{x+1}^{\beta}( \eta_r^n )  \Bigg\}^{2} \Bigg].  \nonumber 
	\end{align}
	From Lemma \ref{CSfubest}, the term in the last line of the last display can be rewritten as
		\begin{align}
	 2 \chi(\rho_{\alpha}) \chi(\rho_{\beta}) (t-s)^{2} \frac{ \big[ \Theta(n) K_n \big]^2 }{n}  \sum_{y} \sum_{z=1}^{L_n-1}   \big[  H ( \tfrac{ y+z }{n}  ) - H ( \tfrac{ y }{n}  )  \big]^{2} a^2(z).   \label{lemrepxzx1a3}
	\end{align}
	On the other hand, from the Cauchy-Schwarz inequality, Fubini's Theorem and \eqref{corrzero},  the term in \eqref{lemrepxzx1a} is bounded from above by
      \begin{align}
	 & 2 \chi(\rho_{\alpha}) \chi(\rho_{\beta}) (t-s)^{2} \frac{ \big[ \Theta(n) K_n \big]^2 }{n}  \sum_{y} \Bigg\{ \sum_{z=1}^{L_n-1}   [  H ( \tfrac{ y + z }{n}  ) - H ( \tfrac{y }{n}  )  ] a(z)  \Bigg\}^2. \label{lemrepxzx1a2}
	\end{align}
Combining \eqref{lemrepxzx1a3} and \eqref{lemrepxzx1a2} with \eqref{claimuseest9}, the proof ends. 
    \end{proof}
We proceed by stating the estimate corresponding to Lemma \ref{lemobe}.	
	\begin{lem}  \label{lemobe2}
		Let $\gamma >0$, $0 \leq s \leq t \leq T$, $n \in \mathbb{N}_+$ and $ \widehat{b}_n \in \mathbb{R}$. Denote $C_{\alpha,\beta}:=8 \chi(\rho_{\alpha}) \chi(\rho_{\beta})$. Then
		\begin{align*}
			\mathbb{E}_{\nu_{\rho}} \Bigg[ \Bigg\{ \widehat{b}_n \int_s^t \dd r \sum_{x}   \nabla H ( \tfrac{ \lfloor x - r v_n \rfloor }{n}  )    \big[  \bar{\xi} _x^{\alpha}( \eta_r^n ) \bar{\xi} _{x+1}^{\beta}( \eta_r^n ) - \Psi_{x, \ell_n}^{\alpha,\beta}(\eta_r^n) \big]  \Bigg\}^{2} \Bigg] \leq C_{\alpha,\beta} (t-s)^2 n (\widehat{b}_n)^2 \| \nabla H \|^{2}_{2,n}.
		\end{align*}	
	\end{lem}
	\begin{proof}
From a convex inequality, the expectation in the statement is bounded from above by
\begin{align}
& 2 \mathbb{E}_{\nu_{\rho}} \Bigg[ \Bigg\{ \widehat{b}_n \int_s^t \dd r \sum_{x}   \nabla H ( \tfrac{ \lfloor x - r v_n \rfloor }{n}  )     \Psi_{x, \ell_n}^{\alpha,\beta}(\eta_r^n)   \Bigg\}^{2} \Bigg] \nonumber \\
+ & 2 \mathbb{E}_{\nu_{\rho}} \Bigg[ \Bigg\{ \widehat{b}_n \int_s^t \dd r \sum_{x}   \nabla H ( \tfrac{ \lfloor x - r v_n \rfloor }{n}  )    \bar{\xi} _x^{\alpha}( \eta_r^n ) \bar{\xi} _{x+1}^{\beta}( \eta_r^n )   \Bigg\}^{2} \Bigg] \nonumber \\
\leq & \frac{C_{\alpha,\beta}}{2}  (t-s)^{2} n \frac{(\widehat{b}_n)^2 }{\ell_n} \| \nabla H \|^{2}_{2,n}+2 \mathbb{E}_{\nu_{\rho}} \Bigg[ \Bigg\{ \widehat{b}_n \int_s^t \dd r \sum_{x}   \nabla H ( \tfrac{ \lfloor x - r v_n \rfloor }{n}  )   \bar{\xi} _x^{\alpha}( \eta_r^n ) \bar{\xi} _{x+1}^{\beta}( \eta_r^n ) \Bigg\}^{2} \Bigg]. \label{explemobe2a}
\end{align}	
The inequality in the last line is due to \eqref{cor28a}. Now for every $x \in \mathbb{Z}$, define $F_x: \Omega \mapsto \mathbb{R}$ by 
\begin{align*}
\forall \eta \in \Omega, \quad F_x(\eta):= \bar{\xi} _x^{\alpha}( \eta ) \bar{\xi} _{x+1}^{\beta}( \eta ).
\end{align*}
From Remark \ref{remind}, we get that \eqref{hiplemsupCS1} and \eqref{hiplemsupCS2} are satisfied for $K=L=1$. Combining this with Lemma \ref{lemsupCS}, the expectation in \eqref{explemobe2a} is bounded from above by $C_{\alpha,\beta} (t-s)^2 n (\widehat{b}_n)^2 \| \nabla H \|^{2}_{2,n} / 4$. Since $\ell_n \geq 1$, the proof ends.
	\end{proof}
Next we state the estimate corresponding to Lemma \ref{lemtbe}.	
	\begin{lem}  \label{lemtbe2}
		Let $\gamma >0$, $0 \leq s \leq t \leq T$, $n \in \mathbb{N}_+$ and $ \widehat{b}_n \in \mathbb{R}$. Denote $C_{\alpha,\beta}:=8 \chi(\rho_{\alpha}) \chi(\rho_{\beta})$. For $1 \leq \ell_n \leq L_n$, it holds
		\begin{align*}
			\mathbb{E}_{\nu_{\rho}} \Bigg[ \Bigg\{ \widehat{b}_n \int_s^t \dd r \sum_{x}   \nabla H ( \tfrac{ \lfloor x - r v_n \rfloor }{n}  )    \big[   \Psi_{x, \ell_n}^{\alpha,\beta}(\eta_r^n) - \Psi_{x,L_n}^{\alpha,\beta}(\eta_r^n)  \big]  \Bigg\}^{2} \Bigg] \leq C_{\alpha,\beta} (t-s)^{2} n \frac{(\widehat{b}_n)^2}{\ell_n} \| \nabla H \|^{2}_{2,n}.
		\end{align*}	
	\end{lem}
	\begin{proof}
From a convex inequality and \eqref{cor28a}, the expectation in the statement is bounded from above by
\begin{align*}
& 2 \mathbb{E}_{\nu_{\rho}} \Bigg[ \Bigg\{ \widehat{b}_n \int_s^t \dd r \sum_{x}   \nabla H ( \tfrac{ \lfloor x - r v_n \rfloor }{n}  )     \Psi_{x, \ell_n}^{\alpha,\beta}(\eta_r^n)   \Bigg\}^{2} \Bigg]  +  2 \mathbb{E}_{\nu_{\rho}} \Bigg[ \Bigg\{ \widehat{b}_n \int_s^t \dd r \sum_{x}   \nabla H ( \tfrac{ \lfloor x - r v_n \rfloor }{n}  )     \Psi_{x, L_n}^{\alpha,\beta}(\eta_r^n)   \Bigg\}^{2} \Bigg]  \\
\leq & \frac{C_{\alpha,\beta}}{2}  (t-s)^{2} n \frac{(\widehat{b}_n)^2 }{\ell_n} \| \nabla H \|^{2}_{2,n} + \frac{C_{\alpha,\beta}}{2} (t-s)^{2} (t-s)^{2} n \frac{(\widehat{b}_n)^2 }{L_n} \| \nabla H \|^{2}_{2,n}. 
\end{align*}	
Since $\ell_n \leq L_n$, the proof ends.
	\end{proof}
The next step to obtain an upper bound of the form $(t-s)^{1 + \delta_1}$, for some $\delta_1 >0$ is to make use of the following remark.	
\begin{rem} \label{remtight}
	For any $n \in \mathbb{N}_+$, let $F_n: [0, \infty) \mapsto  [0, \infty)$ such that 
	\begin{align*}
		\forall t \geq 0, \quad F_n(t) \leq C \min \{ t f(n), \; t^2 g(n)  \},
	\end{align*}
	for some positive constant $C$; and some functions $f,g: \mathbb{N} \mapsto \mathbb{R}$. Now let $\delta \in [0, \; 1]$.
	\begin{itemize}
		\item
		If $t \geq 0$ and $n \in \mathbb{N}_+$ are such that $t \leq f(n) / g(n)$, then 
		\begin{align*}
			\min \{ t f(n), \; t^2 g(n)  \} = t^2 g(n) = t^{1+\delta} g(n) t^{1-\delta} \leq  t^{1+\delta} g(n) [f(n) / g(n)]^{1-\delta} = t^{1+\delta} [f(n) ]^{1 - \delta} [g(n) ]^{ \delta}.  
		\end{align*}
		\item
		If $t \geq 0$ and $n \in \mathbb{N}_+$ are such that $t \geq f(n) / g(n)$, then $t^{-1} \leq g(n) / f(n)$ and
		\begin{align*}
			\min \{ t f(n), \; t^2 g(n)  \} = t f(n) = t^{1+\delta} f(n) \big( t^{-1} \big)^{\delta} \leq  t^{1+\delta} f(n) [g(n) / f(n)]^{\delta} = t^{1+\delta} [f(n) ]^{1 - \delta} [g(n) ]^{ \delta}.  
		\end{align*}
	\end{itemize}
	In this way we conclude that
	\begin{align} \label{eqremtight}
		\forall t \geq 0, \; \forall n \in \mathbb{N}_+, \; \forall \delta \in [0, \; 1], \quad \min \{ t f(n), \; t^2 g(n)  \} \leq t^{1+\delta} [f(n) ]^{1 - \delta} [g(n) ]^{ \delta}.
	\end{align}
\end{rem}
The upper bounds which are (on time) at most linear are the ones provided by Lemmas \ref{lem31jarabur}, \ref{lemobe} and \ref{lemtbe}.
Combining Lemmas \ref{lemrepxzx1} and \ref{lem31jarabur} with \eqref{eqremtight}, we get the following result.
\begin{cor} \label{corlem31jarabur}
Let $\gamma \geq 1$, $0 \leq s \leq t \leq T$, $\delta \in [0, \; 1]$ and $n \in \mathbb{N}_+$. It holds
\begin{align*}
	&	\mathbb{E}_{\nu_{\rho}} \Bigg[ \Bigg\{ \frac{\Theta(n) K_n}{  \sqrt{n} } \int_s^t \dd r \sum_{x} \sum_{z=1}^{L_n-1}   [  H ( \tfrac{ \lfloor x+z - r v_n \rfloor }{n}  ) - H ( \tfrac{ \lfloor x - r v_n \rfloor }{n}  )  ] a(z)   \bar{\xi}_x^{\alpha}( \eta_r^n ) \big[ \bar{\xi} _{x+z}^{\beta}( \eta_r^n ) - \bar{\xi} _{x+1}^{\beta}( \eta_r^n ) \big] \Bigg\}^{2} \Bigg] \\
	\leq & \widehat{C}(t-s)^{1+\delta} \Bigg( \frac{\big[ \Theta(n) K_n\big]^{\delta+1} }{n^{2}} + \big[  K_n \big]^2 f_{\delta}(n)   \Bigg),
\end{align*} 
where $f_{\delta}:\mathbb{N}_+ \mapsto \mathbb{R}$ is given by
\begin{align} \label{deffedltan}
\forall n \in \mathbb{N}_+, \quad f_{\delta}(n):= 
\begin{dcases}
 (L_n/n)^{ 1 - \delta } \big[ \log(L_n) \big]^{2 \delta} , \quad & \gamma  =1, \\
  (L_n/n)^{ (2-\gamma)(1 - \delta) }     n^{2 \delta (\gamma-1) }, \quad & 1 < \gamma  <2, \\
 n^{2 \delta} \big[ \log(L_n) \big]^{1 - \delta} \big[ \log(n) \big]^{-1-\delta}, \quad & \gamma  =2, \\
 n^{2 \delta}, \quad & \gamma > 2. 
\end{dcases}
\end{align}
Furthermore, if $H \in C_c^2(\mathbb{R})$, the role of $\widehat{C}$ is fulfilled by some constant $C(H)$ satisfying \eqref{Cinvshift}.
\end{cor}
Combining Lemmas \ref{lemobe} and \ref{lemobe2} with \eqref{eqremtight}, we get the following result. Below and in the remainder of this section, $\widetilde{C}$ denotes some constant independent of $H$, $n$, $s$ and $t$.
\begin{cor} \label{corlemobe}
Let $\gamma \geq 1$, $0 \leq s \leq t \leq T$, $\delta \in [0, \; 1]$ and $n \in \mathbb{N}_+
$. For $1 \leq \ell_n$, it holds
\begin{align*}
&	\mathbb{E}_{\nu_{\rho}} \Bigg[\Bigg\{ \frac{\Theta(n) K_n}{n \sqrt{n} } \sum_{z=1}^{L_n-1} za(z)  \int_s^t \dd r \sum_{x}   \nabla H ( \tfrac{ \lfloor x - r v_n \rfloor }{n}  )   \big[  \bar{\xi} _x^{\alpha}( \eta_r^n ) \bar{\xi} _{x+1}^{\beta}( \eta_r^n ) - \Psi_{x, \ell_n}^{\alpha,\beta}(\eta_r^n) \big] \Bigg\}^{2} \Bigg] \\
	\leq & \widetilde{C} \| \nabla H \|^{2}_{2,n} (t-s)^{1 + \delta} \Bigg( \frac{\big[ \Theta(n) K_n\big]^{\delta+1} }{n^{2}} +  \frac{\big[ \Theta(n) \big]^{1+\delta} (K_n)^2  }{n^{2}}   \ell_n^{2 - 2 \delta} \Bigg)  \big\{ \mathbbm{1}_{ \{ \gamma > 1 \} } + \big[ \log(L_n) \big]^{2} \mathbbm{1}_{ \{ \gamma = 1 \} } \big\}.
\end{align*}
\end{cor}
Combining Lemmas \ref{lemtbe} and \ref{lemtbe2} with \eqref{eqremtight}, we get the following result.
\begin{cor} \label{corlemtbe}
Let $\gamma \geq 1$, $0 \leq s \leq t \leq T$, $\delta \in [0, \; 1]$ and $n \in \mathbb{N}_+$. For $1 \leq \ell_n \leq L_n$, it holds
	\begin{align*}
	&	\mathbb{E}_{\nu_{\rho}} \Bigg[ \Bigg\{ \frac{\Theta(n) K_n}{n \sqrt{n} } \sum_{z=1}^{L_n-1} za(z)  \int_s^t \dd r \sum_{x}   \nabla H ( \tfrac{ \lfloor x - r v_n \rfloor }{n}  )    \big[   \Psi_{x, \ell_n}^{\alpha,\beta}(\eta_r^n) - \Psi_{x,L_n}^{\alpha,\beta}(\eta_r^n)  \big]  \Bigg\}^{2} \Bigg] \\
		\leq & \widetilde{C} \| \nabla H \|^{2}_{2,n} (t-s)^{1 + \delta} \Bigg( \frac{\big[ \Theta(n) K_n\big]^{\delta+1} }{n^{2}} +  \frac{\big[ \Theta(n) \big]^{1+\delta} (K_n)^2  }{n^{2}}  L_n^{(\gamma-1)(1-\delta)}  \ell_n^{ \delta} \Bigg)  \big\{ \mathbbm{1}_{ \{ \gamma > 1 \} } + \big[ \log(L_n) \big]^{2} \mathbbm{1}_{ \{ \gamma = 1 \} } \big\}.
	\end{align*}
\end{cor}
Now we state a corollary which allows us to approximate $\mcb{B}_t^{n}(H)$ for $\gamma \in [1, \;2]$, under some conditions.
\begin{cor} \label{corngamgr1}
	Let $ \gamma \in [1, \; 2]$, $\lambda \in (0, 1)$, $0 \leq s \leq t \leq T$ and $\omega = \min \{\gamma, \; 3/2\}$. Assume that
    \begin{equation} \label{ubomega}
    \exists C >0, \; \forall n \in \mathbb{N}_+, \quad  \Theta(n) K_n \leq n^{\omega}.   
    \end{equation}
    Then there exists some $\delta_1 \in (0, \;1]$ and some $f: \mathbb{N}_+ \mapsto \mathbb{R}$ satisfying $\lim_{n \rightarrow \infty} f(n)=0$, such that
	\begin{align*}
		 \mathbb{E}_{\nu_{\rho}} \Bigg[ \Bigg\{ & \frac{\Theta(n) K_n}{\sqrt{n}} \int_s^t \dd r \sum_x \sum_{y=x+1}^{\infty} [  H ( \tfrac{y - r v_n }{n}  ) - H ( \tfrac{x - r v_n }{n}  )  ] a(y-x)   \bar{\xi} _x^{\alpha}( \eta_r^n ) \bar{\xi} _y^{\beta}( \eta_r^n ) \\
		- & \frac{\Theta(n) K_n}{n \sqrt{n} } \sum_{z=1}^{L_n-1}  za(z) \int_s^t \dd r \sum_{x}   \nabla H ( \tfrac{ \lfloor x - r v_n \rfloor }{n}  )   \Psi_{x,L_n}^{\alpha,\beta}(\eta_r^n)  \Bigg\}^{2} \Bigg] \leq \widehat{C} (t-s)^{1 + \delta_1} f(n), 
	\end{align*}
for every $n \in \mathbb{N}_+$, where $L_n:=n^{1 - \lambda}$. Furthermore, if $H \in C_c^2(\mathbb{R})$, the role of $\widehat{C}$ is fulfilled by some constant $C(H)$ satisfying \eqref{Cinvshift}.
\end{cor}
\begin{proof}
Combining Lemmas \ref{lemaproxintpart}, \ref{varfarterms}, \ref{varmidterms} and \ref{lem46} with Corollaries \ref{corlem31jarabur}, \ref{corlemobe} and \ref{corlemtbe}, the expectation in the last display is bounded from above by $\widehat{C}$, times
\begin{align*}
		& (t-s)^2 \Bigg( \frac{  \big[ \Theta(n) K_n \big]^{2}  }{ n^{4}  } + \frac{ [ \Theta(n) K_n ]^{2} }{n^{2 \gamma+1}} +  \frac{ [ \Theta(n) K_n ]^{2} }{n^{2}}  L_n^{1-2 \gamma}  + \frac{\big[ \Theta(n) K_n \big]^{2} }{n^{4} } \Bigg\{1 + \int_1^{L_n} u^{2 - 2 \gamma} \; \dd u \Bigg\} \Bigg) \\
		+ & (t-s)^{1 + \delta} \Bigg( \frac{\big[ \Theta(n) K_n\big]^{\delta+1} }{n^{2}} +  \big[  K_n \big]^2 \Bigg\{ f_{\delta}(n)   	+   \frac{1}{n^2}   \ell_n^{2 - 2 \delta} \big[ \Theta(n) \big]^{\delta + 1} +  \frac{1}{n^2}   L_n^{(\gamma-1)(1-\delta)} \ell_n^{\delta}\big[ \Theta(n) \big]^{\delta + 1} \Bigg\} \Bigg) \\
		\leq &   (t-s)^2 \frac{ [ \Theta(n) K_n ]^{2} }{n^{4}} \Bigg( 2  + 2  n^{3-2 \gamma}  +  \int_1^{n} u^{2 - 2 \gamma} \; \dd u  \Bigg) + (t-s)^{1 + \delta} \frac{\big[ \Theta(n) K_n\big]^{\delta+1} }{n^{2}}  \\
		+ & (t-s)^{1 + \delta} \big[ \Theta(n)  K_n \big]^2 \Bigg\{ f_{\delta}(n)   	+   \frac{1}{n^2}   \ell_n^{2 - 2 \delta} \big[ \Theta(n) \big]^{\delta + 1} +  \frac{1}{n^2}   L_n^{(\gamma-1)(1-\delta)} \ell_n^{\delta}\big[ \Theta(n) \big]^{\delta + 1} \Bigg\} \big[ \Theta(n) \big]^{-2}, 
	\end{align*}
for any $\delta \in [0, 1]$, where $f_{\delta}(n)$ is given by \eqref{deffedltan} and $\ell_n$ is such that $1 \leq \ell_n < L_n$. From \eqref{ubomega} , the last display is bounded from above by $\widehat{C}$, times
	\begin{align*}
	& (t-s)^2 \frac{ n^{2 \omega} }{n^{4}} \Bigg(      n^{3-2 \gamma}  +  \int_1^{n} u^{2 - 2 \gamma} \; \dd u  \Bigg) + (t-s)^{1 + \delta} \frac{ n^{ \omega(\delta+1)} }{n^{2}}  \\
	 + & (t-s)^{1 + \delta} n^{2 \omega} \Bigg\{ f_{\delta}(n)   	+   \frac{1}{n^2}   \ell_n^{2 - 2 \delta} \big[ \Theta(n) \big]^{\delta + 1} +  \frac{1}{n^2}   L_n^{(\gamma-1)(1-\delta)} \ell_n^{\delta}\big[ \Theta(n) \big]^{\delta + 1} \Bigg\} \big[ \Theta(n) \big]^{-2}. 
	\end{align*}
Recalling that $L_n = n^{1 - \lambda}$ and $\omega = \min \{\gamma, \; 3/2\}$, the last display is bounded from above by $\widehat{C}$, times
\begin{align}
	& (t-s)^2 \Bigg( n^{2 \omega - 2 \gamma - 1} +  n^{2 \omega - 4} \int_1^{n} u^{2 - 2 \gamma} \; \dd u \Bigg)  + (t-s)^{1 + \delta} \frac{ n^{ \omega(\delta+1)} }{n^{2}} \label{tight1} \\
	+ &  (t-s)^{1 + \delta}  \big[ \Theta(n) \big]^{-2} n^{2 \omega} f_{\delta}(n) \label{tight2} \\
	+ & (t-s)^{1+\delta} n^{2 \omega-2} \big[ \Theta(n) \big]^{\delta - 1} \ell_n^{2 - 2 \delta}  \big\{ \mathbbm{1}_{ \{ \gamma > 1 \} } + \big[ \log(L_n) \big]^{2} \mathbbm{1}_{ \{ \gamma = 1 \} } \big\}   \label{tight3} \\ 
	+& (t-s)^{1+\delta} n^{2 \omega-2} \big[ \Theta(n) \big]^{\delta - 1} L_n^{(\gamma-1)(1-\delta)} \ell_n^{\delta}  \big\{ \mathbbm{1}_{ \{ \gamma > 1 \} } + \big[ \log(L_n) \big]^{2} \mathbbm{1}_{ \{ \gamma = 1 \} } \big\}. \label{tight4} 
\end{align}
Since $\omega = \min \{\gamma, \; 3/2\}$, we have that
\begin{align*}
	n^{2 \omega - 2 \gamma - 1} +  n^{2 \omega - 4} \int_1^{n} u^{2 - 2 \gamma} du \leq C(\gamma) \Bigg\{ \frac{1}{n} + \mathbbm{1}_{ \{ \gamma < 3/2 \} } n^{2 \gamma - 4} n^{3 - 2 \gamma }  + \mathbbm{1}_{ \{ \gamma \geq 3/2 \} } \frac{\log(n)}{n} \Bigg\} \leq C(\gamma) \frac{\log(n)}{n},
\end{align*}
and we conclude that the first term in \eqref{tight1} vanishes, as $n \rightarrow \infty$. In order to control the remaining term in \eqref{tight1}, it is enough to choose $\delta$ such that $0 < \delta < 1/3$, since $\omega \leq 3/2$. 

At this point, we recall that $L_n = n^{1 - \lambda}$.

Next, we analyze the term in \eqref{tight2}, treating the cases $\gamma=1$, $1 < \gamma < 2$ and $\gamma = 2$ separately.
\begin{itemize}
\item 
If $\gamma = 1$, then $\omega=1$, $\Theta(n)=n$ and the term in \eqref{tight2} can be rewritten as 	
\begin{align*}
	(t-s)^{1 + \delta} n^{-2} n^{2 }  \big(n^{-\lambda} \big)^{ 1 - \delta } \big[ \log \big( n^{1-\lambda} \big)  \big]^{2 \delta} = (1 - \lambda)^{2 \delta}	(t-s)^{1 + \delta} [\log(n)]^{ 2 \delta} n^{ \lambda (\delta-1) }.
\end{align*}
In order to control this term, it is enough to choose any $\delta, \lambda \in (0, \;1)$.
	\item 
	If $1 < \gamma < 2$, then $\Theta(n)=n^{\gamma}$ and the term in \eqref{tight2} can be rewritten as
\begin{align*}
(t-s)^{1 + \delta} n^{ \delta [2 (\gamma-1) + \lambda (2 -\gamma) ] - [ 2\gamma - 2 \omega + \lambda (2 - \gamma)]}.
\end{align*}
In order to control this term, it is enough to choose any $\lambda \in (0, \;1)$, and $\delta$ such that
\begin{align*}
	0 < \delta < \frac{\lambda (2 - \gamma) + 2 \gamma - 2 \omega }{\lambda (2 - \gamma) +2 \gamma - 2  }.
\end{align*}
We observe that $\lambda (2 - \gamma) + 2 \gamma - 2 \omega \geq \lambda (2 - \gamma) >0$, since $\omega \leq \gamma$.
\item 
If $\gamma = 2$, then $\omega=3/2$, $\Theta(n)=n^{2} / \log(n)$ and the term in \eqref{tight2} can be rewritten as
\begin{align*}
(1 - \lambda)^{1 - \delta}	(t-s)^{1 + \delta} [\log(n)]^{2 - 2 \delta} n^{ 2 \delta  - 1 } .
\end{align*}
In order to control this term, it is enough to choose any $\lambda \in (0, \;1)$, and $\delta$ such that $0 < \delta < 1/2$.
\end{itemize}  
Next, we analyze the term in \eqref{tight3}, treating the cases $\gamma=1$, $1 < \gamma < 2$ and $\gamma = 2$ separately.
\begin{itemize}
\item 
If $\gamma = 1$, then $\omega=1$, $\Theta(n)=n$ and the term in \eqref{tight3} can be rewritten as
 \begin{align*}
(t-s)^{1 + \delta} n^{ \delta  - 1} \ell_n^{2 - 2 \delta} (1 - \lambda)^{2} [\log(n)]^{ 2 }.
 \end{align*}
In order to control this term, it is enough to choose $\delta$ and $\ell_n$ such that
\begin{align*}
	0 < \delta < 1, \quad \ell_n:=n^{\beta}, \quad 0 < \beta < \min \Bigg\{ 1 - \lambda, \; \frac{1}{2} \Bigg\}.
\end{align*}	
	\item 
If $1 < \gamma < 2$, then $\Theta(n)=n^{\gamma}$ and the term in \eqref{tight3} can be rewritten as
\begin{align*}
	(t-s)^{1 + \delta} n^{ \delta \gamma -\gamma + 2 \omega -2} \ell_n^{2 - 2 \delta} \leq (t-s)^{1 + \delta} n^{ \delta \gamma - (\gamma-1)} \ell_n^{2 - 2 \delta}.
\end{align*}
In order to control this term, it is enough to choose $\delta$ and $\ell_n$ such that
\begin{align*}
	0 < \delta < \frac{\gamma - 1 }{\gamma}, \quad \ell_n:=n^{\beta}, \quad 0 < \beta < \min \Bigg\{ 1 - \lambda, \; \frac{\gamma-1 - \delta \gamma }{2 - 2 \delta} \Bigg\}.
\end{align*}
\item 
If $\gamma = 2$, then $\omega=3/2$, $\Theta(n)=n^{2} / \log(n)$ and the term in \eqref{tight3} can be rewritten as
\begin{align*}
	(t-s)^{1 + \delta} [\log(n)]^{1 - \delta} n^{ 2 \delta  - 1 } \ell_n^{2 - 2 \delta}.
\end{align*}
In order to control this term, it is enough to choose $\delta$ and $\ell_n$ such that
\begin{align*}
	0 < \delta < \frac{ 1 }{2}, \quad \ell_n:=n^{\beta}, \quad 0 < \beta < \min \Bigg\{ 1 - \lambda, \; \frac{1 - 2\delta  }{2 - 2 \delta} \Bigg\}.
\end{align*}
\end{itemize}  	
Next, we analyze the term in \eqref{tight4}, treating the cases $\gamma = 1$, $1 < \gamma < 2$ and $\gamma = 2$ separately.
\begin{itemize}
\item 
If $\gamma = 1$, then $\omega=1$, $\Theta(n)=n$ and the term in \eqref{tight4} can be rewritten as	
 \begin{align*}
	(t-s)^{1 + \delta} n^{ \delta  - 1} \ell_n^{ \delta} (1 - \lambda)^{2} [\log(n)]^{ 2 }.
\end{align*}
In order to control this term, it is enough to choose $\delta$ and $\ell_n$ such that
\begin{align*}
	0 < \delta < 1, \quad \ell_n:=n^{\beta}, \quad 0 < \beta < \min \Bigg\{ 1 - \lambda, \; \frac{1-\delta}{\delta} \Bigg\}.
\end{align*}		
	\item 
If $1 < \gamma < 2$, then $\Theta(n)=n^{\gamma}$ and the term in \eqref{tight4} can be rewritten as
\begin{align*}
	(t-s)^{1+\delta}  n^{ \delta (1 - \lambda + \lambda \gamma)  - \lambda  ( \gamma - 1) } \ell_n^{\delta}.
\end{align*}	
In order to control this term, it is enough to choose $\delta$ and $\ell_n$ such that
\begin{align*}
	0 < \delta < \frac{ \lambda ( \gamma - 1)}{ 1 - \lambda + \lambda \gamma}, \quad \ell_n:=n^{\beta}, \quad 0 < \beta < \min \Bigg\{ 1 - \lambda, \; \frac{  \lambda ( \gamma - 1)  - \delta( 1 - \lambda + \lambda \gamma)  }{ \delta} \Bigg\}.
\end{align*}
\item
If $\gamma = 2$, then $\omega=3/2$, $\Theta(n)=n^{2} / \log(n)$ and the term in \eqref{tight4} can be rewritten as
\begin{align*}
	(t-s)^{1 + \delta} [\log(n)]^{1 - \delta}  n^{ \delta (1 + \lambda)  - \lambda  } \ell_n^{\delta}.
\end{align*}
In order to control this term, it is enough to choose $\delta$ and $\ell_n$ such that
\begin{align*}
	0 < \delta < \frac{ \lambda }{ 1 + \lambda }, \quad \ell_n:=n^{\beta}, \quad 0 < \beta < \min \Bigg\{ 1 - \lambda, \; \frac{  \lambda   - \delta( 1 + \lambda )  }{ \delta} \Bigg\}.
\end{align*}	
\end{itemize}
\end{proof}
Thanks to Corollary \ref{corngamgr1}, it only remains to treat the term equivalent to the one inside the expectation in \eqref{cor28a}, for the choice $\widehat{b}_n:= \Theta(n) K_n n^{-3/2} \sum_{z=1}^{L_n-1}  za(z)$, where $L_n=n^{1 - \lambda}$. We stress that $\lambda$ is a value in $(0, 1)$ that can be chosen freely. We treat the aforementioned term in the next couple of results.
\begin{cor} \label{cornlgamless32}
	Let $\gamma \in [1, \; 3/2)$ and fix $\lambda$ such that $0 < \lambda < 3-2 \gamma$. For every $n \in \mathbb{N}_+$, let $L_n:=n^{1 - \lambda}$. There exists some $f: \mathbb{N}_+ \mapsto \mathbb{R}$ satisfying $\lim_{n \rightarrow \infty} f(n)=0$, such that for any $0 \leq s \leq t \leq T$,
	\begin{align*}
	\forall n \in \mathbb{N}_+, \quad	\mathbb{E}_{\nu_{\rho}} \Bigg[ \Bigg\{  \frac{\Theta(n) K_n}{n \sqrt{n} } \sum_{z=1}^{L_n-1}  za(z) \int_s^t \dd r \sum_{x}   \nabla H ( \tfrac{ \lfloor x - r v_n \rfloor }{n}  )   \Psi_{x,L_n}^{\alpha,\beta}(\eta_r^n)  \Bigg\}^{2} \Bigg] \leq \widehat{C} (t-s)^{2} f(n). 
	\end{align*}
	\end{cor}
\begin{proof}
From \eqref{cor28a}, the expectation in the last display is bounded from above by $\widehat{C}$, times
\begin{align*}
(t-s)^{2} \frac{n}{L_n} \Bigg( \frac{\Theta(n) K_n}{n \sqrt{n} } \sum_{z=1}^{L_n-1}  za(z) \Bigg)^{2} \leq C(\gamma)  (t-s)^{2} \frac{ \big[ \Theta(n) K_n \big]^{2} }{n^{2} L_n}  \Bigg(  \sum_{z=1}^{L_n-1}  z^{-\gamma} \Bigg)^{2}.
\end{align*}
By assumption, $\gamma < 3/2 < 2$, thus $\Theta(n)=n^{\gamma}$. Moreover, $(K_n)_{n \in \mathbb{N}_+}$ is an uniformly bounded sequence. Hence, the expression in the last display is bounded from above by $\widehat{C}$, times
\begin{align*}
	  (t-s)^{2} \frac{ \big[ n^{\gamma} \big]^{2} }{n^{2} L_n}  \Bigg(  \sum_{z=1}^{L_n-1}  z^{-\gamma} \Bigg)^{2} \leq
\begin{dcases}
(t-s)^{2} \frac{ [ \log(L_n) ]^{2}}{L_n} = (1-\lambda)^{2} (t-s)^{2} \frac{ [ \log(n) ]^{2}}{n^{1-\lambda}}, \quad & \gamma=1, \\
C(\gamma)(t-s)^{2} \frac{ n^{2 \gamma}}{n^{2} L_n} = C(\gamma)(t-s)^{2} n^{2 \gamma + \lambda -3}, \quad & 1 < \gamma < 3/2.
\end{dcases}	  
\end{align*}	
Since $\lambda  < 3 - 2 \gamma$ and $\gamma \geq 1$, we have that $\lambda < 1$ and $2 \gamma + \lambda -3 < 0$. This ends the proof!
\end{proof}
Next we treat the term equivalent to the one inside the expectation in \eqref{cor28a} for larger values of $\gamma$, as it is stated in the next result. In the remainder of this section, $\widetilde{C}$ denotes some constant (that may change from line to line) independent of $n \in \mathbb{N}_+$, $0 \leq s \leq t \leq T$ and $H$.
\begin{cor} \label{cornlgamgrea32}
	Let $\gamma \geq 3/2$, $\lambda  \in (0, \; 1)$ and $0 \leq s \leq t \leq T$. For every $n \in \mathbb{N}_+$, it holds 
	\begin{align} \label{contbox}
		\mathbb{E}_{\nu_{\rho}} \Bigg[ \Bigg\{  \frac{\Theta(n) K_n}{n \sqrt{n} }  \int_s^t \dd r \sum_{x}   \nabla H ( \tfrac{ \lfloor x - r v_n \rfloor }{n}  )   \Psi_{x, \ell_n}^{\alpha,\beta}(\eta_r^n)  \Bigg\}^{2} \Bigg] \leq \widetilde{C} {\| \nabla H \|^{2}_{2,n}} (t-s)^{\frac{2\gamma-1}{\gamma}} \frac{ \big[ \Theta(n) K_n \big]^{2} }{n^{2} [\Theta(n)]^{1/\gamma}}, 
	\end{align}
	where $\ell_n:=n^{1 - \lambda}$. 
\end{cor}
\begin{proof}
From \eqref{cor28a}, the expectation in \eqref{contbox} is bounded from above by
\begin{align} \label{bound1}
\widetilde{C} \| \nabla H \|^{2}_{2,n}	(t-s)^{2} \frac{ \big[ \Theta(n) K_n \big]^{2} }{n^{2} \ell_n}.
\end{align}
In what follows, we treat two cases separately: $(t-s) \leq (\ell_n)^{\gamma} [\Theta(n)]^{-1}$; and $(t-s) > (\ell_n)^{\gamma} [\Theta(n)]^{-1}$. 
\begin{itemize}
	\item 
If $(t-s) \leq (\ell_n)^{\gamma} [\Theta(n)]^{-1}$, the term in \eqref{bound1} can be rewritten as $\widetilde{C}$, times
\begin{align*}
	\| \nabla H \|^{2}_{2,n} (t-s)^{\frac{2\gamma-1}{\gamma}} (t-s)^{\frac{1}{\gamma}} \frac{ \big[ \Theta(n) K_n \big]^{2} }{n^{2} \ell_n} \leq  \| \nabla H \|^{2}_{2,n} (t-s)^{\frac{2\gamma-1}{\gamma}} \big\{ (\ell_n)^{\gamma} [\Theta(n)]^{-1}  \big\}^{\frac{1}{\gamma}} \frac{ \big[ \Theta(n) K_n \big]^{2} }{n^{2} \ell_n}, 
\end{align*}
and we obtain the upper bound in \eqref{contbox}. This end the proof for this case.
\item 
If $(t-s) > (\ell_n)^{\gamma} [\Theta(n)]^{-1}$, define $L_n$ such that $(L_n)^{\gamma} = (t-s) \Theta(n)$. In particular, $(L_n)^{\gamma} > (\ell_n)^{\gamma}$ and we conclude from Lemma \ref{lemtbe} that
\begin{align} \label{bound2}
	& \mathbb{E}_{\nu_{\rho}} \Bigg[ \Bigg\{ \frac{\Theta(n) K_n}{n \sqrt{n} }  \int_s^t \dd r \sum_{x}   \nabla H ( \tfrac{ \lfloor x - r v_n \rfloor }{n}  )    \big[   \Psi_{x, \ell_n}^{\alpha,\beta}(\eta_r^n) - \Psi_{x,L_n}^{\alpha,\beta}(\eta_r^n)  \big]  \Bigg\}^{2} \Bigg]  \leq \widetilde{C} \| \nabla H \|^{2}_{2,n} (t-s)   \frac{ \big[ \Theta(n) K_n \big]^{2} }{ \Theta(n)n^{2} } (L_n)^{\gamma-1}. 
\end{align}
On the other hand, from \eqref{cor28a} we get that
\begin{align*}
	\mathbb{E}_{\nu_{\rho}} \Bigg[ \Bigg\{  \frac{\Theta(n) K_n}{n \sqrt{n} }  \int_s^t \dd r \sum_{x}   \nabla H ( \tfrac{ \lfloor x - r v_n \rfloor }{n}  )   \Psi_{x,L_n}^{\alpha,\beta}(\eta_r^n)  \Bigg\}^{2} \Bigg]	\leq  \widetilde{C} \| \nabla H \|^{2}_{2,n} (t-s)^{2} \frac{ \big[ \Theta(n) K_n \big]^{2} }{n^{2} L_n}.
\end{align*}
Combining the last display with \eqref{bound2}, we conclude that the expectation in \eqref{contbox} is bounded from above by $\widetilde{C}$, times
\begin{align*}
\| \nabla H \|^{2}_{2,n} (t-s) \big[ \Theta(n) K_n \big]^{2} \Bigg\{ \frac{ (L_n)^{\gamma-1} }{ \Theta(n)n^{2} }  + \| \nabla H \|^{2}_{2,n}   \frac{t-s }{n^{2} L_n} \Bigg\}= 2 \| \nabla H \|^{2}_{2,n} (t-s) \frac{ \big[ \Theta(n) K_n \big]^{2} }{ \Theta(n)n^{2} } (L_n)^{\gamma-1}.
\end{align*}
The last equality holds, due to the definition of $L_n$. Now from $(t-s) > (\ell_n)^{\gamma} [\Theta(n)]^{-1}$, we get that $(t-s) \leq (\ell_n)^{\gamma} [\Theta(n)]^{-1}$. Therefore, the last display can be rewritten as $2 \| \nabla H \|^{2}_{2,n}$, times
\begin{align*}
	 (t-s)^{\frac{2\gamma-1}{\gamma}} \big[ (t-s)^{-1/\gamma} L_n \big]^{\gamma -1}  \frac{ \big[ \Theta(n) K_n \big]^{2} }{ \Theta(n)n^{2} } =  (t-s)^{\frac{2\gamma-1}{\gamma}} \Big\{ \big[ \Theta(n)  \big]^{1/\gamma} \Big\}^{\gamma -1}  \frac{ \big[ \Theta(n) K_n \big]^{2} }{ \Theta(n)n^{2} },	
\end{align*}
leading to the upper bound in \eqref{contbox}. This end the proof for this case.
\end{itemize}
\end{proof}
Next we observe that Corollary \ref{corngamgr1} does not include the case $\gamma >2$. Moreover, for $\gamma > 2$, we get from \eqref{timescale} that the upper bound in Corollary \ref{cornlgamgrea32} is $\big[ K_n n^{1 - \frac{1}{\gamma}} \big]^{2} > [ K_n \sqrt{n} \; ]^{2}$; in particular, this upper bound gets worse, as $\gamma$ increases. For those reasons, we state a couple of results in order to complete the analysis of the regime $\gamma >2$.
\begin{lem} \label{lemgamgrea2}
	Let $\gamma \geq 2$ and $0 \leq s \leq t \leq T$. For every $n \in \mathbb{N}_+$, it holds 
	\begin{align} \label{contbox2}
		\mathbb{E}_{\nu_{\rho}} \Bigg[ \Bigg\{ \frac{\Theta(n) K_n}{n \sqrt{n} } \int_s^t \dd r  \sum_{x} \nabla H ( \tfrac{ \lfloor x - r v_n \rfloor }{n})     \bar{\xi} _x^{\alpha}( \eta_r^n )  \bar{\xi} _{x+1}^{\beta}( \eta_r^n )   \Bigg\}^{2} \Bigg] \leq \widetilde{C} \| \nabla H \|^{2}_{2,n} (t-s)^{3/2} [ K_n \sqrt{n} \;]^{2}.
	\end{align}
\end{lem}
\begin{proof}
	From Lemma \ref{lemsupCS}, the expectation in \eqref{contbox2} is bounded from above by
	\begin{align} \label{bound1b}
	\widetilde{C} \| \nabla H \|^{2}_{2,n}	(t-s)^{2} \frac{ \big[ \Theta(n) K_n \big]^{2} }{n^{2}}  = \widetilde{C} \| \nabla H \|^{2}_{2,n}(t-s)^{2} \frac{ \big[ n^{2} K_n \big]^{2} }{n^{2}} = \widetilde{C} \| \nabla H \|^{2}_{2,n} (t-s)^{2} [ K_n \sqrt{n} \;]^{2} n.
	\end{align}
	In what follows, we treat two cases separately: $\sqrt{t-s} \leq n^{-1}$; and $\sqrt{t-s} > n^{-1}$. 
	\begin{itemize}
		\item 
		If $\sqrt{t-s} \leq n^{-1}$, the term in \eqref{bound1b} can be rewritten as $\widetilde{C} \| \nabla H \|^{2}_{2,n}$, times
		\begin{align*}
			(t-s)^{3/2} [ K_n \sqrt{n} \;]^{2} \sqrt{t-s} n \leq  (t-s)^{3/2} [ K_n \sqrt{n} \;]^{2}  n^{-1} n =  (t-s)^{3/2} [ K_n \sqrt{n} \;]^{2}, 
		\end{align*}
		and we obtain the upper bound in \eqref{contbox2}. This end the proof for this case.
		\item 
		If $\sqrt{t-s} > n^{-1}$, let $L_n:=n \sqrt{t-s} \geq 1$. Applying arguments analogous to the ones in the proof of Theorem 6.2 in \cite{ABC}, it holds
		\begin{equation}  \label{2bndexpdif}
			\begin{split}
				&\mathbb{E}_{\nu_{\rho}} \Bigg[ \Bigg\{  \frac{\Theta(n) K_n}{n \sqrt{n} } \int_s^t \dd r  \sum_{x} \nabla H ( \tfrac{ \lfloor x - r v_n \rfloor }{n})  \big[  \bar{\xi} _x^{\alpha}( \eta_r^n )  \bar{\xi} _{x+1}^{\beta}( \eta_r^n )  - \Psi_{x,L_n}^{\alpha,\beta}(\eta_r^n) \big] \Bigg\}^{2} \Bigg] \\
				\leq & \widetilde{C} \| \nabla H \|^{2}_{2,n} [ K_n \sqrt{n} \;]^{2} (t-s) \Bigg( \frac{L_n}{n} + \frac{(t-s) n }{ (L_n)^2 } \Bigg). 
			\end{split}
		\end{equation} 
		Moreover, from \eqref{cor28a}, for every $0 \leq s \leq t \leq T$, it holds
		\begin{align*}
			\mathbb{E}_{\nu_{\rho}} \Bigg[ \Bigg\{ \frac{\Theta(n) K_n}{n \sqrt{n} } \int_s^t \dd r  \sum_{x} \nabla H ( \tfrac{ \lfloor x - r v_n \rfloor }{n})     \Psi_x^{L_n}(\eta_r^n)   \Bigg\}^{2} \Bigg] \leq \widetilde{C} \| \nabla H \|^{2}_{2,n} (t-s)^{2} [ K_n \sqrt{n} \;]^{2} \frac{n}{L_n}.
		\end{align*}
		Combining the last display with inequality $(u+v)^2 \leq 2 u^2 + 2 v^2$ and \eqref{2bndexpdif}, we conclude that
		\begin{align*}
			& \mathbb{E}_{\nu_{\rho}} \Bigg[ \Bigg\{ \frac{\Theta(n) K_n}{n \sqrt{n} } \int_s^t \dd r  \sum_{x} \nabla H ( \tfrac{ \lfloor x - r v_n \rfloor }{n})     \bar{\xi} _x^{\alpha}( \eta_r^n )  \bar{\xi} _{x+1}^{\beta}( \eta_r^n )   \Bigg\}^{2} \Bigg]  \\
			\leq & \widetilde{C} \| \nabla H \|^{2}_{2,n} [ K_n \sqrt{n} \;]^{2} (t-s) \Bigg( \frac{L_n}{n} + \frac{(t-s) n }{ (L_n)^2 } + \frac{(t-s) n }{ L_n } \Bigg) \\
			\leq& \widetilde{C} \| \nabla H \|^{2}_{2,n}[ K_n \sqrt{n} \;]^{2} (t-s) \Bigg( \frac{n \sqrt{t-s}}{n} + \frac{2(t-s) n }{ n \sqrt{t-s} } \Big) = \widetilde{C} \| \nabla H \|^{2}_{2,n} (t-s)^{3/2} [ K_n \sqrt{n} \;]^{2}.
		\end{align*}
		Above we applied the fact that $L_n \geq 1$. This leads to the upper bound in \eqref{contbox2} and ends the proof for this case.
	\end{itemize}
\end{proof}
We end this section by stating a corollary which allows us to approximate $\mcb{B}_t^{n}(H)$ for $\gamma >2$, under some conditions.
\begin{cor} \label{corngamgr2}
	Let $ \gamma >2$ and $0 \leq s \leq t \leq T$. Assume that $K_n \sqrt{n} \leq \widehat{C}$, for every $n \in \mathbb{N}_+$. Then there exists some $\delta_1 \in (0, \;1]$ and some $f: \mathbb{N}_+ \mapsto \mathbb{R}$ satisfying $\lim_{n \rightarrow \infty} f(n)=0$, such that
	\begin{align*}
		\mathbb{E}_{\nu_{\rho}} \Bigg[ \Bigg\{ & \frac{\Theta(n) K_n}{\sqrt{n}} \int_s^t \dd r \sum_x \sum_{y=x+1}^{\infty} [  H ( \tfrac{y - r v_n }{n}  ) - H ( \tfrac{x - r v_n }{n}  )  ] a(y-x)   \bar{\xi} _x^{\alpha}( \eta_r^n ) \bar{\xi} _y^{\beta}( \eta_r^n ) \\
		- & \frac{\Theta(n) K_n}{n \sqrt{n} } \sum_{z=1}^{L_n-1}  za(z) \int_s^t \dd r \sum_{x}   \nabla H ( \tfrac{ \lfloor x - r v_n \rfloor }{n}  )   \bar{\xi} _x^{\alpha}( \eta_r^n )  \bar{\xi} _{x+1}^{\beta}( \eta_r^n )   \Bigg\}^{2} \Bigg] \leq \widehat{C} (t-s)^{1 + \delta_1} f(n), 
	\end{align*}
	for every $n \in \mathbb{N}_+$ and every $L_n$ such that $L_n \geq n^{2/3}$. Furthermore, if $H \in C_c^2(\mathbb{R})$, the role of $\widehat{C}$ is fulfilled by some constant $C(H)$ satisfying \eqref{Cinvshift}.
\end{cor}
\begin{proof}	
From Lemmas \ref{lemaproxintpart}, \ref{varfarterms}, \ref{varmidterms} and Corollary \ref{corlem31jarabur}, the expectation in the last display is bounded from above by $\widehat{C}$, times
\begin{align*}
&	(t-s)^{2} \Bigg\{\frac{  \big[ \Theta(n) K_n \big]^{2}  }{ n^{4}  }   + \frac{ [ \Theta(n) K_n ]^{2} }{n^{2} L_n^{2 \gamma - 1} } \Bigg\} + (t-s)^{1+\delta} \Bigg( \frac{\big[ \Theta(n) K_n\big]^{\delta+1} }{n^{2}} + \big[  K_n n^{\delta} \big]^2    \Bigg) \\
\leq &2 (t-s)^{2} \frac{  \big[ \Theta(n) K_n \big]^{2}  }{ n^{4}  } + (t-s)^{1+\delta} \Bigg( \frac{\big[ \Theta(n) K_n\big]^{\delta+1} }{n^{2}} + \big[  K_n n^{\delta} \big]^2    \Bigg),
\end{align*}
for any $\delta \in [0, \; 1]$. The inequality in the last display comes from the assumption that $L_n \geq n^{2/3}$. Since $\gamma > 2$, we get from \eqref{timescale} that $\Theta(n)=n^{2}$, thus the expression in the last display can be rewritten as
\begin{align*}
\end{align*}
\begin{align*}
	& (t-s)^{1+\delta} \Big\{2 (t-s)^{1-\delta} \big[  K_n \big]^{2} + (K_n)^{\delta+1} n^{2 \delta}  + \big[  K_n n^{\delta} \big]^2    \Big\} \\
	\leq & (t-s)^{1+\delta} \Big\{2 T^{1-\delta} \big[  K_n \sqrt{n} \; \big]^{2} n^{-1} + [ K_n \sqrt{n} \;]^{1+\delta} n^{ \frac{3 \delta-1}{2}} +  [ K_n \sqrt{n} \;]^{2} n^{2 \delta - 1}  \Big\}.
\end{align*}
The inequality in the last display comes from $0 \leq t-s \leq T$ and $\delta \leq 1/2 < 1$.  From the assumption that $K_n \sqrt{n} \leq \widehat{C}$, by choosing any $\delta \in (0, 1/3)$, the proof ends.  
\end{proof}

\section{Characterization of the limit points} \label{seccharac}

The goal of this section is to characterize the limit points of $\big\{ \big( \mathcal{Z}_t^{n,\pm}  \big)_{0 \leq t \leq T}: \; n \in \mathbb{N}_+  \big\}$ as random elements satisfying the  conditions stated in Definition \ref{defspde} or Definition \ref{defspdefbe}, {depending on whether Hypothesis \ref{hyplin} or Hypothesis \ref{hypnonlin} holds.} 
The fields $ \mathcal{Z}_t^{n, \pm}$ are given by \eqref{defZtnHpm}; this means that $D_1$, $D_2$, $\lambda$ and $v$ are chosen in such a way that \eqref{defD2pm}, \eqref{deflambda} and \eqref{vsdif} are satisfied. From Proposition \ref{tightseqZpm}, the sequence $\big\{ \big( \mathcal{Z}_t^{n,\pm}  \big)_{0 \leq t \leq T}: \; n \in \mathbb{N}_+  \big\}$ is tight with respect to the Skorohod topology of $ \mcb{D} \big( [0,T], \;  \mathcal{S}'(\mathbb{R}) \; \big)$.
We begin by observing that from Proposition \ref{convinifie}, any limit point $\big( \mathcal{Z}_t^{\pm}  \big)_{0 \leq t \leq T}$ of the aforementioned sequence is a stationary stochastic process. 

\subsection{Identification of the drift operator.} The next result will be useful to show that the quadratic fluctuation generated by the long–jump operator acting on a smooth test function vanishes. This allows the discrete operators to converge to their continuous analogues and to identify the candidate drift in the limit martingale problem. We present the proof for Lemma \ref{convL2lapyesvel} below in Appendix \ref{useest10}.
\begin{lem} \label{convL2lapyesvel}
	Let $\gamma > 0$ and $v: \mathbb{N}\in \mathbb{N}_+ \mapsto \mathbb{R}$. Recall the definition of $q_n^r$ in \eqref{defqnr}. Then
\begin{align}
& \lim_{n \rightarrow \infty} \frac{1}{n}  \sup_{r \in [0,T]} \sum_x \big[ \Theta(n)  \mathbb{L}_n^{\gamma,s} H (\tfrac{x + q_n^r}{n})  - \mathbb{L}^{\gamma} H (\tfrac{x + q_n^r}{n}) \big]^{2}=0, \label{limL2normlapfrac} \\
&  \lim_{n \rightarrow \infty} \frac{1}{n}  \sup_{r \in [0,T]} \sum_x \big[ \widehat{\mathbb{L}}_{n, \lambda}^{ \gamma} H (\tfrac{x + q_n^r}{n})  - \widehat{\mathbb{L}}_{ \lambda}^{ \gamma} H (\tfrac{x + q_n^r}{n}) \big]^{2} = 0. \label{limL2normlapfraclat}
\end{align}
\end{lem}
As consequence of Lemma \ref{convL2lapyesvel} the time-integrated microscopic drift converges in $L^2(\nu_{\rho})$ to the continuum drift, as $n \rightarrow \infty$.
\begin{prop}
It holds
\begin{align*}
\lim_{n \rightarrow \infty} \mathbb{E}_{\nu_{\rho}} \Bigg[ \Bigg\{  \mcb{I}_t^{n} (H) - \mcb{I}_s^{n} (H) - \int_s^t \mcb{Z}_r^n \big( \mathbb{L}^{\gamma} H -  \widehat{\mathbb{L}}_{ \lambda}^{ \gamma} H \big) \; \dd r  \Bigg\}^{2} \Bigg] =0.
\end{align*}
\end{prop}
\begin{proof}
From \eqref{defZtnH} and \eqref{defItn}, the expectation can be rewritten as
\begin{align} \label{expprop52}
\mathbb{E}_{\nu_{\rho}} \Bigg[ \Bigg\{ \int_s^t \frac{1}{\sqrt{n}} \sum_{z } [D_1  \bar{\xi}_z^{A}( \eta_r^n )  + D_2 \bar{\xi}_z^{B}( \eta_r^n ) ] \; F_{n,r}^H(z) \; \dd r  \Bigg\}^{2} \Bigg],
\end{align}
where for every $n \in \mathbb{N}$, $r \in [0, T]$ and $z \in \mathbb{Z}$, $F_{n,r}^H(z)$ is given by
\begin{align*}
F_{n,r}^H(z):=[ \Theta(n) \; \mathbb{L}_n^{\gamma} H ( \tfrac{z - r v_n }{n}  ) - \mathbb{L}^{\gamma} H ( \tfrac{z - r v_n }{n}  ) ] + [ \widehat{\mathbb{L}}_{n, \lambda}^{ \gamma} H ( \tfrac{z - r v_n }{n}  ) - \widehat{\mathbb{L}}_{ \lambda}^{ \gamma} H ( \tfrac{z - r v_n }{n}  ) ].
\end{align*}
Thus, from the H\"older inequality, the expectation in \eqref{expprop52} is bounded from above by
\begin{align*}
 & \frac{t-s}{n} \mathbb{E}_{\nu_{\rho}} \Bigg[  \int_s^t  \sum_{z,w } \big\{ D_1  \bar{\xi}_z^{A}( \eta_r^n )  + D_2 \bar{\xi}_z^{B}( \eta_r^n ) \big\}  F_{n,r}^H(z) \big\{ D_1  \bar{\xi}_w^{A}( \eta_r^n )  + D_2 \bar{\xi}_w^{B}( \eta_r^n ) \big\} F_{n,r}^H(w) \; \dd r   \Bigg] \\
 =& \frac{t-s}{n}   \int_s^t \sum_{z } D_3 \big\{F_{n,r}^H(z) \big\}^2 \dd r \leq  \frac{(t-s)^2}{n} D_3  \sup_{ r \in [0,T] } \Bigg\{ \sum_{z }  \big\{F_{n,r}^H(z) \big\}^2 \Bigg\},
\end{align*}
where $D_3$ is given by \eqref{defD3}. Above, we applied the Fubini's Theorem and Remark \ref{remind}. Combining the last display with \eqref{limL2normlapfrac} and \eqref{limL2normlapfraclat}, the proof ends.
\end{proof}

\subsection{Decoupling of the + and - fields.} For every $z \in \mathbb{Z}$ fixed, define $f_z, g_z: \Omega \mapsto \mathbb{R}$ by
\begin{align} \label{deffg}
\forall \eta \in \Omega, \quad f_z(\eta):= D_1 \overline{\xi}_z^{A}( \eta ) + D_2^{+} \overline{\xi}_z^{B}( \eta ) \quad \text{and}  \quad  g_z(\eta):= D_1 \overline{\xi}_z^{A}( \eta ) + D_2^{-} \overline{\xi}_z^{B}( \eta ).
\end{align}
Next, we observe that for every $t \in [0, \;T]$ and $G, H \in \mathcal{S}(\mathbb{R})$, we get from \eqref{defZtnH} that
\begin{align*}
\mathbb{E}_{\nu_{\rho}} [ \mathcal{Z}_t^{n,+}(G) \mathcal{Z}_t^{n,-}(H) ] = & \frac{1}{n} \sum_{x,y}  G \big( \tfrac{x - t v_n^{+}}{n} \big) \; H \big( \tfrac{y - t v_n^{-}}{n} \big)  \mathbb{E}_{\nu_{\rho}} [ f_x ( \eta_t^n)  g_y ( \eta_t^n) ]. 
\end{align*}
For every $t \in [0, T]$, the random variables $\{\xi_x^{A}( \eta_t^n ), \xi_x^{B}( \eta_t^n ), \; x \in \bb Z \}$ are independent under $\nu_{\rho}$ we get that $\mathbb{E}_{\nu_{\rho}} [ f_x ( \eta_t^n)  g_y ( \eta_t^n) ] =0$ whenever $x \neq y$, which leads to 
\begin{equation} \label{decorrfiel}
\forall G, H \in \mathcal{S}(\mathbb{R}), \quad \mathbb{E}_{\nu_{\rho}} [ \mathcal{Z}_t^{n,+}(G) \mathcal{Z}_t^{n,-}(H) ] = \frac{1}{n} \sum_{x}  G \big( \tfrac{x - t v_n^{+}}{n} \big) \; H \big( \tfrac{x - t v_n^{-}}{n} \big)  \mathbb{E}_{\nu_{\rho}} [ f_x ( \eta_t^n)  g_x ( \eta_t^n) ] =0.
\end{equation}
In the last equality we used the fact (due to \eqref{defD2pm}, \eqref{deffg} and \eqref{simp}) that
\begin{equation*} 
\forall z \in \mathbb{Z}, \quad \mathbb{E}_{\nu_{\rho}} [ f_z ( \eta)  g_z ( \eta) ] = \frac{2 (D_1)^2 + 2 D_2^{+} D_2^{-} - D_1 (D_2^{+} + D_2^{-} ) }{9} =0.
\end{equation*}
Now for any $n \in \mathbb{N}_+$, $t \in [0,T]$ and $G, H \in \mathcal{S}(\mathbb{R})$, we get from polarization identities that
\begin{align*}
2 \langle \mathcal{M}_t^{n, +}(G),  \mathcal{M}_t^{n,-}(H) \rangle_t = \langle \mathcal{M}_t^{n, +}(G) +  \mathcal{M}_t^{n,-}(H) \rangle_t - \langle \mathcal{M}_t^{n, +}(G) \rangle_t - \langle   \mathcal{M}_t^{n,-}(H) \rangle_t.
\end{align*}
Therefore, by combining the last display with Proposition \ref{propquadvarM} and \eqref{genxiza}, \eqref{genxizb}, \eqref{ezalneza},  \eqref{ezblneza}, \eqref{ezalnezb}, \eqref{ezblnezb}, we get after performing some algebraic manipulations that
\begin{equation} \label{croossvar}
\begin{split}
 \langle \mathcal{M}_t^{n, +}(G), \;  \mathcal{M}_t^{n,-}(H) \rangle_t =& \frac{\Theta(n)}{2n} \int_0^t \dd s \sum_{z,w} P_{z,w,n}^{s,v^+,v^-} (G,H)  h_{z,w}(\eta_s^n),
\end{split}
\end{equation}
where for every $z,w \in \mathbb{Z}$, $h_{z,w}: \Omega \mapsto \mathbb{R}$ is given by
\begin{align*}
h_{z,w}(\eta):= [f_w( \eta ) - f_z( \eta )] [g_w( \eta ) - g_z( \eta )] [ p(w-z) r^n_{z,w}(\eta) + p(z-w) r^n_{w,z}(\eta) ],
\end{align*}
and for any $s \in [0, T]$, $n \in \mathbb{N}$, $P_{z,w,n}^{s,v^+,v^-} (G,H)$ is given by
\begin{align*}
P_{z,w,n}^{s,v^+,v^-} (G,H):= [ G ( \tfrac{w - s v^{+}_n  }{n}  ) - G ( \tfrac{z - s v^{+}_n  }{n}  ) ] [ H ( \tfrac{w - s v^{-}_n  }{n}  ) - H ( \tfrac{z - s v^{-}_n  }{n}  )].
\end{align*}
Combining \eqref{defD2pm} with \eqref{deffg} and \eqref{simp}, we get
\begin{align*}
\forall z,w \in \mathbb{Z}, \quad \mathbb{E}_{\nu_{\rho}} [ h_{z,w}(\eta) ] =4 s(w-z) \frac{2 (D_1)^2 + 2 D_2^{+} D_2^{-} - D_1 (D_2^{+} + D_2^{-} ) }{9} =0.
\end{align*}
Combining the last expression with \eqref{croossvar} and Fubini's Theorem, we get that
\begin{equation} \label{expcrossvar}
\begin{split}
&\mathbb{E}_{\nu_{\rho}} \big[ \langle \mathcal{M}_t^{n, +}(G), \;  \mathcal{M}_t^{n,-}(H) \rangle_t \big] \\
=& \frac{\Theta(n)}{2n} \int_0^t \dd s \sum_{z,w}[ G ( \tfrac{w - s v^{+}_n  }{n}  ) - G ( \tfrac{z - s v^{+}_n  }{n}  ) ] [ H ( \tfrac{w - s v^{-}_n  }{n}  ) - H ( \tfrac{z - s v^{-}_n  }{n}  )] \mathbb{E}_{\nu_{\rho}} [ h_{z,w}(\eta_s^n) ]=0.
\end{split}
\end{equation}
By performing some computations, we get
\begin{align*}
\forall G, H \in \mathcal{S}(\mathbb{R}), \quad \lim_{n \rightarrow \infty} \mathbb{E}_{\nu_{\rho}} \Big[ \Big\{ \langle \mathcal{M}_t^{n, +}(G), \;  \mathcal{M}_t^{n,-}(H) \rangle_t - \mathbb{E}_{\nu_{\rho}} \big[ \langle \mathcal{M}_t^{n, +}(G), \;  \mathcal{M}_t^{n,-}(H) \rangle_t \big] \Big\}^2 \Big] =0.
\end{align*}
Combining the last expression with \eqref{expcrossvar}, we conclude that for every $G, H \in \mathcal{S}(\mathbb{R})$ and every $t \in [0,T]$ fixed, the sequence of random variables $\big( \langle \mathcal{M}_t^{n, +}(G), \;  \mathcal{M}_t^{n,-}(H) \rangle_t \big)_{n \in \mathbb{N}_+}$ converges in probability to zero, as $n \rightarrow \infty$. 
  
Next, from Proposition \ref{convinifie}, $( \mathcal{Z}_0^n)_{n \in \mathbb{N}_+}$ converges in distribution to $\mathcal{Z}_0$, a mean zero Gaussian field with covariance given by \eqref{covinifie}, with $D$ being replaced by $D_3$. This ensures that the first condition stated in Definitions \ref{defspde} and \ref{defspdefbe} is fulfilled. Furthermore, from \eqref{decorrfiel}, we conclude that $\big( \mathcal{Z}_t^{+}, \; \mathcal{Z}_t^{-}  \big)_{0 \leq t \leq T}$, where $(\mathcal{Z}_t^{+})_{0 \leq t \leq T}$ and $(\mathcal{Z}_t^{-})_{0 \leq t \leq T}$ are uncorrelated stochastic processes.

\subsection{Identification of the quadratic variation.} In order to complete the characterization of the limit points, in the remainder of this section we fix $H \in \mathcal{S}(\mathbb{R})$. From Proposition \ref{convmart}, we have that $\big\{ \big(  \mathcal{M}_t^{n,\pm}(H)  \big)_{0 \leq t \leq T}, \; n \in \mathbb{N}_+\big\}$ converges with respect to the Skorohod topology of $\mathcal{D}([0, T], \mathbb{R})$ to $\big(  \mathcal{M}_t^{\pm}(H)  \big)_{0 \leq t \leq T}$, as $n \rightarrow \infty$. Furthermore, $\big(  \mathcal{M}_t^{\pm}(H)  \big)_{0 \leq t \leq T}$ is a continuous martingale on $[0, T]$. We focus now on the second half of the second condition stated in Definitions \ref{defspde} and \ref{defspdefbe}, i.e., we will prove that the process $\mathcal{N}_t(H)$ given by
\begin{align*} \mathcal{N}^{\pm}_t(H) :=[ \mathcal{M}^{\pm}_t(H)]^2  -2  D_3^{\pm} t \;    \mathcal{P}^{\gamma} H,
\end{align*} 
is also a continuous martingale. In order to do so, for every $n \in \mathbb{N}_+$, define the martingale $\mathcal{N}_t^{n, \pm}(H):=[\mathcal{M}_t^{n, \pm}(H)]^2 - \langle \mathcal M^{n, \pm}(H)  \rangle_t$ for every $t \in [0,T]$. From Proposition \ref{convmart} we get that $\mathcal{N}_t^{n,\pm}(H)$ converges in distribution to $\mathcal{N}_t^{\pm}(H)$ as $n \rightarrow \infty$. At this point we claim that there exists some constant $C_0(H,T)$ depending only on $H$ and $T$ such that
\begin{equation} \label{claimNt}
\forall n \in \mathbb{N}_+, \; \forall t \in [0,T], \quad  \quad \mathbb{E}_{\nu_{\rho}} \big[ \big\{ \mathcal{N}_t^{n,\pm}(H) \big\}^2 \big] \leq C_0(H,T).
\end{equation}
By assuming \eqref{claimNt}, $\big\{ \big(  \mathcal{N}_t^{n,\pm}(H)  \big)_{0 \leq t \leq T}, \; n \in \mathbb{N}_+\big\}$ is a uniformly integrable sequence of martingales, therefore its limit $\big(  \mathcal{N}_t^{\pm}(H)  \big)_{0 \leq t \leq T}$ is itself a martingale. In order to obtain \eqref{claimNt}, we observe from \eqref{limpropvarvarquad} that there exists some constant $C_1(H,T)>0$ depending only on $H$ and $T$ such that
	\begin{align} \label{bndvarquadvar2}
		\forall n \in \mathbb{N}_+, \; \forall t \in [0,T], \quad \mathbb{E}_{\nu_{\rho}} \big[ \big(  \langle \mathcal M^{n, \pm} (H)  \rangle_t - \mathbb{E}_{\nu_{\rho}} [   \langle \mathcal M^{n, \pm} (H)  \rangle_t ] \big)^2 \big] \leq C_1(H, T).
	\end{align}
On the other hand, from \eqref{expquadvar} and \eqref{limexpquadvar0a}, there exists some constant $C_1(H,T)>0$ depending only on $H$ and $T$ such that
\begin{align*}
\forall n \in \mathbb{N}_+, \; \forall t \in [0,T], \quad  \E_{\nu_{\rho}} \big[ \langle \mathcal{M}^{n,\pm}(H) \rangle_t \big] \leq C_2(H,T).
\end{align*}
Combining the last line with \eqref{bndvarquadvar2} and the inequality $(u+v)^2 \leq 2 (u^2 + v^2)$, we conclude that
\begin{align} \label{bndvarquadvar3}
\forall n \in \mathbb{N}_+, \; \forall t \in [0,T], \quad \mathbb{E}_{\nu_{\rho}} \big[ \big(  \langle \mathcal M^n (H)  \rangle_t \big)^2 \big] \leq 2 \big[ C_1(H,T) + \big\{C_2(H,T)\big\}^2 \big].
\end{align} 
Furthermore, from \eqref{bndvarquadvar2} and Lemma 3 in \cite{Dittrich91}, there exists a constant $C_3(H,T)$ such that
\begin{align*}
\forall n \in \mathbb{N}_+, \; \forall t \in [0,T], \quad \E_{\nu_{\rho}} \big[ \big\{  \mathcal{M}_t^{n,\pm}(H) \big\}^2 \big] \leq C_3(H,T).
\end{align*}
Combining the last line with \eqref{bndvarquadvar2} and the inequality $(u+v)^2 \leq 2 (u^2 + v^2)$, we obtain \eqref{claimNt}.

\subsection{Identification of the limiting equations.} We still need to show that the first half of the second condition stated in Definitions \ref{defspde} and \ref{defspdefbe} is satisfied, depending on whether Hypothesis \ref{hyplin} or Hypothesis \ref{hypnonlin} holds. Keeping this in mind, from \eqref{decompZ}, for every $n \in \mathbb{N}_+$, it holds
\begin{align} \label{decompZpm}
\mathcal{Z}_{t}^{n,\pm}(H) = \mathcal{Z}_{0}^{n,\pm}(H) + \mathcal{M}_t^{n,\pm}(H) + \mcb{I}_t^{n,\pm}(H) + \mcb{B}_t^{n,\pm}(H).
\end{align}
Next, from Proposition \ref{convinifie}, $ \big( \mathcal{Z}_{0}^{n,\pm}(H) \; \big)_{n \in \mathbb{N}_+}$ converges in distribution to $\mathcal{Z}_{0}^{\pm}(H)$ with respect to the Skorohod topology of $\mathcal{D}([0, T], \mathbb{R})$, as $n \rightarrow \infty$. At this point we observe that the sequences $\big\{ \big(  \mcb{I}_t^{n,\pm}(H)  \big)_{0 \leq t \leq T}, \; n \geq 1\big\}$ and $\big\{ \big(  \mcb{B}_t^{n,\pm}(H)  \big)_{0 \leq t \leq T}, \; n \in \mathbb{N}_+ \big\}$ are not necessarily convergent with respect to the Skorohod topology of $\mathcal{D}([0, T], \mathbb{R})$. Nevertheless, under at least one of the Hypotheses \ref{hyplin} or \ref{hypnonlin}, we get from Remark \ref{rem:KCtight} that the aforementioned sequences are both tight with respect to the Skorohod topology of $\mathcal{D}([0, T], \mathbb{R})$. Therefore, there exists some subsequence $(n_j)_{j \in \mathbb{N}_+}$ such that $\big\{ \big(  \mcb{I}_t^{n_j,\pm}(H)  \big)_{0 \leq t \leq T}, \; j \in \mathbb{N}_+ \big\}$ and $\big\{ \big(  \mcb{B}_t^{n_j,\pm}(H)  \big)_{0 \leq t \leq T}, \; j \in \mathbb{N}_+\big\}$ converge with respect to the Skorohod topology of $\mathcal{D}([0, T], \mathbb{R})$ to $\big(  \mcb{I}_t^{\pm}(H)  \big)_{0 \leq t \leq T}$ and to $\big(  \widehat{\mcb{B}}_t^{\pm}(H)  \big)_{0 \leq t \leq T}$, respectively, as $j \rightarrow \infty$. 

In particular, $\big\{ \big(  \mathcal{Z}_{t}^{n_j,\pm}(H)  \big)_{0 \leq t \leq T}, \; j \in \mathbb{N}_+\big\}$ converges with respect to the Skorohod topology of $\mathcal{D}([0, T], \mathbb{R})$ to $\big( \mathcal{Z}_{t}^{\pm}(H)  \big)_{0 \leq t \leq T}$, as $j \rightarrow \infty$. Thus, applying \eqref{decompZpm} for $n=n_j$ and taking $j \rightarrow \infty$, we conclude that
\begin{equation} \label{limdecompZpm}
\forall t \in [0, \; T], \quad \mathcal{Z}_{t}^{\pm}(H) = \mathcal{Z}_{0}^{\pm}(H) + \mathcal{M}_t^{\pm}(H) + \mcb{I}_t^{\pm}(H) + \widehat{\mcb{B}}_t^{\pm}(H).
\end{equation}
Now from Lemma \ref{convL2lapyesvel}, 
\begin{align*}
\mcb{I}_t^{\pm}(H) = \int_0^t \mcb{Z}_r \big( \mathbb{L}^{\gamma} H -  \widehat{\mathbb{L}}_{ \lambda}^{ \gamma} H \big) \; \dd r = \mathcal{I}_t^{\pm}(H),
\end{align*}
where $\mathcal{I}_t(H)$ is given by \eqref{intprocess}.

\subsubsection{Identification of the limiting equation under Hypothesis \ref{hyplin}.}

If Hypothesis \ref{hyplin} holds, from Proposition \ref{varnlinvan} and the Kolmogorov-Centsov Theorem, the sequence $\big\{ \big(  \mcb{B}_t^{n,\pm}(H)  \big)_{0 \leq t \leq T}, \; n \in \mathbb{N}_+ \big\}$ converges with respect to the Skorohod topology of $\mathcal{D}([0, T], \mathbb{R})$ to zero, as $n \rightarrow \infty$. By combining this with \eqref{limdecompZpm}, we conclude that
\begin{align*}
\forall t \in [0, \; T], \quad \mathcal{Z}_{t}^{\pm}(H) = \mathcal{Z}_{0}^{\pm}(H) + \mathcal{M}_t^{\pm}(H) + \mcb{I}_t^{\pm}(H).
\end{align*}
In particular, the second condition stated in Definition \ref{defspde} is satisfied. In this way, we conclude that $(\mathcal{Z}_t^{+})_{0 \leq t \leq T}$ and $(\mathcal{Z}_t^{-})_{0 \leq t \leq T}$ are uncorrelated stationary solutions of \eqref{spdeteo}. This ends the proof of Theorem \ref{clt}.

\subsubsection{Identification of the limiting equation under Hypothesis \ref{hypnonlin}.} \label{secidelimhip2}

If Hypothesis \ref{hypnonlin} holds, according to Remark \ref{remgam32}, the nonlinearity in the equation survives and $\gamma \in [3/2,2) \cup (2, \infty)$. We define $\omega_{\gamma}:=\gamma-1>0$, if $\gamma \in [3/2,2)$; and $\omega_{\gamma}:=1$, if $\gamma > 2$. 

Recall the definition of $\mcb{B}_t^{n, \alpha , \beta}(H)$ in \eqref{defBtnalbe}, of $K^{\star}$ in \eqref{defKstar}, of $m_a$ in \eqref{defma} and of $\Psi_{x, \varepsilon n}^{\alpha, \beta}$ in \eqref{defpsixL}.  We claim that
\begin{equation}  \label{apxepsgen}
\mathbb{E}_{\nu_{\rho}} \Bigg[ \Bigg\{  \mcb{B}_t^{n, \alpha , \beta}(H) - \mcb{B}_s^{n, \alpha , \beta}(H)-  K^{\star} m_a \int_s^t \dd r \sum_{x}   \nabla H ( \tfrac{ \lfloor x - r v_n \rfloor }{n}  )   \Psi_{x, \varepsilon n}^{\alpha, \beta}(\eta_r^n)  \Bigg\}^{2} \Bigg] \leq \widehat{C} (t-s) \varepsilon^{\omega_{\gamma}},	
\end{equation}
for any $\varepsilon \in (0,1/2)$, as long as $n \in \mathbb{N}_+$ is large enough. Above and in the remainder of this section, $\widehat{C}$ denotes a constant (which may change from line to line) independent of $s$, $t$ and $\varepsilon$. In order to obtain \eqref{apxepsgen}, it is enough to
\begin{itemize}
    \item 
apply Corollary \ref{corngamgr1} (with $L_n=\sqrt{n}$) and Corollary \ref{corlemtbe} (with $\delta=0$, $\ell_n = \sqrt{n}$ and $L_n = \varepsilon n$), if $\gamma \in [3/2,2)$;
\item 
apply \ref{corngamgr2} and \eqref{2bndexpdif} (both with $L_n= \varepsilon n$), if $\gamma > 2$.
\end{itemize}
Next, from \eqref{defpsixL} we get that 
	\begin{align*}
		\overleftarrow{\xi}_x^{\alpha, \varepsilon n}(\eta_r^n) =&\frac{1}{\varepsilon n} \sum_{j}\mathbbm{1}_{ \{ [-\varepsilon, \; 0)   \}  } ( \tfrac{-j}{n} ) \bar{\xi} _{x-j}^{\alpha}( \eta_r^n ) =\frac{1}{ n} \sum_{j} \overleftarrow{\iota}_{\varepsilon}( \tfrac{-j}{n} ) \bar{\xi} _{x-j}^{\alpha}( \eta_r^n ) =\frac{1}{ n} \sum_{y} \overleftarrow{\iota}_{\varepsilon}( \tfrac{y-x}{n} ) \bar{\xi} _{y}^{\alpha}( \eta_r^n ),  \\
		\overrightarrow{\xi}_x^{\beta, \varepsilon n}(\eta_r^n)=& \frac{1}{\varepsilon n} \sum_{k}\mathbbm{1}_{ \{  (0, \; \varepsilon ]   \}  } ( \tfrac{-k}{n} ) \bar{\xi} _{x-k}^{\beta}( \eta_r^n ) =\frac{1}{ n} \sum_{k} \overrightarrow{\iota}_{\varepsilon}( \tfrac{-k}{n} ) \bar{\xi} _{x-k}^{\beta}( \eta_r^n )= \frac{1}{ n} \sum_{z} \overrightarrow{\iota}_{\varepsilon}( \tfrac{z-x}{n} )  \bar{\xi} _{z}^{\beta}( \eta_r^n ).
	\end{align*}
In this way, we conclude that for every $n \in \mathbb{N}_+$, any $r \in [0,T]$ and any $\varepsilon \in (0, 1/2)$, it holds
\begin{align}
 \Psi_{x, \varepsilon n}^{\alpha, \beta}(\eta_r^n) = \frac{1}{ n^{2} } \sum_{y,z} \overleftarrow{\iota}_{\varepsilon}( \tfrac{y-x}{n} ) \overrightarrow{\iota}_{\varepsilon}( \tfrac{z-x}{n} ) \bar{\xi} _{y}^{\alpha}( \eta_r^n ) \bar{\xi} _{z}^{\beta}( \eta_r^n ), \label{approxepsn}
\end{align}
where $\overleftarrow{\iota}_{\varepsilon} : \mathbb{R} \mapsto \mathbb{R}$ and $\overrightarrow{\iota}_{\varepsilon} : \mathbb{R} \mapsto \mathbb{R}$ are given by \eqref{defiotaeps}. Due to Theorem \ref{thm:Cross}, we now proceed by approximating $\overleftarrow{\iota}_{\varepsilon}$ and $\overrightarrow{\iota}_{\varepsilon}$ by smooth functions of compact support. In order to do this, the following lemma will be helpful.
\begin{lem} \label{lemvarfgeps}
Let $f_{\varepsilon}, g_{\varepsilon}: \mathbb{R} \mapsto \mathbb{R}$ be such that the supports of $f_{\varepsilon}$ and $g_{\varepsilon}$ are contained in $[- \varepsilon, \; 0)$ and $(0, \; \varepsilon]$, respectively. Recall the definition of $ \| \cdot \|_{2,n}$ in \eqref{L2discnorm}. It holds
\begin{align*} 
& \mathbb{E}_{\nu_{\rho}} \Bigg[ \Bigg\{  \frac{1}{ n^{2} } \int_s^t \dd r \sum_{x}   \nabla H ( \tfrac{ \lfloor x - r v_n \rfloor }{n}  ) \sum_{y,z} f_{\varepsilon}( \tfrac{y-x}{n} ) g_{\varepsilon}( \tfrac{z-x}{n} ) \bar{\xi} _{y}^{\alpha}( \eta_r^n ) \bar{\xi} _{z}^{\beta}( \eta_r^n ) \Bigg\}^{2} \Bigg] \\
 \leq & 2 \chi(\rho_{\alpha}) \chi(\rho_{\beta})  \| \nabla H  \|^2_{2,n} (t-s)^2    \| f_{\varepsilon} \|^2_{2,n} \; \| g_{\varepsilon} \|^2_{2,n} \varepsilon.
\end{align*}
\end{lem}
\begin{proof}	
From the Cauchy-Schwarz inequality, the expectation in the statement of Lemma \ref{lemvarfgeps} is bounded from above by
\begin{align}
& \frac{t-s}{n^4} \mathbb{E}_{\nu_{\rho}} \Bigg[   \int_s^t \dd r \Bigg\{ \sum_{x}   \nabla H ( \tfrac{ \lfloor x - r v_n \rfloor }{n}  ) Z_x^{\varepsilon}(\eta_r^n) W_x^{\varepsilon}(\eta_r^n) \Bigg\}^{2} \Bigg],  \label{estepsn}
\end{align} 
where for every $\varepsilon >0$, $n \in \mathbb{N}_+$ and $x \in \mathbb{Z}$, define $Z_x^{\varepsilon}, W_x^{\varepsilon}: \Omega \mapsto \mathbb{R}$ by
\begin{align*}
\forall \eta \in \Omega, \quad Z_x^{\varepsilon}(\eta):= \sum_y \bar{\xi}_{x+y}^{\alpha}( \eta ) f_{\varepsilon} \big( \tfrac{y}{n} \big), \quad W_x^{\varepsilon}(\eta):= \sum_y \bar{\xi}_{x+y}^{\alpha}( \eta ) g_{\varepsilon} \big( \tfrac{y}{n} \big). 
\end{align*}
In particular, for any $\varepsilon >0$, $n \in \mathbb{N}_+$, $x \in \mathbb{Z}$ and $r \in [s, \; t]$, it holds
\begin{align*}
\mathbb{E}_{\nu_{\rho}} \big[  \big\{ Z_x^{\varepsilon}(\eta_r^n)  \big\}^2 \big] 
=&   \sum_{y,z}  f_{\varepsilon} \big( \tfrac{y}{n} \big)  \sum_z  f_{\varepsilon} \big( \tfrac{z}{n} \big) \mathbb{E}_{\nu_{\rho}} \big[ \bar{\xi}_{x+y}^{\alpha}( \eta_r^n ) \bar{\xi}_{x+z}^{\alpha}( \eta_r^n ) \big].
\end{align*}
We observe that $\mathbb{E}_{\nu_{\rho}} \big[  \big\{ W_x^{\varepsilon}(\eta_r^n)  \big\}^2 \big]$ can be computed in an analogous way. From \eqref{corrzero}, we get
\begin{equation}  \label{expZW2}
\mathbb{E}_{\nu_{\rho}} \big[  \big\{ Z_x^{\varepsilon}(\eta_r^n)  \big\}^2 \big] = \chi(\rho_{\alpha}) \sum_{y} \big[  f_{\varepsilon} \big( \tfrac{y}{n} \big) \big]^2, \quad \mathbb{E}_{\nu_{\rho}} \big[  \big\{ W_x^{\varepsilon}(\eta_r^n)  \big\}^2 \big] = \chi(\rho_{\beta}) \sum_{y} \big[  g_{\varepsilon} \big( \tfrac{y}{n} \big) \big]^2.
\end{equation}
Since the supports of $f_{\varepsilon}$ and $g_{\varepsilon}$ are contained in $[- \varepsilon, \; 0)$ and $(0, \; \varepsilon]$, respectively, from \eqref{corrzero} we get
\begin{equation} \label{uncorr}
\forall x,y \in \mathbb{Z}: |y-x| \geq 2 \varepsilon n, \quad \mathbb{E}_{\nu_{\rho}} \big[  Z_x^{\varepsilon}(\eta_r^n) W_x^{\varepsilon}(\eta_r^n) Z_y^{\varepsilon}(\eta_r^n) W_y^{\varepsilon}(\eta_r^n) \big]=0.
\end{equation}
The display in the last line motivates us to write the decomposition $\mathbb{Z} = \cup_{k=0}^{2 \varepsilon n - 1} \mcb X_k$, where $\mcb X_k := \{ x \in \mathbb{Z}: \quad \exists q \in \mathbb{Z}: x - q \varepsilon n =k  \}$. In this way, the term in \eqref{estepsn} can be rewritten as
\begin{align*}
& \frac{t-s}{n^4} \mathbb{E}_{\nu_{\rho}} \Bigg[   \int_s^t \dd r \Bigg\{ \sum_{k=0}^{2 \varepsilon n-1} \sum_{x \in \mcb X_k}   \nabla H ( \tfrac{ \lfloor x - r v_n \rfloor }{n}  ) Z_x^{\varepsilon}(\eta_r^n) W_x^{\varepsilon}(\eta_r^n) \Bigg\}^{2} \Bigg].
\end{align*}
Combining a convex inequality, Fubini's Theorem and the Cauchy-Schwarz inequality, the last display is bounded from above by
\begin{align*}
 & \frac{t-s}{n^4} \mathbb{E}_{\nu_{\rho}} \Bigg[   \int_s^t \dd r (2 \varepsilon n) \sum_{k=0}^{2 \varepsilon n-1} \Bigg\{  \sum_{x \in \mcb X_k}   \nabla H ( \tfrac{ \lfloor x - r v_n \rfloor }{n}  ) Z_x^{\varepsilon}(\eta_r^n) W_x^{\varepsilon}(\eta_r^n) \Bigg\}^{2} \Bigg] \\
\leq & \frac{2 \varepsilon (t-s)}{n^3} \int_s^t \dd r  \Bigg\{\sum_{k=0}^{2 \varepsilon n-1} \sum_{x \in \mcb X_k} \big\{ \nabla H ( \tfrac{ \lfloor x - r v_n \rfloor }{n}  ) \big\}^2  \mathbb{E}_{\nu_{\rho}} \big[  \big\{ Z_x^{\varepsilon}(\eta_r^n)  \big\}^2 \big] \cdot \mathbb{E}_{\nu_{\rho}} \big[  \big\{  W_x^{\varepsilon}(\eta_r^n) \big\}^2 \big]    \Bigg\}.
\end{align*}
The proof ends by making use of \eqref{expZW2} in order to rewrite the last display as
\begin{align*}
& \frac{2 \varepsilon (t-s)}{n^3} \int_s^t \dd r  \Bigg\{\sum_{k=0}^{2 \varepsilon n-1} \sum_{x \in \mcb X_k} \big\{ \nabla H ( \tfrac{ \lfloor x - r v_n \rfloor }{n}  ) \big\}^2  \chi(\rho_{\alpha}) \sum_{y} \big[  f_{\varepsilon} \big( \tfrac{y}{n} \big) \big]^2 \chi(\rho_{\beta}) \sum_{z} \big[  g_{\varepsilon} \big( \tfrac{z}{n} \big) \big]^2 \Bigg\} \\
= & \frac{2 \varepsilon (t-s)}{n^3} \chi(\rho_{\alpha}) \chi(\rho_{\beta}) \sum_{y} \big[  f_{\varepsilon} \big( \tfrac{y}{n} \big) \big]^2  \sum_{z} \big[  g_{\varepsilon} \big( \tfrac{z}{n} \big) \big]^2  \int_s^t \dd r  \Bigg\{  \sum_{x} \big\{ \nabla H ( \tfrac{ \lfloor x - r v_n \rfloor }{n}  ) \big\}^2   \Bigg\}. 
\end{align*} 
\end{proof}
Now we will make use of the following result, which is proved in Appendix \ref{useest11}.
\begin{lem} \label{lemaproxeps}
For every $\varepsilon \in (0, \; 1/2)$, there exist $f_{\varepsilon}, g_{\varepsilon} \in C_c^{\infty}(\mathbb{R})$ such that
\begin{equation} \label{aproxepsk}
	\forall \varepsilon \in (0, \; 1/2), \quad \|   \overleftarrow{\iota}_{\varepsilon} - f_{\varepsilon} \|^2_{2,n} = \|   \overrightarrow{\iota}_{\varepsilon} - g_{\varepsilon} \|^2_{2,n} \leq 2 \varepsilon^{2}.
\end{equation}	
\end{lem}
Next, observe that for every $\varepsilon >0$ and any $h_{\varepsilon}: \mathbb{R} \mapsto \mathbb{R}$ such that $\sup_{u \in \mathbb{R}} |h_{\varepsilon}(u)| \leq \varepsilon^{-1}$, it holds
\begin{equation} \label{discnorm2h}
 \| h_{\varepsilon} \|^2_{2,n} = \frac{1}{n} \sum_x \big[ h \big( \tfrac{x}{n}  \big)  \big]^2  \leq \frac{1}{n} \sum_{x=1}^{\varepsilon n} \Bigg( \frac{1}{\varepsilon} \Bigg)^2  = \frac{1}{\varepsilon},
\end{equation}
as long as the support of $h_{\varepsilon}$ is either contained in $[- \varepsilon, \; 0)$, or contained in $(0, \; \varepsilon]$. Thus, combining Lemma \ref{lemvarfgeps} with \eqref{aproxepsk} and \eqref{discnorm2h}, we get
\begin{equation} \label{aproxconveps}
\begin{split}
 \mathbb{E}_{\nu_{\rho}} \Bigg[ \Bigg\{  \frac{1}{ n^{2} } \int_s^t \dd r \sum_{x}   \nabla H ( \tfrac{ \lfloor x - r v_n \rfloor }{n}  ) \Bigg[ &\sum_{y,z} \overleftarrow{\iota}_{\varepsilon}( \tfrac{y-x}{n} ) \overrightarrow{\iota}_{\varepsilon}( \tfrac{z-x}{n} ) \bar{\xi} _{y}^{\alpha}( \eta_r^n ) \bar{\xi} _{z}^{\beta}( \eta_r^n )  \\
- & \sum_{y,z} f_{\varepsilon}( \tfrac{y-x}{n} ) g_{\varepsilon}( \tfrac{z-x}{n} ) \bar{\xi} _{y}^{\alpha}( \eta_r^n ) \bar{\xi} _{z}^{\beta}( \eta_r^n ) \Bigg] \Bigg\}^{2} \Bigg] \leq  \widehat{C}  (t-s)^2 \varepsilon. 
\end{split}
\end{equation}
Combining the last display with \eqref{apxepsgen} and \eqref{approxepsn}, we have
\begin{equation}  \label{apxepsgensm}
\begin{split}
	\mathbb{E}_{\nu_{\rho}} \Bigg[ \Bigg\{ & \mcb{B}_t^{n, \alpha , \beta}(H) - \mcb{B}_s^{n, \alpha , \beta}(H) \\
	- & \frac{K^{\star} m_a }{n^{2}} \int_s^t \dd r \sum_{x}   \nabla H ( \tfrac{ \lfloor x - r v_n \rfloor }{n}  ) \sum_{y,z} f_{\varepsilon}( \tfrac{y-x}{n} ) g_{\varepsilon}( \tfrac{z-x}{n} ) \bar{\xi} _{y}^{\alpha}( \eta_r^n ) \bar{\xi} _{z}^{\beta}( \eta_r^n )  \Bigg\}^{2} \Bigg] \leq \widehat{C} (t-s) \varepsilon^{\omega_{\gamma}}.
\end{split}	
\end{equation}
In order to perform the change of variable $w=\lfloor x - r v_n \rfloor$ inside the integral of the last display, we observe that $-	x = -\big(x - r v_n - \lfloor x - r v_n \rfloor \big) - r v_n - \lfloor x - r v_n \rfloor$, for every $x \in \mathbb{Z}$. Thus, the double sum over $y$ and $z$ in \eqref{apxepsgensm} can be rewritten as
\begin{align*}
& \sum_{y,z} f_{\varepsilon}\Big( \tfrac{y-\big(x - r v_n - \lfloor x - r v_n \rfloor \big) - r v_n - \lfloor x - r v_n \rfloor}{n} \Big) g_{\varepsilon}\Big( \tfrac{z-\big(x - r v_n - \lfloor x - r v_n \rfloor \big) - r v_n - \lfloor x - r v_n \rfloor}{n} \Big) \bar{\xi} _{y}^{\alpha}( \eta_r^n ) \bar{\xi} _{z}^{\beta}( \eta_r^n ).
\end{align*}
Keeping this in mind, we state and prove the following result.
\begin{prop} \label{propchanvar}
There exists some positive constant $C$ independent of $n$, $s$ and $t$ such that
\begin{align*}
\mathbb{E}_{\nu_{\rho}} \Bigg[ \Bigg\{ \int_s^t \dd r \sum_{x}   \nabla H ( \tfrac{ \lfloor x - r v_n \rfloor }{n}  ) \sum_{y,z} \frac{\bar{\xi} _{y}^{\alpha}( \eta_r^n ) \bar{\xi} _{z}^{\beta}( \eta_r^n ) }{n^{2}}  \Big[ & f_{\varepsilon}( \tfrac{y-x}{n} ) g_{\varepsilon}( \tfrac{z-x}{n} ) \\
- & f_{\varepsilon}( \tfrac{y- r v_n - \lfloor x - r v_n \rfloor}{n} ) g_{\varepsilon}( \tfrac{z- r v_n - \lfloor x - r v_n \rfloor}{n} ) \Bigg\}^{2} \Bigg] \leq \frac{C(t-s)^{2}}{n^2}.
\end{align*}
\end{prop}  
\begin{proof}
For every $x,y,z \in \mathbb{Z}$, $r \in {0,T}$ and $n \in \mathbb{N}_+$, let $q_{x}^{r,n}:= r v_n + \lfloor x - r v_n \rfloor$ and
\begin{align} \label{defqxyzrn}
q_{x,y,z}^{r,n}:= f_{\varepsilon}( \tfrac{y-x}{n} ) g_{\varepsilon}( \tfrac{z-x}{n} ) -  f_{\varepsilon}( \tfrac{y - q_{x}^{r,n}}{n} ) g_{\varepsilon}( \tfrac{z - q_{x}^{r,n}}{n} ).
\end{align}
Since the support of $f_{\varepsilon}$, resp. $g_{\varepsilon}$, is contained on $[0, \varepsilon]$, resp. on $[-\varepsilon, 0]$, we get 
\begin{equation} \label{suppqxyzrn}
\forall x,y,z \in \mathbb{Z}, \; \forall n \in \mathbb{N}_+, \; \forall r \in [0,T] ¸\quad \max \{|y-x|, |z-x|\} \geq 2 \varepsilon n \Rightarrow q_{x,y,z}^{r,n}=0.
\end{equation}
In the last line we combined the assumption that $\varepsilon n \geq 1$ with the fact that $0 \leq w - q_w^{r,n} \leq 1$, for any $w \in \mathbb{Z}$, $n \in \mathbb{N}_+$ and $r \in [0,T]$. Next we obtain an uniform upper bound for $q_{x,y,z}^{r,n}$. Observe that
\begin{align*}
	q_{x,y,z}^{r,n}= \big[ f_{\varepsilon}( \tfrac{y-x}{n} ) -f_{\varepsilon}( \tfrac{y - q_{x}^{r,n}}{n} ) \big] g_{\varepsilon}( \tfrac{z-x}{n} ) + \big[ g_{\varepsilon}( \tfrac{z-x}{n} ) - g_{\varepsilon}( \tfrac{z - q_{x}^{r,n}}{n} ) ] f_{\varepsilon}( \tfrac{y - q_{x}^{r,n}}{n} ).
\end{align*}
Therefore, an application of the Mean Value Theorem leads to
\begin{align*}
	|	q_{x,y,z}^{r,n} | \leq  \| \nabla f_{\varepsilon} \|_{\infty} \frac{|x -q_{x}^{r,n} | }{n} \| g_{\varepsilon} \|_{\infty} + \| \nabla g_{\varepsilon} \|_{\infty} \frac{|x -q_{x}^{r,n} | }{n} \| f_{\varepsilon} \|_{\infty}.
\end{align*} 
Combining the last display with the fact that $0 \leq w - q_w^{r,n} \leq 1$, for any $w \in \mathbb{Z}$, $n \in \mathbb{N}_+$ and $r \in [0,T]$, we get
\begin{equation} \label{suprqxyzrn}
	\forall x,y,z \in \mathbb{Z}, \; \forall n \in \mathbb{N}, \; \forall r \in [0,T] ¸\quad |	q_{x,y,z}^{r,n} | \leq \frac{ \| \nabla f_{\varepsilon} \|_{\infty} \| g_{\varepsilon} \|_{\infty} +  \| \nabla g_{\varepsilon} \|_{\infty}  \|  f_{\varepsilon} \|_{\infty} }{n} = \frac{C(\varepsilon)}{n}.
\end{equation}
From \eqref{defqxyzrn}, the integrand inside the expectation in the statement of the claim can be rewritten as
\begin{align*}
& \sum_{x} \nabla H ( \tfrac{ \lfloor x - r v_n \rfloor }{n}  ) \sum_{y,z} \frac{\bar{\xi} _{y}^{\alpha}( \eta_r^n ) \bar{\xi} _{z}^{\beta}( \eta_r^n ) }{n^{2}} q_{x,y,z}^{r,n}.
\end{align*} 
From the Cauchy-Schwarz inequality with an application of Fubini's Theorem, the expectation in the statement of the claim is bounded from above
\begin{align*}
& (t-s)\mathbb{E}_{\nu_{\rho}} \Bigg[   \int_s^t \Bigg\{  \sum_{x}   \nabla H ( \tfrac{ \lfloor x - r v_n \rfloor }{n}  ) \sum_{y,z} \frac{\bar{\xi} _{y}^{\alpha}( \eta_r^n ) \bar{\xi} _{z}^{\beta}( \eta_r^n ) }{n^{2}}  q_{x,y,z}^{r,n}\Bigg\}^{2} \dd r  \Bigg].
\end{align*}
From an application of Fubini`s Theorem, the expectation in the last line can be rewritten as
\begin{align*}
&  \int_s^t      \sum_{x_1, x_2, y_1,y_2,z_1,z_2}  \nabla H ( \tfrac{ \lfloor x_1 - r v_n \rfloor }{n}  )  \nabla H ( \tfrac{ \lfloor x_2 - r v_n \rfloor }{n}  ) \frac{q_{x_1,y_1,z_1}^{r,n} q_{x_2,y_2,z_2}^{r,n}}{n^{4}}  \mathbb{E}_{\nu_{\rho}} \big[  \bar{\xi} _{y_1}^{\alpha}( \eta_r^n ) \bar{\xi} _{z_1}^{\beta}( \eta_r^n ) \bar{\xi} _{y_2}^{\alpha}( \eta_r^n ) \bar{\xi} _{z_2}^{\beta}( \eta_r^n ) \big] \dd r \\
=& (I) + (II) + (III) + (IV), 
\end{align*}
where $(I)$, $(II)$, $(III)$ and $(IV)$ are given by 
\begin{align*}
& (I):= \int_s^t      \sum_{x_1, x_2, y_1}  \nabla H ( \tfrac{ \lfloor x_1 - r v_n \rfloor }{n}  )  \nabla H ( \tfrac{ \lfloor x_2 - r v_n \rfloor }{n}  ) \frac{q_{x_1,y_1,y_1}^{r,n} q_{x_2,y_1,y_1}^{r,n}}{n^{4}}  \mathbb{E}_{\nu_{\rho}} \big[  \bar{\xi} _{y_1}^{\alpha}( \eta_r^n ) \bar{\xi} _{y_1}^{\beta}( \eta_r^n ) \bar{\xi} _{y_1}^{\alpha}( \eta_r^n ) \bar{\xi} _{y_1}^{\beta}( \eta_r^n ) \big] \dd r \\	
& (II):=  \int_s^t    \underset{ z_1 \neq y_1 } { \sum_{x_1, x_2, y_1,z_1} } \nabla H ( \tfrac{ \lfloor x_1 - r v_n \rfloor }{n}  )  \nabla H ( \tfrac{ \lfloor x_2 - r v_n \rfloor }{n}  ) \frac{q_{x_1,y_1,z_1}^{r,n} q_{x_2,y_1,z_1}^{r,n}}{n^{4}}  \mathbb{E}_{\nu_{\rho}} \big[  \bar{\xi} _{y_1}^{\alpha}( \eta_r^n ) \bar{\xi} _{z_1}^{\beta}( \eta_r^n ) \bar{\xi} _{y_1}^{\alpha}( \eta_r^n ) \bar{\xi} _{z_1}^{\beta}( \eta_r^n ) \big] \dd r \\
& (III):= \int_s^t    \underset{ y_2 \neq y_1 } {  \sum_{x_1, x_2, y_1,y_2} } \nabla H ( \tfrac{ \lfloor x_1 - r v_n \rfloor }{n}  )  \nabla H ( \tfrac{ \lfloor x_2 - r v_n \rfloor }{n}  ) \frac{q_{x_1,y_1,y_2}^{r,n} q_{x_2,y_2,y_1}^{r,n}}{n^{4}}  \mathbb{E}_{\nu_{\rho}} \big[  \bar{\xi} _{y_1}^{\alpha}( \eta_r^n ) \bar{\xi} _{y_2}^{\beta}( \eta_r^n ) \bar{\xi} _{y_2}^{\alpha}( \eta_r^n ) \bar{\xi} _{y_1}^{\beta}( \eta_r^n ) \big] \dd r \\
 & (IV):=  \int_s^t     \underset{ y_2 \neq y_1 } {  \sum_{x_1, x_2, y_1,y_2}  } \nabla H ( \tfrac{ \lfloor x_1 - r v_n \rfloor }{n}  )  \nabla H ( \tfrac{ \lfloor x_2 - r v_n \rfloor }{n}  ) \frac{q_{x_1,y_1,y_1}^{r,n} q_{x_2,y_2,y_2}^{r,n}}{n^{4}}  \mathbb{E}_{\nu_{\rho}} \big[  \bar{\xi} _{y_1}^{\alpha}( \eta_r^n ) \bar{\xi} _{y_1}^{\beta}( \eta_r^n ) \bar{\xi} _{y_2}^{\alpha}( \eta_r^n ) \bar{\xi} _{y_2}^{\beta}( \eta_r^n ) \big] \dd r.
\end{align*}
In what follows, we analyze each of the four terms in the last line separately. In the remainder of the proof, $C(\rho_{\alpha}, \rho_{ \beta } )$ denotes some constant depending only on $\rho_{\alpha}$ and $\rho_{\beta}$, which may change from line to line. 
\begin{itemize}
	\item 
 From Remark \ref{remind} and \eqref{suppqxyzrn}, we get that (I) can be rewritten as 
\begin{align*}
	 & C(\rho_{\alpha}, \rho_{ \beta } ) \int_s^t      \sum_{x_1, x_2, y_1}  \nabla H ( \tfrac{ \lfloor x_1 - r v_n \rfloor }{n}  )  \nabla H ( \tfrac{ \lfloor x_2 - r v_n \rfloor }{n}  ) \frac{q_{x_1,y_1,y_1}^{r,n} q_{x_2,y_1,y_1}^{r,n}}{n^{4}} \dd r \\
	= & C(\rho_{\alpha}, \rho_{ \beta } ) \int_s^t \sum_{x} \sum_{w=x-4 \varepsilon n}^{x + 4 \varepsilon n} \sum_{y=x-2 \varepsilon n}^{x + 2 \varepsilon n}  \nabla H ( \tfrac{ \lfloor x - r v_n \rfloor }{n}  )  \nabla H ( \tfrac{ \lfloor w - r v_n \rfloor }{n}  ) \frac{q_{x,y,y}^{r,n} q_{w,y,y}^{r,n}}{n^{4}} \dd r.
\end{align*} 
Combining the last display with \eqref{suprqxyzrn}, we conclude that $| (I)|$ is bounded from above by
\begin{align*}
 C(\rho_A, \rho_B) \int_s^t \sum_{x} \sum_{w=x-4 \varepsilon n}^{x + 4 \varepsilon n} \sum_{y=x-2 \varepsilon n}^{x + 2 \varepsilon n} | \nabla H ( \tfrac{ \lfloor x - r v_n \rfloor }{n}  ) | \; \| \nabla H \|_{\infty} \frac{C(\varepsilon)}{n^{6}} \dd r =  C(\rho_A, \rho_B)\frac{C(\varepsilon, H)}{n^{3}} (t-s). 
\end{align*}
\item 
From Remark \ref{remind}, \eqref{suppqxyzrn} and and \eqref{suprqxyzrn}, we have that $|(II)|$ is bounded from above by
\begin{align*}
	 & C(\rho_{\alpha}, \rho_{ \beta } ) \int_s^t \sum_{x} \sum_{w=x-4 \varepsilon n}^{x + 4 \varepsilon n} \sum_{y=x-2 \varepsilon n}^{x + 2 \varepsilon n} \sum_{z=x-2 \varepsilon n}^{x + 2 \varepsilon n} | \nabla H ( \tfrac{ \lfloor x - r v_n \rfloor }{n}  ) | \; |  \nabla H ( \tfrac{ \lfloor w - r v_n \rfloor }{n}  ) | \frac{|q_{x,y,z}^{r,n}| \cdot |q_{w,y,z}^{r,n}|}{n^{4}} \dd r \\
     \leq & C(\rho_{\alpha}, \rho_{ \beta } ) \int_s^t \sum_{x} \sum_{w=x-4 \varepsilon n}^{x + 4 \varepsilon n} \sum_{y=x-2 \varepsilon n}^{x + 2 \varepsilon n} \sum_{z=x-2 \varepsilon n}^{x + 2 \varepsilon n} | \nabla H ( \tfrac{ \lfloor x - r v_n \rfloor }{n}  ) | \; \| \nabla H \|_{\infty} \frac{C(\varepsilon)}{n^{6}} \dd r \\
	= & C(\rho_{\alpha}, \rho_{ \beta } ) \frac{C(\varepsilon, H)}{n^{2}} (t-s).
\end{align*} 
 \item 
From Remark \ref{remind}, \eqref{suppqxyzrn} and and \eqref{suprqxyzrn}, we get that $|(III)|$ is bounded from above by
\begin{align*}
	 & C(\rho_{\alpha}, \rho_{ \beta } ) \int_s^t \sum_{x} \sum_{w=x-4 \varepsilon n}^{x + 4 \varepsilon n} \sum_{y=x-2 \varepsilon n}^{x + 2 \varepsilon n} \sum_{z=x-2 \varepsilon n}^{x + 2 \varepsilon n} | \nabla H ( \tfrac{ \lfloor x - r v_n \rfloor }{n}  ) | \; |  \nabla H ( \tfrac{ \lfloor w - r v_n \rfloor }{n}  ) | \frac{|q_{x,y,z}^{r,n}| \cdot |q_{w,z,y}^{r,n}|}{n^{4}} \dd r \\	
     \leq & C(\rho_{\alpha}, \rho_{ \beta } ) \int_s^t \sum_{x} \sum_{w=x-4 \varepsilon n}^{x + 4 \varepsilon n} \sum_{y=x-2 \varepsilon n}^{x + 2 \varepsilon n} \sum_{z=x-2 \varepsilon n}^{x + 2 \varepsilon n} | \nabla H ( \tfrac{ \lfloor x - r v_n \rfloor }{n}  ) | \; \| \nabla H \|_{\infty} \frac{C(\varepsilon)}{n^{6}} \dd r \\
	= & C(\rho_{\alpha}, \rho_{ \beta } ) \frac{C(\varepsilon, H)}{n^{2}} (t-s).
\end{align*}
 \item 
From Remark \ref{remind}, \eqref{suppqxyzrn} and and \eqref{suprqxyzrn}, we get that $|(IV)|$ is bounded from above by
  \begin{align*}
  	 & C(\rho_{\alpha}, \rho_{ \beta } ) \int_s^t \sum_{x} \sum_{w} \sum_{y=x-2 \varepsilon n}^{x + 2 \varepsilon n} \sum_{z=w-2 \varepsilon n}^{w + 2 \varepsilon n} | \nabla H ( \tfrac{ \lfloor x - r v_n \rfloor }{n}  ) | \; |  \nabla H ( \tfrac{ \lfloor w - r v_n \rfloor }{n}  ) | \frac{|q_{x,y,z}^{r,n}| \cdot |q_{w,z,y}^{r,n}|}{n^{4}} \dd r \\
     \leq & C(\rho_{\alpha}, \rho_{ \beta } ) \int_s^t \sum_{x} \sum_{w} \sum_{y=x-2 \varepsilon n}^{x + 2 \varepsilon n} \sum_{z=w-2 \varepsilon n}^{w + 2 \varepsilon n} | \nabla H ( \tfrac{ \lfloor x - r v_n \rfloor }{n}  ) | \; |  \nabla H ( \tfrac{ \lfloor w - r v_n \rfloor }{n}  ) | \frac{C(\varepsilon)}{n^{6}} \dd r \\
  	= &  C(\rho_{\alpha}, \rho_{ \beta } ) \frac{C(\varepsilon, H)}{n^{2}} (t-s).
  \end{align*} 	
\end{itemize}
 \end{proof}
By combining Proposition \ref{propchanvar} with \eqref{apxepsgensm} and performing the change of variables $\lfloor x - r v_n^{\pm} \rfloor \mapsto w$, we have that for every $\varepsilon \in (0,1/2)$ and every $n$ large enough, it holds
\begin{align*}
	\mathbb{E}_{\nu_{\rho}} \Bigg[ \Bigg\{ & \mcb{B}_t^{n, \alpha , \beta}(H) - \mcb{B}_s^{n, \alpha , \beta}(H) \\
	- & \frac{K^{\star} m_a }{n^{2}} \int_s^t \dd r \sum_{w}   \nabla H ( \tfrac{ w }{n}  ) \sum_{y,z} f_{\varepsilon}^{w/n}( \tfrac{y- r v_n }{n} ) g_{\varepsilon}^{w/n}( \tfrac{z- r v_n }{n} )  \bar{\xi} _{y}^{\alpha}( \eta_r^n ) \bar{\xi} _{z}^{\beta}( \eta_r^n )  \Bigg\}^{2} \Bigg] \leq \widehat{C} (t-s) \varepsilon^{\omega_{\gamma}},
\end{align*}
where for every $u \in \mathbb{R}$, $f_{\varepsilon}^{u},g_{\varepsilon}^u \in C_c^{\infty}(\mathbb{R})$ are defined by \eqref{smoothshift}. Now from \eqref{defYalp} and \eqref{translop}, we rewrite the last display as
\begin{equation} \label{aproxBngood}
\begin{split}
	\mathbb{E}_{\nu_{\rho}} \Bigg[ \Bigg\{ & \mcb{B}_t^{n, \alpha , \beta}(H) - \mcb{B}_s^{n, \alpha , \beta}(H) \\- & \frac{K^{\star} m_a }{n} \int_s^t \dd r \sum_{w}   \nabla H ( \tfrac{ w }{n}  ) \mathcal{Y}_r^{n, \alpha} ( T_{ r v_n } f_{\varepsilon}^{w/n} ) \; \mathcal{Y}_r^{n, \beta} ( T_{ r v_n } g_{\varepsilon}^{w/n} )   \Bigg\}^{2} \Bigg] \leq \widehat{C} (t-s) \varepsilon^{\omega_{\gamma}}.
\end{split}
\end{equation}
From \eqref{defZtnH}, we observe that 
for any $\star, \square \in \{-,+\}$ and any $G,H \in \mathcal{S}(\mathbb{R})$, it holds
\begin{align*} 
\mathcal{Z}_r^{n, \star}(G) \mathcal{Z}_r^{n,\square}(H) 
=&   (D_1)^2 \mathcal{Y}_t^{n, A} ( T_{ r v_n^{\star} } G ) \mathcal{Y}_t^{n, A} ( T_{ r v_n^{\square} } H ) + D_1 D_2^{\square} \mathcal{Y}_t^{n, A} ( T_{ r v_n^{\star} } G ) \mathcal{Y}_t^{n, B} ( T_{ r v_n^{\square} } H )  \\
+ &  D_1 D_2^{\star} \mathcal{Y}_t^{n, B} ( T_{ r v_n^{\star} } G ) \mathcal{Y}_t^{n, A} ( T_{ r v_n^{\square} } H ) + D_2^{\star} D_2^{\square}   \mathcal{Y}_t^{n, B} ( T_{ r v_n^{\star} } G ) \mathcal{Y}_t^{n, B} ( T_{ r v_n^{\square} } H )  .
\end{align*}
Combining the last display with \eqref{defBtn}, \eqref{defBtnalbe} and \eqref{aproxBngood}, we have that for any $0 \leq s \leq t \leq T$, $\varepsilon < 1/2$ fixed and $n$ large enough, $\mcb{B}_t^{n, \pm}(H) - \mcb{B}_s^{n, \pm}(H)$ can be replaced by
\begin{align}
& \frac{2 K^{\star} m_a }{n} \frac{(E_A-2 E_B + E_C \mp  \sqrt{\Delta})(E_A-E_C)}{\pm 2  D_1 \sqrt{\Delta} }  \int_s^t \dd r \sum_{w}   \nabla H ( \tfrac{ w }{n}  )  \mathcal{Z}_r^{n, \pm} ( f_{\varepsilon}^{w/n} ) \; \mathcal{Z}_r^{n, \pm} ( g_{\varepsilon}^{w/n} ) \label{KPZbeh} \\
+ & \frac{2 K^{\star} m_a }{n}  k_2^{\pm} \int_s^t \dd r \sum_{x}   \nabla H ( \tfrac{ w }{n}  )   \mathcal{Z}_r^{n, +} ( f_{\varepsilon}^{ \frac{w+r( v_n^{\pm} - v_n^{+}   ) }{n} } ) \; \mathcal{Z}_r^{n, -} ( g_{\varepsilon}^{ \frac{w+r( v_n^{\pm} - v_n^{-}   ) }{n} } ) \nonumber  \\
 + & \frac{2 K^{\star} m_a }{n}  k_2^{\pm} \int_s^t \dd r \sum_{x}   \nabla H ( \tfrac{ w }{n}  )  \mathcal{Z}_r^{n, -} ( f_{\varepsilon}^{ \frac{w+r( v_n^{\pm} - v_n^{-}   ) }{n} } ) \; \mathcal{Z}_r^{n, +} ( g_{\varepsilon}^{ \frac{w+r( v_n^{\pm} - v_n^{+}   ) }{n} } ), \nonumber
\end{align}
with an $L^2(\nu_{\rho})$-error of order $(t-s) \varepsilon^{\omega_{\gamma}}$.At this point, by choosing $\alpha_n=v_n^{\pm} - v_n^{+}$ and $\beta_n=v_n^{\pm} - v_n^{-}$ (or $\alpha_n=v_n^{\pm} - v_n^{-}$ and $\beta_n=v_n^{\pm} - v_n^{+}$), we get from \eqref{vsdif} that the condition \eqref{e:seqConst} is satisfied. On the other hand, by combining the Cauchy-Schwarz inequality with \eqref{1bndL2Z} and \eqref{boundtimeZhol}, we get that \eqref{e:supP} and \eqref{e:HolP} are also satisfied. Therefore, since $f_{\varepsilon}, g_{\varepsilon} \in C_c^{\infty}(\mathbb{R})$, from every $\varepsilon >0$ fixed, we get from Theorem \ref{thm:Cross} that the terms in the second and third lines of the last display converge in distribution to zero, as $n \rightarrow \infty$. 

Recalling the definition of $\kappa_{\gamma}$ in \eqref{defkgampm}, we rewrite the term in \eqref{KPZbeh} as
\begin{align*}
- \frac{\kappa_{\gamma}^{\pm}}{n} \int_s^t \dd r \sum_{w}   \nabla H ( \tfrac{ w }{n}  ) \mathcal{Z}_r^{n, \pm} ( f_{\varepsilon}^{w/n} ) \; \mathcal{Z}_r^{n, \pm} ( g_{\varepsilon}^{w/n} ) .
\end{align*}
From \eqref{aproxconveps}, we get that with an added $L^2(\nu_{\rho})$-error of order $(t-s)^2 \varepsilon \leq T \varepsilon^{1-\omega_{\gamma}} (t-s) \varepsilon^{\omega_{\gamma}}$, the term in the last line can be replaced by
\begin{align*}
- \frac{\kappa_{\gamma}^{\pm}}{n} \int_s^t \dd r \sum_{w}   \nabla H ( \tfrac{ w }{n}  ) \mathcal{Z}_r^{n, \pm} \big( \big[ \overleftarrow{\iota}_{\varepsilon} \big]^{w/n} \big) \; \mathcal{Z}_r^{n, \pm} \big( \big[ \overrightarrow{\iota}_{\varepsilon} \big]^{w/n} \big).
\end{align*}
Now fix $0 \leq \delta < \varepsilon < 1/2$. If $3/2 \leq \gamma < 2$, resp. if $\gamma >2$, we get from an application of Lemma \ref{lemtbe} for $\ell_n=\delta n$ and $L_n= \varepsilon n$, resp. combining the inequality $(u+v)^{2} \leq 2 (u^{2}+v^{2})$ with an application of \eqref{2bndexpdif} for $L_n = \delta n$ and an application of \eqref{2bndexpdif} for $L_n = \varepsilon n$, we get that
\begin{align*}
	\mathbb{E}_{\nu_{\rho}} \Bigg[ \Bigg\{  \frac{\Theta(n) K_n}{n \sqrt{n} } \int_s^t \dd r  \sum_{x} \nabla H ( \tfrac{ \lfloor x - r v_n \rfloor }{n})  \big[ \Psi_{x, \delta n}^{\alpha,\beta}(\eta_r^n)  - \Psi_{x, \varepsilon}^{\alpha,\beta}(\eta_r^n) \big] \Bigg\}^{2} \Bigg] \leq \widehat{C} (t-s) \varepsilon^{\omega_{\gamma}}.
\end{align*}
Thus, by taking $n \rightarrow \infty$ in the last display and applying arguments analogous to the ones described in \cite{ABC}, we conclude that the energy estimate given by \eqref{eqdefEE} is satisfied.
Furthermore, we conclude from \eqref{limdecompZpm} that
\begin{align*}
\forall t \in [0, \; T], \quad \mathcal{Z}_{t}^{\pm}(H) = \mathcal{Z}_{0}^{\pm}(H) + \mathcal{M}_t^{\pm}(H) + \mcb{I}_t^{\pm}(H) - \kappa_{\gamma}^{\pm} \mathcal{B}_t(H),
\end{align*}
where $\mathcal{B}_t(H):= \lim_{\varepsilon \rightarrow 0^+} \mathcal{B}_{0,t}^{\varepsilon}(H)$ and $\mathcal{B}_{s,t}^{\varepsilon}(H)$ is given by \eqref{defprocB}. In particular, the second condition stated in Definition \ref{defspdefbe} is satisfied. 

Finally, we stress that by performing analogous computations as the one presented above for the process reversed in time, we have that the third condition stated in Definition \ref{defspdefbe} is also satisfied.

In this way, we conclude that $(\mathcal{Z}_t^{+})_{0 \leq t \leq T}$ and $(\mathcal{Z}_t^{-})_{0 \leq t \leq T}$ are uncorrelated stationary solutions of \eqref{spdeteofbe}. This ends the proof of Theorem \ref{clt2}.

\section{Crossed fields}\label{sec:Cross}

The goal of this section is to show that the integrals of the product of two fluctuation fields 
evolving on different time frames vanish as $n \rightarrow \infty$. 
To do so, we need a version of~\cite[Theorem 5.1]{ABC} for fields which are non-periodic 
and are defined in infinite volume. 

\begin{thm} \label{thm:Cross}
For every $n \in \mathbb{N}_+$, let $\mathcal{X}^{1,n},\mathcal{X}^{2,n}$ be two time-dependent fields taking values in $ \mathcal{D} \big( [0, T],  { \color{orange} L^2 (\mathbb{R})}' \; \big)$ and defined on the same probability space. Let $(\alpha^{(i)}_n)_{n \in \mathbb{N}}$, $i=1,2$, be two sequences of constants such that 
\begin{equation} \label{e:seqConst}
 \lim_{n \rightarrow \infty} \frac{|u \alpha^{(1)}_n + v  \alpha^{(2)}_n|}{n}= \infty,
\end{equation}
for all $(u, v) \in A$ and $A$ a subset of $\mathbb{R}^2$ whose complement has 
$0$-Lebesque measure. 
Assume that there exists $\delta >0$ such that for every $f,g \in C_c^r(\mathbb{R})$ for some $r \in \mathbb{N}$,
there are finite constants $C_1(f,g), C_2(f,g)>0$ non-decreasing in $T$ and independent of $n$ for which 
\begin{align}
\sup_{0 \leq s \leq T}\mathbb{E} \Big[ \big|  \mathcal{X}_s^{1,n} (f) \mathcal{X}_s^{2,n} (g) \big|   \Big] &\leq C_1(f,g)\,, \label{e:supP} \\
 \mathbb{E} \Big[ \big|  \mathcal{X}_t^{1,n} (f) \mathcal{X}_t^{2,n} (g) - \mathcal{X}_s^{1,n} (f) \mathcal{X}_s^{2,n} (g) \big|   \Big] &\leq C_2(f,g) (t-s)^{\delta}  \,,\qquad \text{for all $s,t\in[0,T]$.}  \label{e:HolP}
\end{align}
Further, assume that $C_1$ depends bilinearly only on the $L^2(\R)$-norms of $f,g$ of their first $1<r<\infty$ derivatives 
and that $C_2$ is invariant under translations of $f$ and $g$, i.e. for any $u,v\in\R$, $C_2(f,g)=C_2(f^u,g^v)$. 

Then, for any $t \in [0, T]$, any $f \in \mathcal{S}( \mathbb{R} )$ and any $h_1, h_2 \in C_c^{\infty}(\mathbb{R})$, 
\begin{equation}   \label{e:CrossLim}
\lim_{n \rightarrow \infty }   \mathbb{E}\Bigg[ \Bigg| \int_0^t {\rm d} s \frac{1}{n} \sum_{w}   \nabla f ( \tfrac{ w }{n}  ) \, \mathcal{X}_s^{1,n} \Big( h_1^{ \frac{w + \alpha^{(1)}_n s }{n} } \Big)   \mathcal{X}_s^{2,n} \Big( h_2^{ \frac{w + \alpha^{(2)}_n s}{n} } \Big)  \Bigg|  \Bigg] = 0,
\end{equation}
where for $w \in \mathbb{Z}$ and $n \in \mathbb{N}_+$, $h_1^{ \frac{w + \alpha^{(1)}_n }{n} }$ and $h_2^{ \frac{w + \alpha^{(2)}_n }{n} }$ are given by \eqref{smoothshift}.
\end{thm}

The proof of the above statement is more delicate compared to that in~\cite{ABC} as we 
work infinite volume. Nonetheless, a crucial tool is the same version of the 
Riemeann-Lebesgue lemma proved therein~\cite[Proposition B.1]{ABC}, that 
we here recall for the reader's convenience. 

\begin{prop}\label{p:RL}
Let $(\alpha^{(i)}_n)_{n \in \mathbb{N}_+}$, $i=1,2$, be two sequences of 
diverging sequence of constants such that~\eqref{e:seqConst} holds. 
Let $T>0$ and let $\{\mcb P^n_s\colon s\in[0,T]\}$ be a 
real--valued stochastic processes on $[0,T]$.  
Assume that there exist $\alpha\in(0,1)$ and $C=C(T)>0$, such that uniformly in $n$,
\begin{align}
\sup_{0\leq s\leq T}\E\left[ |\mcb P_s^{n}|\right]&\leq C\,,\label{e:sup}\\
\E\left[|\mcb P_t^{n}-\mcb P_s^{n}|\right]&\leq C(t-s)^{\alpha}\,,\qquad \text{for all $s,t\in[0,T]$.}\label{e:Hol}
\end{align}
Then, for any $t\in[0,T]$, and Lebesgue-a.e. $k_1,k_2\in\R$
we have
\begin{equation}\label{e:RL}
\lim_{n\to\infty} \E\left[\left| \int_0^t\mcb P_s^{n} e^{-2\pi \iota  \frac{k_1\alpha_n^{(1)}+k_2\alpha_n^{(2)}}{n}}\dd s\right|\right]=0\,. 
\end{equation}
\end{prop}
\begin{proof}
The proof of the statement is identical to that of~\cite[Proposition B.1]{ABC}. Compared to it, 
we note that there was no need to impose in~\eqref{e:sup} that the supremum 
is inside of the expectation. 
\end{proof}

The true difficulty in the proof of Theorem~\ref{thm:Cross} lies in the fact that 
the Fourier elements $e_k(x)=e^{2\pi\iota kx}$ 
do not form a {\it countable} basis of $L^2(\R)$. For our purposes not only we 
need a countable basis but we also require it to satisfy suitable regularity and decay properties. 
A convenient choice turns out to be that of wavelets (see~\cite{Dau,Mey}), whose main features 
we now briefly summarise\footnote{For what we actually need, we will mainly follow~\cite[Section 3.1]{Hai}}. 
Let $r>0$ and $\phi\in L^2(\R)$ be a $\mathcal{C}_c^r(\R)$-{\it scaling function}, i.e. $\phi$ is such that 
\begin{enumerate}[noitemsep]
\item for every $k\in\Z$, $\int_\R \phi(x)\phi(x+k)\dd x=\delta_{k,0}$, 
\item there exists constants $\{a_k\}_k$ such that for every $y$, $\phi(y)=\sum_k a_k\phi(2y-k)$, 
\item $\phi$ belongs to $\mathcal{C}_c^r(\R)$ and is compactly supported, 
\end{enumerate}
whose existence was shown in~\cite{Dau1}. 
Upon defining, for $x\in\Z$, $\phi^x$ according to~\eqref{smoothshift}, the first of the properties 
listed above ensures that the family $\{\phi^x\}_{x\in\Z}$ is orthonormal. It turns out that 
the complement of the subspace of $L^2(\R)$ generated by it can be easily 
described. Indeed, it is possible to find finitely many non-zero coefficients $\{b_k\}_k$ such that, 
defining 
\begin{equ}[e:mother]
\psi(y):= \sum_{k\in\Z}b_k\phi(2y-k)\,,\qquad y\in\R\,,
\end{equ}
and, for $m\in\NN$ and 
$x\in\Lambda_m:=\{2^{-m}k\colon k\in\Z\}$, 
\begin{equ}[e:scaling]
\psi^x_m(y):= 2^{m/2}\psi(2^m(y-x))\,,\qquad y\in\R\,
\end{equ}
one has $\langle \psi^x_m, \psi_n^y\rangle_{L^2(\R)}=\delta_{m,n}\delta_{x,y}$ and the set 
\begin{equ}[e:basis]
\{\psi_m^x\colon m\geq 0,\,x\in\Lambda_m\}
\end{equ}
forms an orthonormal basis of $L^2(\R)$, where, to unify the notations 
we set $\psi^x_0:= \phi^x$ for $x\in\Lambda_0$. Note that the function $\psi$, 
which is called {\it mother wavelet}, 
is also in $\cC^r(\R)$, is compactly supported and annihilates polynomials up to order $j\leq r$, i.e. 
\begin{equ}[e:pol]
\int_\R x^j\psi(x)\dd x=0\,.
\end{equ} 
Further properties, which will be useful in this context, are summarised in Appendix~\ref{a:Wavelets}.  

\begin{proof}
Let us first note that we can replace the Riemann-sum at the l.h.s. of~\eqref{e:CrossLim} with an integral. Indeed, 
setting $\hat\alpha_n^{(i)}:=\alpha_n^{(i)}/n$, 
it is easy to see that 
\begin{equs}
\mathbb{E}\Bigg[ \Bigg|\int_0^t {\rm d} s \Bigg\{ \frac{1}{n} \sum_{w}   \nabla f ( \tfrac{ w }{n}  ) \, \mathcal{X}_s^{1,n} \Big( h_1^{ \frac{w  }{n}+ \hat\alpha^{(1)}_n s } \Big)   \mathcal{X}_s^{2,n} \Big( h_2^{ \frac{w  }{n}+ \hat\alpha^{(2)}_n s } \Big)-\int_\R\dd z\,f'(z)\mathcal{X}_s^{1,n} \Big( h_1^{ z+\hat\alpha^{(1)}_n s  } \Big)   \mathcal{X}_s^{2,n} \Big( h_2^{ z+\hat\alpha^{(2)}_n s } \Big) \Bigg\} \Bigg| \Bigg]
\end{equs}
is bounded above by a quantity which vanishes as $n\to\infty$. 
Next, we expand each of the factors in the product using a  
wavelet orthonormal basis of $L^2(\R)$ as in~\eqref{e:basis}, whose scaling function and 
mother wavelet are $\cC^\ell_c(\R)$ for $\ell>r$ sufficiently large (it will be fixed below) 
and $r$ as in the statement.  
In particular, using 
\begin{equ}
h_i^{ z+\hat\alpha^{(i)}_n s }(y)= \sum_{m\geq 0, x\in\Lambda_m} \langle h_i^{ z+\hat\alpha^{(i)}_n s}, \psi_m^x\rangle_{L^2(\R)} \psi_m^x(y)
\end{equ}
we easily deduce 
\begin{equs}
\lim_{n\to\infty}&\mathbb{E}\Bigg[ \Bigg|\int_0^t {\rm d} s\int_\R\dd z\,f'(z)\mathcal{X}_s^{1,n} \Big( h_1^{ z+\hat\alpha^{(1)}_n s  } \Big)   \mathcal{X}_s^{2,n} \Big( h_2^{ z+\hat\alpha^{(2)}_n s } \Big)\Bigg| \Bigg]\\
&\leq \lim_{n\to\infty} \sum_{m_1,m_2\geq 0}\mathbb{E}\Bigg[ \Bigg|\int_0^t\dd s\int_\R\dd z\, f'(z)\,\prod_{i=1}^2\sum_{x\in\Lambda_{m_i}}\,\mathcal{X}_s^{i,n}(\psi_{m_i}^{x_i})\langle h_i^{ z+\hat\alpha^{(i)}_n s }, \psi_{m_i}^{x_i}\rangle_{L^2(\R)}\Bigg| \Bigg]\,.
\end{equs}
At this stage, we want to pass the limit in $n$ inside the sum over $m_1,m_2$, 
for which we will use dominated convergence whose 
applicability is guaranteed by~\eqref{e:supP} and Lemma~\ref{l:WaveBounds}. 
To see this, let $A^n_{m_{i}}$ be the set of $x\in\Lambda_{m_i}$ 
such that $\langle h_i^{ z+\hat\alpha^{(i)}_n s }, \psi_{m_i}^{x_i}\rangle_{L^2(\R)}\neq 0$. 
Thanks to Lemma~\ref{l:WaveBounds} the cardinality of $A^n_{m_{i}}$ is finite and independent of 
$z+\hat\alpha^{(i)}_n s$ (and thus of $n$). 
Then, 
\begin{equs}
\mathbb{E}&\Bigg[ \Bigg|\int_0^t\dd s\int_\R\dd zf'(z)\prod_{i=1}^2\sum_{x\in\Lambda_{m_i}}\,\mathcal{X}_s^{i,n}(\psi_{m_i}^{x_i})\langle h_i^{ z+\hat\alpha^{(i)}_n s }, \psi_{m_i}^{x_i}\rangle_{L^2(\R)}\Bigg| \Bigg]\\
&\leq \int_0^t \dd s \int_\R \dd z|f'(z)| \sum_{i=1,2}\sum_{x_i\in A^n_{m_i}} \mathbb{E}\Big[ \Big|\mathcal{X}_s^{1,n} \big( \psi_{m_1}^{x_1}\big)   \mathcal{X}_s^{2,n} \big( \psi_{m_2}^{x_2}\big)\Big|\Big] |\langle h_1^{ z+\hat\alpha^{(1)}_n s }, \psi_{m_1}^{x_1}\rangle_{L^2(\R)}||\langle h_2^{ z+\hat\alpha^{(2)}_n s }, \psi_{m_2}^{x_2}\rangle_{L^2(\R)}|\\
&\lesssim t \int_\R \dd z|f'(z)| \sum_{i=1,2}\sum_{x_i\in A^n_{m_i}} C_1(\psi_{m_1}^{x_1},\psi_{m_2}^{x_2})2^{m_1-m_1\ell} 2^{m_2-m_2\ell} \lesssim t 2^{m_1+m_2}2^{\frac{m_1 r}{2}+\frac{m_2 r}{2}}2^{m_1-m_1\ell} 2^{m_2-m_2\ell} \int_\R \dd z|f'(z)| 
\end{equs}
where in the last two steps we used~\eqref{e:supP} and~\eqref{e:BoundSmooth} first, and then 
the cardinality of $A^n_{m_j}$ and~\eqref{e:BoundsWavelets}. 
Upon choosing $\ell$ large enough the r.h.s. is summable so that we deduce 
\begin{equs}
\lim_{n\to\infty}& \sum_{m_1,m_2\geq 0}\mathbb{E}\Bigg[ \Bigg|\int_0^t\dd s\int_\R\dd z\, f'(z)\,\prod_{i=1}^2\sum_{x\in\Lambda_{m_i}}\,\mathcal{X}_s^{i,n}(\psi_{m_i}^{x_i})\langle h_i^{ z+\hat\alpha^{(i)}_n s }, \psi_{m_i}^{x_i}\rangle_{L^2(\R)}\Bigg| \Bigg]\\
&= \sum_{m_1,m_2\geq 0}\lim_{n\to\infty}\mathbb{E}\Bigg[ \Bigg|\int_0^t\dd s\int_\R\dd z\, f'(z)\,\prod_{i=1}^2\sum_{x\in\Lambda_{m_i}}\,\mathcal{X}_s^{i,n}(\psi_{m_i}^{x_i})\langle h_i^{ z+\hat\alpha^{(i)}_n s }, \psi_{m_i}^{x_i}\rangle_{L^2(\R)}\Bigg| \Bigg]\\
&\leq \sum_{m_1,m_2\geq 0}\int_\R\dd z\, |f'(z)|\,\lim_{n\to\infty}\mathbb{E}\Bigg[ \Bigg|\int_0^t\dd s\prod_{i=1}^2\sum_{x\in\Lambda_{m_i}}\,\mathcal{X}_s^{i,n}(\psi_{m_i}^{x_i})\langle h_i^{ z+\hat\alpha^{(i)}_n s }, \psi_{m_i}^{x_i}\rangle_{L^2(\R)}\Bigg| \Bigg]
\end{equs}
the last step being a consequence of monotone convergence theorem. 
\medskip

As a consequence, the statement follows provided we show 
that, for {\it any fixed value} of $m_1$ and $m_2$, we have 
\begin{equ}[e:LASTstep]
\lim_{n\to\infty}\mathbb{E}\Bigg[ \Bigg|\int_0^t\dd s\prod_{i=1}^2\sum_{x\in\Lambda_{m_i}}\,\mathcal{X}_s^{i,n}(\psi_{m_i}^{x_i})\langle h_i^{ z+\hat\alpha^{(i)}_n s }, \psi_{m_i}^{x_i}\rangle_{L^2(\R)}\Bigg| \Bigg]=0\,.
\end{equ}
By Plancherel's identity,  
\begin{equ}
\langle h_i^{ z+\hat\alpha^{(i)}_n s }, \psi_{m_i}^{x_i}\rangle_{L^2(\R)}=\int_\R \,\dd k\, \hat h_i(k)\widehat{\psi^{m_i}_{0}}(k)e^{-2\pi\iota (z+\hat\alpha^{(i)}_n s+x_i)k}
\end{equ}
so that 
\begin{equs}[e:Intermezzo]
\,&\mathbb{E}\Bigg[ \Bigg|\int_0^t\dd s\prod_{i=1}^2\sum_{x\in\Lambda_{m_i}}\,\mathcal{X}_s^{i,n}(\psi_{m_i}^{x_i})\langle h_i^{ z+\hat\alpha^{(i)}_n s }, \psi_{m_i}^{x_i}\rangle_{L^2(\R)}\Bigg| \Bigg]\\
&\leq \int_{\R^2}\dd k_1\dd k_2\Big(\prod_{j=1}^2|\hat h_j(k_j)| |\widehat{\psi_{m_j}^{0}}(k_j)|\Big)\E\Big[\Big|\int_0^t \dd s \sum_{\substack{x_i\in A^n_{m_i}\\ i=1,2}}\mathcal{X}_s^{1,n}(\psi_{m_1}^{x_1}) \mathcal{X}_s^{2,n}(\psi_{m_2}^{x_2})e^{-2\pi\iota \big(\hat\alpha^{(1)}_n k_1+\hat\alpha^{(2)}_n k_2\big)s}\Big|\Big]\,,
\end{equs}
where, once again, we used that $\langle h_i^{ z+\hat\alpha^{(i)}_n s }, \psi_{m_i}^{x_i}\rangle_{L^2(\R)}$ 
is different from $0$ only for $x_i\in A^n_{m_i}$. 

We now want to pass the limit in $n$ inside the integral over $k_1$ and $k_2$. 
For this, it suffices to show that the integrand can be upper bounded by an integrable 
function of $k_1$ and $k_2$ and then apply dominated convergence. 
Notice first that 
\begin{equs}
\E&\Big[\Big|\int_0^t \dd s \sum_{\substack{x_i\in A^n_{m_i}\\ i=1,2}}\mathcal{X}_s^{1,n}(\psi_{m_1}^{x_1}) \mathcal{X}_s^{2,n}(\psi_{m_2}^{x_2})e^{-2\pi\iota \big(\hat\alpha^{(1)}_n k_1+\hat\alpha^{(2)}_n k_2\big)s}\Big|\Big]\\
&\leq t \sum_{\substack{x_i\in A^n_{m_i}\\ i=1,2}}\sup_{s\leq T} \E\Big|\mathcal{X}_s^{1,n}(\psi_{m_1}^{x_1}) \mathcal{X}_s^{2,n}(\psi_{m_2}^{x_2})\Big|\leq t \sum_{\substack{x_i\in A^N_{m_i}\\ i=1,2}} C_1(\psi_{m_1}^{x_1},\psi_{m_2}^{x_2})
\end{equs}
and the right hand side is finite and independent of $k_1,k_2$. 
Therefore, 
\begin{equs}
\Big(\prod_{j=1}^2|\hat h_j(k_j)| |\widehat{\psi_{m_j}^{0}}(k_j)|\Big)\E\Big[\Big|\int_0^t \dd s \sum_{\substack{x_i\in A^n_{m_i}\\ i=1,2}}\mathcal{X}_s^{1,n}(\psi_{m_1}^{x_1}) \mathcal{X}_s^{2,n}(\psi_{m_2}^{x_2})e^{-2\pi\iota \big(\hat\alpha^{(1)}_n k_1+\hat\alpha^{(2)}_n k_2\big)s}\Big|\Big]\lesssim \prod_{j=1}^2|\hat h_j(k_j)| |\widehat{\psi_{m_j}^{0}}(k_j)|
\end{equs}
and the r.h.s. is integrable in $k_1,k_2$, so that we get 
\begin{equs}
\,&\lim_{n\to\infty}\int_{\R^2}\dd k_1\dd k_2\Big(\prod_{j=1}^2|\hat h_j(k_j)| |\widehat{\psi_{m_j}^{0}}(k_j)|\Big)\E\Big[\Big|\int_0^t \dd s \sum_{\substack{x_i\in A^n_{m_i}\\ i=1,2}}\mathcal{X}_s^{1,n}(\psi_{m_1}^{x_1}) \mathcal{X}_s^{2,n}(\psi_{m_2}^{x_2})e^{-2\pi\iota \big(\hat\alpha^{(1)}_n k_1+\hat\alpha^{(2)}_n k_2\big)s}\Big|\Big]\\
&=\int_{\R^2}\dd k_1\dd k_2\Big(\prod_{j=1}^2|\hat h_j(k_j)| |\widehat{\psi_{m_j}^{0}}(k_j)|\Big)\lim_{n\to\infty}\E\Big[\Big|\int_0^t \dd s \sum_{\substack{x_i\in A^n_{m_i}\\ i=1,2}}\mathcal{X}_s^{1,n}(\psi_{m_1}^{x_1}) \mathcal{X}_s^{2,n}(\psi_{m_2}^{x_2})e^{-2\pi\iota \big(\hat\alpha^{(1)}_n k_1+\hat\alpha^{(2)}_n k_2\big)s}\Big|\Big]\,.
\end{equs}
But now, to prove that the inner limit vanishes and thus show that~\eqref{e:LASTstep} holds, 
it suffices to apply Proposition~\ref{p:RL} to the process 
\begin{equ}
\mcb P^n_s:=\sum_{i=1,2}\sum_{x_i\in A^n_{m_i}}\mathcal{X}_s^{1,n}(\psi_{m_1}^{x_1}) \mathcal{X}_s^{2,n}(\psi_{m_2}^{x_2})\,.
\end{equ}
This is allowed provided $\mcb P^n$ satisfies~\eqref{e:sup} and~\eqref{e:Hol}, 
that we now verify. 
For~\eqref{e:sup}, note that 
\begin{equs}
\sup_{0\leq s\leq T} \E\Big[ \Big| \sum_{i=1,2}\sum_{x_i\in A^n_{m_i}}\mathcal{X}_s^{1,n}(\psi_{m_1}^{x_1}) \mathcal{X}_s^{2,n}(\psi_{m_2}^{x_2})\Big|\Big]&\leq \sup_{0\leq s\leq T}\sum_{i=1,2}\sum_{x_i\in A^n_{m_i}}\E\Big[ \Big| \mathcal{X}_s^{1,n}(\psi_{m_1}^{x_1}) \mathcal{X}_s^{2,n}(\psi_{m_2}^{x_2})\Big|\Big]\\
&\leq \sum_{i=1,2}\sum_{x_i\in A^n_{m_i}} C_1(\psi_{m_1}^{x_1},\psi_{m_2}^{x_2})\lesssim 2^{m_1+m_2}\sup _{x_1, x_2}C_1(\psi_{m_1}^{x_1},\psi_{m_2}^{x_2})
\end{equs} 
where, in the last supremum, $x_1$ and $x_2$ respectively range over $\Lambda_{m_1}$ and $\Lambda_{m_2}$, 
and we used~\eqref{e:sup} and that the cardinality of $A_{m_i}^n$ is independent of $n$ and of order $2^{m_i}$. 
Now, since by assumption $C_1$ depends bilinearly on the $L^2(\R)$-norms of 
its arguments and their derivatives, and these norms of $\psi^x_m$ can be bounded uniformly 
in $x$ thanks to~\eqref{e:BoundsWavelets}, we conclude that the r.h.s. above is bounded uniformly in $n$, 
so that indeed~\eqref{e:sup} holds. 

For~\eqref{e:Hol}, we take $0\leq s<t\leq T$ and argue as follows  
\begin{equs}
\E&\Big[ \Big| \sum_{i=1,2}\sum_{x_i\in A^n_{m_i}}\mathcal{X}_t^{1,n}(\psi_{m_1}^{x_1}) \mathcal{X}_t^{2,n}(\psi_{m_2}^{x_2})-\sum_{i=1,2}\sum_{x_i\in A^n_{m_i}}\mathcal{X}_s^{1,n}(\psi_{m_1}^{x_1}) \mathcal{X}_s^{2,n}(\psi_{m_2}^{x_2}) \Big|\Big]\\
&\leq \sum_{i=1,2}\sum_{x_i\in A^n_{m_i}} \E\Big[ \Big| \mathcal{X}_t^{1,n}(\psi_{m_1}^{x_1}) \mathcal{X}_t^{2,n}(\psi_{m_2}^{x_2})-\mathcal{X}_s^{1,n}(\psi_{m_1}^{x_1}) \mathcal{X}_s^{2,n}(\psi_{m_2}^{x_2}) \Big|\Big]\\
&\leq  (t-s)^\delta\sum_{i=1,2}\sum_{x_i\in A^n_{m_i}} C_2(\psi_{m_1}^{x_1},\psi_{m_2}^{x_2})\lesssim (t-s)^\alpha,
\end{equs} 
where the last line follows as, $C_2(f,g)$ is invariant under translations of $f$ and $g$ 
\end{proof}

\appendix

\section{Computations regarding the generator} \label{appa}

From \eqref{defxialp}, we have that
\begin{equation} \label{condxi}
	\forall x \in \mathbb{Z}, \; \forall \eta \in \Omega, \; \forall \alpha \neq \beta \in S, \; \forall k \in \mathbb{N}, \quad [\xi_x^{\alpha_1}( \eta )]^k = \xi_x^{\alpha}( \eta ) \quad \text{and} \quad \xi_x^{\alpha}( \eta ) \xi_x^{\beta}( \eta )=0. 
\end{equation}
The first, resp. second, equality in \eqref{condxi} comes from the fact that $\xi_x^{\alpha}( \eta ) \in \{0,1\}$, resp. the site $x$ cannot be occupied by a particle of type $\alpha$ and a particle of type $\beta$ simultaneously. Combining this with \eqref{genABC}, we conclude that
\begin{equation} \label{genxiza}
	\begin{split}
		&(\mcb L^n \xi_z^{A}) (\eta) =  \sum_{y}\big\{ p(z-y) \big[ r^n_{A, C}   \xi_y^{A}( \eta ) - r^n_{C, A}   \xi_z^{A}( \eta ) \big] + 	p(y-z) \big[ r^n_{C, A}   \xi_y^{A}( \eta ) - r^n_{A, C}   \xi_z^{A}( \eta ) \big] \big\} \\
		+ &  \sum_{y} p(z-y) \big\{  [r^n_{C, A}  - r^n_{A, C}   ] \xi_y^{A}( \eta ) \xi_z^{A}( \eta ) + [r^n_{A, B}  - r^n_{A, C}   ] \xi_y^{A}( \eta ) \xi_z^{B}( \eta )  + [r^n_{C, A}  - r^n_{B, A}   ] \xi_z^{A}( \eta ) \xi_y^{B}( \eta ) \big\} \\
		+&	\sum_{y} p(y-z) \big\{  [r^n_{A, C}  - r^n_{C, A}   ] \xi_y^{A}( \eta ) \xi_z^{A}( \eta ) + [r^n_{B, A}  - r^n_{C, A}   ] \xi_y^{A}( \eta ) \xi_z^{B}( \eta )  + [r^n_{A, C}  - r^n_{A, B}   ] \xi_z^{A}( \eta ) \xi_y^{B}( \eta ) \big\}.
	\end{split}
\end{equation}
Exchanging $A$ and $B$ in the last display, we get
\begin{equation} \label{genxizb}
	\begin{split}
		&(\mcb L^n \xi_z^{B}) (\eta) = \sum_{y}\big\{ p(z-y) \big[ r^n_{B, C}   \xi_y^{B}( \eta ) - r^n_{C, B}   \xi_z^{B}( \eta ) \big] + 	p(y-z) \big[ r^n_{C, B}   \xi_y^{B}( \eta ) - r^n_{B, C}   \xi_z^{B}( \eta ) \big] \big\} \\
		+ &  \sum_{y} p(z-y) \big\{  [r^n_{C, B}  - r^n_{B, C}   ] \xi_y^{B}( \eta ) \xi_z^{B}( \eta ) + [r^n_{B, A}  - r^n_{B, C}   ] \xi_y^{B}( \eta ) \xi_z^{A}( \eta )  + [r^n_{C, B}  - r^n_{A, B}   ] \xi_z^{B}( \eta ) \xi_y^{A}( \eta ) \big\} \\
		+&	\sum_{y} p(y-z) \big\{  [r^n_{B, C}  - r^n_{C, B}   ] \xi_y^{B}( \eta ) \xi_z^{B}( \eta ) + [r^n_{A, B}  - r^n_{C, B}   ] \xi_y^{B}( \eta ) \xi_z^{A}( \eta )  + [r^n_{B, C}  - r^n_{B, A}   ] \xi_z^{B}( \eta ) \xi_y^{A}( \eta ) \big\}.
	\end{split}
\end{equation} 
Combining \eqref{condxi}, \eqref{genxiza} and \eqref{genxizb}, we conclude that
\begin{equation} \label{ezalneza}
	\begin{split}
		\xi_z^{A}( \eta ) (\mcb L^n \xi_z^{A}) (\eta) =& -  \xi_z^{A}( \eta ) \sum_{w}\big\{ p(z-w)    r^n_{C, A}    + 	p(w-z)  r^n_{A, C}    \big\} \\
		+ &  \sum_{w} \big\{ p(z-w)  r^n_{C, A}  +  p(w-z)     r^n_{A, C} \big\}  \xi_w^{A}( \eta ) \xi_z^{A}( \eta ) \\
		+&  \sum_{w} \big\{ p(z-w) [r^n_{C, A}  - r^n_{B, A}   ] + p(w-z) [r^n_{A, C}  - r^n_{A, B}   ]  \big\} \xi_z^{A}( \eta ) \xi_w^{B}( \eta ),
	\end{split}
\end{equation}
\begin{equation} \label{ezblneza}
	\xi_z^{B}( \eta )  (\mcb L^n \xi_z^{A}) (\eta) = \sum_{w}  \big\{  p(z-w) r^n_{A, B}       + p(w-z)   r^n_{B, A}        \big\} \xi_w^{A}( \eta ) \xi_z^{B}( \eta ),
\end{equation}
\begin{equation} \label{ezalnezb}
	\xi_z^{A}( \eta )	(\mcb L^n \xi_z^{B}) (\eta) =  \sum_{w} \big\{  p(z-w)  r^n_{B, A}      	+	 p(w-z)    r^n_{A, B}     ]   \big\} \xi_w^{B}( \eta ) \xi_z^{A}( \eta ),
\end{equation}
\begin{equation} \label{ezblnezb}
	\begin{split}
		\xi_z^{B}( \eta )	(\mcb L^n \xi_z^{B}) (\eta) =& - \xi_z^{B}( \eta ) \sum_{w}\big\{ p(z-w)  r^n_{C, B}     + 	p(w-z)  r^n_{B, C}   \big\} \\
		+ &  \sum_{w}  \big\{ p(z-w) r^n_{C, B} +  p(w-z)   r^n_{B, C} \big\} \xi_w^{B}( \eta ) \xi_z^{B}( \eta )  \\
		+ &  \sum_{w}  \big\{ p(z-w) [r^n_{C, B}  - r^n_{A, B}   ] +  p(w-z) [r^n_{B, C}  - r^n_{B, A}   ]  \big\}  \xi_z^{B}( \eta ) \xi_w^{A}( \eta ).
	\end{split}
\end{equation}

\section{Some useful deterministic estimates} \label{useest}

The following remark will be useful in this section.
\begin{rem}
For all $G \in \mathcal{S}(\mathbb{R})$, define $F^G: \mathbb{R} \mapsto \mathbb{R}$ by 
\begin{align} \label{defF0}
	\forall u \in \mathbb{R}, \quad F^G(u):= \sup_{v: |v-u| \leq |u|/2} \nabla G(v).
\end{align}
In particular, $F^G \in L^2(\mathbb{R})$, since for every $k \in \mathbb{N}:=\{0,1,\ldots\}$, there exists $C_{G,k} >0$ such that
	\begin{align} 
&2^k \big\{	\sup_{u \in \mathbb{R}} u^{2k} |\nabla G(u) | + 4 \sup_{u \in \mathbb{R}} |\nabla G(u) |  u^{2k+2} \big\} \leq C_{G,k}; \nonumber \\
&\int_{\mathbb{R}} u^{2k} [F^G(u)]^{2} \dd u \leq \int_{\mathbb{R}} \Bigg( \frac{C_{G,k}}{u^{2}+1} \Bigg)^{2} \dd u < \infty. \label{boundintfg}
	\end{align}
\end{rem}
\subsection{Proof for \eqref{errorexpquadvar}} \label{useest1}

From a second-order Taylor expansion of $H$, we get
\begin{align*}
\forall n \in \mathbb{N}_+, \; \forall r \in [0,T], \; \forall y \in \mathbb{Z}, \quad H ( \tfrac{ y + q_{n}^r  }{n}  ) - H ( \tfrac{ y   }{n}  ) = \frac{q_{n}^r }{n} \nabla H ( \tfrac{ y   }{n}  ) + \frac{ \big(q_{n}^r \big)^2 }{2 n^2} \Delta H (  \tfrac{ \omega_{r,n}^{y} }{n}  ),
\end{align*}
where $\omega_{r,n}^{y} \in \big[ y, y + q_{n}^r  \big]$. Thus, the display inside the supremum in \eqref{errorexpquadvar} can be rewritten as
\begin{align*}
 \frac{\big(  q_{n}^r \big)^2 }{n^2}   \sum_{ z, w } \big\{ \big[ \nabla H ( \tfrac{ w   }{n}  )  - \nabla H ( \tfrac{ z  }{n}  ) + \frac{q_{n}^r}{2n}  \big( \Delta H (  \tfrac{ \omega_{r,n}^{z} }{n} ) - \Delta H (  \tfrac{ \omega_{r,n}^{w} }{n} )  \big) \big] \big\}^{2} s(w-z). 
\end{align*}
From \eqref{defqnr} we get $|q_n^r| \leq 1$. Thus, from the inequality $(a+b+c)^2 \leq 3(a^2 + b^2 + c^2)$, the supremum in \eqref{errorexpquadvar} is bounded from above by
\begin{align}
&  \sup_{r \in [0,T]} \Bigg\{ \frac{2 \Theta(n)}{ n^5}   \sum_{ z } \big[ \Delta H (  \tfrac{ \omega_{r,n}^{z} }{n} ) \big]^2 \sum_w s(w-z) \Bigg\}  +    \frac{3  \Theta(n)}{ n^3}   \sum_{ z, w } \big[ \nabla H ( \tfrac{ w   }{n}  )  - \nabla H ( \tfrac{ z  }{n}  ) \big]^2 s(w-z) \nonumber \\ 
\leq &    2  \sup_{r \in [0,T]} \Bigg\{ \frac{ \Theta(n)}{ n^5}   \sum_{ z } \big[ \Delta H (  \tfrac{ \omega_{r,n}^{z} }{n} ) \big]^2  \Bigg\} + \frac{\widetilde{C}_H}{n^{2}} \mathcal{P}^{ \gamma} \nabla H.  \label{errorexpquadvar3a1}
\end{align}
The inequality in the last line comes from \eqref{limexpquadvar0a}. In order to finish the proof for \eqref{errorexpquadvar}, it is enough to estimate the leftmost term in the last line. We do so by considering the cases $H \notin C_c^1(\mathbb{R})$ and $H \in C_c^1(\mathbb{R})$ separately.
\begin{itemize}
\item
If $H \notin C_c^1(\mathbb{R})$, we have that $H, \nabla H \in \mathcal{S}(\mathbb{R})$. Thus, for any $n \in \mathbb{N}_+$ and $r \in [0,T]$ fixed, an application of the Mean Value Theorem leads to 
\begin{align*} 
	\sum_{ z } \big[ \Delta H (  \tfrac{ \omega_{r,n}^{z} }{n} ) \big]^2 \leq & 3 \| \Delta H \|^{2}_{\infty} + \sum_{ |z| \geq 2 } \big[ F^{\nabla H} ( \tfrac{ z     }{n}  ) \big]^{2} \leq  3 \| \Delta H \|^{2}_{\infty} + \sum_{ x } \big[ F^{\nabla H} ( \tfrac{x   }{n}  ) \big]^{2}, 
\end{align*}
where $F^{\nabla H} \in L^2(\mathbb{R})$ is given by \eqref{defF0}. The later upper bound is independent of $r \in [0,T]$. Thus, the leftmost term in \eqref{errorexpquadvar3a1} is bounded from above by
\begin{align*}
	  2 \frac{\Theta(n)}{ n^4}  \Bigg\{  3 \| \Delta H \|^{2}_{\infty} + \frac{1}{n} \sum_{ x } \big[ F^{\nabla H} ( \tfrac{x   }{n}  ) \big]^{2}\Bigg\} 
	\leq &  \widetilde{C}_H \big\{   \| \Delta H \|^{2}_{\infty} +  \| F^{ \nabla H} \|_{ L^2(\mathbb{R}) }^{ 2} \big\}   \frac{\Theta(n)}{n^{4}}.
\end{align*}
Combining this with \eqref{timescale}, the proof for \eqref{errorexpquadvar} ends for the case $H \notin C_c^1( \mathbb{R})$.
\item
If $H \in C_c^1(\mathbb{R})$, there exist $a_H < b_H \in \mathbb{R}$ such that $H(u)=0$ whenever $u \notin (a_H,b_H)$. Thus, the leftmost term in \eqref{errorexpquadvar3a1} is bounded from above by
\begin{align*}
 \sup_{r \in [0,T]} 2 \Bigg\{ \frac{ \Theta(n)}{ n^5}   \sum_{ z=-1 + a_H n }^{b_H n} \big[ \| \Delta H \|_{\infty} \big]^2  \Bigg\} \leq 2 (b_H - a_H + 1) \big[ \| \Delta H \|_{\infty} \big]^2  \frac{\Theta(n)}{n^{4}}.
\end{align*}
This ends the proof for \eqref{errorexpquadvar}.
\end{itemize}

\subsection{Proof for \eqref{boundexpquadvar}} \label{useest2}

    If $a_H < b_H \in \mathbb{R}$ are such that $H(u)=0$ whenever $u \notin (a_H,b_H)$, the term inside the supremum in \eqref{boundexpquadvar} is equal to
\begin{align}
	& \frac{c_{\gamma}}{n} \sum_{x = - \infty}^{(a_H-1)n}  \Theta(n) \sum_{y = a_H n - x }^{b_H n - x} \big[ H ( \tfrac{x + q_n^r + y }{n}  )  - H ( \tfrac{x + q_n^r  }{n}  )  \big]^{2} | y|^{-\gamma-1}  \mathbbm{1}_{ \{ y \neq 0 \} }  \label{expquadvar1a} \\
	+ & \frac{c_{\gamma}}{n} \sum_{x = (b_H+1)n}^{\infty} \Theta(n) \sum_{y = a_H n - x }^{b_H n - x} \big[ H ( \tfrac{x + q_n^r + y }{n}  )  - H ( \tfrac{x + q_n^r  }{n}  )  \big]^{2} |y |^{-\gamma-1} \mathbbm{1}_{ \{ y \neq 0 \} } \label{expquadvar1b} \\
		+ & \frac{c_{\gamma}}{n} \sum_{x = (a_H-1)n +1}^{(b_H+1)n-1} \Theta(n) \sum_{y } \big[ H ( \tfrac{x + q_n^r + y }{n}  )  - H ( \tfrac{x + q_n^r  }{n}  )  \big]^{2} |y |^{-\gamma-1} \mathbbm{1}_{ \{ y \neq 0 \} }. \label{expquadvar1c}
\end{align}
In order to estimate the term in \eqref{expquadvar1a}, we observe that for every $n \in \mathbb{N}_+$ and every $\alpha > 1$ fixed, it holds
\begin{equation} \label{ineqalp}
	\frac{1}{n}  \sum_{y = a_H n - x  }^{b_H n - x}  \Bigg( \frac{y}{n} \Bigg)^{-\alpha} \leq  C(\alpha) \int_{a_H  - x/n}^{b_H - x/n} v^{-\alpha} \; \dd v \leq  \frac{C(\alpha)}{\alpha - 1} (b_H - a_H  ) \Bigg(a_H  - \frac{x}{n}\Bigg)^{- \alpha  },
\end{equation}
as long as $x < a_H n$. Above, $C(\alpha)$ is some positive constant depending only on $\alpha$. The last inequality comes from an application of the Mean Value Theorem to the function $f_{\alpha} : (0, \infty) \mapsto \mathbb{R}$, given by $ f_{\alpha}(u)=u^{-( \alpha - 1) }$. 

Applying \eqref{ineqalp} for the choice $\alpha=1+\gamma$, the term in \eqref{expquadvar1a} is bounded from above by
\begin{align*}
& \frac{\Theta(n)}{n^{\gamma}} \frac{ C(\gamma) (b_H - a_H  ) \| H \|_{\infty}^{2} }{n} \sum_{x = - \infty}^{(a_H-1)n}  \Bigg(a_H  - \frac{x}{n}\Bigg)^{- \gamma - 1  } \leq  \frac{\Theta(n)}{n^{\gamma}} (b_H - a_H  ) \| H \|_{\infty}^{2} C(\gamma) \int_1^{\infty} u^{-\gamma-1} \; \dd u. 
\end{align*}
 Applying an analogous reasoning for estimating the term in \eqref{expquadvar1b}, we conclude that the sum of the terms in \eqref{expquadvar1a} and \eqref{expquadvar1b} is bounded from above by  
\begin{equation} \label{ubexpquadvar1ab}
C(\gamma) \frac{\Theta(n)}{n^{\gamma}} (b_H - a_H  ) \| H \|_{\infty}^{2} \leq C(\gamma)  (b_H - a_H  ) \| H \|_{\infty}^{2}.
\end{equation}
Next, we bound the term in \eqref{expquadvar1c} from above by
\begin{align*}
& \frac{2c_{\gamma}}{n} \|  H \|_{\infty}^{2} \sum_{x = (a_H-1)n +1}^{(b_H+1)n-1} \frac{\Theta(n)}{ n^{\gamma} } \frac{1}{n}  \sum_{z = n }^{\infty} \Bigg( \frac{z}{n} \Bigg)^{-1-\gamma} + \frac{2c_{\gamma}}{n} \| \nabla H \|_{\infty}^{2} \sum_{x = (a_H-1)n +1}^{(b_H+1)n-1} \frac{\Theta(n)}{ n^{2} }   \sum_{z = 1 }^{n - 1} z^{1-\gamma} \\
\leq & C(\gamma) \|  H \|_{\infty}^{2} (b_H - a_H +2)  + \frac{C(\gamma)}{n} \| \nabla H \|_{\infty}^{2} \sum_{x = (a_H-1)n +1}^{(b_H+1)n-1} \frac{\Theta(n)}{ n^{2} }   \sum_{z = 1 }^{n - 1} z^{1-\gamma}.
\end{align*}
Above we applied the Mean Value Theorem. Now from \eqref{timescale}, we get that $\frac{\Theta(n)}{ n^{2} } \sum_{z = 1 }^{n - 1} z^{\gamma-1}$ is bounded from above by $C(\gamma)$, for any $\gamma > 0$. Combining this with \eqref{ubexpquadvar1ab}, the proof for \eqref{boundexpquadvar} ends.

\subsection{Proof for Lemmas \ref{lemvarprinc} and \ref{lemvarext}}  

In order to obtain Lemmas \ref{lemvarprinc} and \ref{lemvarext}, Lemma \ref{opdiscfarorig} below will be useful for estimating the contribution of
sites $x$ that are macroscopically far from the support of $H$ when $H$ has compact support.
\begin{lem} \label{opdiscfarorig}
	Assume that  $a_H < b_H \in \mathbb{R}$ are such that $H(u)=0$ whenever $u \notin (a_H,b_H)$. Then,
	\begin{align*}
		& \frac{1}{n} \sum_{x = - \infty}^{(a_H-1)n} \Bigg\{ \Theta(n) \sum_y | H ( \tfrac{x + q_n^r + y }{n}  )  - H ( \tfrac{x + q_n^r  }{n}  )  |\cdot | y|^{-\gamma-1} \Bigg\}^{2}    \\
		+ & \frac{1}{n} \sum_{x = (b_H+1)n}^{\infty} \Bigg\{ \Theta(n) \sum_y | H ( \tfrac{x + q_n^r + y }{n}  )  - H ( \tfrac{x + q_n^r  }{n}  )  | \cdot | y |^{-\gamma-1} \Bigg\}^{2} \leq C(\gamma) \| H \|_{\infty}^2 \frac{[\Theta(n)]^2}{n^{2\gamma}} (b_H - a_H)^2.
	\end{align*}
	\end{lem}
\begin{proof}
	In what follows we estimate only the first sum in the last display, but we observe that the reasoning for the second sum is analogous. For $x$ outside the enlarged support
$[(a_H-1)n,(b_H+1)n]$, the value $H\!\left(\frac{x+q_r^n}{n}\right)$ vanishes, and the difference
$H\!\left(\frac{x+q_r^n+y}{n}\right)-H\!\left(\frac{x+q_r^n}{n}\right)$ can be nonzero only if the jump
$y$ is large enough to reach the support of $H$. Keeping this in mind, by choosing $\alpha=\gamma+1$ in \eqref{ineqalp}, we get that the first sum in the statement of Lemma \ref{opdiscfarorig} is bounded from above by
	\begin{align*}
		& \frac{\| H \|_{\infty}^2}{n} \sum_{x = - \infty}^{(a_H-1)n}  \Bigg\{ \frac{\Theta(n)}{n^{\gamma}} \frac{1}{n}  \sum_{y = a_H n - x }^{b_H n - x}  \Bigg( \frac{y}{n} \Bigg)^{-\gamma-1} \Bigg\}^2 \\
		\leq &  C(\gamma) \| H \|_{\infty}^2 \frac{[\Theta(n)]^2}{n^{2\gamma}} (b_H - a_H  )^2 \int_{1}^{\infty} u^{-2 \gamma - 2} \; \dd u =  C(\gamma) \| H \|_{\infty}^2 \frac{[\Theta(n)]^2}{n^{2\gamma}} (b_H - a_H  )^2.
	\end{align*}
\end{proof}

\subsubsection{Proof for Lemma \ref{lemvarprinc}} \label{useest31}

We begin by observing that $\frac{1}{n} \sum_{r=n}^{\infty}  \Big( \frac{r}{n} \Big)^{-\gamma-1} \leq C(\gamma) \int_1^{ \infty } u^{-\gamma-1} du$, thus
\begin{equation} \label{boundGfar}
\Theta(n)  \sum_{|r| \geq n } | H ( \tfrac{u + r }{n}  )  - H ( \tfrac{u }{n}  )  | \; |r|^{-\gamma-1} \leq C(\gamma)  \|  H \|_{\infty} \frac{\Theta(n)}{n^{\gamma}} \frac{1}{n} \sum_{r=n}^{\infty}  \Big( \frac{r}{n} \Big)^{-\gamma-1} \leq C(\gamma)  \|  H \|_{\infty}.
\end{equation}
Moreover, applying a second-order Taylor expansion on $H$ around $u/n$, we get
\begin{equation} \label{boundGnear}
\Theta(n)  \sum_{|r| \leq n } | H ( \tfrac{u + r }{n}  )  - H ( \tfrac{u }{n}  )  | \; |r|^{-\gamma-1} \mathbbm{ 1 }_{ \{ r \neq 0 \} } \leq C(\gamma)  \| \Delta H \|_{\infty} \frac{\Theta(n)}{n^{2}} \sum_{r=1}^{n-1} r^{1-\gamma} \leq C(\gamma)  \| \Delta H \|_{\infty}.
\end{equation}
Combing \eqref{defKnsa} with \eqref{boundGfar} and \eqref{boundGnear}, we get
\begin{align}  \label{suplapfracdiscLinfty}
\forall n \in \mathbb{N}_+, \quad \sup_{u \in \mathbb{R}} \Theta(n) \big|( \mathbb{L}_n^{\gamma,s} H )( \tfrac{u }{n}  ) \big| \leq C(\gamma) \big( \|  H \|_{\infty} +  \| \Delta H \|_{\infty} \big).
\end{align}
In the remainder of the proof, we consider the cases $H \notin C_c^2(\mathbb{R})$ and $H \in C_c^2(\mathbb{R})$ separately.
\begin{itemize}
\item
If $H \notin C_c^2(\mathbb{R})$, we have that $H, \nabla H \in \mathcal{S}(\mathbb{R})$. Now we claim that for every $B \geq 2$, it holds
\begin{equation} \label{claimfracdiscfarorig}
\frac{1}{n} \sum_{|x| \geq B n } \big[  \Theta(n) ( \mathbb{L}_n^{\gamma,s} H )( \tfrac{x + q_n^r }{n}  ) \big]^{2} \leq 
\begin{dcases}
 C(\gamma,H)  \int_{|u| \geq B} \big\{ \big[ F^{\nabla H}(u) \big]^2 u^4 + I_H^{\gamma}(u) \big\} \dd u,  \quad & 0 < \gamma < 2, \\   
  C(\gamma,H) \Bigg[ \int_{|u| \geq B} \big[ F^{\nabla H}(u) \big]^2 u^4 \dd u +B^{1-2 \gamma}  \Bigg],  \quad & \gamma \geq 2,
\end{dcases}
\end{equation}
where $F^{\nabla H}$ is given by \eqref{defF0}; and for $0 < \gamma < 2$, $I_H^{\gamma}: \mathbb{R} \mapsto \mathbb{R}$ is given by
\begin{align*} 
\forall u \in \mathbb{R}, \quad	I_H^{\gamma}(u):=  \int_{|v| \geq |u|/2} \frac{[H(u+v)-H(u)]^2}{|v|^{1+\gamma}} \dd v. 
\end{align*}
In particular, we get that $ \int_{\mathbb{R}} |I_H^{\gamma}(u)| \dd u \leq \mathcal{P}^{ \gamma} H < \infty$. Now we observe that \eqref{suplapfracdisc1} is a direct consequence of \eqref{suplapfracdiscLinfty} and \eqref{claimfracdiscfarorig}. In order to obtain \eqref{claimfracdiscfarorig}, we observe that from the inequality $(u+v)^2 \leq 2(u^2+v^2)$, the term in the left-hand side of \eqref{claimfracdiscfarorig} is bounded from above by
\begin{align} \label{L2lapfrac}
& \frac{2}{n} \sum_{|x| \geq Bn }  \Bigg\{ \frac{\Theta(n)}{n^2}    | F^{ \nabla H}  ( \tfrac{x + q_n^r }{n}  )  | 2  \sum_{y=1}^{|x|/2} y^2 s(y) \Bigg\}^2 + \frac{2}{n} \sum_{|x| \geq Bn }  \Bigg\{  \Theta(n)  \sum_{|y| \geq  |x|/2 } | H ( \tfrac{x + q_n^r + y }{n}  )  - H ( \tfrac{x + q_n^r }{n}  )  |s(y) \Bigg\}^2.
\end{align}
Above we applied a second-order Taylor expansion on $H$ around $(x + q_n^r)/n$. The leftmost sum over $x$ in the last line is bounded from above by 
\begin{align*}
C(\gamma) \frac{1}{n} \sum_{|x| \geq Bn }  \Bigg\{     | F^{\nabla H}  ( \tfrac{x + q_n^r }{n}  )  | \frac{(x + q_n^r)^2}{n^2}  \Bigg\}^2 \leq  C(\gamma,H)  \int_{|u| \geq B} \big[ F^{\nabla H}(u) \big]^2 u^4 \dd u.
\end{align*}
In order to estimate the rightmost sum over $x$ in \eqref{L2lapfrac}, we treat the cases $0 < \gamma < 2$ and $\gamma \geq 2$ separately.
\begin{itemize}
	\item 
If $0 < \gamma < 2$, applying the H\"older inequality,  the rightmost sum over $x$ in \eqref{L2lapfrac} is bounded from above by
\begin{align*}
&	\frac{\Theta(n)}{n} C(\gamma) \sum_{|x| \geq B n }      \sum_{|y| \geq  |x|/2 } \big[ H ( \tfrac{x + q_n^r + y }{n}  )  - H ( \tfrac{x + q_n^r }{n}  )  \big]^2 |y|^{-1-\gamma} \\
	\leq & C(\gamma,H)   \int_{|u| \geq B} \int_{|v| \geq |u|/2} \frac{[H(u+v)-H(u)]^2}{|v|^{1+\gamma}} \dd v \; \dd u = C(\gamma,H)  \int_{|u| \geq B}  I_H^{\gamma}(u) \; \dd u.
\end{align*} 
\item
If $\gamma \geq 2$, the rightmost sum over $x$ in \eqref{L2lapfrac} is bounded from above by
\begin{align*}
	\frac{C(\gamma)}{n} \sum_{|x| \geq Bn }  \Bigg\{ \frac{\Theta(n)}{n^{\gamma}} \frac{1}{n}  \sum_{|y| \geq  |x|/2 }  \| H \|_{\infty}  \Bigg( \frac{|y|}{n} \Bigg)^{-\gamma-1} \Bigg\}^2 \leq C(\gamma)  \int_{|u| \geq B} \Bigg\{ \int_{|v| \geq |u|/2} \frac{\| H \|_{\infty}}{|v|^{1+\gamma}} \dd v \Bigg\}^2  \dd u. 
\end{align*}
\end{itemize}
This ends the proof of \eqref{claimfracdiscfarorig}, and thus the proof of \eqref{suplapfracdisc1} for the case $H \notin C_c^2( \mathbb{R} )$.
\item
If $H \in C_c^2(\mathbb{R})$, there exist $a_H < b_H \in \mathbb{R}$ such that $H(u)=0$ whenever $u \notin (a_H,b_H)$. Thus the term inside the supremum in \eqref{suplapfracdisc1} is bounded from above by
\begin{align*}
	& \frac{c_{\gamma}}{n} \sum_{x = - \infty}^{(a_H-1)n} \Bigg\{ \Theta(n) \sum_{y = a_H n - x }^{b_H n - x} \big| H ( \tfrac{x + q_n^r + y }{n}  )  - H ( \tfrac{x + q_n^r  }{n}  )  \big| \; | y|^{-\gamma-1}  \Bigg\}^{2}  \\
	+ & \frac{c_{\gamma}}{n} \sum_{x = (b_H+1)n}^{\infty} \Bigg\{ \Theta(n) \sum_{y = a_H n - x }^{b_H n - x} \big| H ( \tfrac{x + q_n^r + y }{n}  )  - H ( \tfrac{x + q_n^r  }{n}  )  \big| \; |y |^{-\gamma-1} \Bigg\}^{2} + \frac{1}{n} \sum_{x = (a_H-1)n +1}^{(b_H+1)n-1} \big[  \Theta(n) ( \mathbb{L}_n^{\gamma} H )( \tfrac{x + q_n^r }{n}  ) \big]^{2}.
\end{align*}
Combining Lemma \ref{opdiscfarorig} with \eqref{suplapfracdiscLinfty}, the proof ends by bounding the last display from above by
\begin{align*}
C(\gamma) \| H \|_{\infty}^2 \frac{[\Theta(n)]^2}{n^{2\gamma}} (b_H - a_H)^2 + \frac{1}{n} \sum_{x = (a_H-1)n +1}^{(b_H+1)n-1} \big\{  C(\gamma) \big( \|  H \|_{\infty} +  \| \Delta H \|_{\infty} \big) \}^{2}.
\end{align*}
\end{itemize}

\subsubsection{Proof for Lemma \ref{lemvarext}} \label{useest32}

We begin the proof by claiming that
\begin{align}  \label{suplapfraclatdiscLinfty}
	\forall n \in \mathbb{N}_+, \quad \sup_{u \in \mathbb{R}} \big| \widehat{\mathbb{L}}_{n, \lambda}^{ \gamma} H ( \tfrac{u}{n}  ) \big| \leq 	C(\gamma)	|\lambda | K_n C(\gamma) (\| \Delta H \|_{\infty}  +   \|  \nabla H \|_{\infty} + \|   H \|_{\infty}).
\end{align}
In order to obtain the last display, we analyze the cases $\gamma \in (0,1)$ and $\gamma \geq 1$ separately.
\begin{itemize}
	\item 
	If $\gamma \in (0,1)$, we get from \eqref{Lnbgam} and \eqref{mngamma} that $\sup_{ u \in \mathbb{R} } \big| \widehat{\mathbb{L}}_{n, \lambda}^{ \gamma} H ( \tfrac{u}{n}  ) \big|$ is bounded from above by
	\begin{align*}
		 |\lambda | K_n 
		\Theta(n) \sup_{ u \in \mathbb{R} }  \Bigg\{ \sum_{r } | H ( \tfrac{u + r }{n}  )  - H ( \tfrac{u }{n}  )  | \; |a(r)| \Bigg\} \leq C(\gamma) |\lambda | K_n \big\{ \| \Delta H \|_{\infty}  +   \|  H \|_{\infty} \}.
	\end{align*}	
The inequality in the last line holds due to \eqref{defpsa}, \eqref{boundGnear} and \eqref{boundGfar}. 	
	\item 
	If $\gamma \geq 1$, we get from \eqref{Lnbgam} and \eqref{mngamma} that $\sup_{ u \in \mathbb{R} } \big| \widehat{\mathbb{L}}_{n, \lambda}^{ \gamma} H ( \tfrac{u}{n}  ) \big|$ is bounded from above by
	\begin{align*}
|\lambda | K_n 
		\Theta(n) \sup_{ u \in \mathbb{R} }  \Bigg\{ & \sum_{|r| > n } | H ( \tfrac{u + r }{n}  )  - H ( \tfrac{u }{n}  )  | \; |a(r)|  + \mathbbm{1}_{ \{ \gamma > 1 \} } \sum_{|r| > n } |  \tfrac{r}{n} \nabla H ( \tfrac{u}{n}  ) | \; |a(r)| \\
        + & \sum_{|r| \leq n } | H ( \tfrac{u + r }{n}  )  - H ( \tfrac{u }{n}  ) - \tfrac{r}{n} \nabla H ( \tfrac{u}{n}  ) | \; |a(r)|  \Bigg\}.  
	\end{align*}
	The leftmost sum, resp. rightmost sum in the first line of the last display is bounded from above by $C(\gamma) \| H \|_{\infty}$, resp. by $C(\gamma) \| \nabla H \|_{\infty}$, due to \eqref{defpsa} and \eqref{boundGfar}. Finally, applying a second-order Taylor expansion on $H$ around $u/n$, we get that the sum in the second line of the last display is bounded from above by $C(\gamma) \| \Delta H \|_{\infty}$.  
\end{itemize}
This ends the proof of \eqref{suplapfraclatdiscLinfty}. In the remainder of the proof, we consider the cases $H \notin C_c^2(\mathbb{R})$ and $H \in C_c^2(\mathbb{R})$ separately.
\begin{itemize}
\item
If $H \notin C_c^2(\mathbb{R})$, we have that $H, \nabla H \in \mathcal{S}(\mathbb{R})$. Now we claim that for every $B \geq 3$, it holds
\begin{equation} \label{claimfraclatdiscfarorig}
\begin{split}
&  \sum_{|x| \geq B n } \frac{1}{n} \big[  ( \widehat{\mathbb{L}}_{n, \lambda}^{ \gamma} H )( \tfrac{x + q_n^r }{n}  ) \big]^{2} \\
\leq & 
\begin{dcases}
(\lambda K_n)^{2} C(\gamma,H)  \int_{|u| \geq B} \big\{ \big[ F^{ H}(u) \big]^2  + \widehat{I}_H^{\gamma}(u) \big\} \dd u,  \quad & 0 < \gamma < 1, \\
(\lambda K_n)^{2} C(\gamma,H)  \int_{|u| \geq B} \big\{ \big[ F^{\nabla H}(u) \big]^2  + \widehat{I}_H^{\gamma}(u) \big\} \dd u,  \quad &  \gamma =1, \\    
(\lambda K_n)^{2} C(\gamma,H)  \int_{|u| \geq B} \big\{ \big[ F^{\nabla H}(u) \big]^2  + \widehat{I}_H^{\gamma}(u) + \big[ \nabla H(u) \big]^2  \big\} \dd u,  \quad &  1 < \gamma < 2, \\    
(\lambda K_n)^{2} C(\gamma,H) \Bigg\{ \int_{ |u| \geq B }   \big[ F^{\nabla H} ( u  ) \big]^{2} u^2 \dd u + B^{3 - 2 \gamma} \Bigg\},  \quad & \gamma \geq 2,
\end{dcases}
\end{split}
\end{equation}
where $F^{H}$ and $F^{\nabla H}$ are given by \eqref{defF0}; and for $0 < \gamma < 2$, $\widehat{I}_H^{\gamma}: \mathbb{R} \mapsto \mathbb{R}$ is given by
\begin{align*} 
\forall u \in \mathbb{R}, \quad	\widehat{I}_H^{\gamma}(u):=  \int_{|v| \geq 1} \frac{[H(u+v)-H(u)]^2}{|v|^{1+\gamma}} \dd v. 
\end{align*}
In particular, we get that $ \int_{\mathbb{R}} |\widehat{I}_H^{\gamma}(u)| \dd u \leq \mathcal{P}^{ \gamma} H < \infty$. Now we observe that \eqref{suplapfraclatdisc1} is a direct consequence of \eqref{suplapfraclatdiscLinfty} and \eqref{claimfraclatdiscfarorig}. In order to obtain \eqref{claimfraclatdiscfarorig}, we consider the following cases.
\begin{itemize}
	\item 
	If $\gamma \in (0,1)$, we get from \eqref{mngamma}	that $m_n^{\gamma}=0$. From the inequality $(u+v)^{2} \leq 2(u^{2}+v^{2})$, the sum over $x$ in \eqref{claimfraclatdiscfarorig} is bounded from above by
	\begin{align}
		& \frac{2 (\lambda K_n)^{2}}{n} \sum_{|x| \geq Bn } \Bigg\{ \Bigg( \Theta(n) \sum_{ |y| \leq n } | H ( \tfrac{x + q_n^r + y }{n}  )  - H ( \tfrac{ x + q_n^r }{n}  )  | \; | a(y)|  \Bigg)^{2}  \Bigg\} \nonumber \\
		+ & \frac{2 (\lambda K_n)^{2}}{n} \sum_{|x| \geq Bn } \Bigg\{ \Bigg( \Theta(n) \sum_{ |y| > n } | H ( \tfrac{x + q_n^r + y }{n}  )  - H ( \tfrac{ x + q_n^r }{n}  )  | \; | a(y)|  \Bigg)^{2}  \Bigg\}.   \label{suplapfraclatdiscL2b}
	\end{align}
Thus, in order to obtain \eqref{claimfraclatdiscfarorig}, it is enough to prove that
\begin{align}
& \frac{1}{n} \sum_{|x| \geq Bn } \Bigg\{ \Bigg( \Theta(n) \sum_{ |y| \leq n } | H ( \tfrac{x + q_n^r + y }{n}  )  - H ( \tfrac{ x + q_n^r }{n}  )  | \; | a(y)|  \Bigg)^{2}  \Bigg\} \leq C (\gamma,H) \int_{ |u| \geq B }   \big[ F^{H} ( u  ) \big]^{2} \dd u, \nonumber \\
& \frac{1}{n} \sum_{|x| \geq Bn } \Bigg\{ \Bigg( \Theta(n) \sum_{ |y| > n } | H ( \tfrac{x + q_n^r + y }{n}  )  - H ( \tfrac{ x + q_n^r }{n}  )  | \; | a(y)|  \Bigg)^{2}  \Bigg\} \leq  C (\gamma,H)  \int_{|u| \geq B} \widehat{I}_H^{\gamma}(u) \; \dd u. \label{suplapfraclatdiscL2b1}
\end{align}
In order to obtain the first upper bound, it is enough to apply the Mean Value Theorem; in this way, the term in the left-hand side in the first line of the last display is bounded from above by
\begin{align*}
		\frac{C (\gamma)}{n}   \sum_{|x| \geq Bn }  \big[ F^{H} ( \tfrac{ x + q_n^r }{n}  ) \big]^{2}     \Bigg[ \frac{1}{n}  \sum_{ y =1 }^{n} \Bigg( \frac{y}{n} \Bigg)^{-\gamma}  \Bigg]^{2} \leq & \frac{C (\gamma)}{n}   \sum_{|x| \geq Bn }  \big[ F^{H} ( \tfrac{ x + q_n^r }{n}  ) \big]^{2} \Bigg[ \int_0^{1} u^{-\gamma} \; \dd u \Bigg]^{2}.
	\end{align*}	
From the H\"older inequality, the left-hand side of \eqref{suplapfraclatdiscL2b1} is bounded from above by	
\begin{align*}
&	\frac{\Theta(n)}{n} C(\gamma) \sum_{|x| \geq B n }      \sum_{|y| \geq  n } \big[ H ( \tfrac{x + q_n^r + y }{n}  )  - H ( \tfrac{x + q_n^r }{n}  )  \big]^2 |y|^{-1-\gamma} \\
	\leq & C(\gamma,H)  \int_{|u| \geq B} \int_{|v| \geq 1} \frac{[H(u+v)-H(u)]^2}{|v|^{1+\gamma}} \dd v \; \dd u = C(\gamma,H)  \int_{|u| \geq B} \widehat{I}_H^{\gamma}(u) \; \dd u.
\end{align*} 
	\item 
	If $\gamma =1$, from \eqref{mngamma}, the sum over $x$ in \eqref{claimfraclatdiscfarorig} is bounded from above by \eqref{suplapfraclatdiscL2b} (whicn can be estimated by applying \eqref{suplapfraclatdiscL2b1}), plus
	\begin{align} \label{suplapfraclatdiscL2a}
		& \frac{2(\lambda K_n)^{2}}{n} \sum_{|x| \geq Bn } \Bigg\{ \Bigg( \Theta(n) \sum_{ |y| \leq n } \big| H ( \tfrac{x + q_n^r + y }{n}  )  - H ( \tfrac{ x + q_n^r }{n}  ) - \tfrac{y}{n} \nabla H ( \tfrac{ x + q_n^r }{n}  )  \big| \; | a(y)|  \Bigg)^{2}  \Bigg\} 
	\end{align}
In order to treat the term in the last line, we claim that
	\begin{equation} \label{suplapfraclatdiscL2a1}
\begin{split}
	&	\frac{1}{n} \sum_{|x| \geq Bn } \Bigg\{ \Bigg( \Theta(n) \sum_{ |y| \leq n } \big| H ( \tfrac{x + q_n^r + y }{n}  )  - H ( \tfrac{ x + q_n^r }{n}  ) - \tfrac{y}{n} \nabla H ( \tfrac{ x + q_n^r }{n}  )  \big| \; | a(y)|  \Bigg)^{2}  \Bigg\}  \\
    \leq & C(\gamma,H) \int_{ |u| \geq B }   \big[ F^{\nabla H} ( u  ) \big]^{2} \dd u.
    \end{split}
	\end{equation}
Assuming that \eqref{suplapfraclatdiscL2a1} holds, we  obtain \eqref{claimfraclatdiscfarorig}. In order to prove \eqref{suplapfraclatdiscL2a1}, we apply a Taylor expansion on $H$ around $(x+q_n^r)/n$ in order to bound the term in the first line of the last display from above by
	\begin{align*}
		\frac{1}{n} C (\gamma)  \sum_{|x| \geq Bn }  \big[ F^{\nabla H} ( \tfrac{ x + q_n^r }{n}  ) \big]^{2}     \Bigg[ \frac{1}{n}  \sum_{ y =1 }^{n} \Bigg( \frac{y}{n} \Bigg)^{1-\gamma}  \Bigg]^{2} \leq & \frac{1}{n} C (\gamma)  \sum_{|x| \geq Bn }  \big[ F^{\nabla H} ( \tfrac{ x + q_n^r }{n}  ) \big]^{2} \Bigg[ \int_0^{1} u^{1-\gamma} \; \dd u \Bigg]^{2}. 
	\end{align*}
	\item 
	If $\gamma \in (1,2)$, we get from \eqref{mngamma}	that the sum over $x$ in \eqref{claimfraclatdiscfarorig} is bounded from above by \eqref{suplapfraclatdiscL2a}, plus the double of \eqref{suplapfraclatdiscL2b}, plus
	\begin{align*}
		\frac{(\lambda K_n)^{2}}{n} \sum_{|x| \geq Bn } \Bigg\{ \Bigg( \Theta(n) \sum_{ |y| > n } \big|  \tfrac{y}{n} \nabla H ( \tfrac{ x + q_n^r }{n}  )  \big| \; | a(y)|   \Bigg)^{2}  \Bigg\}.
	\end{align*}
	The terms in \eqref{suplapfraclatdiscL2a} and \eqref{suplapfraclatdiscL2b} can be estimated by applying \eqref{suplapfraclatdiscL2a1} and \eqref{suplapfraclatdiscL2b1}, respectively. Moreover, the term in the last display is bounded from above by $(\lambda K_n)^{2}$, times
	\begin{align*}
		\frac{1}{n} \sum_{|x| \geq Bn } \big[ \nabla H ( \tfrac{ x + q_n^r }{n}  ) \big]^{2} C(\gamma) \Bigg\{  \frac{1}{n} \sum_{ y = n }^{\infty} \Bigg( \frac{y}{n} \Bigg)^{-\gamma}  \Bigg\}^{2} 	\leq & \frac{1}{n} \sum_{|x| \geq Bn } \big[ \nabla H ( \tfrac{ x + q_n^r }{n}  ) \big]^{2} C(\gamma) \Bigg\{  \int_1^{\infty} u^{-\gamma} \; \dd u  \Bigg\}^{2}. 
	\end{align*}
	\item 
	If $\gamma \geq 2$, for every $x \in \mathbb{Z}$ such that $|x| \geq 3n$, we get from \eqref{defpsa}, \eqref{Lnbgam} and \eqref{mngamma} that $\big| \widehat{\mathbb{L}}_{n, \lambda}^{ \gamma} H (\tfrac{x+ q_{n}^{r} }{n}) \big|$ is be bounded from above by
	\begin{align*}
		&  C(\gamma) | \lambda | K_n \Bigg\{ \frac{  \Theta(n)}{ n^2}  F^{\nabla H}(\tfrac{x+ q_{n}^{r}}{n})   \sum_{r =1 }^{|x|/3} r^{1-\gamma}  + \Theta(n) \| \nabla  H \|_{\infty} \sum_{r =|x|/3 }^{\infty} r^{-\gamma} + \frac{  \Theta(n)}{ n} \|  H \|_{\infty} \sum_{r =|x|/3 }^{\infty} r^{-1-\gamma} \Bigg\} \\
		\leq & C(\gamma) | \lambda | K_n \Bigg\{  F^{\nabla H}(\tfrac{x+ q_{n}^{r}}{n}) \Bigg| \frac{x+ q_{n}^{r}}{n} \Bigg|  +  \frac{  \Theta(n)}{ n^{\gamma}} ( \| \nabla H \|_{\infty} + \|  H \|_{\infty}  )   \Bigg| \frac{x}{n} \Bigg|^{1-\gamma}   \Bigg\}.
	\end{align*}
Now from the last display, the sum over $x$ in \eqref{claimfraclatdiscfarorig} is bounded from above by
	\begin{align*}
		&\frac{1}{n}   \sum_{|x| \geq Bn } \Bigg\{ C(\gamma) | \lambda | K_n \Bigg(F^{\nabla H}(\tfrac{x+ q_{n}^{r}}{n}) \Bigg| \frac{x+ q_{n}^{r}}{n} \Bigg|  +  \frac{  \Theta(n)}{ n^{\gamma}} ( \| \nabla H \|_{\infty} + \|  H \|_{\infty}  )   \Bigg| \frac{x}{n} \Bigg|^{1-\gamma} \Bigg) \Bigg\}^{2} \\
		\leq & C(\gamma,H)   (\lambda K_n)^{2}  \int_{ |u| \geq B } \{  \big[ F^{\nabla H} ( u  ) \big]^{2} u^2 + u^{2-2 \gamma} \} \; \dd u .
	\end{align*}
\end{itemize}
This ends the proof of \eqref{claimfraclatdiscfarorig}, and thus the proof of \eqref{suplapfraclatdisc1} for the case $H \notin C_c^2(\mathbb{R})$.
\item
If $H \in C_c^2(\mathbb{R})$, there exist $a_H < b_H \in \mathbb{R}$ such that $H(u)=0$ whenever $u \notin (a_H,b_H)$. Thus the term inside the supremum in \eqref{suplapfraclatdisc1} is bounded from above by
\begin{align*}
		& \frac{C(\gamma) (\lambda K_n)^{2} }{n} \sum_{x = - \infty}^{(a_H-1)n} \Bigg\{ \Theta(n) \sum_{y = a_H n - x }^{b_H n - x} \big| H ( \tfrac{x + q_n^r + y }{n}  )  - H ( \tfrac{x + q_n^r  }{n}  )  \big| \; | y|^{-\gamma-1}  \Bigg\}^{2}  \\
		+ & \frac{C(\gamma) (\lambda K_n)^{2} }{n} \sum_{x = (b_H+1)n}^{\infty} \Bigg\{ \Theta(n) \sum_{y = a_H n - x }^{b_H n - x} \big| H ( \tfrac{x + q_n^r + y }{n}  )  - H ( \tfrac{x + q_n^r  }{n}  )  \big| \; |y |^{-\gamma-1} \Bigg\}^{2} \\
		+& \frac{1}{n} \sum_{x = (a_H-1)n +1}^{(b_H+1)n-1} \big[  \widehat{\mathbb{L}}_{n, \lambda}^{ \gamma} H (\tfrac{x+ q_{n}^{r} }{n})\big]^{2}.
	\end{align*}
Combining Lemma \ref{opdiscfarorig} with \eqref{suplapfraclatdiscLinfty}, the proof ends by bounding the last display from above by
	\begin{align*}
		C(\gamma) (\lambda K_n)^{2} \| H \|_{\infty}^2 \frac{[\Theta(n)]^2}{n^{2\gamma}} (b_H - a_H)^2 + \frac{1}{n} \sum_{x = (a_H-1)n +1}^{(b_H+1)n-1} C(\gamma) \big\{ \| \Delta H \|_{\infty} + \| \nabla H \|_{\infty} + \|  H \|_{\infty}  	 \}^{2}.
	\end{align*}
\end{itemize}

\subsection{Proof of \eqref{claimuseest4}} \label{useest4}

From a second-order Taylor expansion of $H$ 
and a convex inequality, the term inside the supremum in \eqref{claimuseest4} is bounded from above by
\begin{align}
& \frac{3}{n^2} \sum_{x,y}   \big\{ \nabla H ( \tfrac{ y }{n}  ) - \nabla H ( \tfrac{ x }{n}  ) \big\}^{2}  a^{2}(y-x) + \frac{C(\gamma)}{n^4} \Bigg\{ \sum_{x}  \big[  \Delta H(\omega_{x,r}^{v,n}) \big]^{2} + \sum_{y}  \big[  \Delta H(\omega_{y,r}^{v,n}) \big]^{2} \Bigg\} \label{bound21a}, 
\end{align}
where for any $z \in \mathbb{Z}$, $\omega_{z,r}^{v,n}$ is some number in $[z, z + q_{n}^{ r }  ]$. Above we applied the fact that $|q_{n}^{ r }| \leq 1$ and that $\sum_{z} a^{2}(z) \leq C(\gamma)$. In the remainder of the proof, we consider the cases $H \notin C_c^2(\mathbb{R})$ and $H \in C_c^2(\mathbb{R})$ separately.
\begin{itemize}
\item
If $H \notin C_c^2(\mathbb{R})$, we have that $H, \nabla H \in \mathcal{S}(\mathbb{R})$. By applying arguments analogous to the ones used to estimate the leftmost sum over $z$ in \eqref{errorexpquadvar3a1}, the rightmost term in \eqref{bound21a} is bounded from above by
\begin{align} \label{upbound21a}
 &	 \frac{C(\gamma)}{n^3}  \frac{  1  }{ n} \Bigg\{ \big[  \| \Delta H \|_{\infty} \big]^{2}	+  \sum_x \big[ F^{\nabla H} ( \tfrac{x   }{n}  ) \big]^{2}\Bigg\}  \leq  \frac{ C(\gamma,H)  }{ n^3}.
\end{align}	
where $F^{\nabla H} \in L^2(\mathbb{R})$ is given by \eqref{defF0}. Now we rewrite the leftmost term in \eqref{bound21a} as	
	\begin{align}
		&   \frac{  3  }{ n^2} \Bigg\{ \sum_{x,y: |y-x| \geq n} \big\{ \nabla H ( \tfrac{ y }{n}  ) - \nabla H ( \tfrac{ x }{n}  ) \big\}^{2}  a^{2}(y-x) + \sum_{x,y: |y-x| < n} \big\{ \nabla H ( \tfrac{y }{n}  ) - \nabla H ( \tfrac{ x }{n}  ) \big\}^{2}  a^{2}(y-x) \Bigg\} \nonumber  \\
		\leq &  \frac{  C(\gamma)  }{ n} \sum_{z=n}^{\infty} z^{-2 \gamma-2}  \frac{1}{n} \sum_{w} \big\{ \nabla H ( \tfrac{ w }{n}  )  \big\}^{2}  \label{bound21b2}  \\
        + & \frac{ C(\gamma)    }{ n^{3}  } \sum_{z=1}^{n} z^{-2 \gamma} \Bigg\{ \frac{1}{n} \sum_{w: |w| < 2n} \big\{ \| \Delta H \|_{\infty} \big\}^{2} + \frac{1}{n}  \sum_{w: |w| \geq 2n} \big\{  F^{\nabla H} ( \tfrac{w}{n} ) \big\}^{2}   \Bigg\}   \label{bound21b1} \\
        \leq & \frac{C(\gamma,H)}{n^{2\gamma +2}} + \frac{C(\gamma,H)}{n^{3}} \leq \frac{C(\gamma,H)}{n^{3}}. \nonumber 
	\end{align}
Combining the last display with \eqref{upbound21a}, the proof of \eqref{claimuseest4} follows for the case $G \notin C_c^2(\mathbb{R})$. The term in \eqref{bound21b2}, resp. \eqref{bound21b1}, comes from the inequality $(u+v)^2 \leq 2(u^2+v^2)$, resp. arguments analogous to the ones used to estimate the leftmost sum over $z$ in \eqref{errorexpquadvar3a1}. Above we used the fact that $\gamma \geq 1 > 1/2$, in order to ensure that the sum over $z$ in \eqref{bound21b1} is convergent; and to get the last inequality in the fourth line of the last display.
    
\item
If $H \in C_c^2(\mathbb{R})$, there exist $a_H < b_H \in \mathbb{R}$ such that $H(u)=0$ whenever $u \notin (a_H,b_H)$. Thus, the rightmost term in \eqref{bound21a} can be rewritten as
\begin{align} \label{upbound2aaa}
& \frac{ C(\gamma) }{n^{4}}  \Bigg\{  \sum_{y=a_H n}^{-1+b_H n} \big[  \Delta H(\omega_{y,r}^{v,n}) \big]^{2} +  \sum_{y=a_H n}^{-1+b_H n} \big[  \Delta H(\omega_{y,r}^{v,n}) \big]^{2} \Bigg\}  \leq \frac{C(\gamma)}{n^3}   \| \Delta H \|^{2}_{\infty} (b_H - a_H).   
\end{align}
Now from the assumption that $a_H < b_H \in \mathbb{R}$ are such that $H(u)=0$ whenever $u \notin (a_H,b_H)$, \eqref{bound21a} can be rewritten as
\begin{align}
	&   \frac{  C(\gamma)  }{ n^{2} } \sum_{x = - \infty}^{(a_H-1)n}   \sum_{y = a_H n - x }^{b_H n - x} \big[ \nabla H ( \tfrac{x + y }{n}  )  - \nabla H ( \tfrac{x   }{n}  )  \big]^{2} | y|^{-2 \gamma-2}  \mathbbm{1}_{ \{ y \neq 0 \} }  \label{bound21a1a} \\
	+ &  \frac{  C(\gamma)  }{ n^{2} } \sum_{x = (b_H+1)n}^{\infty}  \sum_{y = a_H n - x }^{b_H n - x} \big[ \nabla H ( \tfrac{x  + y }{n}  )  - \nabla H ( \tfrac{x   }{n}  )  \big]^{2} | y|^{-2 \gamma-2}  \mathbbm{1}_{ \{ y \neq 0 \} } \label{bound21a1b} \\
		+ &  \frac{  C(\gamma)  }{ n^{2} } \sum_{x = (a_H-1)n +1}^{(b_H+1)n-1}  \sum_{y = a_H n - x }^{b_H n - x} \big[ \nabla H ( \tfrac{x  + y }{n}  )  - \nabla H ( \tfrac{x   }{n}  )  \big]^{2} | y|^{-2 \gamma-2}  \mathbbm{1}_{ \{ y \neq 0 \} }. \label{bound21a1c}
\end{align}
From an application of \eqref{ineqalp} with $\alpha=2(\gamma+1)$, \eqref{bound21a1a} can be  bounded from above by
\begin{align*}
 \frac{   C(\gamma)  }{ n^{2 \gamma+2} } (b_H - a_H  ) \| \nabla H \|_{\infty}^{2}  \int_1^{\infty} u^{- 2(\gamma+1)} \dd u.
\end{align*}
From a similar argument as the one we used to estimate  \eqref{bound21a1a}, we conclude that the sum of \eqref{bound21a1a} and \eqref{bound21a1b} is bounded from above by
\begin{equation} \label{bound21a1ab}
 \frac{   C(\gamma)  }{ n^{2 \gamma+2} }    (b_H - a_H  ) \| \nabla H \|_{\infty}^{2}  \leq  \frac{   C(\gamma)  }{ n^{3} } (b_H - a_H  ) \| \nabla H \|_{\infty}^{2}, 
\end{equation}
The last inequality since $\gamma \geq 1 \geq 1/2$. It remains to treat \eqref{bound21a1c}. From the Mean Value Theorem, this term is bounded from above by 
\begin{align*}
&  \frac{ C(\gamma)    }{ n^{4} } \| \Delta H \|_{\infty}^{2} \sum_{x = (a_H-1)n +1}^{(b_H+1)n-1}    \sum_{z = 1 }^{(1+b_H - a_H) n - 1} z^{-2\gamma}   \leq   \frac{ C(\gamma)   }{ n^{3} } \| \Delta H \|_{\infty}^{2} (b_H - a_H +2).
\end{align*}
Combining the last display with \eqref{bound21a1ab}, the proof of \eqref{claimuseest4} ends. Above we used the convergence of the sum over $z$, since $\gamma \geq 1 > 1/2$.
\end{itemize}

\subsection{Proof of \eqref{claimuseest5}} \label{useest5}

From the fact that $(u+v)^2 \leq 2(u^2+v^2)$, the term inside the supremum in \eqref{claimuseest5} is bounded from above by
\begin{align} \label{claimuseest5a}
2 \sum_{z,w:|w - z| \geq n}  [  H ( \tfrac{z + q_r^n }{n}  )  ]^{2} a^{2}(w-z)    \leq \widetilde{C}_{H} n  \sum_{z = n}^{\infty}   a^{2}(z)  = \frac{C(\gamma;H) }{n^{2 \gamma}} \frac{1}{n} \sum_{z = n}^{\infty}  \big( \tfrac{z}{n} \big)^{-2 \gamma-2} = \frac{C(\gamma;H) }{n^{2 \gamma}}, 
\end{align}
where $\widetilde{C}_{H}$ is a constant depending only on $H$. This ends the proof of \eqref{claimuseest5} for the general case. 

It remains to show that the constant $\widehat{C}$ in \eqref{claimuseest5} satisfies \eqref{Cinvshift} when $H \in C_c^{0}( \mathbb{R} )$. Keeping this in mind, if $a_H < b_H \in \mathbb{R}$ are such that $H(u)=0$ whenever $u \notin (a_H,b_H)$, the proof ends by rewriting the double sum over $z$ and $w$ in \eqref{claimuseest5a} as
\begin{align*}
 2  \sum_{z=n}^{\infty} a^2(z) \sum_{x=a_H n}^{b_H n}  [  H ( \tfrac{x + q_n^r  }{n}  )  ]^{2} = n (b_H - a_H + 1 ) \| H \|_{\infty}^2 \frac{C(\gamma)}{n} \sum_{z = n}^{\infty}  \big( \tfrac{z}{n} \big)^{-2 \gamma-2} = \frac{(b_H - a_H + 1 ) \| H \|_{\infty}^2 }{n^{2 \gamma}} C(\gamma). 
\end{align*}

\subsection{Proof of \eqref{claimuseest6}} \label{useest6}

We begin by rewriting the term inside the supremum in \eqref{claimuseest6} as
\begin{equation} \label{bndxynotfar}
\sum_{z,w: \; L_n \leq |w - z| \leq n} \big[  H ( \tfrac{ w }{n}  ) - H ( \tfrac{ z }{n}  ) \big]^{2}  a^{2}(w-z)  \leq   \frac{ C(\gamma)  }{ n^2  } \sum_{z,w: \; L_n \leq |w - z| \leq n} \big\{  \nabla H (\omega_{n}^{z,w}) \big\}^{2}  |w-z|^{-2 \gamma},  
\end{equation}
for some $\omega_{n}^{z,w}$ between $z/n$ and $w/n$. Above we applied the Mean Value Theorem. By applying arguments analogous to the ones used to estimate the sum over $z$ on the right-hand side of \eqref{errorexpquadvar3a1}, the proof for \eqref{claimuseest6} in the general case ends. It remains to show that the constant $\widehat{C}$ in \eqref{claimuseest6} satisfies \eqref{Cinvshift} when $H \in C_c^2(\mathbb{R})$. It is enough to observe that If $a_H < b_H \in \mathbb{R}$ are such that $H(u)=0$ whenever $u \notin (a_H,b_H)$, then the leftmost term in \eqref{bndxynotfar} is bounded from above by
\begin{align*}
2 \sum_{w=(a_H-1)n}^{b_H n} \; \; \sum_{z: L_n \leq |w-z| \leq n} \big[  H ( \tfrac{ w }{n}  ) - H ( \tfrac{ z }{n}  ) \big]^{2}  a^{2}(w-z) 
\leq C(\gamma) (b_H - a_H + 3 ) \frac{ \| \nabla H \|^2_{\infty} }{n}  \sum_{x=L_n}^{n} x^{- 2 \gamma}.
\end{align*}

\subsection{Proof of \eqref{claimuseest7}} \label{useest7}

We consider the cases $H \notin C_c^1(\mathbb{R})$ and $H \in C_c^1(\mathbb{R})$ separately.
\begin{itemize}
\item
If $H \notin C_c^1(\mathbb{R})$, we have that $H \in \mathcal{S}(\mathbb{R})$. Thus, from \eqref{defpsa} and an application of the Mean Value Theorem, the leftmost term in \eqref{claimuseest7} is bounded from above by
\begin{align*}
& \frac{C(\gamma)}{n} \sum_{z=1}^{L_n-1} z^{1-\gamma} \Bigg\{ \frac{1}{n} \sum_{|w| < 2 L_n}  \| \nabla H \|^2_{\infty} + \frac{1}{n} \sum_{|w| \geq 2 L_n} | F^H ( \tfrac{w}{n} )|^2  \Bigg\} \leq \frac{\widetilde{C}_H}{n} C(\gamma) \sum_{z=1}^{L_n-1} z^{1-\gamma}.
\end{align*}
\item
If $H \in C_c^1(\mathbb{R})$, there exist $a_H < b_H \in \mathbb{R}$ such that $H(u)=0$ whenever $u \notin (a_H,b_H)$. Thus the leftmost double sum over $w$ and $z$ in \eqref{ub36a1} can be rewritten as
\begin{align*}
 \sum_{w=(a_H-1) n}^{(b_H+1) n} \sum_{z=1}^{L_n-1}   |  H ( \tfrac{ w+z  }{n}  ) - H ( \tfrac{ w }{n}  )  |^2 \; | a(z) |  \leq \frac{C(\gamma)}{n} (b_H-a_H+3)  \| \nabla H \|_{\infty}^2 \sum_{z=1}^{L_n-1}  z^{1-\gamma}. 
\end{align*}
Above we applied \eqref{defpsa} and the Mean Value Theorem. \end{itemize}   

\subsection{Proof of \eqref{claimuseest8}} \label{useest8}

We begin by rewriting the term inside the supremum in \eqref{claimuseest8} as
\begin{align} \label{bnd1lem46}
\sum_{|x| < 2 L_n } \sum_{z=1}^{L_n-1} \Bigg\{  z^{2} a(z) \Delta H (\omega_{x,z}^{r,n}) \Bigg\}^2 + \sum_{|x| \geq 2 L_n} \sum_{z=1}^{L_n-1} \Bigg\{ z^{2} a(z) \Delta  H(\omega_{x,z}^{r,n}) \Bigg\}^2.
\end{align}
Since $L_n \leq n$, the first double sum in the last line is bounded from above by
\begin{align} \label{ublem46}
    \sum_{|x| < 2 n } \|\Delta H \|^2_{\infty} \sum_{z=1}^{L_n-1}   z^{4} [a(z)]^2   \leq n C(\gamma)  \|\Delta H \|^2_{\infty} \Bigg\{1 + \int_1^{L_n} u^{2 - 2 \gamma} \; \dd u \Bigg\}.
\end{align}
In the remainder of the proof, we consider the cases $H \notin C_c^2(\mathbb{R})$ and $H \in C_c^2(\mathbb{R})$ separately.
\begin{itemize}
\item
If $H \notin C_c^2(\mathbb{R})$, we have that $H, \nabla H \in \mathcal{S}(\mathbb{R})$. Thus, from \eqref{defpsa} and the Mean Value Theorem, we have that the rightmost double sum in \eqref{bnd1lem46} is bounded from above by
\begin{align*}
\sum_{|x| \geq 2 L_n} \sum_{z=1}^{L_n-1} \Bigg\{  z^{2} |a(z)| F^{\nabla H} ( \tfrac{x}{n} ) \Bigg\}^2 \leq  n C(\gamma,H)   \sum_{z=1}^{L_n-1} z^{2-2 \gamma} \leq  n C(\gamma,H)    \Bigg\{1 + \int_1^{L_n} u^{2 - 2 \gamma} \;  \dd u \Bigg\},
\end{align*}
where $F^{\nabla H} \in L^2(\mathbb{R})$ is given by \eqref{defF0}. By combining the last display with \eqref{ublem46}, the proof of \eqref{claimuseest8} ends for the case $H \notin C_c^2(\mathbb{R})$.
\item
If $H \in C_c^2(\mathbb{R})$, there exist $a_H < b_H \in \mathbb{R}$ such that $H(u)=0$ whenever $u \notin (a_H,b_H)$. Thus, from \eqref{defpsa}, the rightmost double sum in \eqref{bnd1lem46} is bounded from above by
\begin{align*}
\sum_{y = (a_H-1) n}^{( b_H + 1 ) n} \|\Delta H \|^2_{\infty} \sum_{z=1}^{L_n-1}   z^{4} [a(z)]^2   \leq n C(\gamma) (b_H - a_H + 3 )  \|\Delta H \|^2_{\infty}  \Bigg\{1 + \int_1^{L_n} u^{2 - 2 \gamma} \; \dd u \Bigg\}.
\end{align*}
By combining the last display with \eqref{ublem46}, the proof of \eqref{claimuseest8} ends.
\end{itemize}

\subsection{Proof of \eqref{claimuseest9}} \label{useest9}

From a second-order Taylor expansion of $H$ around $y/n$ and a convex inequality, the term in \eqref{claimuseest9} is bounded from above by
\begin{align}
 & \frac{2}{n^2}  \sum_{y} [ \nabla H ( \tfrac{y }{n}  ) ]^2  \Bigg(  \Bigg\{  \sum_{z=1}^{L_n-1} z |a(z)|  \Bigg\}^2 +   \sum_{z=1}^{L_n-1} z^2 a^2(z) \Bigg) \label{lemrepxzx1a3a}  \\
+ & \frac{1}{n^4} \sum_{y} \Bigg\{ \sum_{z=1}^{L_n-1}   z^2 |a(z)| \Delta H ( \omega_{y,z}  )   \Bigg\}^2 + \frac{1}{n^4} \sum_{y} \sum_{z=1}^{L_n-1}   z^4 a^2(z) [\Delta H ( \omega_{y,z}  )  ]^{2}, \label{lemrepxzx1a3b}
\end{align} 
 where $\omega_{y,z} \in (y/n, \; z/n)$. From \eqref{defpsa}, the term in \eqref{lemrepxzx1a3a} is bounded from above by
\begin{equation} \label{lemrepxzx1a4}
\frac{C(\gamma)}{n} \| \nabla H \|^2_{2,n}  \big\{ \mathbbm{1}_{ \{ \gamma > 1 \} } + \big[ \log(L_n) \big]^{2} \mathbbm{1}_{ \{ \gamma = 1 \} } \big\}.
\end{equation}
Above we used the assumption that $\gamma \geq 1$. It remains to estimate the term in \eqref{lemrepxzx1a3b}. In order to do so, we consider two cases separately: $H \notin C_c^2(\mathbb{R})$ and $H \in C_c^2(\mathbb{R})$. 
\begin{itemize}
\item
If $H \notin C_c^2(\mathbb{R})$, we have that $H, \nabla H \in \mathcal{S}(\mathbb{R})$. Thus, from a convex inequality, the term in \eqref{lemrepxzx1a3b} is bounded from above by 
	\begin{align*}
	 &  \sum_{| y | \leq 2 L_n } \Bigg( \Bigg\{ \sum_{z=1}^{L_n-1} \frac{z^2}{n^2} \| \Delta H \|_{\infty} |a(z)|    \Bigg\}^2 +  \sum_{z=1}^{L_n-1} \frac{z^4}{n^4} \big[ \| \Delta H \|_{\infty}  \big]^2 a^2(z) \Bigg) \\
	+ &  \sum_{| y | > 2 L_n } \Bigg( \Bigg\{ \sum_{z=1}^{L_n-1} \frac{z^2}{n^2} F^{\nabla H} ( \tfrac{ y      }{n}  )  |a(z)|    \Bigg\}^2 +  \sum_{z=1}^{L_n-1} \frac{z^4}{n^4} \big[ F^{\nabla H} ( \tfrac{ y      }{n}  ) \big]^2 a^2(z) \Bigg) \leq \frac{C(\gamma, H)}{n},
	\end{align*}
where $F^{\nabla H} \in L^{2}(\mathbb{R})$ is given by \eqref{defF0}. The inequality in the last line comes from combining arguments analogous to the ones used to estimate the leftmost sum over $z$ in \eqref{errorexpquadvar3a1} with \eqref{defpsa}, the assumption that $\gamma \geq 1$ and the fact that $L_n \leq n$. Combining the last display with  \eqref{lemrepxzx1a4}, the proof for \eqref{claimuseest9} ends in this case.
\item
If $H \in C_c^2(\mathbb{R})$, there exist $a_H < b_H \in \mathbb{R}$ such that $H(u)=0$ whenever $u \notin (a_H,b_H)$. Thus, the term in \eqref{lemrepxzx1a3b} can be rewritten as
\begin{align*}
\frac{1}{n^4} \sum_{y=(a_H-1)n}^{b_H n} \Bigg\{  \Bigg[ \sum_{z=1}^{L_n-1}   z^2 |a(z)| \Delta H ( \omega_{y,z}  )   \Bigg]^2 +  \sum_{z=1}^{L_n-1}   z^4 a^2(z) [\Delta H ( \omega_{y,z}  )  ]^{2} \Bigg\} \leq \frac{C(\gamma)}{n}  (b_H - a_H + 1) \| \Delta H \|_{\infty}.
\end{align*}
Above we used the fact that $L_n \leq n$ and the assumption that $\gamma \geq 1$. Combining the last display with  \eqref{lemrepxzx1a4}, the proof for \eqref{claimuseest9} ends.
\end{itemize}

\subsection{Proof of Lemma \ref{convL2lapyesvel}} \label{useest10}

We begin by observing that
\begin{align*}
& \limsup_{B \rightarrow \infty} \limsup_{n \rightarrow \infty} \frac{1}{n}  \sup_{r \in [0,T]} \sum_{|x| \geq Bn } \big[  \mathbb{L}^{\gamma} H (\tfrac{x + q_n^r}{n}) \big]^{2} = \limsup_{B \rightarrow \infty} \int_{|u| \geq B} \big[  \mathbb{L}^{\gamma} H (u) \big]^{2} \dd u =0,  \\
& \limsup_{B \rightarrow \infty} \limsup_{n \rightarrow \infty} \frac{1}{n} \sup_{r \in [0,T]} \sum_{|x| \geq Bn } \big[ \Theta(n)  \mathbb{L}_n^{\gamma,s} H (\tfrac{x + q_n^r}{n})  \big]^{2} =0.
\end{align*}
The double limit in the first line, resp. second line of the last display holds, from the fact that $\mathbb{L}^{\gamma} H \in L^2(\mathbb{R})$, resp. the fact that \eqref{claimfracdiscfarorig} holds. Therefore, in order to obtain \eqref{limL2normlapfrac}, it is enough to prove that
\begin{align*}
\lim_{n \rightarrow \infty} \sup_{u \in \mathbb{R}} \big| \Theta(n)  \mathbb{L}_n^{\gamma} H (\tfrac{u}{n}) - \mathbb{L}^{\gamma,s} H (\tfrac{u}{n}) \big| = 0.
\end{align*}
If $0 < \gamma < 2$, resp. $\gamma \geq 2$, the last limit holds by applying arguments analogous to the ones in the proof of (A.4) and (A.11) in \cite{CGJ2}, resp. (A.6) and (A.7) in \cite{symflucped}. This ends the proof for \eqref{limL2normlapfrac}.

Now we observe that
\begin{align*}
& \limsup_{B \rightarrow \infty} \limsup_{n \rightarrow \infty} \frac{1}{n}  \sup_{r \in [0,T]} \sum_{|x| \geq Bn } \big[ \widehat{\mathbb{L}}_{ \lambda}^{ \gamma} H (\tfrac{x + q_n^r}{n}) \big]^{2} = \limsup_{B \rightarrow \infty} \int_{|u| \geq B} \big[  \widehat{\mathbb{L}}_{ \lambda}^{ \gamma} H (u) \big]^{2} \dd u =0,  \\
& \limsup_{B \rightarrow \infty} \limsup_{n \rightarrow \infty} \frac{1}{n}  \sup_{r \in [0,T]} \sum_{|x| \geq Bn } \big[ \widehat{\mathbb{L}}_{n, \lambda}^{ \gamma} H (\tfrac{x + q_n^r}{n})  \big]^{2} =0.
\end{align*}
The double limit in the first line, resp. second line of the last display holds, from the fact that $\widehat{\mathbb{L}}_{ \lambda}^{ \gamma} H \in L^2(\mathbb{R})$, resp. the fact that \eqref{claimfraclatdiscfarorig} holds. Therefore, in order to obtain \eqref{limL2normlapfraclat}, it is enough to prove that
\begin{align*}
\lim_{n \rightarrow \infty} \sup_{u \in \mathbb{R}} \big| \widehat{\mathbb{L}}_{n, \lambda}^{ \gamma} H (\tfrac{u}{n}) - \widehat{\mathbb{L}}_{ \lambda}^{ \gamma} H (\tfrac{u}{n}) \big| = 0.
\end{align*}
If $0 < \gamma < 2$, resp. $\gamma \geq 2$, the last limit holds by applying arguments analogous to the ones in Appendix B of \cite{jarafluc}, resp. by combining \eqref{defLhatlamb}, \eqref{suplapfraclatdiscLinfty} and \eqref{timescale}. This ends the proof for \eqref{limL2normlapfraclat}.

\subsection{Proof of Lemma \ref{lemaproxeps}} \label{useest11}

In order to construct $f_{\varepsilon}$ and $g_{\varepsilon}$, define the (classical) mollifier $\phi \in C_c^{\infty}(\mathbb{R})$ by
\begin{align*}
	\forall u \in \mathbb{R}, \quad \phi(u):= C_{\phi} \exp \Bigg(  \frac{1}{u^2 - 1} \Bigg) \mathbbm{1}_{ \{ (-1, \; 1) \} }(u),
\end{align*} 
where $C_{\phi}$ is some positive constant such that $\int_{\mathbb{R}} \phi(u) \; \dd u = 1$. Next, for every $\delta >0$, define $\phi_{\delta} \in C_c^{\infty}(\mathbb{R})$ by
\begin{align*}
	\forall u \in \mathbb{R}, \quad \phi_{\delta}(u):=\frac{1}{\delta}  \phi \Bigg( \frac{u}{\delta} \Bigg).
\end{align*}
Now for every $\varepsilon \in (0, \; 1/2)$, define $F_{\varepsilon} \in C_c^{\infty}(\mathbb{R})$ by
\begin{align*}
	F_{\varepsilon}(u):= \frac{1}{\varepsilon} \int_{\mathbb{R}} \mathbbm{1}_{ (- a, \; a)  } (u-v) \phi_{ b }(v) \; \dd v = \frac{1}{\varepsilon} \int_{-b}^b \mathbbm{1}_{ (- a, \; a)  } (u-v) \phi_{ b }(v) \; \dd v \leq \frac{1}{\varepsilon},
\end{align*}
where $a:=(2 \varepsilon - \varepsilon^{4})/4$ and $b:=\varepsilon^{4}/4$.
In particular, if $|u| \leq a - b = \varepsilon(1 - \varepsilon^{3} )/2$, then $|v| < b \Rightarrow |u-v| < a$. Moreover, if $|u| \geq a + b = \varepsilon/2$, then $|v| < b \Rightarrow |u-v| \geq a$. This leads do
\begin{align*}
	\forall u \in [- (a-b), \; a-b ], \quad & F_{\varepsilon}(u)= \frac{1}{\varepsilon} \int_{-b}^b \mathbbm{1}_{ (- a, \; a)  } (u-v) \phi_{ b }(v) \; \dd v = \frac{1}{\varepsilon} \int_{-b}^b  \phi_{ b }(v) \; \dd v = \frac{1}{\varepsilon}, \\
	\forall u \notin (- \varepsilon/2, \; \varepsilon/2 ), \quad & F_{\varepsilon}(u)= \frac{1}{\varepsilon} \int_{-b}^b \mathbbm{1}_{ (- a, \; a)  } (u-v) \phi_{ b }(v) \; \dd v  = 0.
\end{align*}
Therefore, we conclude that
\begin{align*}
	\| F_{\varepsilon} - \mathbbm{1}_{ [- \varepsilon/2, \; \varepsilon/2]  } \|^{2}_{2,n} =& \frac{2}{n} \sum_{ x= (b-a)n }^{\varepsilon n /2}  \big[ F_{\varepsilon}( \tfrac{x}{n})  - \varepsilon^{-1} \big]^{2} \leq \frac{2}{n} \sum_{ x= (b-a)n }^{\varepsilon n /2}  \big[  \varepsilon^{-1} \big]^{2} \leq 2 \varepsilon^{2}.
\end{align*}
In order to finish the proof, it is enough to define $f_{\varepsilon}, g_{\varepsilon}: \mathbb{R} \mapsto \mathbb{R}$ by
\begin{align*}
	\forall u \in \mathbb{R}, \quad f_{\varepsilon}(u):=F_{\varepsilon} ( u + \tfrac{\varepsilon}{2} ), \quad g_{\varepsilon}(u):=F_{\varepsilon} ( u - \tfrac{\varepsilon}{2} ).
\end{align*} 

\section{Proof of Lemmas \ref{lemobe} and \ref{lemtbe}} \label{appc}

\subsection{Proof of Lemma \ref{lemobe}} \label{prooflemobe}

Combining \eqref{defpsixL} with the inequality $(u+v)^2 \leq 2 (u^2+v^2)$, the expectation in the statement of Lemma \ref{lemobe} is bounded from above by
\begin{align}
&2 \mathbb{E}_{\nu_{\rho}} \Bigg[ \Bigg\{ \widehat{b}_n \int_0^t \dd r \sum_{x}   \nabla H ( \tfrac{ \lfloor x - r v_n \rfloor }{n}  )  \bar{\xi} _x^{\alpha}( \eta_r^n )  \big[   \bar{\xi} _{x+1}^{\beta}( \eta_r^n ) - \overrightarrow{\xi}_x^{\beta, \ell_n}(\eta_r^n)  \big]  \Bigg\}^{2} \Bigg] \label{varlemobe1} \\
+ &2 \mathbb{E}_{\nu_{\rho}} \Bigg[ \Bigg\{ \widehat{b}_n \int_0^t \dd r \sum_{x}   \nabla H ( \tfrac{ \lfloor x - r v_n \rfloor }{n}  )    \big[ 
\bar{\xi} _x^{\alpha}( \eta_r^n )  - \overleftarrow{\xi}_x^{\alpha, L}(\eta_r^n) \big] \overrightarrow{\xi}_x^{\beta, L}(\eta_r^n)  \Bigg\}^{2} \Bigg]. \label{varlemobe2}
\end{align}
Here we detail the reasoning only for \eqref{varlemobe1}, but we observe that the strategy for treating \eqref{varlemobe2} is analogous. From Lemma 4.3 in \cite{CLO}, the expectation in \eqref{varlemobe1} is bounded from above by
\begin{equation} \label{suplemobe}
\begin{split}
C \int_0^{t} \dd r \sup_{g \in L^2(\nu_{\rho} ) } \Bigg\{ &2 \widehat{b}_n \sum_{x} \nabla H ( \tfrac{ \lfloor x - r v_n \rfloor }{n}  ) \int_{\Omega} \bar{\xi}_x^{\alpha}( \eta ) \big[ \bar{\xi}_{x+1}^{\beta}( \eta ) - \overrightarrow{\xi}_x^{\beta, \ell_n}(\eta_r^n)  \big] g(\eta) d \nu_{\rho} -   \langle g, - \mcb S^n g \rangle_{\nu_{\rho}} \Bigg\},
\end{split}
\end{equation}
where $\mcb S^n$ is the symmetric part of the generator $\mcb L^n$, given by \eqref{genABC}. Doing similar steps as above in the proof of Lemma \ref{lem31jarabur}, i.e. a telescopic argument and a change of variables, we get 
\begin{align*}
 \int_{\Omega} \bar{\xi}_x^{\alpha}( \eta ) \big[ \bar{\xi}_{x+1}^{\beta}( \eta ) - \overrightarrow{\xi}_x^{\beta, \ell_n}(\eta_r^n)  \big] g(\eta) \; \dd \nu_{\rho} = \frac{1}{\ell_n} \sum_{y=2}^{\ell_n} \sum_{z=x+1}^{x+y-1}  \int_{\Omega} \bar{\xi}_x^{\alpha}( \eta )  \bar{\xi}_{z+1}^{\beta}( \eta )  \big[g(\eta^{z,z+1}) - g(\eta) \big] \dd \nu_{\rho} (\eta).
\end{align*}
From Young's inequality, we obtain
\begin{align*}
& \Bigg| \int_{\Omega} \bar{\xi}_x^{\alpha}( \eta ) \big[ \bar{\xi}_{x+1}^{\beta}( \eta ) - \overrightarrow{\xi}_x^{\beta, \ell_n}(\eta_r^n)  \big] g(\eta) \; \dd \nu_{\rho} \Bigg|\leq \sum_{z=x+1}^{x+\ell_n-1} \Bigg\{ \frac{\beta(n,r,x)}{2} + \frac{I_{z,z+1}(g,  \nu_{\rho} )}{2 \beta(n,r,x)}  \Bigg\},
\end{align*} 
where $I_{z,z+1}(g,  \nu_{\rho} )$ is given by \eqref{defIywg} and $\beta(n,r,x)$ will be chosen later. 
Therefore, the first term inside the supremum in \eqref{suplemobe} is bounded from above by
\begin{align*}
2 |\widehat{b}_n| \sum_{x} \sum_{z=x+1}^{x+\ell_n-1} \big| \nabla H ( \tfrac{ \lfloor x - r v_n \rfloor }{n}  ) \big|  \Bigg\{ \frac{\beta(n,r,x)}{2} + \frac{I_{x+1,x+z}(g,  \nu_{\rho} )}{2 \beta(n,r,x)}  \Bigg\}.
\end{align*}
By choosing $\beta(n,r,x)= C_n \big| \nabla H ( \tfrac{ \lfloor x - r v_n \rfloor }{n}  ) \big|$ in the last display (for some $C_n$, depending only on $n$, to be chosen later), we conclude that the last display can be rewritten as
\begin{align*}
& |\widehat{b}_n| C_n \sum_{x} \sum_{z=x+1}^{x+\ell_n-1}   \big| \nabla H ( \tfrac{ \lfloor x - r v_n \rfloor }{n}  ) \big|^2    +  \frac{|\widehat{b}_n|}{  C_n } \sum_{x} \sum_{z=x+1}^{x+\ell_n-1} I_{z,z+1}(g,  \nu_{\rho} ) \\
=& |\widehat{b}_n| (\ell_n-1) C_n \sum_{w}    \big| \nabla H ( \tfrac{ w }{n}  ) \big|^2    +  \frac{|\widehat{b}_n|}{  C_n } \sum_{x} \sum_{z=x+1}^{x+\ell_n-1} I_{z,z+1}(g,  \nu_{\rho} ). 
\end{align*}
Recall the positive constant $C_1$ given in \eqref{lbnddirg}. Combining the last line with \eqref{lbnddirg}, we get that the expression inside the supremum in \eqref{suplemobe} is bounded from above by
\begin{align*}
  n |\widehat{b}_n| \ell_n C_n   \| \nabla H \|^{2}_{2,n} +  \sum_{w}  I_{w, w+1}(g,  \nu_{\rho} ) \Bigg\{ \frac{| \widehat{b}_n| \ell_n}{  C_n }   - \Theta (n) C_1 \Bigg\} = \frac{n}{\Theta (n) C_1} \big( | \widehat{b}_n| \ell_n \big)^2  \| \nabla H \|^{2}_{2,n},
\end{align*}
where the equality comes from the choice $C_n =  | \widehat{b}_n| \ell_n \big[ \Theta (n) C_1 \big]^{-1}$. By combining the last display with \eqref{suplemobe}, the proof ends.

\subsection{Proof of Lemma \ref{lemtbe}} \label{prooflemtbe}

In order to show Lemma \ref{lemtbe}, we will make use of Proposition \ref{prop48} below, which is analogous to Proposition 4.9 in \cite{jarafluc}. 
\begin{prop} \label{prop48}
	Let $(k_j)_{j \in \mathbb{Z} } \subset \mathbb{Z}$ be such $k_j - k_{j-1} >0$, for any $j \in \mathbb{Z}$. Moreover, for any $j \in \mathbb{Z}$ let $F_j: [0, T] \times \Omega \mapsto \mathbb{R}$ be such that, for every $s \in [0, \; T]$, $\supp \big( F_j(s, \cdot) \; \big) \subset \{k_{j-1}+1, \ldots ,  k_j\}$, and $\mathbb{E}_{\nu_{\rho}} \big[ F_j(s, \eta) \big]=0$. Therefore, there exists some universal constant $\kappa$ such that
	\begin{align*}
		\forall t \in [0, T], \quad \mathbb{E}_{\nu_{\rho}} \Bigg[ \sup_{s \in [0,t] } \Bigg(  \int_{0}^s \sum_{j} F_j(r, \; \eta_r^n) \; \dd r \Bigg)^2  \Bigg] \leq  \kappa \int_0^{t}  \sum_{j}   \frac{(k_j - k_{j-1})^{\gamma}}{\Theta(n)} \mathbb{E}_{\nu_{\rho}} \big[ \big\{ F_j(s, \eta) \}^{2} \big] \dd s.
	\end{align*}
\end{prop} 
\begin{proof}
This is the analogue of \cite[Proposition~4.9]{jarafluc},
with the time-acceleration $\Theta(n)$ in place of $n^\alpha$.
Let $F(s,\eta):=\sum_j F_j(s,\eta)$. By the Kipnis--Varadhan inequality 
\cite[Proposition~4.7]{jarafluc} applied to the process with generator $\Theta(n) \mcb L_n$,
$$
\EE_{\nu_\rho}\Big[\sup_{0\le u\le t}\Big(\int_0^u F(r,\eta_r^n) \; \dd r\Big)^2\Big]
\;\le\; 14\int_0^t \|F(r,\cdot)\|_{-1,\Theta(n)}^2 \; \dd r,
$$
where $\|\cdot\|_{-1,\Theta(n)}$ denotes the $H_{-1}$-norm associated to $\Theta(n) \mcb L^n$
Using the block decomposition of $F$ and the spectral gap estimate
(see \cite[Proposition~A.1 and Corollary~A.3]{jarafluc}), one obtains the upper bound
$$
\|F(r,\cdot)\|_{-1,\Theta(n)}^2 \;\le\; C\sum_j \frac{(k_j-k_{j-1})^\gamma}{\Theta(n)}
\,\EE_{\nu_\rho}\big[F_j(r,\eta)^2\big],$$
with $C$ being an universal constant. Combining the two displays gives the claim.
\end{proof}
Now we are ready to prove Lemma \ref{lemtbe}. 
\begin{proof}[Proof for Lemma \ref{lemtbe}]
Observe that there exists $J \in \mathbb{N}$ such that $2^J \ell_n \leq L_n <  2^{J+1} \ell_n$. Next, for every $i \in \{0, 1, \ldots, J\}$, we define $B_{i} := 2^{i} \ell_n$; and $B_{J+1} :=  L_n$. In particular,
\begin{equation} \label{uboratB}
\forall j \in \{0,1, \ldots, J\}, \quad B_{i+1} \leq 2 B_{i}.
\end{equation}
With this notation and from Minkowski's inequality we can bound the  the expectation in the statement from above by 
\begin{align} 
\Bigg( \sum_{i=0}^J \mathbb{E}_{\nu_{\rho}} \Bigg[      \Bigg\{   \sum_{j=1}^{B_{i+1}} \widehat{b}_n \int_0^t \dd r  \sum_{x \in \mathbb{A}_{i,j}}   \nabla H ( \tfrac{ \lfloor x - r v_n \rfloor }{n}  )    \big[   \Psi_{x, B_{i}}^{\alpha,\beta}(\eta_r^n) - \Psi_{x, B_{i+1}}^{\alpha,\beta}(\eta_r^n)  \big]  \Bigg\}^2  \Bigg]^{1/2} \Bigg)^2, \label{explemtbe2}
\end{align}
where for any $i \in \{0,1,\ldots,J\}$ fixed, for every $j \in \{1, 2, \ldots, B_{i+1}\}$, $\mathbb{A}_{i,j} \subset \mathbb{Z}$ is defined by
\begin{align*}
\mathbb{A}_{i,j}:=\big\{ y \in \mathbb{Z}: \quad \exists q \in \mathbb{Z}:  y = B_{i+1} q + j-1 \big\}.
\end{align*}
From the Cauchy-Schwarz inequality,  \eqref{explemtbe2} is bounded from above by
\begin{align}
 & ( \widehat{b}_n)^2  \Bigg\{ \sum_{i=0}^J \sqrt{B_{i+1}} \Bigg(  \sum_{j=1}^{B_{i+1}} \mathbb{E}_{\nu_{\rho}} \Bigg[           \Bigg\{  \int_0^t \dd r  \sum_{x \in \mathbb{A}_{i,j}}   \nabla H ( \tfrac{ \lfloor x - r v_n \rfloor }{n}  )    \big[   \Psi_{x, B_{i}}^{\alpha,\beta}(\eta_r^n) - \Psi_{x, B_{i+1}}^{\alpha,\beta}(\eta_r^n)  \big]  \Bigg\}^2  \Bigg] \Bigg)^{1/2} \Bigg\}^2. \label{explemtbe3}
\end{align}
Now, fix $i \in \{0,1, \ldots, J\}$. For every $i \in \{1, \ldots, B_{i+1}\}$ and $k \in \mathbb{Z}$, define $F_{j,k}: [0, T] \times \Omega \mapsto \mathbb{R}$ by
\begin{align*}
\forall (r, \eta) \in [0, T] \times \Omega, \quad  F_{j,k}(r, \eta):= \nabla H ( \tfrac{ \lfloor B_{i+1} k + j-1 - r v_n \rfloor }{n}  )    \big[   \Psi_{B_{i+1} k + j-1, B_{i}}^{\alpha,\beta}(\eta) - \Psi_{B_{i+1} k + j-1, B_{i+1}}^{\alpha,\beta}(\eta)  \big].
\end{align*}
Therefore, the sum over $j$ in \eqref{explemtbe3} can be rewritten as
\begin{align*}
 & \sum_{j=1}^{B_{i+1}} \mathbb{E}_{\nu_{\rho}} \Bigg[ \Bigg\{  \int_0^t \dd r  \sum_{k } F_{j,k}(r, \eta_r^n)  \Bigg\}^{2} \Bigg] 
\leq   \kappa \int_0^{t}  \sum_{k}   \frac{(B_{i+1})^{\gamma}}{\Theta(n)} \mathbb{E}_{\nu_{\rho}} \big[ \big\{ F_{j,k}(r, \eta) \}^{2} \big] \dd r \\
=& \sum_{j=1}^{B_{i+1}} \kappa \int_0^{t} \sum_{j=1}^{B_{i+1}} \sum_{k}   \frac{(B_{i+1})^{\gamma}}{\Theta(n)} [ \nabla H ( \tfrac{ \lfloor B_{i+1} k + j-1 - r v_n \rfloor }{n}  ) ]^2 \mathbb{E}_{\nu_{\rho}} \big[ \big\{ \Psi_{B_{i+1} k + j-1, B_{i}}^{\alpha,\beta}(\eta) - \Psi_{B_{i+1} k + j-1, B_{i+1}}^{\alpha,\beta}(\eta)  \}^{2} \big] \dd r \\
\leq & 2 \kappa \chi(\rho_{\alpha}) \chi(\rho_{\beta})   \frac{(B_{i+1})^{\gamma}}{\Theta(n)} \int_0^{t} \sum_{j=1}^{B_{i+1}} \sum_{k}   [ \nabla H ( \tfrac{ \lfloor B_{i+1} k + j-1 - r v_n \rfloor }{n}  ) ]^2 \Bigg\{ \frac{1}{(B_{i})^{2}} + \frac{1}{(B_{i+1})^{2}} \Bigg\}   \dd r \\
\leq & 10 \kappa \chi(\rho_{\alpha}) \chi(\rho_{\beta}) t \frac{(B_{i+1})^{\gamma-2}}{\Theta(n)} n  \| \nabla H \|^2_{2,n},
\end{align*}
where the first, resp. second, resp. third inequality in the last display comes from Proposition \ref{prop48}, resp. a convex  inequality and \eqref{varpsi}, resp. \eqref{uboratB}. Therefore, the display in \eqref{explemtbe3} is bounded from above by
\begin{align*}
  & 10 \kappa \chi(\rho_{\alpha}) \chi(\rho_{\beta}) t ( \widehat{b}_n)^2  \frac{n  \| \nabla H \|^2_{2,n}}{\Theta(n)}  \Bigg\{ \sum_{i=1}^{J+1} \big(B_{i})^{(\gamma-1)/2}  \Bigg\}^2.
\end{align*}	
From the definition of $(B_{i})_{i \in \{0, 1, \ldots, J+1\}}$, we get that $B_{i} \leq 2^i \ell_n$ for every $i \in \{1, \ldots, J+1\}$. Moreover, $2^{J} \ell_n \leq L_n$. Thus, the proof ends by observing that
\begin{align*}
\Bigg\{ \sum_{i=1}^{J+1} \big(B_{i})^{(\gamma-1)/2}  \Bigg\}^2 \leq \Bigg\{ \sum_{i=1}^{J+1} \big(2^i \ell_n)^{(\gamma-1)/2}  \Bigg\}^2 \leq C(\gamma) (2^{J} \ell_n )^{\gamma-1} \leq C(\gamma) L_n^{\gamma-1}. 
\end{align*}    
\end{proof}

\section{Wavelets properties}\label{a:Wavelets}

In section~\ref{sec:Cross}, we summarised a few elements of wavelets theory, 
specifically concerning of scaling function $\phi$ and mother wavelet $\psi$. 
In the same notations introduced therein, 
the next lemma states some useful properties of the orthonormal basis~\eqref{e:basis} they induce. 

\begin{lem}\label{l:WaveBounds}
For $r>1$, let $\phi\in\cC^r$ be a scaling function and $\psi$ the associated mother wavelet. 
Then, there exists a constant $C=C(\phi,\psi)>0$ 
such that for any $p\in\NN$, $p<r$, all $z\in\Lambda_0$, $m\geq 1$, $x\in\Lambda_m$ and $v\in\R$, we have 
\begin{equ}[e:BoundsWavelets]
\|\partial_y^p\phi^z(\cdot-v)\|_{L^2(\R)}\leq C\,,\qquad\text{and}\qquad \|\partial_y^p\psi_m^x(\cdot-v)\|_{L^2(\R)}\leq C 2^{\frac{mp}{2}}\,.
\end{equ}
Let $h$ be a smooth and compactly supported function. Then, for $v\in\R$, 
$\langle h(\cdot-v), \phi_z\rangle_{L^2(\R)}$ and 
$\langle h(\cdot-v), \psi_m^x\rangle_{L^2(\R)}$ are 
non-zero for at most finitely many $z\in\Lambda_0$ and $O(2^m)$ $x\in\Lambda_m$, $m\geq 1$, 
where in both cases the exact number is independent of $v$. 
Furthermore, there exists a constant $C=C(h, \phi,\psi)>0$ such that 
for all $z\in\Lambda_0$, $m\geq 1$, $x\in\Lambda_m$ and $v\in\R$, and any $j\in\NN$, $j<r$, we have 
\begin{equ}[e:BoundSmooth]
\big|\langle h(\cdot-v), \phi^z\rangle_{L^2(\R)}\big|\leq C\,,\qquad \text{and}\qquad \big|\langle h(\cdot-v), \psi_m^x\rangle_{L^2(\R)}\big|\leq C 2^{m-mj}
\end{equ}
\end{lem}
\begin{proof}
The bound~\eqref{e:BoundsWavelets} is a simple consequence of~\eqref{e:scaling} and the compact support 
of both $\phi$ and $\psi$. Let us show it for $\psi$
\begin{equs}
\|\partial_y^p\psi_m^x(\cdot-v)\|^2_{L^2(\R)}&=2^{pm} 2^m \int_\R\big|(\partial_x^p\psi)(2^m(y-x-v))\big|^2\dd y=2^{pm}\int_{\R}\big|(\partial_x^p\psi)(y)\big|^2\dd y\\
&=2^{pm}\int_{\supp(\psi)}\big|(\partial_x^p\psi)(y)\big|^2\dd y\lesssim 2^{pm}\|\psi\|^2_{\cC^r}
\end{equs}
where the hidden constant in the last step is the volume of the support of $\psi$. 

Let $h$ be smooth and compactly supported. Then, since also 
$h(\cdot-v)$, $\phi$ and $\psi$ have compact support and the size of the support is independent of $v$, 
the first part is immediate. Concerning the bounds in~\eqref{e:BoundSmooth}, we focus on the second. 
Note that by Plancherel we have 
\begin{equs}
\langle h(\cdot-v), \psi_m^x\rangle_{L^2(\R)}&=\int_\R h(y-v)\psi_m^x(y)\dd y=\int_\R \hat h(k)\widehat{\psi_m^x}(k) e^{2\pi\iota kv}\dd k\\
&=2^{-\frac{m}{2}}\int_\R \hat h(k)\hat\psi(2^{-m}k)e^{2\pi\iota k(v+x)}\dd k\\
&=2^{\frac{m}{2}}\int_\R \hat h(2^m k)\hat\psi(k)e^{2\pi\iota 2^m k(v+x)}\dd k\,.
\end{equs}
Now, since $ h$ is smooth and compactly supported, its Fourier transform decays polynomially 
fast for $k$ large, which means that for any $j\in \NN$ (and we take $1\leq j\leq r$) there exists a constant $C'>0$ depending only 
on the size of the support and the $\cC^j$-norm of $ h$ such that for all $k\in\R$
$|\hat h(k)|\leq C' |k|^{-j}$. Therefore, 
\begin{equs}
\big|\langle  h(\cdot-v), \psi_m^x\rangle_{L^2(\R)}\big|\leq 2^m\int_\R |\hat h(2^m k)||\hat\psi(k)|\dd k\leq C' 2^{m}2^{-mj}\int_\R\frac{1}{|k|^j}|\hat\psi(k)|\dd k\,.
\end{equs}
Thus, we are left to prove that the integral at the right hand side is finite. 
To do so, we split the integral in two parts corresponding to $|k|>1$ and $|k|\leq 1$. The former 
can be easily controlled via Cauchy-Schwarz, i.e. 
\begin{equ}
\int_{|k|>1}\frac{1}{|k|^j}|\hat\psi(k)|\dd k\leq \Big(\int_1^\infty \frac{1}{k^{2j}}\dd k\Big)^{\frac12} \Big(\int_\R|\hat\psi(k)|^2\dd k\Big)^{\frac12}\leq \frac{1}{\sqrt{2j-1}}\|\psi\|_{L^2(\R)}
\end{equ}
and the right hand side is finite. For the other, by~\eqref{e:pol}, for any $\ell\leq r$, we have 
\begin{equs}[e:vanishingFT]
\partial^\ell_k\hat\psi(0)&=\partial^\ell_k\Big(\int_\R\psi(x)e^{-2\pi\iota k x}\dd x\Big)\Big|_{k=0}=\Big((-2\pi\iota)^\ell \int_\R x^\ell\psi(x)e^{-2\pi\iota k x}\dd x\Big)\Big|_{k=0}\\
&=-(2\pi\iota)^\ell \int_\R x^\ell\psi(x)\dd x=0\,.
\end{equs}
As a consequence, by Taylor's theorem, for every $k\in[-1,1]$, we get 
\begin{equ}
|\hat\psi(k)|=\Big|\hat\psi(k)-\sum_{\ell\leq j-1}\partial_k^\ell\hat\psi(0) \frac{k^\ell}{\ell!}\Big|\leq \Big(\sup_{k\in[-1,1]}|\partial_{k}^j\hat\psi(k)|\Big)\frac{|k|^j}{j!}\lesssim |k|^j
\end{equ}
where the last step can be deduced with an argument similar to that in~\eqref{e:vanishingFT} together with the fact that 
$\psi$ is compactly supported. The previous bound immediately implies the result. 
\end{proof}

\section{Uniqueness of the solution in the critical case}\label{app:Unique}

In the present appendix, we sketch how to adapt the results in~\cite{GPP} to show uniqueness 
of the martingale problem in Definition~\ref{defspde} in the case of $\gamma =
3 / 2$ and thus conclude the proof of Proposition~\ref{uniqlaw}. 
The main difference between the two settings is the presence of the operator 
$\widehat{\mathbb{L}}^{\gamma}_{\lambda}$ but we will argue that 
the latter is well-behaved and that the methods in~\cite{GPP} still apply. 
Let us begin with a minor lemma that derives the Fourier representation 
of both $\mathbb{L}^{\gamma}$ and $\widehat{\mathbb{L}}^{\gamma}_{\lambda}$. 
Throughout, we write $\mathbb{D}$ for the symbol of the Fourier multiplier of the gradient, i.e. 
such that, in particular,  $\Delta f = | 2 \pi \mathbb{D} |^2 f$.

\begin{lem}\label{lemFrac}
Assume $\gamma\in(1,2)$. Then, there exist constants $\hat{C} = \hat{C}
  (\gamma)$ and $C=C(\gamma) \in \mathbb{R}$ such that for every $G \in \mathcal{S}(\mathbb{R})$, we have 
\begin{align}
\widehat{\mathbb{L}}^{\gamma}_{\lambda} G &= \iota \lambda \hat{C} \mathrm{sgn}
     (\mathbb{D}) | \mathbb{D} |^{\gamma} G\,, \label{LhatFourier}\\
\mathbb{L}^{\gamma} G &= - C | \mathbb{D} |^{\gamma} G\,, \label{LFourier}
\end{align}
  in the sense of Fourier multipliers.
\end{lem}
\begin{proof}
We will prove the statement only for $\widehat{\mathbb{L}}^{\gamma}_{\lambda}$, the other being completely analogous. 
Recall that for $G \in \mathcal{S}(\mathbb{R})$, by Fourier inversion, we know that 
$G (x) = \int_{\mathbb{R}} \hat{G} (k) e_k (x) \text{d} k$, 
where $\hat G$ is the Fourier transform of $G$ and $e_k (u) = e^{2 \pi i k u}$. 
Now, let $\varepsilon>0$ and denote by $\widehat{\mathbb{L}}^{\gamma, \varepsilon}_{\lambda}$ 
the operator given by 
\[ \widehat{\mathbb{L}}^{\gamma, \varepsilon}_{\lambda} G (u) = \lambda K
     \frac{c_+ - c_-}{3} \int_{\mathbb{R}} 1_{| v | > \varepsilon}
     \frac{\text{sgn} v}{| v |^{1 + \gamma}} [G (u + v) - G (u) - v \nabla G
     (u)] \text{d} v. \]
Note that, since $\gamma>1$ and the Fourier transform of the integrand at the 
r.h.s. is integrable, we can write  
\begin{align*}
    \widehat{\mathbb{L}}^{\gamma, \varepsilon}_{\lambda} G (u) & =  \lambda K
    \frac{c_+ - c_-}{3} \int_{\mathbb{R}} \hat{G} (k) \int_{\mathbb{R}} 1_{| v
    | > \varepsilon} \frac{\text{sgn} v}{| v |^{1 + \gamma}} [e_k (u) e_k (v)
    - e_k (u) - 2 \pi i k v e_k (u)] \text{d} v \; \text{d} k \\
    &=\lambda K\frac{c_+ - c_-}{3} \int_{\mathbb{R}} \hat{G} (k) e_k(u) \int_{\mathbb{R}} 1_{| v
    | > \varepsilon} \frac{\text{sgn} v}{| v |^{1 + \gamma}} [e_k (v)
    - 1 - 2 \pi i k v ] \text{d} v \; \text{d} k\\
    &=\iota \lambda K\frac{c_+ - c_-}{3} \int_{\mathbb{R}} \hat{G} (k) e_k(u) (2 \pi | k |)^{\gamma} \text{sgn} (k) \int_{\mathbb{R}}
    \frac{1}{i} 1_{| v (2 \pi k)^{- 1} | > \varepsilon} \frac{\text{sgn}
    (v)}{| v |^{1 + \gamma}} [e^{i v} - 1 - i v] \text{d} v \; \text{d} k,
  \end{align*}
where in the last step we applied a simple change of variables. Consider now the integral 
in $v$. Its integrand is integrable uniformly in $\varepsilon$ ($e^{i v} - 1 - i v$ behaves like $v^2$ around $v = 0$ 
and grows at most as $v$ for $v \rightarrow \infty$, and $\gamma\in(1,2)$), so that, by Dominated Convergence, 
we can take the limit as $\varepsilon\to 0$. Upon setting 
\[
\hat C(\gamma):=\int_{\mathbb{R}}\frac{\text{sgn}
    (v)}{| v |^{1 + \gamma}} [e^{i v} - 1 - i v] \text{d} v
\]
the statement follows. 
\end{proof}

\subsection{The generator and the associated martingale problem}

The bulk of the analysis in~\cite{GPP} consists in the construction of the generator of the Markov process 
determined by the solution of~\eqref{spdefbe}. For this, let us first introduce the function spaces 
which are required to properly define the operators of interest. 

\subsubsection{Functional analytic setup}

In the setting of {\cite[Sec. 1.1]{GPP}}, let $H = L^2 (\mathbb{R})$, $V =
H^{\gamma / 2} (\mathbb{R})$, $\mu$ be the law on $\mathcal{S}'
(\mathbb{R})$ of the mean-zero Gaussian field with covariance in~\eqref{covinifie} 
and $\mathcal{C}$ be the set of all
cylinder functions, i.e. $\mathcal{C}\ni F (\mathcal{Z}) = \Phi (\mathcal{Z} (f))$ for $\Phi$ 
a polynomial in $n$ variables and $f \in \mathcal{S} (\mathbb{R})^n$. 
For $F\in \mathcal{C}$, we define norms
\begin{align*}
\| F \|_{\mathcal{H}^1_0}^2 &:=\mathbb{E} \big[\langle D F, \mathbb{L}^{\gamma} D
   F \rangle_{L^2 (\mathbb{R})}\big] + \| F \|^2_{L^2 (\mu)}\\
\| F\|_{\mathcal{H}^{- 1}_0} &:= \sup_{G \in \mathcal{C}, \| G
   \|_{\mathcal{H}^1_0} = 1} \langle F, G \rangle_{L^2 (\mu)}, 
\end{align*}
where $D$ is the Malliavin derivative, $\mathbb{L}^{\gamma}$ the operator in~\eqref{defLlamb} 
and we define the spaces $\mathcal{H}^{\pm 1}_0$ as 
the completion of $\mathcal{C}$ under these norms.  A key point in \cite{GPP} (and previous works) is that one can make sense of the (formal) bilinear form associated to the non-linearity
\[
B(\mathcal{Z},\mathcal{Z}')=\kappa\nabla (\mathcal{Z} \mathcal{Z}'),
\]
as an element
\begin{align}\label{BinHminusone}
    \langle B(\cdot,\cdot),H\rangle_{L^2(\mathbb{R})}\in \mathcal{H}^{-1}_0,
\end{align}
for $H\in\mathcal{S}$. For $\alpha \in \{ - 1, 1 \}$, we will also need 
\[ \| F \|_{\mathcal{H}^{\pm 1}_{\alpha}} = \| (1 + \mathcal{N}) F
   \|_{\mathcal{H}^{\pm 1}_{\alpha}}, F \in \mathcal{C}, \]
where $\mathcal{N}$ is the so-called-number operator, i.e. $\mathcal{N} F = \delta D F$ 
for $\delta$ the Gaussian divergence w.r.t. $\mu$. At last, we define the spaces 
$\mathcal{H}^{\pm 1}_{\alpha}$ again via completion.
%
%
%
%
%
%

\subsubsection{Construction of the generator and uniqueness for the associated
martingale problem} Our goal now is to properly define the generator of~\eqref{spdefbe}. To be more precise, we 
introduce an operator $\mathcal{L}$ and show that it is the generator of a strongly continuous contraction 
semigroup on $L^2(\mu)$. Formally, we would like to define 
\begin{equation}\label{gen}
\mathcal{L}:= \mathcal{L}_0+\mathcal{A}_{\gamma, \lambda}+ \mathcal{G}
\end{equation}
where $\mathcal{L}_0+\mathcal{A}_{\gamma, \lambda}$ is the generator of $OU(\gamma,\lambda,D)$ in~\eqref{spde}, 
with $\mathcal{L}_0, \mathcal{A}_{\gamma, \lambda}$ respectively its symmetric and antisymmetric parts defined 
on $\mathcal{C}$ as 
\begin{align}
\mathcal{L}_{0} F (\mathcal{Z}) &:= \delta
  ({\mathbb{L}}^{\gamma} D F) (\mathcal{Z})\,,
  \label{defmathcalL0}\\
\mathcal{A}_{\gamma, \lambda} F (\mathcal{Z}) &:= \delta
  (\widehat{\mathbb{L}}^{\gamma}_{\lambda} D F) (\mathcal{Z})\,,
  \label{defmathcalA}
\end{align}
and $\mathcal{G}$ the generator of the drift in~\eqref{spdefbe} due to the non-linearity. The problem 
is that the latter is not a priori well-defined. In the next lemma, we summarise the construction of 
$\mathcal{G}$ and the main estimates $\mathcal{G}$ and $\mathcal{A}_{\gamma, \lambda}$ satisfy. For this purpose, 
we also introduce the approximation of $\mathcal{G}$ given by 
\begin{align}
\mathcal{G}^{\epsilon} F (\mathcal{Z}) &:= \int_{\mathbb{R}} \overrightarrow{\iota}_{\varepsilon}\ast B(\overrightarrow{\iota}_{\varepsilon}\ast\mathcal{Z},\overrightarrow{\iota}_{\varepsilon}\ast\mathcal{Z})(u) D_u F \;\text{d} u ,\qquad F\in\mathcal{C}.
  \label{defmathcalGeps}
\end{align}
for $B$ as in~\eqref{BinHminusone}. 

\begin{lem}\label{antisymm}
For any $\alpha \in \{ - 1, 0, 1 \}$, there exists a constant $C>0$ independent of $\varepsilon$ such that for all $F\in\mathcal{C}$, we have 
\begin{equation}\label{BoundAgamma}
\| \mathcal{A}_{\gamma, \lambda} F \|_{\mathcal{H}^{- 1}_{\alpha}}
     \lesssim \| F \|_{\mathcal{H}^1_{\alpha}}, \quad\| \mathcal{G}^\varepsilon F \|_{\mathcal{H}^{- 1}_{\alpha}}
     \lesssim \| F \|_{\mathcal{H}^1_{\alpha+1}}\,,
\end{equation}
and the limit 
$
\mathcal{G}F=\lim_\varepsilon \mathcal{G}^\varepsilon F,
$
exists in $\mathcal{H}^{- 1}_{0}$. Thus, $\mathcal{A}_{\gamma, \lambda}$ (resp. $\mathcal{G}$) can be extended to a bounded operator
  $\mathcal{A}_{\gamma, \lambda} \in L (\mathcal{H}^1_{\alpha}, \mathcal{H}^{-
  1}_{\alpha})$ (resp. $\mathcal{G}\in L (\mathcal{H}^1_{\alpha+1}, \mathcal{H}^{-
  1}_{\alpha})$) for any $ \alpha \in \{ - 1, 0, 1 \}$.
  Furthermore, the operators $\mathcal{A}_{\gamma, \lambda}$ and $\mathcal{G}$ are antisymmetric, i.e. $\langle G,
  \mathcal{T} F \rangle_{L^2 (\mu)} = - \langle
  \mathcal{T} G, F \rangle_{L^2 (\mu)}$, for $\mathcal{T}\in\{\mathcal{A}_{\gamma, \lambda},\mathcal{G}\}$ and
  \begin{equation}
    [\mathcal{A}_{\gamma, \lambda}, \mathcal{N}] = 0 , \label{commAN}
  \end{equation}
as well as
\begin{equation*}
    [\mathcal{G}, \mathcal{N}]  \in L (\mathcal{H}^1_{\alpha+1}, \mathcal{H}^{-
  1}_{\alpha}),
  \end{equation*}
  for any $ \alpha \in \{ - 1, 0, 1 \}$.
\end{lem}
\begin{proof}
The statements concerning $\mathcal{G}$ follow from the abstract setting of Lemma 2.14 and Lemma 2.16 in \cite{GPP} as  well as Example 4.4 in the same work. Thus, we only focus on $\mathcal{A}_{\gamma, \lambda}$. 
  Antisymmetry and (\ref{commAN}) follow immediately from (\ref{defmathcalA}). 
Concerning~\eqref{BoundAgamma}, we have
  \begin{align*}
\mathbb{E} [G \mathcal{A} F] &=\mathbb{E} \left[ \int_{\mathbb{R}} D_u G \,\widehat{\mathbb{L}}^{\gamma}_{\lambda}
    D_u F \text{d} u \right] \lesssim \mathbb{E} \left[ \int_{\mathbb{R}} | | \mathbb{D} |^{\gamma /
    2} D_u G |^2 \text{d} u \right]^{1 / 2} \left[ \int_{\mathbb{R}} | |
    \mathbb{D} |^{\gamma / 2} D_u F |^2 \text{d} u \right]^{1 / 2}  \\
    & \lesssim \| G \|_{\mathcal{H}^1_0} \| F \|_{\mathcal{H}^1_0}, & 
  \end{align*}
which implies~\eqref{BoundAgamma} holds with $\alpha=0$. The claim for $\alpha=\pm 1$ is a consequence of~\eqref{commAN}.
\end{proof}

Thanks to the previous lemma, we have shown that the operator $\mathcal{L}$ in~\eqref{gen} 
is indeed well-defined on cylinder functions and can be viewed as an element of 
$L (\mathcal{H}^1_{\alpha}, \mathcal{H}^{- 1}_{\alpha - 1}), \alpha \in
   \{ 0, 1 \}$. Let us define 
\begin{equation}\label{dom}
\mathcal{D}_{\max} (\mathcal{L}) = \{ F \in \mathcal{H}^1_0 : \mathcal{L} F
   \in L^2 (\mu) \}. 
   \end{equation}
In the following proposition, we show that $\mathcal{D}_{\max} (\mathcal{L}$ is indeed a good choice for the domain 
of $\mathcal{L}$. 

\begin{prop}\label{genres}
Let $\mathcal{L}$ be the operator defined according to~\eqref{gen} and $\mathcal{D}_{\max} (\mathcal{L})$ be as in~\eqref{dom}. 
Then, $(\mathcal{D}_{\max} (\mathcal{L}),
  \mathcal{L})$ generates a strongly continuous contraction semigroup on $L^2
  (\mu)$. Moreover, $(1 - \mathcal{L})^{- 1} \in L (L^2 (\mu))$ (uniquely)
  extends to an element $(1 - \mathcal{L})^{- 1} \in L (\mathcal{H}^1_0,
  \mathcal{H}^{- 1}_0)$ and it holds that $(1 - \mathcal{L}) \mathcal{C}
  \subset \mathcal{H}^{- 1}_0$ is dense.
\end{prop}
\begin{proof}
The statement is shown in~\cite[Prop. 3.1, Lem. 3.2]{GPP} in the case $\lambda=0$. 
That said, as shown in~\cite[Theorem
  3.12]{Graefner2025}, the proof therein only relies on the fact
  that
  \[ \mathcal{G} + \mathcal{A}_{\gamma, \lambda} \in L
     (\mathcal{H}^1_{\alpha}, \mathcal{H}^{- 1}_{\alpha - 1}), \quad \alpha \in \{
     0, 1 \}, \]
  is antisymmetric and satisfies the commutator estimate
  \[ [\mathcal{G} + \mathcal{A}_{\gamma, \lambda}, \mathcal{N}] \in L
     (\mathcal{H}^1_{\alpha}, \mathcal{H}^{- 1}_{\alpha - 1}), \quad \alpha \in \{
     0, 1 \} . \]
  But these properties are still fulfilled for $\lambda \neq 0$ due to Lemma
  \ref{antisymm}.
\end{proof}

We are now ready to introduce the (essentially classical) notion of martingale problem associated to a Markov generator.  

\begin{definition}\label{DefmartL} 
Let $\mathcal{Z} \in C (\mathbb{R}_+ ; \mathcal{S}'(\mathbb{R}))$ be a
  stochastic process and assume it is incompressible, i.e. $\sup_{t \leqslant
  T} | \mathbb{E} [F (\mathcal{Z}_t)] | \lesssim_T \| F \|_{L^2 (\mu)}$ for
  all $T > 0$ and $F \in \mathcal{C}$. We say that $\mathcal{Z}$ is a solution to the
  {\textit{martingale problem associated to $(\mathcal{L},\mathcal{D}_{\max} (\mathcal{L}))$ }} if, for all
  $F \in \mathcal{D}_{\max} (\mathcal{L})$, the process 
  \[ F (\mathcal{Z}_t) - F (\mathcal{Z}_0) - \int_0^t \mathcal{L} F
     (\mathcal{Z}_s) \; \dd s, \qquad t \geqslant 0, \]
  is a martingale w.r.t. the filtration generated by $\mathcal{Z}$.
\end{definition}

Thanks to Proposition~\ref{genres}, the martingale problem associated to associated to 
$(\mathcal{L},\mathcal{D}_{\max} (\mathcal{L}))$ admits a unique solution. Indeed, 
the key point is that, for any $F \in
\mathcal{D}_{\max} (\mathcal{L})$, $\mathcal{L} F \in L^2 (\mu)$ and thus the additive functional is
well-defined and controlled by the incompressibility assumption.

\begin{thm}\label{TheoremmartLunique}
Let $\mathcal{Z} \in C (\mathbb{R}_+ ;
  \mathcal{S}')$ be a solution to the martingale problem associated to
  $(\mathcal{L},\mathcal{D}_{\max} (\mathcal{L}))$. Then, the law of $\mathcal{Z}$ is uniquely determined by
  $\text{law} (\mathcal{Z}_0)$.
\end{thm}

\begin{proof}
  The proof is essentially classical and only uses the fact that by
  Proposition \ref{genres} the Cauchy problem $\partial_t F = \mathcal{L} F$
  can be solved in $C^1_T L^2 (\mu) \cap C_T \mathcal{D}_{\max}
  (\mathcal{L})$ given any $F_0 \in \mathcal{D}_{\max} (\mathcal{L})$, see that of~\cite[Thm. 4.8]{GubinelliPerkowski2020}.
\end{proof}

\subsection{Proof of Proposition \ref{uniqlaw} }

With Theorem \ref{TheoremmartLunique} at hand, the proof of Proposition~\ref{uniqlaw} 
boils down to showing that any process $\mathcal{Z}$ solving $SBE(\gamma,\lambda,\kappa,D)$ 
according to Definition~\ref{defspdefbe}, also solves the 
martingale problem associated to $(\mathcal{L},\mathcal{D}_{\max} (\mathcal{L}))$. 
To do so, it turns out to be convenient to first prove that any solution 
of $SBE(\gamma,\lambda,\kappa,D)$  is a solution of an intermediate martingale 
problem in~\cite[Def 2.17]{GPP}\footnote{In ~\cite[Def 2.17]{GPP} this martingale problem is referred to as 'energy solution' but we chose to use a different term in this work in order to avoid confusion.} (see Definition~\ref{defSPDEfberel} below for how it translates in our context), 
and then show that the latter solves the 
martingale problem associated to $(\mathcal{L},\mathcal{D}_{\max} (\mathcal{L}))$. 

\begin{definition}[Intermediate martingale problem]
  \label{defSPDEfberel}We call an $\mathcal{S}' (\mathbb{R})$-valued
  stochastic process $(\mathcal{Z}_t)_{t \geqslant 0}$ on some complete
  probability space $(\Omega, \mathcal{F}, \mathbb{P})$ a \textit{solution to the intermediate martingale
problem} for equation \eqref{spdefbe} if the following conditions are
  satisfied:
  \begin{enumerate}
    \item \label{proppathcont}Path regularity: For all $H \in \mathcal{S}$ the
    process $t \mapsto \mathcal{Z}_t (H)$ is continuous.
    
    \item \label{incompr}Incompressibility: $\sup_{t \leqslant T} | \mathbb{E}
    [F (\mathcal{Z}_t)] | \lesssim \| F \|_{L^2 (\mu)}$ for all $T > 0$ and $F
    \in \mathcal{C}$.
    
    \item \label{propEEnl}Non-linear energy estimate: For all $F \in
    \mathcal{C}$ and for all $T > 0$
    \begin{equation}
      \mathbb{E} \left[ \sup_{t \leqslant T} \left| \int_0^t F (\mathcal{Z}_s)
    \; \dd s \right| \right] \lesssim (T^{1 / 2} + T) \| F
      \|_{\mathcal{H}^{- 1}_0} . \label{EEcyl}
    \end{equation}
    In particular, there exists a unique continuous extension to
    $\mathcal{H}^{- 1}_0$ of the map
    \[ I : \mathcal{C} \rightarrow \bigcap_{T > 0} L^1 (\Omega, C ([0, T],
       \mathbb{R})), \qquad I (F)_t = \int_0^t F (\mathcal{Z}_s) \; \dd s . \]
    \item \label{pvarconv}Weak solution: For $H \in \mathcal{S}$ the process
    \[ M^H_t = \mathcal{Z}_t (H) - \mathcal{Z}_0 (H) - \int_0^t \mathcal{Z}_s
       (\mathbb{L}^{\gamma} H + \widehat{\mathbb{L}}^{\gamma}_{\lambda} H) \; \dd
       s - I_t (\langle B (\cdot, \cdot), H \rangle_{L^2 (\mathbb{R})})
       , \quad t \geqslant 0, \]
    where the last term is well-defined due to \eqref{BinHminusone} and $(\ref{propEEnl})$, is a martingale in the filtration generated by $\mathcal{Z}$ with
    quadratic variation $\langle M^H \rangle_t = 2 t \langle
    \mathbb{L}^{\gamma} H, H \rangle_{L^2 (\mathbb{R})}$. Moreover, there exist $p <
    2$ and for all $T > 0$ the following convergence holds in probability:
    \[ \lim_{N \rightarrow \infty} \left\| I (\langle B (\cdot, \cdot), H
       \rangle_{L^2 (\mathbb{R})}) - \int_0^{\cdot} \langle B
       (\vec{\iota}_{\varepsilon} \ast \mathcal{Z}_s,
       \vec{\iota}_{\varepsilon} \ast \mathcal{Z}_s), H \rangle_{L^2
       (\mathbb{R})} \; \dd s \right\|_{p \text{-var}} = 0, \]
    where $\| X \|_{p - \text{var}}^p = \sup \left\{ \sum_{k = 0}^{n - 1} |
    X_{t_{k + 1}} - X_{t_k} |^p : n \in \mathbb{N}_0, 0 = t_0 < \cdots < t_n =
    T \right\}$.
  \end{enumerate}

\end{definition}

\begin{lem}
  \label{eimpliesInM}
  Any solution $\mathcal{Z}$ to $SBE(\gamma,\lambda,\kappa,D)$ 
according to Definition~\ref{defspdefbe} is a
  solution of the intermediate martingale problem of Definition~\ref{defSPDEfberel}. 
\end{lem}

\begin{proof}
  If $\lambda=0$ then the statement was shown \cite[Lem. 2.19]{GPP} and a similar argument 
  applies in our case. To see this, we will check that a
  solution $\mathcal{Z}$ in the sense of Def \ref{defspdefbe} fulfills the
  conditions of Definition~\ref{defSPDEfberel}. 
  It is clear that properties~\eqref{proppathcont} and~\eqref{incompr} hold. 
  For~\eqref{propEEnl}, we first want to argue that $\mathcal{B}(H)$ in~\eqref{e:NonLin} 
  has zero-quadratic variation. Notice that the energy estimate~\eqref{eqdefEE} and 
  Fatou's Lemma imply 
  \begin{eqnarray*}
    \limsup_{\delta \rightarrow 0} \mathbb{E} \left[ \frac{1}{\delta} \int_0^t
    \mathcal{B}_{s, s + \delta}^2 (H) \right] & \lesssim & \limsup_{\delta
    \rightarrow 0} \mathbb{E} \left[ \frac{1}{\delta} \int_0^t
    (\mathcal{B}_{s, s + \delta}^{\varepsilon} (H))^2 \text{d} s \right]\\
    &  & + \limsup_{\delta \rightarrow 0} \frac{1}{\delta} \int_0^t
    \mathbb{E} [(\mathcal{B}_{s, s + \delta} (H) - \mathcal{B}_{s, s +
    \delta}^{\varepsilon} (H))^2] \text{d} s\\
    & \lesssim & \varepsilon^{\omega} t,
  \end{eqnarray*}
  so that $[\mathcal{B} (H)] = \lim_{\delta} \frac{1}{\delta} \int_0^{\cdot}
  \mathcal{B}_{s, s + \delta}^2 (H)$ exists in the ucp-topology with
  $[\mathcal{B} (H)] = 0$. Thus, $\mathcal{Z}$ is a
  {\textit{controlled process}} in the sense of 
  {\cite[Def. 2.1]{GP}} (since $\int_0^{\cdot} \mathcal{Z}_s
  (\widehat{\mathbb{L}}^{\gamma}_{\lambda} H) \text{d} s$ has vanishing quadratic
  variation) and therefore satisfies the non-linear energy estimate
  (\ref{EEcyl}) by Corollary 3.5 of that work. Furthermore, 
  the operator $I : \mathcal{H}^{- 1}_0 \rightarrow
  \bigcap_{T > 0} L^1 (\Omega, C ([0, T], \mathbb{R}))$ given as in Definition 
  \ref{defSPDEfberel} is well-defined. 
  Now, to check property~\eqref{pvarconv}, we first need to verify that the operator $B$ featuring therein 
  coincides with $\mathcal{B}$ in~\eqref{e:NonLin}. 
  Notice that, thanks to the map $I$ introduced above, we have 
  \[ \overline{\mathcal{B}}_t (H) := I_t (\langle B (\cdot, \cdot), H
     \rangle_{L^2 (\mathbb{R})}) =\lim_{\varepsilon}
     \int_0^t \int_{\mathbb{R}}
     \nabla H (u) (\vec{\iota}_{\varepsilon} \ast \mathcal{Z}_r)^2 (u) \; \text{d}
     u  \]
  where the limit is in $\bigcap_{T > 0} L^1 (\Omega, C ([0, T], \mathbb{R}))$
  and it holds due to {\cite[Lemma 2.16]{GPP}}, so that we are left to check that 
  $\overline{\mathcal{B}} (H) = \mathcal{B} (H)$ almost surely. But
  this true since $I$ is continuous w.r.t. $\mathcal{H}^{- 1}_0$ and $\langle
  B (\vec{\iota}_{\varepsilon} \ast \cdot, \vec{\iota}_{\varepsilon} \ast
  \cdot), H \rangle_{L^2 (\mathbb{R})} - \langle B
  (\vec{\iota}_{\varepsilon} \ast \cdot, \iota^{\leftarrow}_{\varepsilon}
  \ast \cdot), H \rangle_{L^2 (\mathbb{R})} \rightarrow 0$ in
  $\mathcal{H}^{- 1}_0$ by a polarization argument. 
  At last, we point out that the $p$-variation estimate is independent 
  of the term $\widehat{\mathbb{L}}^{\gamma}_{\lambda}$-term and 
  can thus be argued as in {\cite[Lem. 2.16]{GPP}}. 
\end{proof}

\begin{lem}
  \label{itof}Let $\mathcal{Z}$ be a solution to the intermediate martingale problem in Definition~\ref{defSPDEfberel}  
  and let $F \in\mathcal{C}$. Then,
  \[ F (\mathcal{Z}_t) - F (\mathcal{Z}_0) - I_t (\mathcal{L} F)_t, \qquad t
     \geqslant 0, \]
  is a continuous martingale in the filtration generated by $\mathcal{Z}$.
\end{lem}

\begin{proof}
  The proof proceeds as sketched in the  proof of  {\cite[Lemma 3.6]{GPP}}. Since the presence of 
  $\widehat{\mathbb{L}}^{\gamma}_{\lambda}$ does not alter the argument, we do not reproduce it here. 
\end{proof}

\begin{proof}[Proof of Proposition \ref{uniqlaw}]
  Since we know that any solution to $SBE(\gamma,\lambda,\kappa,D)$ 
according to Definition~\ref{defspdefbe} is a
  solution of the intermediate martingale problem of Definition~\ref{defSPDEfberel}, it suffices to show that 
  the latter satisfies the martingale problem associated to $(\mathcal{L},\mathcal{D}_{\max} (\mathcal{L}))$. 
  But this is an immediate consequence of properties~\eqref{incompr},~\eqref{EEcyl} and Lemma~\ref{itof}, 
  together with denseness of $(1- \mathcal{L}) \mathcal{C} \subset \mathcal{H}^{- 1}_0$ from Proposition
  \ref{genres}. Therefore, Theorem~\ref{TheoremmartLunique} concludes the proof. 
\end{proof}

\medskip

\noindent\textbf{Acknowledgements}
The authors would like to thank P. Gon\c{c}alves for elucidating discussions 
and a thorough review of the paper. 
P.C. was funded by the Deutsche Forschungsgemeinschaft (DFG, German Research Foundation) – CRC 1720 – 539309657. G.C. and L.G. gratefully acknowledge financial support via the UKRI FL fellowship “Large-scale
universal behaviour of Random Interfaces and Stochastic Operators” MR/W008246/1. A.O. was funded by the Region Pays de la Loire through the AAP \'Etoiles Montante ULIS 2023DRI00555.

\bibliographystyle{plain}
\bibliography{bibliografia}

\bigskip

\Addresses

\end{document}